\documentclass[11pt,reqno]{amsart}
\usepackage{simplewick}
\usepackage{times}
\usepackage{amsmath,amssymb,amsthm}
\usepackage{color}
\usepackage{hep}
\usepackage{mathrsfs}
\usepackage{leftidx}
\usepackage{graphicx}
\usepackage{float}
\usepackage{enumerate}
\usepackage{titletoc}
\usepackage{epstopdf}
\usepackage{ulem}
\usepackage{appendix}

\allowdisplaybreaks
 \def \no{\nonumber}

\def\R {\mathbb R}

\def\p{\partial}
\def\ve{\varepsilon}
\def\f{\frac}
\def\na{\nabla}
\def\la{\lambda}

\def\al{\alpha}
\def\t{\tilde}

\def\vp{\varphi}
\def\O{\Omega}
\def\th{\theta}
\def\g{\gamma}
\def\G{\Gamma}
\def\si{\sigma}
\def\dl{\delta}
\def\p{\partial}
\def\ve{\varepsilon}
\def\f{\frac}
\def\k{\kappa}

\def\na{\nabla}
\def\la{\lambda}
\def\al{\alpha}

\def\t{\tilde}
\def\o{\omega}
\def\O{\Omega}
\def\vp{\varphi}
\def\th{\theta}

\def\g{\gamma}
\def\G{\Gamma}

\def\si{\sigma}
\def\dl{\delta}

\def\ds{\displaystyle}
 \allowdisplaybreaks[4]

\begin{document}
 \footskip=0pt
 \footnotesep=2pt
\let\oldsection\section
\renewcommand\section{\setcounter{equation}{0}\oldsection}
\renewcommand\thesection{\arabic{section}}
\renewcommand\theequation{\thesection.\arabic{equation}}
\newtheorem{claim}{\noindent Claim}[section]
\newtheorem{theorem}{\noindent Theorem}[section]
\newtheorem{lemma}{\noindent Lemma}[section]
\newtheorem{proposition}{\noindent Proposition}[section]
\newtheorem{definition}{\noindent Definition}[section]
\newtheorem{remark}{\noindent Remark}[section]
\newtheorem{corollary}{\noindent Corollary}[section]
\newtheorem{example}{\noindent Example}[section]

\title{Global smooth solutions to 4D quasilinear wave equations with short pulse initial data}

\author[B.B. Ding]{Ding Bingbing} \email{bbding@njnu.edu.cn}
\address{School of Mathematical Sciences and Mathematical Institute, Nanjing Normal University, Nanjing, 210023, China.}

\author[Z.P. Xin]{Xin Zhouping} \email{zpxin@ims.cuhk.edu.hk}
\address{Institute of Mathematical Sciences, The Chinese University of Hong Kong, Shatin, NT, Hong Kong.}

\author[H.C. Yin]{Yin Huicheng} \email{huicheng@nju.edu.cn, 05407@njnu.edu.cn}
\address{School of Mathematical Sciences and Mathematical Institute, Nanjing Normal University, Nanjing, 210023, China.}
\footnote{Ding Bingbing (13851929236@163.com, bbding@njnu.edu.cn) and Yin Huicheng (huicheng@nju.edu.cn, 05407@njnu.edu.cn) were
supported by the NSFC (No.12331007, No.12071223, No.11971237). Xin Zhouping(zpxin@ims.cuhk.edu.hk )
is partially supported by the Zheng Ge Ru Foundation, Hong Kong RGC Earmarked Research Grants
CUHK-14301421, CUHK-14300819, CUHK-14302819, CUHK-14300917, Basic and Applied Basic Research Foundations of Guangdong
Province 2020131515310002, and the Key Project of National Nature Science Foundation of China (No.12131010).}
\date{}
\maketitle.
\maketitle

\begin{abstract}
In this paper, we establish the global existence of
smooth solutions to general 4D quasilinear wave equations
satisfying the first null condition with the short pulse initial data.
Although the global existence of small data solutions to 4D quasilinear wave equations holds true
without any requirement of null conditions, yet for short pulse data, in general,
it is sufficient and necessary to require the fulfillment of the first null
condition to have global smooth solutions. It is noted that short pulse data are extensions of a class of
spherically symmetric data, for
which the smallness restrictions are imposed on angular directions and along the outgoing directional  derivative $\p_t+\p_r$,
but the largeness is kept for the incoming directional derivative
$\p_t-\p_r$. We expect that here methods can be applied to study  the global smooth  solution or blowup  problem
with short pulse initial data for the general 2D and 3D quasilinear wave equations when the corresponding
null conditions hold or not. On the other hand, as some direct applications of our main results, one can show that for the
short pulse initial data, the smooth solutions to the 4D irrotational compressible Euler equations for Chaplygin gases,
4D nonlinear membrane equations and 4D relativistic membrane equations exist globally since their nonlinearities
satisfy the first null condition; while the smooth solutions to the 4D
irrotational compressible Euler equations for polytropic gases generally blow up in finite time since the
corresponding first null condition does not hold.
\end{abstract}

	
\begin{quote}{\bf Keywords:} Quasilinear wave equations, short pulse initial data, the first null condition,

\qquad \qquad inverse foliation density, Goursat problem, global existence
\end{quote}
\vskip 0.2 true cm
\begin{quote}
\end{quote}

\vskip 0.4 true cm
\tableofcontents

\section{Introduction}\label{in}

\subsection{Formulation of the problem and main results}\label{2}

Consider general $n-$D quasilinear wave equations of the form:
\begin{equation}\label{quasi-0}
\sum_{\al,\beta=0}^ng^{\al\beta}(\phi,\p\phi)\p_{\al\beta}^2\phi=0,
\end{equation}
where $n\ge 2$, $(x^0, x)=(t, x^1, \cdots, x^n)\in [1,\infty)\times\R^n$,
$\p=(\p_0,\p_1,\cdots,\p_n)=(\p_{x^0}, \p_{x^1}, \cdots, \p_{x^n})$,
$g^{\al\beta}(\phi,\p\phi) =g^{\beta\al}(\phi,\p\phi)$ are smooth
functions of their arguments.
In addition, without loss of generality, it is assumed that for small $(\phi, \p\phi)$,
\begin{equation}\label{g}
g^{\al\beta}(\phi, \p\phi)=m^{\al\beta}+g^{\al\beta,0}\phi+\ds\sum_{\gamma=0}^n\tilde g^{\al\beta,
\gamma}\p_\gamma\phi+h^{\al\beta}(\phi,\p\phi)
\end{equation}
with $m^{00}=-1$, $m^{ii}=1$ for $1\le i\le n$, $m^{\al\beta}=0$ for $\al\neq\beta$, $g^{\al\beta,0}$
and $\tilde g^{\al\beta,\gamma}$ being constants, and $h^{\al\beta}(0,0)=\na_{(\phi,\p\phi)} h^{\al\beta}(0,0)=0$.

If \eqref{quasi-0} is equipped with the small initial data
\begin{equation}\label{Y1-0}
\phi(1,x)=\dl \vp_0(x), \p_t\phi(1,x)=\dl \vp_1(x),
\end{equation}
where $\dl>0$ is sufficiently small,  $(\vp_0, \vp_1)(x)\in C_0^{\infty}(\mathbb R^n)$,
$\text{supp$\vp_0$, supp$\vp_1$}\subset B(0, M)$ and $M>0$ is a constant,
then there exist extensive literatures on the global existence of smooth solution to \eqref{quasi-0}-\eqref{Y1-0}.
Indeed, for $n\ge 5$, or  $n=4$ but $g^{\al\beta,0}=0$ for all $0\le\al,\beta\le 4$,
the global existence of smooth solution $\phi$ to  \eqref{quasi-0}-\eqref{Y1-0} is established
and meanwhile $|\p\phi|\le C\dl (1+t)^{-\f{n-1}{2}}$ is obtained in \cite{K-P} (also see Chapter 6 of \cite{H} or \cite{Li-Chen}).
The key points of the analyses are based on
the standard energy estimates for wave equations, the smallness of $\|\phi(t,x)\|_{H_x^{2[\f{n}{2}]+2}(\mathbb R^n)}$
and the following Klainerman-Sobolev inequality:
\begin{equation*}
|\p \phi(t,x)|\le \f{C}{(1+|t-r|)^{\f12}(1+t)^{\f{n-1}{2}}}\ds\sum_{|I|\le [\f{n}{2}]+2}\|Z^I\p \phi(t,x)\|_{L_x^2(\mathbb R^n)},
\end{equation*}
where $Z\in\{\p, x^i\p_j-x^j\p_i, x^i\p_t+t\p_i, 1\le i,j\le n, \ds\sum_{k=1}^nx^k\p_k\}$ and $r=
|x|=\sqrt{(x^1)^2+\cdot\cdot\cdot+(x^n)^2}$.

For $n=4$ and $g^{\al\beta,0}\not=0$ for some $(\al,\beta)$, the global existence of the solution $\phi$ with
$|\phi|\le C\dl (1+t)^{-\f{1}{2}}$ and $|\p\phi|\le C\dl (1+t)^{-\f{3}{2}}$ is  shown in \cite{Lin0},
where the smallness of $\|\phi(t,x)\|_{H_x^6(\mathbb R^4)}$
and the following Hardy-type inequality for $\text{supp$_x\phi(t,x)$}\subset\{x: |x|\le M+t-1\}$
\begin{equation}\label{HC-2}
\bigl\|\f{\phi(t,x)}{1+|t-r|}\bigr\|_{L_x^2(\mathbb R^4)}\le C\|\p \phi(t,x)\|_{L_x^2(\mathbb R^4)}
\end{equation}
are crucial.

When $n=3$,  it follows from
\cite{C1,K1,Lin1,A3,Lin2,Ding1}
that there exists a global smooth solution $\phi$
 with $|\phi|\le C\dl(1+t)^{-1}(1+|r-t|)$ and $|\p\phi|\le C\dl(1+t)^{-1}$ if the first null condition holds (namely, $\ds\sum_{\al,\beta,\gamma=0}^3\tilde g^{\al\beta,\gamma}\xi_\al\xi_\beta\xi_\gamma\equiv0$ for $\xi_0=-1$
and $(\xi_1,\xi_2,\xi_3)\in\mathbb S^2$); otherwise, if the first null condition fails, then there
exists a solution $\phi$ to \eqref{quasi-0}-\eqref{Y1-0}  which blows up in finite time,
one can see \cite{A2,DY,Ding2,J1,J2}.

When $n=2$ and the coefficients $g^{\al\beta}$
are independent of $\phi$ (i.e., $\ds g^{\al\beta}(\phi,\p\phi)=g^{\al\beta}(\p\phi)$),
and one writes that for small $\p\phi$,
\begin{equation}\label{H-01}
\ds g^{\al\beta}(\p\phi)=m^{\al\beta}+\sum_{\gamma=0}^{2}\tilde g^{\al\beta,\gamma}\p_\gamma\phi+\sum_{\gamma,\nu=0}^2h^{\al\beta,\gamma\nu}\p_\gamma\phi\p_\nu\phi+O(|\p\phi|^3),
\end{equation}
where $h^{\al\beta,\gamma\nu}$ are  constants, then it is shown in \cite{A} that the problem \eqref{quasi-0}
with \eqref{Y1-0}
has a global smooth solution $\phi$ satisfying $|\p\phi|\le C\dl(1+t)^{-1/2}$ if both the first and second null conditions hold (that is,
 $\ds\sum_{\al,\beta,\gamma=0}^2\tilde g^{\al\beta,\gamma}\xi_\al\xi_\beta\xi_\gamma\equiv0$ and $\ds\sum_{\al,\beta,\gamma,\nu=0}^2 h^{\al\beta,\gamma\nu}\xi_\al\xi_\beta$ $\xi_\gamma\xi_\nu\equiv0$ for $\xi_0=-1$ and $(\xi_1,\xi_2)\in\mathbb S^1$).
 Furthermore, if either the first or second null condition fails, there exists a
solution $\phi$ to  \eqref{quasi-0}
and \eqref{Y1-0} which blows up in finite time, one can see \cite{A1} and \cite{A2}.

Generally speaking, for the global existence or blowup of smooth
 small solutions to nonlinear wave equations, it is more difficult to study the 2D and 3D cases compared with
the 4D or higher multidimensional cases
 since the solutions to the corresponding linear wave equations admit slower time-decay rates in 2D and 3D cases.
 We now briefly review the proof procedures for the  $n-$D  ($n=2,3$) cases of \eqref{quasi-0} in the references mentioned above.
 Under the null conditions, the proofs of global small data solution $\phi$
are based on the smallness of $\p^{\al}\phi$ ($|\al|\le N_0$ with $N_0\in\mathbb N$ and $N_0$ being appropriately large),
suitably weighted energy estimates for wave equations,  Klainerman-Sobolev inequality,  Hardy-type inequality \eqref{HC-2}, and
the following important estimates:
\begin{equation*}
\begin{array}{l}
\ds |{\tilde g}^{\al\beta,\g}\p^2_{\al\beta}u\p_\g v|\le|Y\p u||\p v|+|\p^2u||Yv|,\\
\ds |h^{\al\beta,\mu\nu}\p^2_{\al\beta}u\p_\mu v\p_\nu w|\le|Y\p u||\p v||\p w|
+|\p^2u||Yv||\p w|+|\p^2u||\p v||Yw|,\\
\end{array}
\end{equation*}
where $Y$ stands for the  one of ``good derivatives"
$\{Y_i=\p_i+\omega_i\p_t: \omega_i=\f{x^i}{r}, i=1, ..., n\}$ since $Y\phi$ admits a better time-decay rate
than that of $\p\phi$. When the null conditions fail, in order to establish the blowup
of $\phi$ (on the blowup time, $\phi$ and $\p\phi$ still remain to be small
but $\p^2\phi$ blows up), near the light conic surface $\{t=r-M+1: t\ge 1\}$, in the domain
$P_{\dl_0}=\{(t,r): t+M-1-\dl_0\le r\le t+M-1, t\ge 1\}$ ($\dl_0>0$ is a suitably fixed constant) one sets
\begin{equation}\label{HC-3}
\begin{array}{l}
\phi(t,x)=\ds\f{\dl}
{r^{\f{n-1}{2}}}G(\si, \o, \tau),
\end{array}
\end{equation}
where $\si=r-t$, $\o=\f{x}{r}\in\mathbb S^{n-1}$, the slow time variable $\tau=\dl\sqrt{t}$ for $n=2$ and $\tau=\dl \ln t$ for $n=3$.
Substituting \eqref{HC-3} into \eqref{quasi-0}
yields such a generalized  $n-$D ($n=2,3$) Burgers equation for $\p_{\si}G$
\begin{equation}\label{HC-4}
\begin{array}{l}
2^{n-2}\p^2_{\si\tau}G+\big(\ds\sum_{\al,\beta=0}^n\tilde g^{\al\beta,0}\o_\al\o_\beta G+\ds\sum_{\al,\beta,\gamma=0}^n\tilde g^{\al\beta,\gamma}\o_\al\o_\beta\o_\gamma\p_{\si}G\big)\p^2_{\si}G\\
=\text{higher order error terms including $\na_{\si,\tau,\o}G, \p_\o^2G, \p_\tau^2G, \p_{\tau\o}^2G, \p_{\si\o}^2G, \p_{\si\tau}^2G$
of $\dl$}.
\end{array}
\end{equation}
Then the blowup of $\nabla^2_{\si,\tau}G$ with the lifespan $T_{\dl}$ can be established by the geometric blowup method
together with the Nash-Moser-H\"ormander iteration which is utilized to treat the free boundary $t=T_{\dl}$. Meanwhile,
in the interior domain
$Q_{\dl_0}=\{(t,r): r\le t+M-1-\dl_0, t\ge 1\}$ of the light cone $\{r\le t+M-1, t\ge 1\}$,
it is shown that $\phi$ exists smoothly before $T_{\dl}+1$.

On the other hand, if \eqref{quasi-0} is equipped with the general smooth initial data of the form
\begin{equation}\label{HC-0}
\phi(1,x)=\vp_0(x), \p_t\phi(1,x)=\vp_1(x),
\end{equation}
and the strict hyperbolicity of \eqref{quasi-0} is assumed
(in this case, \eqref{quasi-0} is demanded to have some special nonlinear structures), the smooth
large solution $\phi$ generally blows
up in finite time (see \cite{Sid}).

In general, in order to guarantee the strict hyperbolicity of \eqref{quasi-0},
the smallness of $(\phi,\p\phi)$ seems to be needed, otherwise, the local well-posedness of the problem \eqref{quasi-0}
with \eqref{HC-0}  is  not expected. This indicates that it is plausible to look for global smooth
solutions to \eqref{quasi-0} with initial data \eqref{HC-0} which are large in the sense that
$(\phi_0,\phi_1)(x)$ are small but $\p^2\phi(1,x)$ may be large. This motivates one to study
the global existence of smooth solutions to \eqref{quasi-0} with the
the following short pulse initial data introduced first by D.Christodoulou \cite{C3}:
\begin{align}\label{HC-00}
\phi|_{t=1}=\delta^{2-\varepsilon}\phi_0(\f{r-1}{\delta},\omega),\
\p_t\phi|_{t=1}=\delta^{1-\varepsilon}\phi_1(\f{r-1}{\delta},\omega),
\end{align}
where $\delta>0$ is a small constant, $\ve$ is some fixed constant satisfying $0<\ve<1$,
$(\phi_0,\phi_1)(s,\o)$ are smooth functions of their arguments. By the way,
only the case of $\varepsilon=\f12$ in \eqref{HC-00} is considered in \cite{C3}.

This leads to two natural questions:

{\bf Problem (A)} Can one show the  global existence of smooth solutions to
\eqref{quasi-0} with \eqref{HC-00} under the corresponding null conditions?

{\bf Problem (B)} Are those null conditions in (A) the sufficient and necessary conditions
for the global existence?

To answer Problems (A) and (B), especially for the more difficult cases of $n=2$ and $n=3$
in \eqref{quasi-0}, as the first step, we now investigate the
higher dimensional case of \eqref{quasi-0} with $n=4$ since
a solution $v$ to the 4D free wave equation $\Box v=0$ with
$(v, \p_tv)(0,x)=(v_0, v_1)(x)$ $\in C_0^{\infty}(\mathbb R^4)$ has a higher
time-decay rate of $|\p v|\le C(1+t)^{-\f{3}{2}}$, which may make it relatively easier to analyze the solutions.
Meanwhile, it is expected that
the ideas developed here may share lights on the general 2D and 3D cases of \eqref{quasi-0} with
\eqref{HC-00} when the corresponding null conditions hold.

Thus, in this paper, we study the 4D quasilinear wave equation
\begin{equation}\label{quasi}
\sum_{\al,\beta=0}^4g^{\al\beta}(\phi,\p\phi)\p_{\al\beta}^2\phi=0
\end{equation}
with the short pulse initial data
\begin{align}\label{Y1-1}
\phi|_{t=1}=\delta^{2-\varepsilon_0/3}\phi_0(\f{r-1}{\delta},\omega),\
\p_t\phi|_{t=1}=\delta^{1-\varepsilon_0/3}\phi_1(\f{r-1}{\delta},\omega),
\end{align}
where $0<\ve_0\leq\f{7}{144}$ is a fixed constant, $(\phi_0,\phi_1)(s,\o)$ are smooth functions defined in $\mathbb R\times \mathbb S^3$
with compact support in $(-1,0)$ for the variable $s$. In addition, the following outgoing constraints are imposed:

\begin{equation}\label{Y-00}
\ds (\p_t+\p_r)^k\O^l\p^q\phi|_{t=1}= O(\delta^{2-\varepsilon_0-|q|}),\ 0\leq k\leq 2,
\end{equation}
where $\O\in\{x^i\p_j-x^j\p_i:1\leq i<j\leq 4\}$ stands for one of the rotational derivatives on $\mathbb S^3$.
It will be shown that \eqref{Y-00} together with \eqref{Y1-1} can guarantee  the smallness of $(\phi, \p\phi)$ for all $t\ge 1$.

It follows from \eqref{Y1-1}-\eqref{Y-00} that
the short pulse data can be regarded as some suitable extensions of a class of ``large" symmetric data, for
which the smallness restrictions are imposed initially on the variations along angular directions and along the ``good" direction tangent
to outgoing light conic surface $\{t=r\}$, but the largeness is kept at least for the second order ``bad" directional derivatives
of $\p_t-\p_r$. Such short pulse data actually provide a powerful framework to study effectively the blowup or the global existence of
smooth solutions to the multi-dimensional hyperbolic systems or the second order quasilinear wave equations
by virtue of the corresponding knowledge from the 1D cases. We illustrate that the problems of global existence
of solutions,  \eqref{quasi-0} and \eqref{Y1-0}, \eqref{quasi} and \eqref{Y1-1}-\eqref{Y-00}, correspond  essentially
to the case of outgoing waves. Indeed, it is interesting to note that for \eqref{quasi-0} and \eqref{Y1-0}, the
small solution admits naturally a better time decay property along $\p_t+\p_r$ than along $\p_t-\p_r$ near the light conic surface $\{t=r-M+1: t\ge 1\}$
due to $(\p_t+\p_r)^k\phi=O(t^{-\f{n-1}{2}(1+k)})$
and $(\p_t-\p_r)^k\phi=O(t^{-\f{n-1}{2}})$  ($k\ge 1, n\ge 2$).
For the large data solution to \eqref{quasi} with \eqref{Y1-1}-\eqref{Y-00}, besides the higher order time-decay
rate  along $\p_t+\p_r$ near the light conic surface as in \eqref{quasi-0} with \eqref{Y1-0}, it can be
derived that $(\p_t+\p_r)^k\phi$ ($0\leq k\leq 2$) admits
a higher order smallness $O(\dl^{2-\ve_0})$, which may be propagated from the outgoing constraint condition \eqref{Y-00}.

The short pulse data chosen in \eqref{Y1-1} originate from the monumental work \cite{C3} which reveals the evolutionary formation
of trapped surfaces in the Einstein vacuum space-times, and subsequently was generalized
in \cite{K-R} by enlarging the admissible set of initial conditions.
For the Cauchy problem of the 3D quasilinear wave equation
$-(1+3G''(0)(\p_t u)^2)\p_t^2 u+\Delta u=0$ with the short pulse initial data
$(u,\p_t u)(-2,x)=\big(\dl^{\f32}u_0(\frac{r-2}{\dl},\o),\dl^{\f12}u_1(\frac{r-2}{\dl},\o)\big)$, where $G''(0)\not=0$ is a constant, $t\ge -2$,
$(u_0,u_1)(s,\o)\in C_0^{\infty}\left((0,1]\times\mathbb{S}^2\right)$,
it is shown in \cite{MY} that the smooth solution $u$
will blow up before $t=-1$ and further the shock is formed due to the compression of the incoming waves under some assumptions on $(u_0,u_1)$.
It should be noted that there exist some global existence results with the short pulse initial data for semilinear wave equations.
For examples, the authors in \cite{MPY} study the 3-D semilinear wave equation systems
$\Box\vp^I=\ds\sum_{\tiny\begin{array}{c}0\le\alpha,\beta\le 3\\ 1\le J,K\le N\end{array}}A_{JK}^{\alpha\beta,I}\p_{\alpha}\vp^J
\p_{\beta}\vp^K$ ($I=1, ..., N$)
with
$(\vp^I, \p_t\vp^I)|_{t=1}=(\dl^{\f12}\vp_0^I(\f{r-1}{\dl},\o), \dl^{-\f12}\vp_1^I(\f{r-1}{\dl},\o))$,
where $A_{JK}^{\alpha\beta,I}$ are constants, $(\vp_0^I,
\vp_1^I)(s,\o)\in C_0^{\infty}((-1,0)\times\mathbb{S}^2)$,
the quadratic nonlinear forms satisfy the corresponding first null condition.
Moreover,  it is further assumed that
\begin{equation}\label{HC-03}
(\p_t+\p_r)^k\slashed\nabla^p\p^q\vp^I|_{t=1}=O(\delta^{1/2-|q|}),\quad k\leq N_0,
\end{equation}
where $\slashed\nabla$ stands for the one of the derivatives on $\mathbb S^2$,
and $N_0\ge 40$ is a sufficiently large integer.
Also recently, by some similar ideas of \cite{MPY}, the authors in \cite{Wang} proved the global existence of the relativistic membrane equation
\begin{equation}\label{F6-0}
\p_t(\ds\frac{\p_t\vp}{\sqrt{1-(\p_t\vp)^{2}+|\nabla\vp|^2}})
-\ds\sum_{i=1}^n\partial_i(\frac{\p_i \vp}{\sqrt{1-(\p_t\vp)^2+|\nabla\vp|^2}})
=0\quad  (n=2,3)
\end{equation}
with the short pulse initial data satisfying
\begin{equation}\label{F7-0}
(\p_t+\p_r)^k\slashed\nabla^l\p^m\varphi|_{t=1}=O(\delta^{1-m}),\quad 1\le k\le N_0.
\end{equation}
Here it should be pointed out that the special structure of \eqref{F6-0} (including the divergence form and cubic nonlinearity),
the largeness of $N_0$ and the higher order smallness of $\delta$ in \eqref{F7-0} play a crucial role on
the direct energy method in \cite{Wang} (in this case, the optical function $t-r$ for the linear wave equation can be
utilized, which leads to some  suitably modified Lorentzian vector fields). However, this approach fails to apply to the general
case here since our optical function
satisfies the nonlinear eikonal equation of \eqref{quasi}, for which the geometry of corresponding characteristic
surface is crucial.

As explained in \cite{Ding6} and \cite{Lu1},
the constraint condition \eqref{Y-00} is generally over-determined for $(\phi_0,\phi_1)$.
Indeed, if the short pulse initial data \eqref{Y1-1} is given, then all the derivatives $\p^\al\phi|_{t=1}=O(\delta^{2-\varepsilon_0/3-|\al|})$ are
derived. This means  $(\p_t+\p_r)^k\O^l\p^q\phi|_{t=1}=O(\delta^{2-\varepsilon_0/3-k-|q|})$ for $0\leq k\leq 2$.
However, the orders of $(\p_t+\p_r)^k\O^l\p^q\phi|_{t=1}$ in \eqref{Y-00} are $O(\delta^{2-\varepsilon_0-|q|})$ rather than
$O(\delta^{2-\varepsilon_0/3-k-|q|})$.
Thus, the choices of $(\phi_0,\phi_1)$ are not arbitrary in general.
Thanks to the null condition in the quasilinear equation \eqref{quasi},
\eqref{Y-00} can be fulfilled by the suitable choice of $(\phi_0,\phi_1)$ through solving the local short pulse data
problem of \eqref{quasi}.

The first null condition for the equation \eqref{quasi} is
\begin{equation}\label{null}
\sum_{\al,\beta,\gamma=0}^4\tilde g^{\al\beta,\gamma}\xi_\al\xi_\beta\xi_\gamma\equiv 0\ \text{for $\xi_0=-1$ and
$(\xi_1, \xi_2, \xi_3, \xi_4)\in \mathbb S^3$}.
\end{equation}

In addition, without loss of generality and for simplicity, one can assume that
\begin{equation}\label{g00}
g^{00}(\phi,\p\phi)=-1.
\end{equation}

Our main result in the paper is stated as
\begin{theorem}\label{main}
Consider the Cauchy problem for the 4D quasilinear wave equation \eqref{quasi}
with the short pulse initial data \eqref{Y1-1} satisfying  \eqref{Y-00}. Assume that
\eqref{g00} holds and the first null condition \eqref{null} is satisfied. Then for $\dl>0$ suitably
small, then there exists a global smooth solution $\phi$ to the problem \eqref{quasi}-\eqref{Y1-1} such that
\begin{equation}\label{HC-F3}
\phi\in C^\infty([1,+\infty)\times\mathbb R^4)\quad\text{and} \  |\phi|\le C\delta^{(4-5\ve_0)/3}t^{-1/2},
\  |\p\phi|\le C\delta^{1-\ve_0}t^{-3/2}
\end{equation}
for all time $t\ge 1$, where $C>0$ is a uniform constant independent of $\dl$ and $\ve_0$.
\end{theorem}

\begin{remark}
Note that $\p^2\phi|_{t=1}=O(\dl^{-\f{\ve_0}{3}})$ is large for \eqref{Y1-1}. Then the global well-posedness
theory of \eqref{quasi} with \eqref{Y1-1}
is totally different from the usual small data solution problem of \eqref{quasi} with \eqref{Y1-0},  where  the smallness of $\|\phi(1,x)\|_{H_x^6(\mathbb R^4)}+\|\p_t\phi(1,x)\|_{H_x^5(\mathbb R^4)}$ is required for the latter (see Chapter 6-7 of \cite{H}).
\end{remark}

\begin{remark}
It is pointed out that the scope of $0<\ve_0\leq\f{7}{144}$ in Theorem \ref{main} is not optimal.
For the optimal scope of $\ve_0$ for the general second order quasilinear wave equations with the related null conditions
and short pulse initial data, one is referred to \cite{Ding6} and \cite{Lu1}.
\end{remark}

\begin{remark}\label{HC-003}
It follows from
the expression of the solution $v$ to the 4D linear wave equation $\Box v=0$ with $(v,\p_tv)(0,x)=(v_0, v_1)(x)$ that both $v$ and $\p v$ are generally large
when $\p_x^2v_0$
or $\p_xv_1$ are large. Note that $(\p^2\phi|_{t=1}, \p(\p_t\phi)|_{t=1})$
are large for the initial data \eqref{Y1-1} with sufficiently small $\dl>0$. Therefore,  the strict hyperbolicity of \eqref{quasi} is generally lost
since the smallness of $(\phi,\p\phi)$ may be violated. Thus,
it is necessary to impose some constraint conditions like \eqref{Y-00} so that the  smallness of $(\phi,\p\phi)$ can be kept
with the development of time (see \eqref{HC-F3} of Theorem \ref{main}).
\end{remark}

\begin{remark} For the 4D quasilinear wave equation \eqref{quasi}
with $(\phi(1,x), \p_t\phi(1,x))\in (H^s(\mathbb R^4),$ $H^{s-1}$ $(\mathbb R^4))$ and $s>\f{7}{2}$, the authors in \cite{Smith}
have established the well-posedness of local solution $\phi\in C([1, T],$ $H^s(\mathbb R^4))\cap C^1([1, T], H^{s-1}(\mathbb R^4))$
with some $T>1$
(see the equation (1.1) with (1.2) and Theorem 1.2 in \cite{Smith}).
If $s>4$, the local well-posedness of  \eqref{quasi} is well known (see \cite{H} or \cite{J2}).
But for the data $(\phi(1,x),$ $\p_t\phi(1,x))$
in \eqref{Y1-1}, $\|\phi(1,x)\|_{H^s}=O(\dl^{\f52-\f{\ve_0}{4}-s})\to\infty$ holds true
when $\dl\to 0+$ and $\f{7}2<s\le 4$.
Therefore, the local well-posedness  of smooth solution $\phi$ to \eqref{quasi} in any fixed interval $[1, T]$
with $T>1$ independent of $\dl$
does not follow \cite{Smith}. Thanks to the special structures
of the short pulse initial data \eqref{Y1-1} with \eqref{Y-00}, based on the first null condition \eqref{null},
both the local and global existence of $\phi$ will be established.
\end{remark}

\begin{remark}	
For second order multidimensional quasilinear wave equations
or multidimensional quasilinear  hyperbolic systems with the genuinely nonlinear structures, the formation of shocks has been studied (see \cite{CM}, \cite{0-Speck}, \cite{LS}, \cite{MY}, and \cite{J}-\cite{S2})
by the remarkable ideas in \cite{C2}. In the breaking-through work \cite{C2},
with the help of differential geometry, D. Christodoulou constructed the ``\textit{inverse foliation density}'' $\mu$ which
measures the compression of the outgoing characteristic surfaces, and showed the formation of shocks
for 3-D relativistic Euler equations
under the smallness assumptions on the initial data. In this process,
D. Christodoulou proved that $\mu$ is positive away from the shock
but approaches $0$ near the blowup curve ($\mu\rightarrow 0+$ corresponds to the intersection of
outgoing characteristic surfaces).
Motivated by the crucial differential geometric methods in \cite{C2,C3,MY,J}, in the present paper, on the contrary, we will show $\mu\ge C>0$ for all time $t\ge 1$
and derive the global energy estimates near the outermost outgoing conic surface $\{t=r\}$ for problem  \eqref{quasi}-\eqref{Y1-1}
with \eqref{Y-00}.
This leads to that  near $\{t=r\}$, the characteristic conic surfaces never intersect
and the solution of \eqref{quasi}-\eqref{Y1-1} exists globally.
\end{remark}

\begin{remark}\label{HC-04}
It is emphasized that if the  null condition \eqref{null} does not hold,
then the solution $\phi$ of \eqref{quasi} may blow up in finite time whether or not  \eqref{Y-00} is fulfilled.
In fact,
for the 4D quasilinear wave equation $(1+\p_t\phi)\p_t^2\phi-\Delta \phi=0$
with the initial data $(\phi|_{t=1}, \p_t\phi|_{t=1})=(\dl^{2-\varepsilon_0/3}\phi_0(\f{r-1}{\dl},\o), \dl^{1-\varepsilon_0/3}\phi_1(\f{r-1}{\dl},\o))$ and $0<\ve_0<3$
(obviously, the corresponding null condition is not fulfilled), by the completely analogous arguments as in \cite{MY},
\cite{Ding6} and \cite{Lu1},
one can prove that the solution $\phi$ blows up before the time $1+O(\dl^{1-\varepsilon_0/3})$ and further the shock is formed under
the assumption of $\p_s\phi_1(s_0,\o_0)\ge 2$ at some point $(s_0,\o_0)\in (0,1]\times\mathbb S^3$.
This indicates that although the global small data solution of \eqref{quasi} exists for generic smooth nonlinearities, yet
for the short pulse initial data \eqref{Y1-1}, the first null condition is not only a sufficient but generally a necessary condition
for the global existence.
\end{remark}

\begin{remark}
For the case of $n\ge 5$ in \eqref{quasi-0}, when the analogous short pulse initial data \eqref{Y1-1} are given
and the constraint condition \eqref{Y-00} is replaced by
\begin{align}\label{HC-F2}
(\p_t+\p_r)^k\O^\kappa\p^q\phi|_{t=1}= O(\delta^{2-\varepsilon_0-|q|})\quad \text{for $0\leq k\leq N_0=[\f{n}{2}]+1$}
\end{align}
with $\O\in\{x^i\p_j-x^j\p_i:1\leq i<j\leq n\}$, then under the first null condition and
suitable scope of $\ve_0$,  \eqref{quasi-0} will have a global smooth solution $\phi$ as in Theorem \ref{main}.
However, with the increasing of $n$ and $N_0$, it is more difficult to find $(\phi_0,\phi_1)$ such that \eqref{HC-F2}
is fulfilled since \eqref{HC-F2} is over-determined for the general form of equation \eqref{quasi-0}.
The choice of $N_0=[\f{n}{2}]+1$ comes essentially from the Sobolev imbedding theorem on the sphere $S^{n-1}$
(see Lemma 11.1 of Section 11 below).
\end{remark}

\begin{remark}
In \cite{MPY} or \cite{Wang}, due to the structure
of the semilinear wave equation systems or the divergence structure of the specific relativistic
membrane equation,  suitable $(\vp_0, \vp_1)$ can be chosen such that \eqref{HC-03} holds for sufficiently large $N_0$.
Note that the largeness of $N_0$ plays an important role in the analysis of \cite{MPY} and \cite{Wang}, which leads actually
to  solve the global small data solution problem  rather than the
large value solution problem in the interior of the light cone. On the other hand, near the light conic surface,
the global smooth solution $\vp$ is obtained through the energy method and by the special
structure of related nonlinear equations (see \cite{MPY} and \cite{Wang}). However, it seems that the
method in \cite{MPY, Wang} fails to the case
for the general form of \eqref{quasi} here and the restricted number $k$ in \eqref{Y-00}.
\end{remark}

\begin{remark}
	It should be noted that there have been a few studies on the local or global existence of solutions
to some special semilinear wave equations and hyperbolic systems in multi-dimensions with general short pulse initial
data by matched asymptotic analysis methods, see \cite{Alt-1, Alt-2, C-R, Hunter, Majda}.
For examples, for the semilinear symmetric hyperbolic systems $\p_tu+\ds\sum_{j=1}^nA_j(t,x)\p_{x_j}u=G(t,x,u)$
with the smooth pulse initial data
$u_0(0,x)=u_0(x,\f{\vp(x)}{\dl})$ ($\p_x\vp(x)\not=0$)
or the semilinear wave equations $\square w+|\p_tw|^{p-1}\p_tw=0$ with the smooth and spherically symmetric  pulse initial data
$(w,\p_tw)(0,x)=(\dl^{\nu+1}w_0(r,\f{r-r_0}{\dl}), \dl^{\nu}w_1(r,\f{r-r_0}{\dl}))$ ($p>1, \nu>0$),
the authors investigated the local existence of $u$ in \cite{Alt-1,Alt-2} and
the local/global existence of $w$ in \cite{C-R} under various assumptions
of $p$ and $\nu$.
In addition, for the quasilinear multidimensional hyperbolic systems, the  formal asymptotic correctors have been discussed
in \cite{Hunter} and \cite{Majda}.
\end{remark}

Next we give some direct applications for Theorem \ref{main} and Remark \ref{HC-04}. First, consider
the 4D compressible isentropic Euler equations
\begin{equation}\label{Euler}
\left\{
\begin{aligned}
&\p_t\rho+div (\rho v)=0,\\
&\p_t(\rho v)+div (\rho v \otimes v)+\nabla p=0,\\
\end{aligned}
\right.
\end{equation}
which express the conservations of the mass and momentum, respectively. Where $(t,x)\in [1,+\infty)\times\mathbb{R}^4$, $\nabla=(\p_1,\p_2, \p_3,\p_4)=(\p_{x^1},\p_{x^2}, \p_{x^3},\p_{x^4})$, $v=(v_1,v_2,v_3,v_4)$, $\rho$, $p$ stand for the velocity, density, pressure
respectively. For the Chaplygin gases, the equation of state (see \cite{CF}) is
\begin{equation}\label{pr}
p(\rho)=P_0-\frac{B}{\rho}
\end{equation}
with $P_0$ and $B$ being positive constants.

Suppose that \eqref{Euler} admits the irrotational smooth initial data
$\big(\rho,v)(1,x)=(\bar\rho+\rho_0(x),v_1^0(x),v_2^0(x),$ $v_3^0(x),v_4^0(x)\big)$,
where $\bar\rho+\rho_0(x)>0$, and $\bar\rho$ is
a positive constant which can be normalized so that the sound speed $c(\bar\rho)=\sqrt{p'(\bar\rho)}=1$. Then by the irrotationality
of \eqref{Euler}, one can
introduce a potential function $\phi$ such that $v=\nabla\phi$, and the Bernoulli's law, $\p_t\phi+\frac{1}{2}|\nabla\phi|^2+h(\rho)=0$,
holds with the enthalpy $h(\rho)$ satisfying $h'(\rho)=\f{c^2(\rho)}\rho$ and $h(\bar\rho)=0$. Therefore,
for smooth irrotational flows, \eqref{Euler} is equivalent to that for the  Chaplygin  gases,
\begin{equation}\label{1.9-01}
\begin{split}
\p_t^2\phi-(1+2\p_t\phi)\triangle\phi+2\sum_{k=1}^4\p_k\phi\p_{tk}\phi
+\sum_{i,j=1}^4\p_i\phi\p_j\phi\p_{ij}^2\phi-|\nabla\phi|^2\triangle\phi=0,
\end{split}
\end{equation}
which
satisfies the first null condition in \eqref{null}.
Therefore, for the short pulse initial data \eqref{Y1-1} with \eqref{Y-00}, Theorem \ref{main} applies, that is,
\begin{corollary}
	For the 4D quasilinear equation \eqref{1.9-01} equipped with \eqref{Y1-1} and \eqref{Y-00},
when $\dl>0$ is small and $0<\ve_0\leq\f{7}{144}$, it has a global smooth solution $\phi$ for $t\geq 1$.
\end{corollary}

In addition, the 4D  relativistic membrane equation and 4D nonlinear membrane equation are, respectively,
\begin{equation}\label{HCCFF-1}
\p_t\Big(\ds\f{\p_t\phi}{\sqrt{1-(\p_t\phi)^2+|\nabla\phi|^2}}\Big)
-\ds\sum_{i=1}^4\p_i\Big(\ds\f{\p_i\phi}{\sqrt{1-(\p_t\phi)^2+|\nabla\phi|^2}}\Big)=0
\end{equation}
and
\begin{equation}\label{HCCFF-2}
\p_t^2\phi
-\ds\sum_{i=1}^4\p_i\Big(\ds\f{\p_i\phi}{\sqrt{1+|\nabla\phi|^2}}\Big)=0.
\end{equation}
Obviously, both \eqref{HCCFF-1} and \eqref{HCCFF-2} fulfill the first null condition \eqref{null} since they don't
contain the quadratic nonlinearity. Hence, Theorem \ref{main} implies
\begin{corollary}
	For the 4D equations \eqref{HCCFF-1} and \eqref{HCCFF-2} equipped with \eqref{Y1-1} and \eqref{Y-00}, when $\dl>0$ is small and $0<\ve_0\leq\f{7}{144}$, they have global smooth solutions for $t\geq 1$.
\end{corollary}

\subsection{Sketch for the proof of Theorem \ref{main}}
Since the proof of Theorem \ref{main} is rather lengthy, for the convenience of readers, we will outline the main
steps and ideas of the analysis in this subsection. Set $C_0=\{(t,x):t\geq 1+2\delta,t=r\}$,
the outermost outgoing conic surface; $A_{2\delta}=\{(t,x):t\geq 1+2\delta,0\leq t-r\leq 2\delta\}$, a domain
containing $C_0$; and $B_{2\delta}=\{(t,x):t\geq 1+\delta,t-r\geq 2\delta\}$, a conic domain inside the light cone
$\{(t,x):t\geq 1+2\delta, r\le t\}$ with the
lateral boundary given by $\tilde C_{2\delta}=\{(t,x):t\geq 1+2\delta,t-r=2\delta\}$. The proof of Theorem \ref{main}
consists of three parts: the local existence of solution $\phi$ for $t\in[1,1+2\delta]$, the global existence of
the solution in $A_{2\delta}$, and the global existence in $B_{2\delta}$.

\subsubsection{Local existence of the solution $\phi$ for $t\in [1, 1+2\delta]$}

By the energy method and the first null condition \eqref{null}, we can obtain the local existence of smooth solution
$\phi$ to \eqref{quasi}-\eqref{Y-00} satisfying $|\p^q\O^l\phi(t,x)|\lesssim
\delta^{2-|q|-\varepsilon_0}$ for $t\in [1, 1+2\delta]$. More importantly, the following basic
estimates for $\phi(1+2\delta,x)$ can be derived:
\begin{align}
&|L^a\p^q\O^\kappa\phi(1+2\delta,x)|\lesssim\delta^{2-|q|-\varepsilon_0}\quad \text{for $r\in[1-2\delta, 1+2\delta]$},\label{local1}\\
&|\underline L^a\p^q\O^\kappa\phi(1+2\delta,x)|\lesssim\delta^{3-|q|-\varepsilon_0}\quad\text{for $r\in [1-3\delta, 1+\delta]$},\label{local2}
\end{align}
where $a\in \mathbb N_0$ can be chosen as large as needed,
$L=\p_t+\p_r$ and $\underline L=\p_t-\p_r$. Note that \eqref{local1} and \eqref{local2}
present the better smallness property of $\phi$ along certain directional derivatives in different space domains
than that on the whole hypersurface $\{(t,x): t=1+2\delta\}$
(i.e., $|L^a\p^q\O^\kappa\phi(1+2\delta,x)|\lesssim\delta^{2-a-|q|-\varepsilon_0}$ and $|\underline L^a\p^q\O^\kappa\phi(1+2\delta,x)|\lesssim\delta^{3-a-|q|-\varepsilon_0}$).
Based on the crucial estimates \eqref{local1}-\eqref{local2} together with the first null condition and a series of
involved analyses,
one can estimate $|L^k\p^q\phi|\lesssim\delta^{2-|q|-\varepsilon_0}t^{-3/2-k}$ on $\tilde C_{2\delta}=\{(t,x):t\geq 1+2\delta,t-r=2\delta\}$ eventually, which is one of the key points
for proving the global existence of smooth solution $\phi$ in $B_{2\dl}$ in the future.

\subsubsection{Global existence of the solution $\phi$ in $A_{2\dl}$}\label{6-2}

Motivated strongly by the strategy of D. Christodoulou in \cite{C2},
we start to construct the solution $\phi$ of \eqref{quasi} in $A_{2\dl}$ by taking somewhat more general computations
from the Lorentzian geometric framework which is suitable for treating the 4D general quasilinear wave equation \eqref{quasi}.
Although the main result in \cite{C2} deals with the finite time blowup of smooth solutions
to the 3-D compressible Euler equations, which seems to be different from our goal to establish the global smooth solution to
\eqref{quasi} with \eqref{Y1-1} and \eqref{Y-00}, yet
we can still make full use of the idea in \cite{C2} to study whether the outgoing characteristic surfaces collapse
since the intersection of characteristic surfaces leads to the finite time formation of shock waves (also see \cite{CM,MY,J}),
while the expansion of characteristic surfaces generally implies the global existence of smooth solutions to
the hyperbolic systems as indicated in \cite{CF,Xin,Xu-1,Xu-4}.
In this process, the key ingredients here are to derive both the precise order of smallness for $\dl$
and the suitable time-decay rates for
the solution $\phi$ and its derivatives in $A_{2\dl}$.

Following \cite{C2} or \cite{J}, for a given smooth solution $\phi$ to \eqref{quasi},
one can define an {\it ``optical function"} $u$ for \eqref{quasi} as
\begin{equation}\label{H0-3}
\left\{
\begin{aligned}
&\ds\sum_{\al,\beta=0}^4g^{\al\beta}(\phi,\p
\phi)\p_\al{u}\p_\beta{u}=0,\\
&u(1+2\delta,x)=1+2\delta-r
\end{aligned}
\right.
\end{equation}
and subsequently the inverse foliation density
\begin{equation}\label{H0-4}
\mu=
-\ds(\sum_{\al=0}^4g^{\al0}\p_\al u)^{-1},
\end{equation}
where $g^{\al0}=g^{\al0}(\phi,\p\phi)$. Set the metric $g=\ds\sum_{\al,\beta=0}^4g_{\al\beta}dx^\al dx^\beta$
with $(g_{\al\beta})$ being the inverse matrix of $(g^{\al\beta})$ and let
\begin{equation}\label{H0-5}
\mathring L=-\mu \ds\sum_{\al,\beta=0}^4g^{\al\beta}\p_\al u\p_\beta.
\end{equation}


Extend the local coordinates $(\th^1,\th^2,\th^3)$ on the standard sphere $\mathbb S^3$ by
\begin{equation}\label{H0-6}
\left\{
\begin{aligned}
&\mathring L\vartheta^A=0,\\
&\vartheta^A|_{t=1+2\dl}=\th^A,
\end{aligned}
\right.
\end{equation}
where and below $A=1,2,3$. As in \cite{J} or \cite{MY}, one can perform the change of coordinates:
$(t, x^1, x^2, x^3,$ $x^4)\longrightarrow (s, u, \vartheta^1, \vartheta^2, \vartheta^3)$ in $A_{2\dl}$ with
\begin{align}\label{H0-7}
&s=t,\quad
u=u(t,x),\quad
\vartheta^A=\vartheta^A(t,x).
\end{align}
Then $X_A:=\f{\p}{\p\vartheta^A}$ are the tangent vectors on the curved
sphere $S_{s,u}:=\{(s',u',\vartheta^1,\vartheta^2,\vartheta^3):s'=s,u'=u\}$.

Under the bootstrap assumptions on the suitable smallness and time-decay rates of $(\phi,\p\phi)$ (more details
see $(\star)$ in Section \ref{BA}) and with the help of the first null condition \eqref{null},
$\mu$ satisfies $\mathring L\mu=O(\delta^{1-\varepsilon_0}s^{-5/2}\mu+\delta^{1-2\varepsilon_0}s^{-3/2})$.
By $\mu(1+2\delta,x)\sim 1$,  $\mathring L\sim\p_t+\p_r$  and through the integration along integral
curves of $\mathring L$, $\mu\sim 1$ can
be derived for small $\dl>0$. The positivity of $\mu$ indicates that the outgoing characteristic conic surfaces never
intersect as long as the smooth solution $\phi$
exists. Set $\varphi=(\varphi_0, \varphi_1, \varphi_2, \varphi_3, \varphi_4):=(\p_t\phi, \p_1\phi, \p_2\phi, \p_3\phi, \p_4\phi)$.
Then a quasilinear wave system for $(\phi,\varphi)$ is derived from \eqref{quasi}
\begin{equation}\label{Y-1}
\begin{split}
&\mu\Box_g\phi=G(\phi,\varphi,\p\varphi),\\
&\mu\Box_{g}\varphi_\gamma=F_\gamma(\phi,\varphi, \p\varphi),\quad\gamma=0,1,2,3,4,
\end{split}
\end{equation}
where $\Box_g=\f{1}{\sqrt{|\det g|}}\p_\al(\sqrt{|\det g|}g^{\al\beta}\p_\beta)$, $G$ and $F_\gamma$ are smooth functions
of their arguments.
To study the nonlinear system \eqref{Y-1}, as usual, we first carry out the energy estimate
for its linearized equation
\begin{equation}\label{Y-2}
\mu\Box_g\Psi=\Phi.
\end{equation}
If $\Psi=\Psi_k=Z^k\varphi$ is chosen in \eqref{Y-2},
where $Z$ is one of  some suitable first order vector fields, then it follows from
\eqref{Y-1} and direct computations of the commutator $[\mu\Box_{g}, Z^k]$
that the term $\Phi=\Phi_k$ in \eqref{Y-2} will contain the $(k+2)$-th order derivatives of $\varphi$.
Especially, $\slashed\nabla Z^{k-1}\chi$ and $\slashed\nabla^2Z^{k-1}\mu$ appear in $\Phi_k$,
where $\slashed\nabla$ represents the covariant derivative on $S_{s,u}$,
and $\chi$ stands for the second fundamental form of $S_{s,u}$ which is defined as
$\chi_{AB}=g(\mathscr D_{X_A} \mathring L, X_B)$
with $\mathscr D$ being the Levi-Civita connection of $g$.
Utilizing  the energy estimate of $\Psi_k$ derived from \eqref{Y-2},
one then can bound the $L^2$ norms of $(k+1)$ order derivatives of $\varphi$ by $L^2$ norm of
$\Phi_k$ and further by the $L^2$ norms of $(k+2)$ order derivatives of $\varphi$.
However, this obviously leads to the loss of derivatives for $\varphi$, and hence,
the estimate of $(k+1)$ order derivatives of $\varphi$
can not be obtained directly from the energy estimate for the equation $\mu\Box_g\Psi_k=\Phi_k$.
To overcome the difficulty, motivated by some ideas in \cite{C2}, one can use the following strategies:

${\bf \bullet}$ First, one can derive the transport equation for $\textrm{tr}\chi$ along the vector field $\mathring L$,
where $\textrm{tr}\chi$ is the trace of $\chi$ on $S_{s,u}$. Note that due to the dependence of $g^{\al\beta}$ on $\phi$ and $\vp$,
it is required to express $\mathring L\textrm{tr}\chi$ by making use of the equation of $\phi$ in \eqref{Y-1}, which leads to the
requirement on the estimate of $\phi$. Second, by
utilizing the Ricci identity, one can find the relation between $\slashed\nabla\chi$ and $\slashed\nabla\textrm{tr}\chi$
from \eqref{Y-1}. By these properties and some elaborate analysis,
the $L^2$ estimates for the highest order derivatives of $\slashed\nabla\chi$ are obtained by taking $L^2$
integration along integral curves of $\mathring L$. In this process, we have to
establish some new estimates for the Riemann curvature on $S_{s,u}$ of $\mathbb R^4$, which is achieved with the aid of the Christoffel symbols
on $S_{s,u}$ (see Lemma \ref{Gc}).

${\bf \bullet}$ According to the elliptic estimate on $\slashed\nabla^2 f$ from $\slashed\triangle f$ (see details in \eqref{ee}),
where $\slashed\triangle$ is the Laplacian operator on $S_{s,u}$, the $L^2$ norm of the highest order derivative
of $\slashed\nabla\mu$ can be controlled by the quantities containing the derivatives of $\slashed\triangle\mu$ and
$\slashed\nabla\mu$. Thanks to the transport equation for $\slashed\triangle\mu$,  the
estimates on the derivatives of $\slashed\triangle\mu$ and
$\slashed\nabla\mu$ may be derived as for $\slashed d{\textrm{tr}}\chi$.

${\bf \bullet}$ Based on the above estimates for the highest order derivatives of $\chi$ and $\mu$, all the terms in $\Phi_{k}$
can be treated by careful classifications. Together with the derived energy estimates on the solution $\Psi$ of \eqref{Y-2}
and some  Sobolev-type
inequalities on the sphere, both the better time-decay rate $s^{-5/2}$ in the right hand side of the
equation  $\mathring L\varphi=O(\delta^{1-\varepsilon_0} s^{-5/2})$ and the  better $L^\infty$ bound of $\varphi$ with order $\delta^{1-\varepsilon_0} s^{-3/2}$
are obtained. Then the basic bootstrap assumptions of $(\phi,\varphi)$ can be closed.

On the other hand, we also need to pay close attention to the treatments of the solution $\phi$ itself and $L\phi$
since the coefficients in \eqref{quasi} includes $\phi$ while it is extremely difficult to derive an effective governing equation for $L\phi$.
More importantly, it should be noted that both the
precise time-decay rate and smallness of $L\phi$ on $\t C_{2\dl}$, such as $|L^k\phi(t,x)|\lesssim\delta^{2-\varepsilon_0}t^{-3/2-k}$
for $(t,x)\in \t C_{2\dl}$, will be crucial in order to show the global solvability of \eqref{quasi} in
$B_{2\dl}$ since the related boundary integrals on $\t C_{2\dl}$ in the energy estimates require the
suitable smallness. We will use the following procedure to derive  the desired time-decay rates and smallness of $(\phi, L\phi)$:

${\bf \bullet}$  We first carry out the energy estimate on $\varphi_\gamma$ by the equation $\mu\Box_{g}\varphi_\gamma=F_\gamma(\phi,\varphi, \p\varphi)$. Then to avoid reconstructing the analogous and complex energy inequalities
for $\phi$ from the equation $\mu\Box_g\phi=G(\phi,\varphi,\p\varphi)$, we utilize the direct relationship
$\varphi_\gamma=\p_{x^\gamma}\phi$ to get the $L^2$ norm of $\phi$ by the energies of $\varphi_\gamma$.
Based on this and Sobolev's embedding theorem, the $L^\infty$ estimates of $\phi$ and its derivatives can be obtained.

${\bf \bullet}$ Note that all the estimates of $\phi$ in $A_{2\dl}$ are made in the modified frame $\{\mathring L, T, R_{ij}\}$,
where $\mathring L$, $T$, $R_{ij}$ are some suitable vectors approximating $L$, $-\p_r=\f{\underline L-L}{2}$,
$\O_{ij}=x^i\p_j-x^j\p_i$ respectively on the time $t=1+2\delta$. To get the $L^\infty$ estimates of $Z^\al\phi$
with $Z\in\{L,\underline L, \O_{ij}\}$
and further obtain the required time-decay rate of $L\phi$, it is required to express the vector fields $\{L,\underline L, \O_{ij}\}$
in terms of $\{\mathring L, T, R_{ij}\}$. A typical such expression is $L=c_{L}\mathring L+c_TT+c_{\O}^{ij}\O_{ij}$,
where $c_L$, $c_T$ and $c^{ij}_{\O}$ are
smooth functions.  One expects that $L\phi$ and its derivatives can admit the desired smallness of higher orders with respect to $\dl$
and faster time-decay rate. However, $T\phi$ and its derivatives have not good enough properties as needed,
it is hoped that the coefficient $c_T$ can make up for these deficiencies. Although the general form of \eqref{quasi}
makes it difficult to analyze the time-decay rates of the related coefficients $c_T$, $c^{ij}_\O$ and so on,
thanks to the first null condition \eqref{null}, one can find their transport equations along the directional derivative $\mathring L$.
Based on this,  the demanded estimates of  $c_L$, $c_T$ and $c^{ij}_{\O}$, then the estimate of $(\phi, L\phi)$
can be obtained.

When the desired time-decay rates and the smallness estimates of $(\phi,\vp)$ are established, the
global existence of the solution to  \eqref{quasi} with \eqref{Y1-1}-\eqref{Y-00}
in $A_{2\dl}$ is shown.

\subsubsection{Global existence of solution $\phi$ inside $B_{2\dl}$}

It follows from the local estimates of the solution $\phi$ to \eqref{quasi} with \eqref{Y1-1}-\eqref{Y-00} in
Theorem \ref{Th2.1} of Section 3 that for $r\in [1-3\delta, 1+\delta]$,
$\p^q\phi(1+2\delta,x)$ admit the high order smallness $O(\delta^{2-\varepsilon_0})$ (compared with $\delta^{2-|q|-\varepsilon_0}$).
In addition, it has been known  that the outgoing characteristic conic surfaces starting from the
domain $\{t=1+2\delta, 1-4\delta\le r\le 1+\delta\}$ in $A_{4\dl}$
are almost straight, and then the resulting domain contains the boundary $\tilde C_{2\delta}$
of $B_{2\dl}$. Noting \eqref{local1}-\eqref{local2} and by the estimates in $A_{4\dl}$ we have established
that on $\tilde C_{2\delta}$  the solution $\phi$ and its derivatives
satisfy $|L^k\phi|\lesssim\delta^{2-\varepsilon_0}t^{-3/2-k}$,
whose right hand side admits the higher order smallness $O(\delta^{2-\varepsilon_0})$ (see \eqref{LLOp} of Section 11).
Based on such ``good" smallness and time-decay rate $t^{-3/2}$ of $\phi$ on $\tilde C_{2\delta}$,
one can solve the global Goursat problem of \eqref{quasi} in $B_{2\dl}$.
To this end, we will carry out the global weighted energy estimates for the solution $\phi$
in $B_{2\dl}$ and make use of the Klainerman-Sobolev inequality to get the time-decay rates of $\p^{\al}\phi$.
However, since the classical Klainerman-Sobolev inequality holds generally on the whole space (see Proposition 6.5.1 in \cite{H}),
we need a modified and refined Klainerman-Sobolev inequality in $B_{2\dl}$.
Using this together with the bootstrap energy assumptions,
the global Goursat problem of \eqref{quasi} in $B_{2\delta}$ can be solved by direct energy
estimates although this is not the usual small data solution problem.

\subsection{Organization of the paper}

Our paper is organized as follows. In Section \ref{i}, the short pulse initial data \eqref{Y1-1}
satisfying  \eqref{Y-00} are shown to exist under the null condition \eqref{null}. In Section \ref{LE}, the local
existence of the solution $\phi$
for $1\leq t\leq 1+2\delta$ is established by the energy method. In Section \ref{p}, we give some preliminaries on
the related differential geometry.
In addition, the equation of $\mu$ is derived, and some basic calculations
are given for the covariant derivatives
of the null frame and the deformation tensors. The crucial bootstrap assumptions
$(\star)$ in $A_{2\dl}$ are given in the beginning of Section \ref{BA}, meanwhile,
under assumptions $(\star)$ we obtain some  estimates of several important quantities which
will be extensively used in subsequent sections. In Section \ref{ho}, under the assumptions $(\star)$,
the higher order smallness for the $L^{\infty}$ estimates of $(\phi, \varphi)$ in $A_{2\dl}$ are established.
In Section \ref{EE}, we carry out the energy estimates for the linearized equation $\mu\Box_g\Psi=\Phi$
and define some suitable higher order weighted energies and fluxes as in \cite{J}. In Section \ref{hoe},
under the assumptions $(\star)$,
the $L^2$ estimates for the higher order derivatives of several basic quantities are derived by
$\sqrt{E_{1,\leq k+2}(t, u)}$ and
$\sqrt{E_{2,\leq k+2}(t, u)}$,
where the energies $E_{1,\leq k+2}(t, u)$ and $E_{2,\leq k+2}(t, u)$ contain the highest order derivatives of
$\vp$. In Section \ref{L2chimu}, the
$L^2$ estimates for the highest order derivatives of tr{$\chi$} and $\slashed\triangle\mu$ are established. In Section \ref{ert},
we deal with the error terms appeared in the energy inequality of Section \ref{EE}.
Based on all the estimates in Section \ref{BA}-\ref{ert},  in Section \ref{YY} we complete the bootstrap argument by choosing the
assumptions $(\star)$ and further establish the global existence of solution $\phi$ to equation \eqref{quasi}
in $A_{2\dl}$. In addition, in the end of Section \ref{YY},
the delicate estimates of $\phi$ on $\t C_{2\delta}$ are derived. Finally,
we prove the global existence of solution $\phi$ inside $B_{2\dl}$ and complete the proof of Theorem \ref{main}
in Section \ref{inside}.

\subsection{Notations}

Through the whole paper, unless stated otherwise, Greek indices $\{\al, \beta, \gamma, \cdots\}$ correspond to the
time-space coordinates in $\{0, 1, 2, 3, 4\}$, Latin indices $\{i, j, a, b, \cdots\}$ correspond to the
spatial coordinates in $\{1, 2, 3, 4\}$, and the Einstein summation convention is used to sum over repeated upper
and lower indices. In addition, $f_1\lesssim f_2$ means that there exists a generic positive constant
$C$ independent of the parameter $\dl>0$ in \eqref{Y1-1} and the variables $(t,x)$ such that $f_1\leq Cf_2$.

Also, the coefficients $g^{\al\beta}(\phi, \p\phi)$ are denoted as $g^{\al\beta}$ for convenience,
and for the Lorentz metric $(g_{\al\beta})$, $(g^{\al\beta})$ represents its inverse spacetime metric.
Finally, the following notations will be used:
\begin{align*}
&\o^i:=\o_i=\f{x^i}{r},\quad i=1,2,3,4,\\
&\o=(\o_1, \o_2, \o_3, \o_4),\\
&t_0:=1+2\delta,\\
&\p_x:=(\p_1, \p_2, \p_3, \p_4),\\
&L:=\p_t+\p_r,\\
&\underline L:=\p_t-\p_r,\\
&\O_{ij}:=x^i\p_j-x^j\p_i,\\
&S:=t\p_t+r\p_r=\f{t-r}{2}\underline L+\f{t+r}{2}L,\\
&H_i:=t\p_i+x^i\p_t=\o_i\big(\f{r-t}{2}\underline L+\f{t+r}{2}L\big)-\f{t}{r}\o^j\O_{ij},\\
&\Sigma_{t}:=\{(s,x): s=t, x\in\mathbb R^4\}.\\
\end{align*}

\section{On the short pulse initial data}\label{i}

In this section, under the first null condition \eqref{null}, we show the existence of the short pulse initial data
\eqref{Y1-1} with  property \eqref{Y-00}.
Set $\tilde\chi(s)\in C^{\infty}(\mathbb R)$ with
\begin{equation*}
\tilde\chi(s)=\left\{
\begin{aligned}
&0\quad \text{if $s\leq 0$},\\
&1\quad\text{if $s\geq 1$},
\end{aligned}
\right.
\end{equation*}
and the re-scaled function $\chi_{\dl}(t,x)=\tilde\chi(\f{|x|-t+\dl}{\dl})$.

Assume that $\psi(t,x)$ satisfies
\begin{equation}\label{Aeq}
\left\{
\begin{aligned}
&\bar g^{\al\beta}(\psi,\p\psi)\p_{\al\beta}^2\psi=0,\\
&\psi(1-\f12\delta,x)=\delta^{2-\varepsilon_0/3}\psi_0(\f{r-1}{\delta},\o),
\p_0\psi(1-\f12\delta,x)=\delta^{1-\varepsilon_0/3}\psi_1(\f{r-1}{\delta},\o),
\end{aligned}
\right.
\end{equation}
where
\[
\bar g^{\al\beta}(\psi,\p\psi)=m^{\al\beta}+\chi_{\dl}\big(g^{\al\beta,0}\psi+\tilde g^{\al\beta,\gamma}\p_\gamma\psi\big)+\chi_{\dl}^2h^{\al\beta}(\psi,\p\psi),
\]
$\psi_0(s,\o)$ and $\psi_1(s,\o)$ are smooth functions with compact supports being in $\{(s,\o):-1<s<-\f12, \o\in\mathbb S^3\}$,
$\varsigma>0$ is a suitable constant which will be determined later, with $m^{\al\beta}, g^{\al\beta,0},
\tilde g^{\al\beta,\gamma}$ and $h^{\al\beta}$ given in \eqref{g}.
We next prove that \eqref{Aeq} has a smooth solution for $t\in[1-\f12\delta, 1]$.

Suppose that for $1-\f12\delta\leq t\leq 1$
and $N_0\in\mathbb N_0$ with $N_0\ge 6$,
\begin{equation}\label{Ba}
|\p^q\O^{\kappa}\psi(t,x)|\leq\delta^{3/2-|q|}\quad (|q|+|\kappa|\leq N_0).
\end{equation}
Define the following energy for \eqref{Aeq} and for $n\in\mathbb N_0$,
$$
M_n(t):=\sum_{|q|+|\kappa|\leq n}\delta^{2|q|}\|\p\p^q\O^\kappa\psi(t,\cdot)\|_{L^2(\mathbb R^4)}^2.
$$
Set $w=\delta^{|q|}\p^q\O^\kappa\psi$ with $|q|+|\kappa|\leq 2N_0-2$. Then it follows from \eqref{Aeq}
and direct computations that
\begin{align}\label{int}
&\int_{1-\f12\delta}^{t}\int_{\Sigma_\tau}(\bar g^{\beta\gamma}\p_{\beta\gamma}^2w \p_0 w)(\tau,x) dxd\tau\no\\
=&\int_{\Sigma_{t}}\big(\bar g^{0\beta}\p_\beta w\p_0 w-\f12\bar g^{\beta\gamma}\p_\beta w\p_\gamma w\big)dx
-\int_{\Sigma_{1-\f12\delta}}\big(\bar g^{0\beta}\p_\beta w\p_0 w-\f12\bar g^{\beta\gamma}\p_\beta w\p_\gamma w\big)dx\no\\
&-\int_{1-\f12\delta}^{t}\int_{\Sigma_\tau}\big((\p_\gamma \bar g^{\gamma\beta})\p_\beta w\p_0 w-\f12(\p_\tau \bar g^{\beta\gamma})\p_\beta w\p_\gamma w\big)(\tau,x) dxd\tau
\end{align}
with
\begin{equation}\label{bgw}
\begin{split}
\bar g^{\beta\gamma}\p_{\beta\gamma}^2w=&\delta^{|q|}[\bar g^{\beta\gamma}\p_{\beta\gamma}^2, \p^q\O^\kappa]\psi\\
=&\delta^{|q|}\sum_{\tiny\begin{array}{c}
\ds|\kappa_1|+|\kappa_2|\leq|\kappa|,\\|q_1|+|q_2|\leq|q|+2, |q_2|\geq 2,\\
|q_1|+|\kappa_1|\geq 1\end{array}}(\p^{q_1}\O^{\kappa_1}\bar g^{\beta\gamma})(\p^{q_2}\O^{\kappa_2}\psi),
\end{split}
\end{equation}
where the unnecessary constant coefficients in the expression of \eqref{bgw} are neglected.

It follows from the bootstrap assumption \eqref{Ba} and \eqref{int} that
\begin{equation}\label{ew}
\begin{split}
&\int_{\Sigma_{t}}\big((\p_tw)^2+|\p_x w|^2\big)(t,x)dx\\
\lesssim &\int_{\Sigma_{1-\f12\delta}}\big((\p_tw)^2
+|\p_x w|^2\big)(1-\f12\delta,x)dx+\int_{1-\f12\delta}^{t}\int_{\Sigma_\tau}\delta^{-1/2}\big((\p_tw)^2+|\p_x w|^2\big)(\tau,x)dxd\tau\\
&+\int_{1-\f12\delta}^{t}\int_{\Sigma_\tau}|\bar g^{\beta\gamma}\p_{\beta\gamma}^2w\p_tw|(\tau,x) dxd\tau.
\end{split}
\end{equation}
Utilizing \eqref{Ba} and \eqref{bgw} yields
\begin{equation}\label{bargw}
	\begin{split}
		|\bar g^{\beta\gamma}\p_{\beta\gamma}^2w|\lesssim&\sum_{\tiny\begin{array}{c}2\leq|q_2|\leq|q|+2,\ |\kappa_2|\leq|\kappa|,\\|q_2|+|\kappa_2|\leq|q|+|\kappa|+1\end{array}}\delta^{|q_2|-\f32}|\p^{q_2}\O^{\kappa_2}\psi|\\
		+&\sum_{|q_1|\leq|q|,\  |\kappa_1|\leq|\kappa|}\delta^{|q_1|-\f12}\big(|\p^{q_1}\O^{\kappa_1}\psi|+|\p\p^{q_1}\O^{\kappa_1}\psi|\big).
	\end{split}
\end{equation}
Due to $\tau\in[1-\f12\delta, 1]$, it then holds that
\begin{equation}\label{po}
	\|\p^{q_1}\O^{\kappa_1}\psi(\tau, \cdot)\|_{L^2(\mathbb R^4)}\lesssim\delta\|\p\p^{q_1}\O^{\kappa_1}\psi(\tau, \cdot)\|_{L^2(\mathbb R^4)}.
\end{equation}

Substituting \eqref{bargw} and \eqref{po} into \eqref{ew}, one gets from the Gronwall's inequality that for $1-\f12\delta\leq t\leq 1$,
\begin{equation}\label{YHC-1}
\begin{split}
M_{2N_0-2}(t)\lesssim M_{2N_0-2}(1-\f12\delta)\lesssim\delta^{3-2\varepsilon_0/3}.
\end{split}
\end{equation}

We next close the bootstrap assumption \eqref{Ba}. By the Sobolev's imbedding theorem on the sphere $\mathbb S^3_r$
with center at the origin and radius $r$, one has
$$
|w(t,x)|\lesssim \f{1}{r^{3/2}}\|\O^{\leq 2}w\|_{L^2(\mathbb S^3_r)}.
$$
This, together with $r\sim 1$ for $t\in [1-\f12\delta,1]$, $(t, x)\in \textrm{supp}\ w$
and \eqref{YHC-1}, yields
\begin{equation}\label{Le}
|\p^q\O^\kappa\psi(t,x)|\lesssim \|\O^{\leq 2}\p^q\O^\kappa\psi\|_{L^2(\mathbb S^3_r)}\lesssim\delta^{1/2}\|\p\O^{\leq 2}\p^q\O^\kappa \psi\|_{L^2(\Sigma_{t})}\lesssim\delta^{2-|q|-\varepsilon_0/3},\quad
\end{equation}
where $|q|+|\kappa|\leq 2N_0-4$. Thus, for small $\dl>0$ and $0<\ve_0<1$, \eqref{Ba} is proved.

Next we improve the $L^\infty$ estimate of $\psi(1,x)$ on some special space domains. To this end,
rewrite the equation in \eqref{Aeq} as
\begin{equation}\label{LLpsi}
\begin{split}
\underline LL\psi=&\f{3}{2r}L\psi-\f{3}{2r}\underline L\psi+\f{1}{r^2}\sum_i\o^j\o^a\O_{ij}\O_{ia}\psi\\
&+{\chi}_\dl\big(g^{\al\beta,0}\psi+\tilde g^{\al\beta,\gamma}\p_\gamma\psi\big)\p_{\al\beta}^2\psi+{\chi_\dl}^2h^{\al\beta}(\psi, \p\psi)\p_{\al\beta}^2\psi
\end{split}
\end{equation}
with $|\tilde g^{\al\beta,\gamma}\p_\gamma\psi\p_{\al\beta}^2\psi|\lesssim(|L\psi|+|\O\psi|)|\p^2\psi|+|\p\psi|(|L\p\psi|+|\O\p\psi|)$
due to the first null condition \eqref{null}.
It follows from \eqref{Le} that the worst term in the expression of $\underline LL\psi$
is ${\chi_\dl}\tilde g^{\al\beta,\gamma}\p_\gamma\psi\p_{\al\beta}^2\psi$. Then this yields $|\underline L L\psi|\lesssim\delta^{1-2\varepsilon_0/3}$.
Together with
the vanishing property of $\psi$ on $\{t=r\}$,
integrating $\underline LL\psi$ along integral curves of $\underline L$ gives that for $(t,r)\in D$
(see Figure 2 below),
\begin{equation}\label{Lpsi}
|L\psi(t,x)|\lesssim\delta^{2-2\varepsilon_0/3}.
\end{equation}
Analogously, one has that  for $(t,r)\in D$,
\begin{equation}\label{Lppsi}
|L\p^q\O^\kappa\psi(t,x)|\lesssim\delta^{2-|q|-2\varepsilon_0/3}\quad \text{for $|q|+|\kappa|\leq 2N_0-6$}.
\end{equation}

Substituting \eqref{Lppsi} into the expression of $\underline LL\psi$ again, and noting that the first null condition \eqref{null}
gives
$|\tilde g^{\al\beta,\gamma}\p_\gamma\psi\p_{\al\beta}^2\psi|\lesssim\delta^{2-\varepsilon_0}$,
then one can see that the worst term
in the expression of $\underline LL\psi$ becomes $-\f{3}{2r}\underline L\psi$.
This shows $|\underline LL\psi|\lesssim\delta^{1-\varepsilon_0/3}$
for $(t,r)\in D$ by \eqref{Le} and \eqref{Lppsi}. Hence the following improved estimate is obtained
\[
|L\psi(t,x)|\lesssim\delta^{2-\varepsilon_0/3},\quad (t,r)\in D.
\]
Similar arguments lead to
\begin{equation*}
|L\p^q\O^\kappa\psi(1,x)|\lesssim\delta^{2-|q|-\varepsilon_0/3}\quad \text{for $|q|+|\kappa|\leq 2N_0-7$
	and $1-\delta\le r\le 1$}.
\end{equation*}
Analogously, for $r\in [1-\delta,1]$, the induction argument yields
\begin{equation}\label{Lp}
|L^k\p^q\O^\kappa\psi(1,x)|\lesssim\delta^{2-|q|-\varepsilon_0/3},\quad 3k+|q|+|\kappa|\leq 2N_0-4.
\end{equation}

\vskip 0.3 true cm
 \centerline{\includegraphics[scale=0.5]{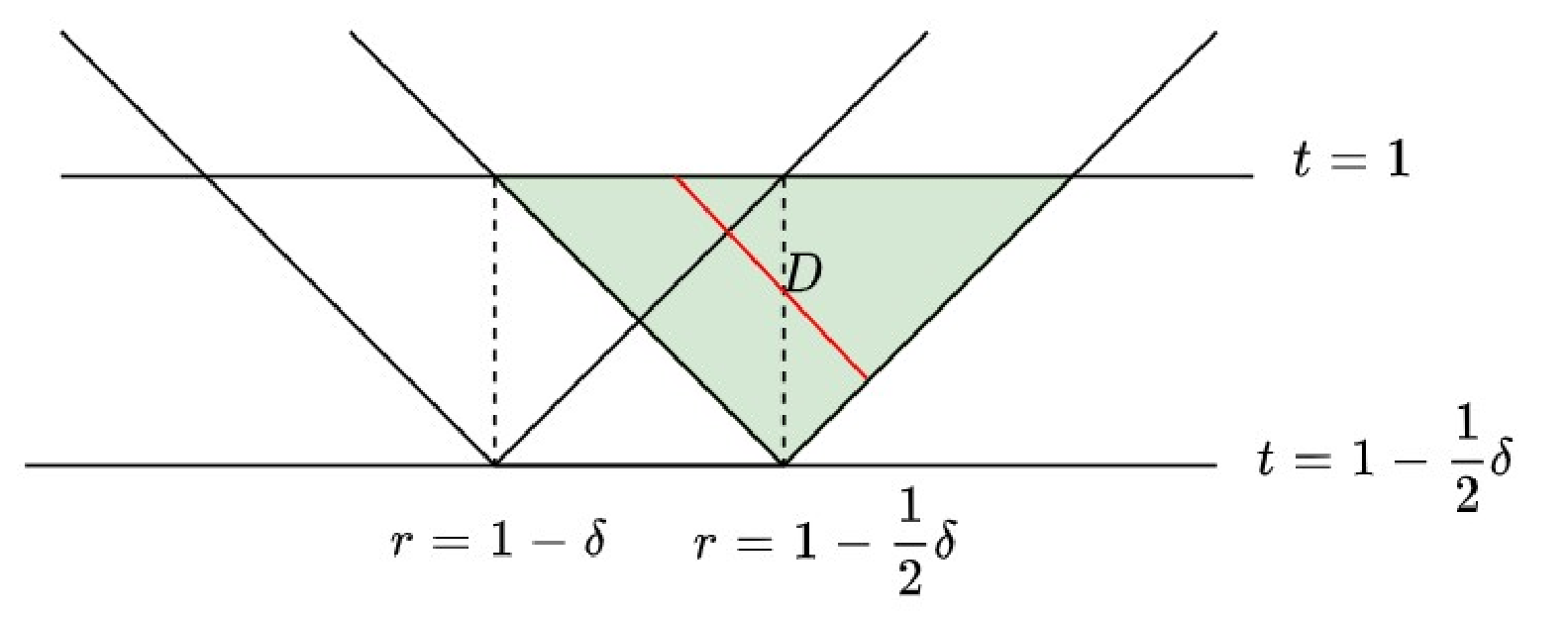}}

\centerline{\bf Figure 2. The domain $D$}
\vskip 0.3 true cm

Choose the initial data for \eqref{quasi} as
\begin{equation}\label{iphi}
\left\{
\begin{aligned}
\phi(1,x)&=(\chi_\dl\psi)(1,x),\\
\p_t\phi(1,x)&=(\p_t(\chi_\dl\psi))(1,x).
\end{aligned}
\right.
\end{equation}
Obviously, the support of $\phi(1,x)$ lies in $\{x:1-\delta<r< 1\}$ (see Figure 3 below).
\vskip 0.3 true cm
\centerline{\includegraphics[scale=0.6]{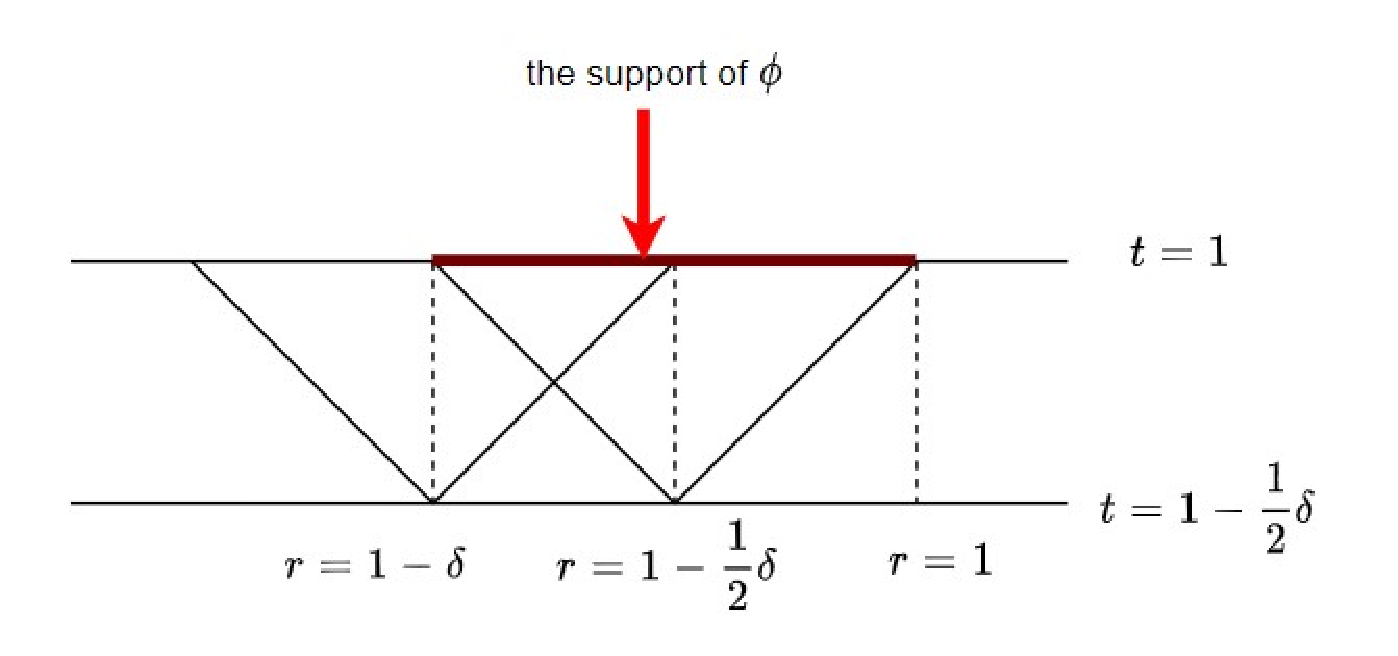}}

\centerline{\bf Figure 3. Support of $\phi(1,x)$}
\vskip 0.3 true cm

Next we check that the initial data \eqref{iphi} satisfy \eqref{Y-00}.
In fact, it follows \eqref{quasi} and \eqref{Aeq} that
\begin{equation}\label{phipsi}
\left\{
\begin{aligned}
g^{\al\beta}(\phi,\p\phi)\p_{\al\beta}^2(\phi-\chi_\dl\psi)=&{\chi_\dl}^2(g^{\al\beta,0}\psi+\tilde{g}^{\al\beta,\gamma}\p_{\gamma}\psi)\p_{\al\beta}^2\psi+{\chi_\dl}^3h^{\al\beta}(\psi,\p\psi)\p_{\al\beta}^2\psi\\
&-(g^{\al\beta,0}\phi+\tilde{g}^{\al\beta,\gamma}\p_{\gamma}\phi)\p_{\al\beta}^2(\chi_\dl\psi)-h^{\al\beta}(\psi,\p\psi)\p_{\al\beta}^2(\chi_\dl\psi)\\
&-\dl^{-1}\tilde{\chi}'(\f{|x|-t+\dl}\dl)\big(\f3r\psi+2L\psi\big),\\
(\phi-\chi_\dl\psi)(1,x)=&0,\\
\big(\p_t(\phi-\chi_\dl\psi)\big)(1,x)=&0.
\end{aligned}
\right.
\end{equation}
It is obvious that
\begin{align}
\big(L(\phi-\chi_\dl\psi)\big)(1,x)=0,\label{Lphipsi}\\
\big(\p_r^2(\phi-\chi_\dl\psi)\big)(1,x)=0,\label{prphipsi}\\
\big(\p_t\p_r(\phi-\chi_\dl\psi)\big)(1,x)=0.\label{ptrphipsi}
\end{align}
In addition, it follows from \eqref{LLpsi} and \eqref{Lp} that
$|\underline L(\f3r\psi+2L\psi)|_D\lesssim\dl^{2-\ve_0}$,
which implies
\[
|\f3r\psi+2L\psi|_D\lesssim\dl^{3-\ve_0}.
\]
This, together with \eqref{phipsi} and \eqref{Lp}, deduces
\begin{equation}\label{ptt}
|\p_t^2(\phi-\chi_\dl\psi)(1,x)|\lesssim\dl^{2-\ve_0}.
\end{equation}
Therefore, by \eqref{prphipsi}, \eqref{ptrphipsi} and \eqref{ptt}, one has
\begin{equation}\label{L2phipsi}
|L^2(\phi-\chi_\dl\psi)(1,x)|\lesssim\dl^{2-\ve_0}.
\end{equation}
Combining \eqref{Lp}, \eqref{phipsi}, \eqref{Lphipsi} and \eqref{L2phipsi} yields
\begin{align*}
&|\phi(1,x)|+|L\phi(1,x)|\lesssim\delta^{2-\varepsilon_0/3},\ |L^2\phi(1,x)|\lesssim\delta^{2-\varepsilon_0}.
\end{align*}
Note that $L^k\p^q\O^{\kappa}\phi$ $(0\leq k\leq 2)$ can be estimated similarly as above after taking suitable derivatives on \eqref{phipsi}
to get the upper bounds of $|L^k\p^q\O^{\kappa}(\phi-\chi_\dl\psi)(1,x)|$ $(0\leq k\leq 2)$ and applying \eqref{Lp} directly. Finally,
we arrive at
 \begin{align*}
 &|\p^q\O^{\kappa}\phi(1,x)|+|L\p^q\O^{\kappa}\phi(1,x)|\lesssim\delta^{2-|q|-\varepsilon_0/3},\\
 &|L^2\p^q\O^{\kappa}\phi(1,x)|\lesssim\delta^{2-|q|-\varepsilon_0}.
 \end{align*}
This means that the initial data \eqref{iphi} satisfy \eqref{Y-00}.

\begin{remark}
Rewriting the initial data \eqref{iphi} as
\begin{equation}\label{F-1}
\left\{
\begin{aligned}
\phi(1,x)&=\psi(1,\f{r-1}{\dl}\cdot\dl\o+\o)\tilde{\chi}(\f{r-1}{\dl}+1):=\delta^{2-\varepsilon_0/3}\phi_0^{\dl}(\f{r-1}{\delta},\omega),\\
\p_t\phi(1,x)&=\p_t\psi(1,\f{r-1}{\dl}\cdot\dl\o+\o)\tilde{\chi}(\f{r-1}{\dl}+1)-\delta^{-1}\psi(1,\f{r-1}{\dl}\cdot\dl\o
+\o)\tilde{\chi}'(\f{r-1}{\dl}+1)\\
&:=\delta^{1-\varepsilon_0/3}\phi_1^{\dl}(\f{r-1}{\delta},\omega).
\end{aligned}
\right.
\end{equation}
Then $(\phi_0^{\dl},\phi_1^{\dl})(s,\omega)$ fulfill the property \eqref{Y-00}. For brevity and without confusions, we
drop the notation $\dl$ in $(\phi_0^{\dl},\phi_1^{\dl})$ as $(\phi_0,\phi_1)$.

\end{remark}

\section{Local existence}\label{LE}

In this section, we apply the energy method and the  distinctive structure of the initial data \eqref{Y1-1}
to establish the local existence of the smooth solution $\phi$ to
equation \eqref{quasi} for $1\leq t\leq t_0$, meanwhile, several
key estimates of $\phi(t_0, x)$ in some special space domains are derived.

\begin{theorem}\label{Th2.1}
Under the assumption \eqref{Y-00} on $(\phi_0,\phi_1)$, when $\delta>0$ is small, equation \eqref{quasi} with \eqref{Y1-1}
admits a local smooth solution $\phi\in C^\infty([1, t_0]\times\mathbb R^4)$. Moreover, for $m\in\mathbb N_0$,
$l\in \mathbb N_0$, $q\in\mathbb N_0^5$ and $\kappa\in \mathbb N_0^6$, it holds that
\begin{enumerate}[(i)]
	\item
	\begin{align}
	&|L^m\p^q\O^\kappa\phi(t_0,x)|\lesssim\delta^{2-|q|-\varepsilon_0},\quad r \in[1-2\delta, 1+2\delta],\label{local1-2}\\
	&|\underline L^l\p^q\O^\kappa\phi(t_0,x)|\lesssim\delta^{3-|q|-\varepsilon_0},\quad r\in [1-3\delta, 1+\delta].\label{local2-2}
	\end{align}
	\item
	\begin{equation}\label{local3-2}
	|\p^q\O^\kappa\phi(t_0,x)|\lesssim
	\left\{
	\begin{aligned}
	\delta^{3-\varepsilon_0},\quad&\text{for}\quad|q|\leq 1,\\
	\delta^{4-|q|-\varepsilon_0},\quad&\text{for}\quad|q|> 1,
	\end{aligned}
	\right.
	\quad r\in [1-3\delta, 1+\delta].
	\end{equation}
	\item
	\begin{equation}\label{local3-3}
	|\underline L^lL^m\O^\kappa\phi(t_0,x)|\lesssim\delta^{2-\varepsilon_0},\quad r\in[1-2\delta,1+\delta].
	\end{equation}
\end{enumerate}

\end{theorem}

\begin{proof} Denote by $Z_g$ any fixed vector filed in $\{S, H_i, i=1,2,3, 4\}$, and construct the energies
$$
\tilde M_n(t):=\sum_{|q|+|\kappa|+k\leq n}\delta^{2|q|}\|\p\p^q\O^\kappa Z_g^k\phi(t,\cdot)\|_{L^2(\mathbb R^4)}^2
$$
for $n\leq \tilde N_0$, $k\leq 2$, $\tilde N_0\in \mathbb N$ and $\tilde N_0\ge 4$.
\vskip 0.3 true cm
\centerline{\includegraphics[scale=0.2]{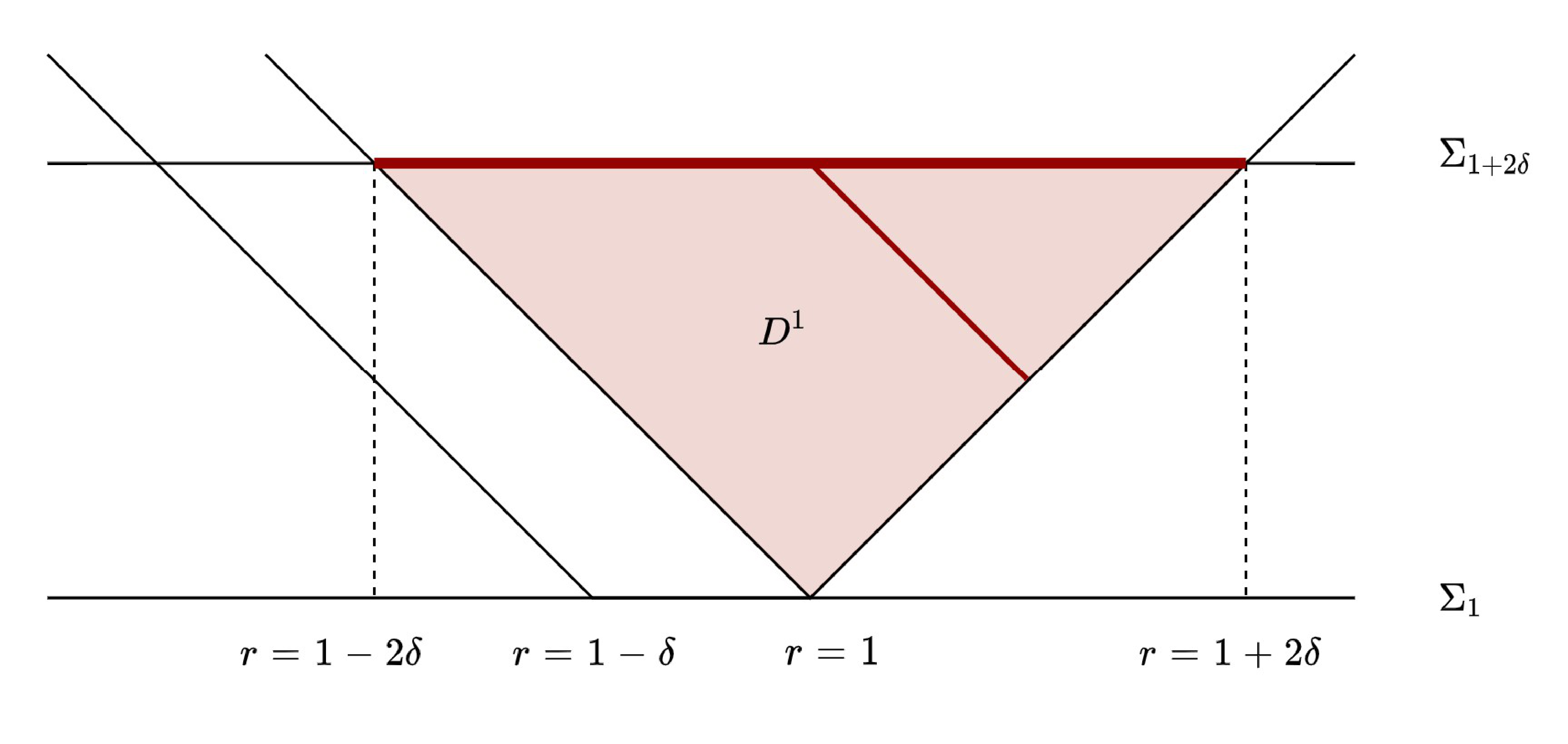}}

\centerline{\bf Figure 4. Space-time domain $D^1=\{(t,r): 1\le t\le t_0, 2-t\le r\le t\}$}
\vskip 0.3 true cm

For suitably small $\dl>0$, by the same arguments as in the proofs of the local existence of the solution $\psi$ to \eqref{Aeq} and the inequality \eqref{Le} in
Section \ref{i} ($\bar g$ and $\psi$ are replaced by $g$ and $Z_g^k\phi$ respectively in \eqref{int}-\eqref{Le}, menawhile
the time interval
$[1-\f12\dl,1]$ is replace by $[1,1+2\dl]$ from \eqref{Ba} to \eqref{Le}),
when $t\in[1,t_0]$ and by \eqref{Y-00}, one can prove that \eqref{quasi} has a smooth solution $\phi$ satisfying
\begin{equation}\label{Lle}
|L^k\p^q\O^\kappa\phi(t, x)|\lesssim\delta^{2-|q|-\varepsilon_0},
\end{equation}
where $|q|+|\kappa|+k\leq 2\tilde N_0-4$ and $k\leq 2$.
Moreover, since $|\p^q\O^\kappa\phi(t,x)|\lesssim\delta^{2-\varepsilon_0-|q|}$ holds in $D^1$ when $|q|+|\kappa|\leq 2\tilde N_0-4$,
after repeating the procedures from \eqref{LLpsi}-\eqref{Lp}, we obtain that for $r\in[1-2\delta,1+2\delta]$,
 \begin{equation}\label{oi}
 |L^m\p^q\O^\kappa\phi(t_0,x)|\lesssim\delta^{2-|q|-\varepsilon_0},\quad 3m+|q|+|\kappa|\leq 2\tilde N_0-4.
 \end{equation}

 \centerline{\includegraphics[scale=0.2]{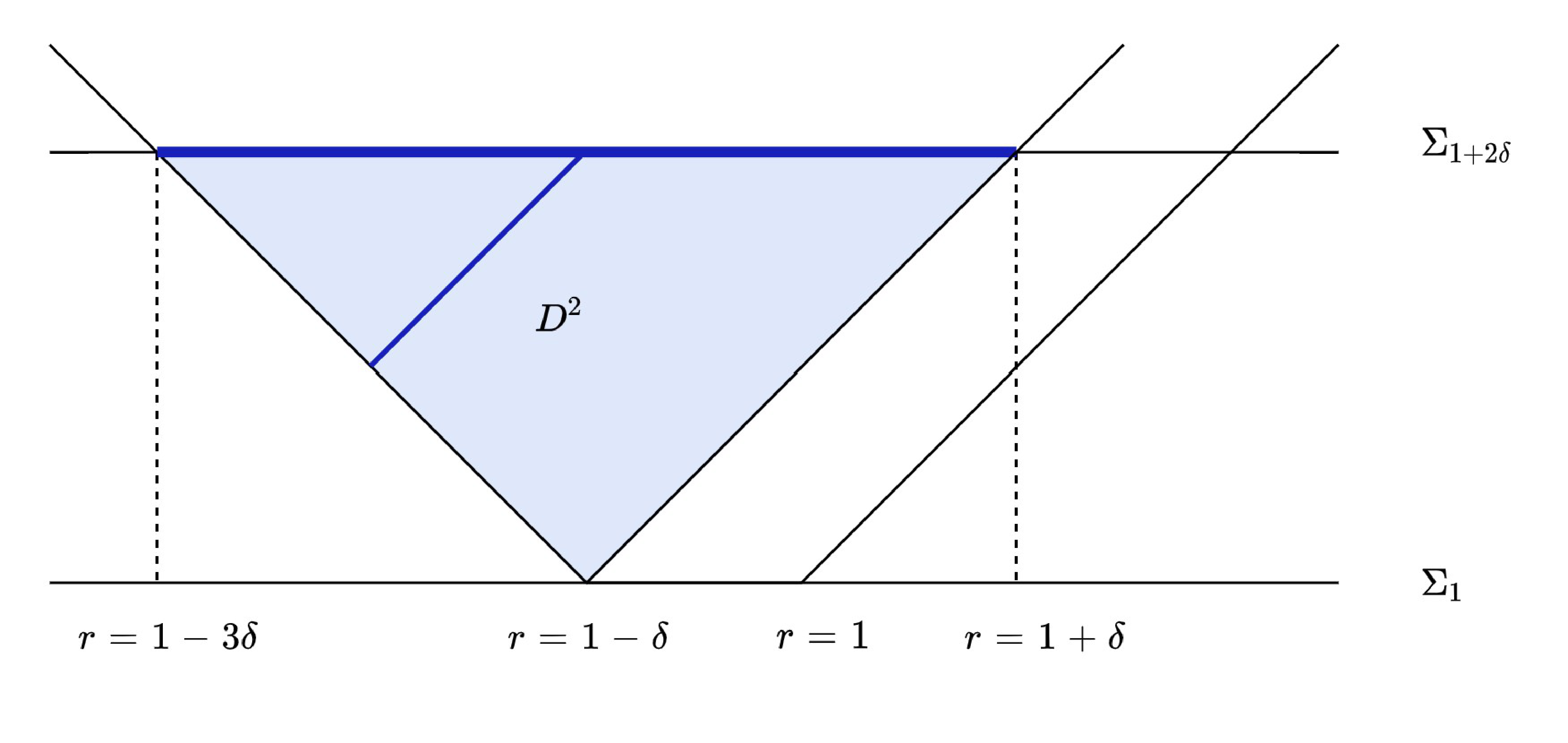}}

\centerline{\bf Figure 5. Space domain $D^2=\{t,r): 1\le t\le t_0, 2-\delta-t\le r\le t-\delta\}$}
\vskip 0.3 true cm

Let $\bar\p\in\{\p_t, \p_r\}$. For any indices $q$ and $\kappa$ with $|q|+|\kappa|\leq 2\tilde N_0-6$, by virtue of
\begin{equation}\label{LUL}
L\bar\p^q\O^\kappa\underline L\phi=\bar\p^q\O^\kappa(L\underline L\phi),
\end{equation}
\begin{equation}\label{me}
L\underline L\phi=\f{3}{2r}L\phi-\f{3}{2r}\underline L\phi+\f{1}{r^2}\sum_i\o^j\o^k\O_{ij}\O_{ik}\phi+\Big(g^{\al\beta,0}\phi
+\tilde g^{\al\beta,\gamma}\p_\gamma\phi+h^{\al\beta}(\phi, \p\phi)\Big)\p_{\al\beta}^2\phi
\end{equation}
and \eqref{Lle}, one has $|L\bar\p^q\O^\kappa\underline L\phi|\lesssim\delta^{1-\varepsilon_0-|q|}$.
This yields $|\underline L\bar\p^q\O^\kappa\phi|\lesssim\delta^{2-\varepsilon_0-|q|}$ in $D^2$
by the integration along integral curves of $L$ when $|q|+|\kappa|\leq2\tilde N_0-6$ (see Figure 5). Thus,
it follows from \eqref{LUL}, \eqref{me} and \eqref{Lle} that in $D^2$,
\begin{equation*}
\begin{split}
&|\underline L\bar\p^q\O^\kappa\phi|\lesssim\delta^{2-\varepsilon_0-|q|},\quad |q|+|\kappa|\leq2\tilde N_0-6,\\
&|L\underline L\bar\p^q\O^\kappa\phi|\lesssim\delta^{2-\varepsilon_0-|q|},\quad |q|+|\kappa|\leq2\tilde N_0-7.
\end{split}
\end{equation*}
By an induction argument, it holds that in $D^2$,
\begin{equation}\label{ii}
\begin{split}
&|\underline L^l\bar\p^q\O^\kappa\phi|\lesssim\delta^{2-\varepsilon_0-|q|},\quad 2l+|q|+|\kappa|\leq2\tilde N_0-4,\\
&|L\underline L^l\bar\p^q\O^\kappa\phi|\lesssim\delta^{2-\varepsilon_0-|q|},\quad 2l+|q|+|\kappa|\leq2\tilde N_0-5.
\end{split}
\end{equation}
This, together with $\p_t=\f12(L+\underline L)$, $\p_i=\f{\o^i}{2}(L-\underline L)-\f1r\o^j\O_{ij}$,  \eqref{me} and \eqref{Lle}, yields that for $2|q|+|\kappa|\leq 2\tilde N_0-4$
and $(t,x)\in D^2$,
 \begin{equation}\label{iid}
 |\p^q\O^\kappa\phi(t,x)|\lesssim
 \left\{
 \begin{aligned}
 \delta^{2-\varepsilon_0},\quad&\text{for}\quad|q|\leq 2,\\
 \delta^{4-\varepsilon_0-|q|},\quad&\text{for}\quad|q|> 2.
 \end{aligned}
 \right.
 \end{equation}
 Next, we use \eqref{me} to improve the estimates in \eqref{ii} and \eqref{iid}. In fact, acting \eqref{me}
 by $\underline L^{l-1}\bar\p^q\O^\kappa$ $(l\geq 1)$, it follows from \eqref{ii} and $\p_i={\o^i}\p_r-\f1r\o^j\O_{ij}$
 that when $2l+|q|+|\kappa|\lesssim 2\tilde N_0-5$,
 	\begin{equation}\label{Lu}
 	|L\underline L^l\bar\p^q\O^\kappa\phi(t,x)|\lesssim\dl^{2-\ve_0-|q|},\ (t,x)\in D^2.
 	\end{equation}
Then integrating \eqref{Lu} along integral curves of $L$ in $D^2$ yields
 	\begin{equation*}
 	|\underline L^l\bar\p^q\O^\kappa\phi(t,x)|\lesssim\dl^{3-\ve_0-|q|},\ (t,x)\in D^2,\ l\geq 1\ \text{and}\ 2l+|q|+|\kappa|\lesssim
 2\tilde N_0-5.
 	\end{equation*}
Analogous treatments from  \eqref{Lle} deduce that in $D^2$,
 	\begin{equation*}
 	\begin{split}
 	&|\bar\p^q\O^\kappa\phi(t,x)|\lesssim\dl^{3-\ve_0-|q|},\ |q|+|\kappa|\leq 2\tilde N_0-4,\\
 	&|L\bar\p^q\O^\kappa\phi(t,x)|\lesssim\dl^{3-\ve_0-|q|},\ |q|+|\kappa|\leq 2\tilde N_0-5.
 	\end{split}
 	\end{equation*}
Therefore, \eqref{iid} can be improved into the more refined  estimate \eqref{local3-2}.
It remains to show \eqref{local3-3}. Note that on $\Sigma_{t_0}$, when $r\in[1-2\delta,1+\delta]$
 and $2l+|q|+|\kappa|\leq 2\tilde N_0-5$, $|L\underline L^l\p^q\O^\kappa\phi|\lesssim\delta^{2-\varepsilon_0-|q|}$ holds by \eqref{ii}.
Assume inductively that the following estimate holds for $r\in[1-2\delta,1+\delta]$,
\begin{equation}\label{ia}
|L^m\underline L^l\p^q\O^\kappa\phi(t_0,x)|\lesssim\delta^{2-\varepsilon_0-|q|},\quad\text{where}\left\{
\begin{aligned}
1\leq m\leq m_0,\quad l\geq 1,\\
3m-3+|q|+|\kappa|\leq 2N_0-4,\\
m+2l+|q|+|\kappa|\leq2N_0-4.
\end{aligned}
\right.
\end{equation}
We next show that \eqref{ia} remains true for $m=m_0+1$.
When $3m_0-1+|q|+|\kappa|\leq 2\tilde N_0-4$ and $m_0+1+2l+|q|+|\kappa|\leq2\tilde N_0-4$,
it follows from \eqref{ia}, \eqref{me} and \eqref{oi} that
\begin{equation}\label{La+1}
\begin{split}
|L^{m_0+1}\underline L^l\bar\p^q\O^\kappa\phi|&\lesssim|L^{m_0+1}\underline L^{\leq l-1}\bar\p^{\leq|q|}\O^\kappa\phi|+\delta^{2-2\varepsilon_0-|q|}\\
  	&\lesssim\cdots\lesssim|L^{m_0+1}\bar\p^{\leq|q|}\O^\kappa\phi|+\delta^{2-2\varepsilon_0-|q|}\\
  	&\lesssim\delta^{2-2\varepsilon_0-|q|}.
  	\end{split}
  	\end{equation}

When $m_0+1+2l+|q|+|\kappa|\leq 2\tilde N_0-4$ and $3m_0+|q|+|\kappa|\leq2\tilde N_0-4$,
  	by \eqref{La+1} and \eqref{me}, one has
  	\begin{equation}\label{F3}
  	|L^{m_0+1}\underline L^l\bar\p^q\O^\kappa\phi|\lesssim|L^{m_0+1}\underline L^{\leq l-1}\bar\p^{\leq|q|}\O^\kappa\phi|+\delta^{2-\varepsilon_0-|q|}\lesssim\delta^{2-\varepsilon_0-|q|}.
  	\end{equation}
Combining \eqref{La+1} and \eqref{F3} yields \eqref{ia}. Then \eqref{local3-3}
 is obtained after setting $|q|=0$ in \eqref{ia}.

Therefore, the proof of Theorem \ref{Th2.1} is finished.
\end{proof}

\section{Some preliminaries}\label{p}

\subsection{The related geometry and definitions}\label{p-0}

In this subsection, we give some preliminaries on the
related geometry and definitions, which will be
utilized as the basic tools to establish the global existence of the
smooth solution $\phi$ to \eqref{quasi} in $A_{2\dl}$.

The definitions of the optical function $u$ and the inverse foliation density $\mu$ have been
given in \eqref{H0-3} and \eqref{H0-4} respectively.
Note that
$$
\tilde{L}=-g^{\al\beta}\p_\al u\p_\beta
$$
is a tangent vector field of the outgoing light cone $\{ u=C\}$. In addition, $\tilde {L}$ is geodesic and $\tilde{L}t=\mu^{-1}$.
Thus, one can rescale $\tilde{L}$ as
$$
\mathring{L}=\mu\tilde{L},
$$
which approximates $L=\p_t+\p_r$ (this is obvious for $t=t_0$).

Set $\tilde T=-g^{\nu 0}\p_\nu-\mathring L$, which is close to $-\p_r$ for $t=t_0$. Then
by the definition of the null frame, let
\begin{align}\label{FC-1}
T=\mu\tilde T,\quad\mathring{\underline L}=\mu\mathring L+2T,
\end{align}
where $\mathring{L}$ and $\mathring{\underline L}$ are two vector fields in the null frame,
moreover $\mathring{\underline L}$ approximates $\underline L=\p_t-\p_r$.

With respect to the other vector fields
$\{X_1, X_2, X_3\}$ in the null frame, they have been chosen in \eqref{H0-6} with $X_A=\f{\p}{\p\vartheta^A}$ $(A=1,2,3)$.
Rewrite $X_A=X_A^\al\p_\al$. Then $X_A^0=0$ due to $\f{\p t}{\p\vartheta^A}=0$.

\begin{lemma}\label{nullframe}
$\{\mathring L, \mathring{\underline L}, X_1, X_2, X_3\}$ constitutes a null frame with respect to the
metric $g$, and admits the following identities:
\begin{align}
&g(\mathring L,\mathring L)=g(\mathring{\underline L}, \mathring{\underline L})=g(\mathring L, X_A)=g(\mathring{\underline L}, X_A)=0,\\
&g(\mathring L, \mathring{\underline L})=-2\mu.
\end{align}
In addition,
\begin{align}
&\mathring L t=1,\quad\mathring L u=0,\\
&\mathring{\underline L}t=\mu,\quad \mathring{\underline L} u=2.
\end{align}
And
\begin{align}\label{LTTT}
g(\mathring L, T)=-\mu,\quad g(T,T)=\mu^2,
\end{align}
\begin{align}\label{T}
Tt=0,\quad T u=1.
\end{align}
\end{lemma}

In the new coordinate $(s, u, \vartheta^1, \vartheta^2, \vartheta^3)$ defined by \eqref{H0-7},
we will focus on  the following subsets (see Figure 6):

\begin{definition}
	\begin{equation}\label{error}
	\begin{split}
	&\Sigma_s^{u}:=\{(s',u',\vartheta): s'=s,0\leq u'\leq  u\},\quad u\in [0, 4\delta],\\
	&C_{u}:=\{(s',u',\vartheta): s'\geq t_0, u'=u\},\\
	&C_{u}^s:=\{(s',u',\vartheta): t_0\leq s'\leq s, u'=u\},\\
	&S_{s, u}:=\Sigma_s\cap C_{u},\\
	&D^{s, u}:=\{(s', u',\vartheta): t_0\leq s'<s, 0\leq u'\leq u\}.
	\end{split}
	\end{equation}
\end{definition}

\vskip 0.2 true cm
\centerline{\includegraphics[scale=0.38]{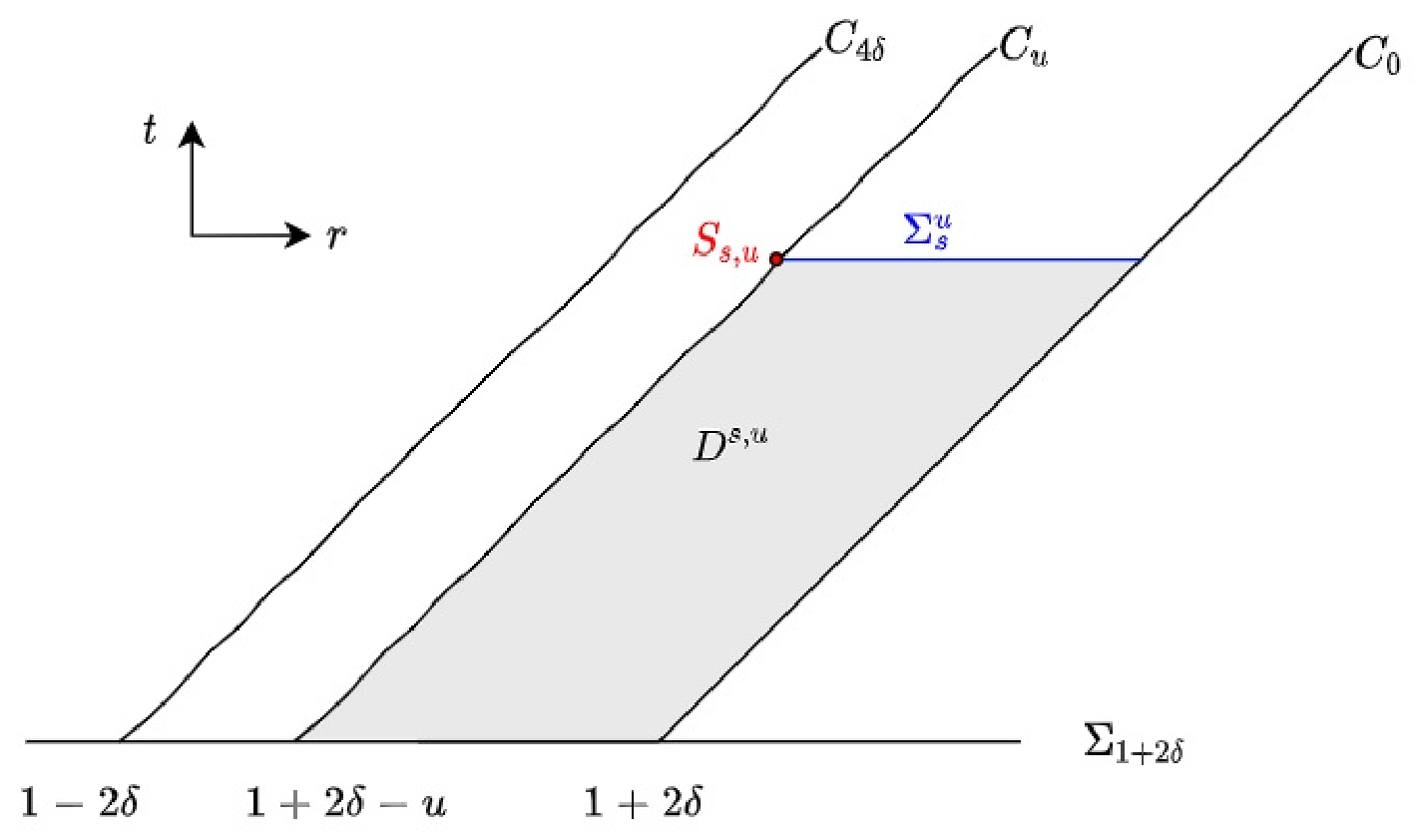}}

\centerline{\bf Figure 6. The indications of some domains}
\vskip 0.3 true cm

Next the following geometric notations will be used.
\begin{definition} For the metric $g$ on the spacetime,
	\begin{itemize}
		\item $\underline g=(g_{ij})$ is defined as the induced metric of $g$ on $\Sigma_t$, i.e., $\underline g(U,V)=g(U,V)$ for any tangent vectors $U$ and $V$ of $\Sigma_t$;
		\item $\slashed \Pi_\al^\beta:=\delta_\al^\beta-\delta_\al^0\mathring L^\beta+\mathring L_\al\tilde{T}^\beta$ is the
projection tensor field on $S_{s,u}$ of type $(1,1)$, where $\delta_\al^\beta$ is Kronecker delta;
		\item define $\slashed\xi:=\slashed\Pi\xi$ as the tensor field on $S_{s,u}$ for any $(m,n)$-type spacetime tensor field
$\xi$, whose components are $$\slashed\xi^{\al_1\cdots\al_m}_{\beta_1\cdots\beta_n}:=(\slashed\Pi\xi)^{\al_1\cdots\al_m}_{\beta_1\cdots\beta_n}
=\slashed\Pi_{\beta_1}^{\beta_1'}\cdots\slashed\Pi_{\beta_n}^{\beta_n'}\slashed\Pi_{\al_1'}^{\al_1}\cdots
\slashed\Pi_{\al_m'}^{\al_m}\xi^{\al_1'\cdots\al_m'}_{\beta_1'\cdots\beta_n'},$$
		in particular, $\slashed g=(\slashed g_{\al\beta})$ is the induced metric of $g$ on $S_{s,u}$;
		\item $(\slashed g^{AB})$ is defined as the inverse matrix of $(\slashed g_{AB})$ with $\slashed g_{AB}=g(X_A, X_B)$;
		\item $\mathscr D$ and $\slashed\nabla$ denote the Levi-Civita connection of $g$ and $\slashed g$, respectively;
		\item $\Box_g:=g^{\al\beta}\mathscr{D}^2_{\al\beta}$, $\slashed\triangle:=\slashed g^{AB}\slashed\nabla^2_{AB}$;
		\item define $\mathcal L_V\xi$ as the Lie derivative of $\xi$ with respect to $V$ and $\slashed{\mathcal L}_V\xi:=\slashed\Pi(\mathcal L_V\xi)$ for any tensor field $\xi$ and vector $V$;
		\item when $\xi$ is a $(m,n)$-type spacetime tensor field, the square of its norm is defined as
		$$
		|\xi|^2:=g_{\al_1\al_1'}\cdots g_{\al_m\al_m'}g^{\beta_1\beta_1'}\cdots g^{\beta_n\beta_n'}\xi_{\beta_1\cdots\beta_n}^{\al_1\cdots\al_m}\xi_{\beta_1'\cdots\beta_n'}^{\al_1'\cdots\al_m'}.
		$$
	\end{itemize}
\end{definition}

In the null frame $\{\mathring L, \mathring{\underline L}, X_1, X_2, X_3\}$,
define {\it{the second fundamental forms}} $\chi$ and $\sigma$  as
\begin{equation}\label{chith}
\chi_{AB}:=g(\mathscr D_A \mathring L, X_B),\quad\sigma_{AB}:=g(\mathscr D_A\tilde T,X_B),
\end{equation}
where $\mathscr D_A$ stands for $\mathscr D_{X_A}$.
Meanwhile, define the {\it one-form tensors} $\zeta$ and $\eta$ as
\begin{equation}\label{zetaeta}
\zeta_A:=g(\mathscr D_A \mathring L,\tilde T),\quad\eta_A:=-g(\mathscr D_A T, \mathring L).
\end{equation}
Then $\mu\zeta_A=-X_A\mu+\eta_A$.
For any vector field $V$, denote its associated {\it{deformation tensor}} by
\begin{equation}\label{dt}
\leftidx{^{(V)}}\pi_{\al\beta}:=g(\mathscr{D}_{\al}V,\p_{\beta})+g(\mathscr{D}_{\beta}V,\p_{\al}).
\end{equation}

On the initial hypersurface $\Sigma_{t_0}^{4\delta}$, it holds that $\tilde T^i=\ds-\f{x^i}{r}+O(\delta^{1-\varepsilon_0})$,
$\mathring L^0=1$, $\mathring L^i=\ds\f{x^i}{r}+O(\delta^{1-\varepsilon_0})$ and
$\chi_{AB}=\ds\f1r\slashed g_{AB}+O(\delta^{1-\varepsilon_0})$.
Note that on $\Sigma_{t_0}$, $r$ is $t_0-u$. For $t\ge t_0$, one can define the ``{\it{error vectors}}" as
\begin{equation}\label{errorv}
\begin{split}
&\check{L}^0:=0,\\
&\check{L}^i:=\mathring L^i-\f{x^i}{\varrho},\\
&\check{T}^i:=\tilde T^i+\f{x^i}{\varrho},\\
&\check{\chi}_{AB}:=\chi_{AB}-\f{1}{\varrho}\slashed g_{AB},
\end{split}
\end{equation}
here and below $\varrho$ is defined as $s
-u$.

It is noted that $\vartheta=(\vartheta^1, \vartheta^2, \vartheta^3)$ are the coordinates on the sphere $S_{s, u}$.
Then in the new coordinate system $(s, u, \vartheta^1, \vartheta^2, \vartheta^3)$,
one has $\mathring L=\f{\p}{\p s}$. In addition, it follows from
\eqref{T} that $T=\f{\p}{\p u}-\Xi$ with $\Xi=\Xi^AX_A$
for some smooth $\Xi^A$, $A=1, 2, 3$. Lemma 3.66 in \cite{J} yields

\begin{lemma}
In domain $D^{s, u}$, the Jacobian determinant of the map $(s, u, \vartheta^1, \vartheta^2, \vartheta^3)\rightarrow
(x^0, x^1, x^2, x^3, x^4)$ is
\begin{equation}\label{Jacobian}
\det\f{\p(x^0, x^1, x^2, x^3, x^4)}{\p(s, u, \vartheta^1, \vartheta^2, \vartheta^3)}
=\mu(\det\underline g)^{-1/2}\sqrt{\det\slashed g},
\end{equation}
where $\det\slashed g=\det (\slashed g_{AB})$.
\end{lemma}

\begin{remark}\label{3.1}
It follows from \eqref{Jacobian} that if the metrics $\underline g$ and $\slashed g$ are regular, i.e., $\det\underline g>0$
and $\det\slashed g>0$, then the transformation of coordinates between $(s, u, \vartheta^1, \vartheta^2, \vartheta^3)$
and $(x^0, x^1, x^2, x^3, x^4)$ makes sense
as long as $\mu>0$.
\end{remark}

On $\mathbb{S}^3$, the standard rotation vector fields $\Omega_{ij}=x^i\p_j-x^j\p_i$ are used
as tangent derivatives.
In order to project $\Omega_{ij}$ on $S_{s, u}$, as in (3.39b) of \cite{J}, one defines
\begin{equation*}
R_{ij}:=\slashed\Pi\Omega_{ij},\quad \slashed d:=\slashed\Pi d
\end{equation*}
as the rotation vectorfields and the differential on $S_{s, u}$, respectively. Note that the explicit expression of $R_{ij}$ is
\begin{equation}\label{R}
R_{ij}=(\delta_l^k-g_{al}\tilde T^a\tilde T^k){\Omega_{ij}}^l\p_k
=\Omega_{ij}-g_{al}\tilde{T}^a{\Omega_{ij}}^l\tilde{T}.
\end{equation}
For the sake of brevity, set
\begin{equation}\label{omega}
\upsilon_{ij}:=g_{al}\tilde{T}^a{\Omega_{ij}}^l=g_{kj}\check{T}^kx^i-g_{ki}\check T^kx^j-(g_{kj}-m_{kj})\f{x^kx^i}\varrho+(g_{ki}-m_{ki})\f{x^kx^j}\varrho.
\end{equation}
Then
$$R_{ij}=\Omega_{ij}-\upsilon_{ij}\tilde{T}.$$

For the domains with $\mu>0$, the following integrations and norms
will be utilized repeatedly in subsequent sections.

\begin{definition}[]
For any continuous function $f$ and tensor field $\xi$, define
\begin{align*}
&\int_{S_{s, u}}f:=\int_{S_{s, u}}fd\nu_{\slashed g}=\int_{\mathbb{S}^3}f(s, u,\vartheta)
\sqrt{\det\slashed g(s, u,\vartheta)}d\vartheta,\\
&\int_{ C^s_{ u}}f:=\int_{t_0}^s\int_{S_{\tau, u}}f(\tau, u,\vartheta) d\nu_{\slashed g}d\tau,\\
&\int_{\Sigma_s^{ u}}f:=\int_0^{ u}\int_{S_{s, u'}}f(s, u',\vartheta) d\nu_{\slashed g}d u',\\
&\int_{D^{s, u}}f:=\int_{t_0}^s\int_0^{ u}\int_{S_{\tau, u'}}f(\tau, u',\vartheta)d\nu_{\slashed g}d u'd\tau,\\
&\|\xi\|_{L^2(S_{s, u})}^2:=\int_{S_{s, u}}|\xi|^2,\quad\|\xi\|_{L^2( C^s_{ u})}^2:=\int_{ C^s_{ u}}|\xi|^2,\\
&\|\xi\|_{L^2(\Sigma_s^{u})}^2:=\int_{\Sigma_s^{ u}}|\xi|^2,\quad\|\xi\|_{L^2(D^{s, u})}^2:=\int_{D^{s, u}}|\xi|^2.
\end{align*}
\end{definition}

Finally, the following contraction notations will be used.
\begin{definition}[]\label{def}
If $\xi$ is a $(0,2)$-type spacetime tensor, $\Lambda$ is a 1-form, $U$ and $V$ are vector fields, then the contraction
of $\xi$ with respect to $U$ and $V$ is defined as
$$
\xi_{UV}:=\xi_{\al\beta}U^{\al}V^{\beta},
$$
and the contraction of $\Lambda$ with respect to $U$ is
\[
\Lambda_U:=\Lambda_{\al}U^{\al}.
\]
\end{definition}

\begin{definition}[]
	If $\xi$ is a $(0,2)$-type tensor on $S_{s,u}$, the trace of $\xi$ can be defined as
	\[
	\text{tr}\xi:=\slashed g^{AB}\xi_{AB}.
	\]
\end{definition}

\subsection{Basic equalities in the null frames}\label{p-1}

In this subsection, we will derive some basic equalities in the frame $\{\mathring L, \mathring{\underline L}, X_1, X_2, X_3\}$
or the frame $\{T, \mathring{\underline L}, X_1, X_2, X_3\}$.

Let
\begin{equation}\label{FG}
\begin{split}
&\ds F_{\al\beta}:=\p_{\phi} g_{\al\beta}\qquad\text{and}\qquad  G_{\al\beta}^\gamma:=\p_{\varphi_\gamma} g_{\al\beta}.
\end{split}
\end{equation}
For any vector fields $\ds U=U^\al\p_{\al}$ and $\ds V=V^\al\p_{\al}$, set $F_{UV}=F_{\al\beta}U^\al V^\beta$ and $G_{UV}^\gamma=G_{\al\beta}^\gamma U^\al V^\beta$ as in Definition \ref{def}.
In addition, define
\begin{equation}\label{FGX}
	(\mathcal{FG})_{UV}(W):=F_{UV}W\phi+G_{UV}^\gamma W\varphi_{\gamma},
\end{equation}
where $U$, $V$, $W$ are any vector fields.

We now derive the transport equation for $\mu$ along $\mathring L$.

\begin{lemma}\label{mu}
$\mu$ satisfies
 \begin{equation}\label{lmu}
 \begin{split}
 \mathring L\mu=-\f12\mu\big\{(\mathcal {FG})_{\mathring L\mathring L}(\mathring L)+2(\mathcal {FG})_{\tilde T\mathring L}({\mathring L})\big\}+\f12F_{\mathring L\mathring L}T\phi+\f12G_{\mathring L\mathring L}^\gamma T\varphi_\gamma.
 \end{split}
 \end{equation}
\end{lemma}

\begin{remark}
As will be seen later, special attention is required to treat the terms containing $T\varphi_\gamma$
since  $T\varphi_\gamma$ has the worse smallness or slower time-decay rate than the
ones of $Z\phi$ and $\bar Z\varphi_\g$,
where $Z\in\{\mathring L, T, X_1, X_2, X_3\}$ and $\bar Z\in\{\mathring L, X_1, X_2, X_3\}$.
\end{remark}

\begin{proof}
 Note that ${\tilde L}=-g^{\al\beta}\p_\al u\p_\beta$ is geodesic, i.e., $\mathscr D_{{\tilde L}}{\tilde L}=0$. Then it holds that
 \begin{equation}\label{ch}
 {\tilde L}^{\al}(\p_\al{\tilde L}^\beta)\p_{\beta}+{\tilde L}^\al{\tilde L}^\beta\Gamma_{\al\beta}^\gamma\p_\gamma=0,
 \end{equation}
 where $\Gamma_{\al\beta}^\gamma$ are Christoffel symbols with
 \begin{equation*}
 \begin{split}
 \Gamma_{\al\beta}^\gamma=\f12g^{\gamma\kappa}\big(F_{\kappa\beta}\p_\al\phi+G_{\kappa\beta}^\nu\p_\al\varphi_\nu
 +F_{\al\kappa}\p_\beta\phi+G_{\al\kappa}^\nu\p_\beta\varphi_\nu-F_{\al\beta}\p_\kappa\phi-G_{\al\beta}^\nu\p_\kappa\varphi_\nu\big).
 \end{split}
 \end{equation*}
By $\tilde L^0=\mu^{-1}$ and $\mathring L=\mu\tilde L$, taking the component relative to $x^0$ in \eqref{ch} yields
 \begin{equation}\label{Y-6}
 \begin{split}
 \mathring L\mu=&\mu\mathring L^\al\mathring L^\gamma\Gamma_{\al\gamma}^0\\
 =&-\f12\mu\big\{F_{\mathring L\mathring L}\mathring L\phi+G_{\mathring L\mathring L}^\gamma\mathring L\varphi_\gamma+2F_{\tilde T\mathring L}\mathring L\phi+2G_{\tilde{T}\mathring L}^\gamma\mathring L\varphi_\gamma\big\}\\
 &+\f12F_{\mathring L\mathring L}T\phi+\f12G_{\mathring L\mathring L}^\gamma T\varphi_\gamma,
 \end{split} \end{equation}
 here the fact $g^{0\kappa}=-\mathring L^\kappa-\tilde T^\kappa$ has been used.
\end{proof}

\begin{remark}\label{R5.2}
The important role played by \eqref{lmu} should be emphasized here. Since the first null condition holds in \eqref{quasi},
the term $\f12G_{\mathring L\mathring L}^\gamma T\varphi_\gamma$ in \eqref{lmu} will admit the better smallness and
faster time-decay rate
than the ones of the single factor $T\varphi_\gamma$. This will imply
$|\mathring L\mu|\lesssim M\delta^{1-\varepsilon_0}s^{-5/2}\mu+M^2\delta^{1-2\varepsilon_0}s^{-3/2}$
and further yield $\mu\geq C$ for some positive constant $C$ (see \eqref{phimu}
in Section \ref{BA}). Therefore, this is in sharp contrast to the case in \cite{C2}, \cite{MY} and \cite{J},
where the expressions of $\mathring L\mu$
contain the pure ``bad" factor $T\vp_\al$ which leads to $\mu\rightarrow 0+$ in finite time (e.g. (2.36) in \cite{MY}).
\end{remark}

Since the quantity ``deformation tensor" defined in \eqref{dt}
 will occur in the subsequent energy estimates, it is necessary  to compute the components
 of $\leftidx{^{(V)}}\pi$ in the null frame $\{\mathring L,\underline{\mathring L},X_1,X_2,X_3\}$.

Set $\leftidx{^{(V)}}{\slashed\pi}_{UA}:=\leftidx{^{(V)}}{\pi}_{UX_A}$ for $U\in\{\mathring L,\underline{\mathring L},X_1,X_2,X_3\}$.
By referring to the components in Proposition 7.7 of \cite{J}, one has

\begin{enumerate}[(1)]
	\item when $V=T$,
	\begin{equation}\label{Lpi}
	\begin{split}
	&\leftidx{^{(T)}}\pi_{\mathring L\mathring L}=0,\quad \leftidx{^{(T)}}\pi_{T\tilde T}=2T\mu,\quad \leftidx{^{(T)}}\pi_{\mathring LT}=-T\mu,\quad
	\leftidx{^{(T)}}{\slashed\pi}_{TA}=0,\\
	&\leftidx{^{(T)}}{\slashed\pi}_{\mathring LA}=-2\mu\zeta_A-\slashed d_A\mu,\quad \leftidx{^{(T)}}{\slashed\pi}_{AB}=2\mu\sigma_{AB};
	\end{split}
	\end{equation}
	
	\item when $V=\mathring L$,
	\begin{equation}\label{uLpi}
	\begin{split}
	&\leftidx{^{(\mathring L)}}\pi_{\mathring L\mathring L}=0,\quad \leftidx{^{(\mathring L)}}\pi_{T\tilde T}=2\mathring L\mu,\quad \leftidx{^{(\mathring L)}}\pi_{\mathring LT}=-\mathring L\mu,\quad \leftidx{^{(\mathring L)}}{\slashed\pi}_{\mathring LA}=0,\\
	&\leftidx{^{(\mathring L)}}{\slashed\pi}_{TA}=2\mu\zeta_A+\slashed d_A\mu,\quad \leftidx{^{(\mathring L)}}{\slashed\pi}_{AB}=2\chi_{AB};
	\end{split}
	\end{equation}
	
	\item when $V=R_{ij}$,
	\begin{equation}\label{Rpi}
	\begin{split}
	&\leftidx{^{(R_{ij})}}\pi_{\mathring L\mathring L}=0,\quad \leftidx{^{(R_{ij})}}\pi_{T\tilde T}=2R_{ij}\mu,\quad \leftidx{^{(R_{ij})}}\pi_{\mathring LT}=-R_{ij}\mu,\\
	&\leftidx{^{(R_{ij})}}{\slashed\pi}_{\mathring LA}=-{R_{ij}}^B\check{\chi}_{AB}+\f12{R_{ij}}^B\{(\mathcal{FG})_{AB}({\mathring L})-(\mathcal{FG})_{\mathring LB}({X_A})\}\\
	&\qquad\qquad\quad-\upsilon_{ij}\{(\mathcal {FG})_{\mathring L\tilde T}({X_A})+\f12(\mathcal{FG})_{\tilde T\tilde T}({X_A})-(\mathcal{FG})_{A\tilde T}({\mathring L})\}\\
	&\qquad\qquad\quad+\f12(\mathcal{FG})_{\mathring LA}({R_{ij}})+g_{ja}\check L^i\slashed d_Ax^a-g_{ia}\check L^j\slashed d_Ax^a,\\
	&\leftidx{^{(R_{ij})}}{\slashed\pi}_{TA}=\mu{R_{ij}}^B\check\chi_{AB}+\upsilon_{ij}\slashed d_A\mu+\f12\mu{R_{ij}}^B\{(\mathcal{FG})_{B\mathring L}({X_A})-(\mathcal{FG})_{AB}({\mathring L})\}\\
	&\qquad\qquad\quad+\f12\mu(\mathcal{FG})_{A\mathring L}({R_{ij}})+\mu(\mathcal{FG})_{A\tilde T}({R_{ij}})+\upsilon_{ij}\{(\mathcal{FG})_{A\tilde T}(T)\\
	&\qquad\qquad\quad-\f12\mu(\mathcal{FG})_{\tilde T\tilde T}(X_A)\}+\mu g_{ja}\check T^i\slashed d_Ax^a
-\mu g_{ia}\check T^j\slashed d_Ax^a,\\
	&\leftidx{^{(R_{ij})}}{\slashed\pi}_{AB}=2\upsilon_{ij}\chi_{AB}+\upsilon_{ij}\{(\mathcal{FG})_{B\mathring L}({X_A})
+(\mathcal{FG})_{B\tilde T}({X_A})+(\mathcal{FG})_{A\mathring L}({X_B})\\
	&\qquad\qquad\quad+(\mathcal{FG})_{A\tilde T}({X_B})-(\mathcal{FG})_{AB}({\mathring L})\}+(\mathcal{FG})_{AB}({R_{ij}})
+\check g_{ja}(\slashed d_Ax^i)\slashed d_Bx^a\\
	&\qquad\qquad\quad-\check g_{ia}(\slashed d_Ax^j)\slashed d_Bx^a+\check g_{ja}(\slashed d_Bx^i)\slashed d_Ax^a
-\check g_{ia}(\slashed d_Bx^j)\slashed d_Ax^a,
	\end{split}
	\end{equation}
where $\check g_{ia}:=g_{ia}-m_{ia}$,  the definitions of $\t T$ and $\check T^i$, $\check L^i$ are given in \eqref{FC-1}
and \eqref{errorv} respectively.	
\end{enumerate}

It follows from the equation \eqref{lmu} and  \eqref{Lpi}-\eqref{Rpi} that
the vector fields $\mathring L$ and $T$ appear frequently. In view of $\tilde T^i=-g^{0i}-\mathring L^i$
and $T=\mu \tilde T$,
it is important to derive the equations for $\mathring L^i$ and $\check L^i$ in the null frame $\{T,\mathring L, X_1, X_2, X_3\}$.

\begin{lemma}\label{4.2}
	It holds that
	\begin{align}
	&\mathring L\mathring L^i=\f12(\mathcal{FG})_{\mathring L\mathring L}({\mathring L})\tilde T^i
-\big\{(\mathcal{FG})_{A\mathring L}({\mathring L})-\f12(\mathcal{FG})_{\mathring L\mathring L}({X_A})\big\}\slashed d^Ax^i,\label{LL}\\
	&\mathring L\big(\varrho\check L^i\big)=\varrho\mathring L\mathring L^i,\label{LeL}\\
	&\slashed d_A\mathring L^i=\chi_{AB}{\slashed d^B}x^i-\{(\mathcal{FG})_{\mathring L\tilde T}({X_A})
+\f12(\mathcal{FG})_{\tilde T\tilde T}({X_A})\}\tilde T^i+\Lambda_{AB}\slashed d^Bx^i,\label{dL}\\
	&\slashed d_A\check{L}^i=\check\chi_{AB}{\slashed d^B}x^i-\{(\mathcal{FG})_{\mathring L\tilde T}({X_A})
+\f12(\mathcal{FG})_{\tilde T\tilde T}({X_A})\}\tilde T^i+\Lambda_{AB}\slashed d^Bx^i,\label{deL}
	\end{align}
	where
	\begin{equation}\label{Lambda}
	\Lambda_{AB}:=-\f12(\mathcal{FG})_{B\mathring L}({X_A})-\f12(\mathcal{FG})_{AB}({\mathring L})+\f12(\mathcal{FG})_{A\mathring L}({X_B}).
	\end{equation}
\end{lemma}

\begin{proof}
\eqref{LL}-\eqref{deL} follow directly from Proposition 4.7 in \cite{J}and \eqref{FG}.
\end{proof}

With the help of Lemma \ref{4.2}, following the proof of Lemma 5.1 in \cite{J}, one can
get the explicit expressions of $\zeta_A$ and $\sigma_{AB}$ as
\begin{align}
&\zeta_A=-\mu^{-1}\slashed d_A\mu+\mu^{-1}\eta_A\no\\
&\quad=-\f12\Big\{(\mathcal{FG})_{\tilde T\mathring L}({X_A})+(\mathcal{FG})_{\tilde T\tilde T}({X_A})-(\mathcal{FG})_{A\tilde T}({\mathring L})+(\mathcal {FG})_{A\mathring L}({\tilde T})\Big\},\label{zeta}\\
&\sigma_{AB}=-\f12\Big\{(\mathcal {FG})_{B\mathring L}({X_A})+(\mathcal {FG})_{B\tilde T}({X_A})+(\mathcal {FG})_{A\mathring L}({X_B})
+(\mathcal {FG})_{A\tilde T}({X_B})\no\\
&\qquad\qquad-(\mathcal{FG})_{AB}({\mathring L})-(\mathcal{FG})_{AB}({\tilde T})\Big\}-\chi_{AB}.\label{theta}
\end{align}
Then it is deduced from \eqref{zeta} that
\begin{equation}\label{TL}
\begin{split}
&T\mathring{L}^i=\{\slashed d_A\mu-\f12\mu(\mathcal{FG})_{\tilde T\tilde T}({X_A})
-(\mathcal{FG})_{A\mathring L}(T)\}\slashed d^Ax^i+\big\{\f12\mu(\mathcal{FG})_{\mathring L\mathring L}({\mathring L})\\
&\qquad\quad+\mu(\mathcal{FG})_{\tilde T\mathring L}({\mathring L})+\f12\mu(\mathcal{FG})_{\tilde T\tilde T}({\mathring L})\big\}\mathring L^i+\f12(\mathcal{FG})_{\mathring L\mathring L}(T)\tilde T^i.
\end{split}
\end{equation}

For later use, one also needs to give the connection coefficients of the related frames.
This can be derived from Lemma 5.1 and Lemma 5.3 of \cite{J}.

\begin{lemma}\label{cd}
The covariant derivatives in the frame $\{T, \mathring L, X_1, X_2, X_3\}$ are
\begin{equation}\label{cdf}
\begin{split}
&\mathscr D_{\mathring L}\mathring L=(\mu^{-1}\mathring L\mu)\mathring L,\quad\mathscr D_{T}\mathring L=-\mathring L\mu\mathring L+\eta^AX_A,\quad\mathscr D_A\mathring L=-\zeta_A\mathring L+{\chi_A}^BX_B,\\
&\mathscr D_{\mathring L}T=-\mathring L\mu\mathring L-\mu\zeta^AX_A,\quad\mathscr D_TT=\mu\mathring L\mu\mathring L
+(\mu^{-1}T\mu+\mathring L\mu)T-\mu(\slashed d^A\mu) X_A,\\
&\mathscr D_AT=\mu\zeta_A\mathring L+\mu^{-1}\eta_AT+\mu{\sigma_A}^BX_B,\\
&\mathscr D_AX_B=\slashed{\nabla}_AX_B+(\sigma_{AB}+\chi_{AB})\mathring L+\mu^{-1}\chi_{AB}T.
\end{split}
\end{equation}
\end{lemma}

\begin{lemma}
The covariant derivatives in the frame $\{\mathring{\underline L}, \mathring L, X_1, X_2, X_3\}$ are
\begin{equation}\label{LuL}
\begin{split}
&\mathscr D_{\mathring {\underline L}}{\mathring L}=-\mathring L\mu\mathring L+2\eta^AX_A,\quad\mathscr D_{\mathring L}\mathring{\underline L}=-2\mu\zeta^AX_A,\\
&\mathscr D_{\mathring{\underline L}}\mathring{\underline L}=(\mu^{-1}\mathring{\underline L}\mu+\mathring
 L\mu)\mathring{\underline L}-(2\mu \slashed d^A\mu) X_A.
\end{split}
\end{equation}
\end{lemma}

In order to estimate $\varphi_\gamma$ and $\phi$, by the energy method, we need to
derive the equations for $\phi$ and $\varphi_\gamma$ under the action of the covariant wave operator $\Box_g=g^{\al\beta}\mathscr{D}^2_{\al\beta}=g^{\al\beta}\p_{\al\beta}^2-g^{\al\beta}\Gamma_{\al\beta}^\lambda\p_\lambda$.
\begin{enumerate}
\item For any $\gamma \in \{0,1,2,3,4\}$,
taking the derivative $\p_{x^\gamma}$ on two sides of \eqref{quasi} yields
\begin{equation}\label{gphi}
\begin{split}
g^{\al\beta}\p_{\al\beta}^2\varphi_\gamma=-\big\{(\p_\phi g^{\al\beta})\p_\gamma\phi+(\p_{\varphi_\lambda}g^{\al\beta})\p_\lambda\varphi_{\gamma}\big\}\p_{\al\beta}^2\phi.
\end{split}
\end{equation}

Note that
\begin{equation}\label{ggg}
g^{\al\beta}g^{\lambda\kappa}(\p_\phi g_{\kappa\beta})=-\p_\phi g^{\al\lambda},\quad g^{\al\beta}g^{\lambda\kappa}(\p_{\varphi_\gamma} g_{\kappa\beta})=-\p_{\varphi_\gamma} g^{\al\lambda}.
\end{equation}
Then
\begin{equation}\label{box}
\begin{split}
-g^{\al\beta}\Gamma_{\al\beta}^\lambda\p_\lambda\varphi_\gamma=&(\p_\phi g^{\al\beta})\p_\al\phi\p_\gamma\varphi_\beta+(\p_{\varphi_\kappa}g^{\al\beta})\p_\kappa\varphi_\al\p_\beta\varphi_\gamma\\
&+\f12(g^{\al\beta}\p_\phi g_{\al\beta})g^{\lambda\kappa}\p_\kappa\phi\p_\lambda\varphi_\gamma+\f12(g^{\al\beta}\p_{\varphi_\nu}
g_{\al\beta})g^{\lambda\kappa}\p_\kappa\varphi_\nu\p_\lambda\varphi_\gamma.
\end{split}
\end{equation}
Thus,
\begin{equation}\label{boxg}
\begin{split}
\Box_g\varphi_\gamma=&\p_\phi g^{\al\beta}\big(-\p_\gamma\phi\p_\al\varphi_\beta+\p_\al\phi\p_\gamma\varphi_\beta\big)
+\p_{\varphi_\nu}g^{\al\beta}\big(-\p_\nu\varphi_\gamma\p_\beta\varphi_\al+\p_\nu\varphi_\al\p_\beta\varphi_\gamma\big)\\
&+\f12(g^{\al\beta}\p_\phi g_{\al\beta})g^{\lambda\kappa}\p_\kappa\phi\p_\lambda\varphi_\gamma
+\f12(g^{\al\beta}\p_{\varphi_\nu}g_{\al\beta})g^{\lambda\kappa}\p_\kappa\varphi_\nu\p_\lambda\varphi_\gamma.
\end{split}
\end{equation}
In addition, it holds that
\begin{align}
&g^{\al\beta}=-\mathring L^\al\mathring L^\beta-\tilde T^\al\mathring L^\beta-\mathring L^\al\tilde T^\beta+\slashed g^{AB}(\slashed d_Ax^\al)(\slashed d_Bx^\beta),\label{gab}\\
&\p_\al=\delta_\al^0\mathring L-\mu^{-1}\mathring L_\al T+g_{\al i}(\slashed d^Ax^i)X_A.\label{pal}
\end{align}
Then in the frame $\{T, \mathring L, X_1, X_2, X_3\}$, \eqref{boxg} can be rewritten as
\begin{equation}\label{ge}
\begin{split}
\mu\Box_g \varphi_\gamma=&\f12(g^{\al\beta}\p_\phi g_{\al\beta})(\mu\slashed d^A\phi)\slashed d_A\varphi_\gamma+\f12(g^{\al\beta}\p_{\varphi_\nu}g_{\al\beta})(\mu\slashed d^A\varphi_\nu)\slashed d_A\varphi_\gamma\\
&+f(\phi, \varphi, \mathring L^i)\left(
\begin{array}{ccc}
\mathring L\phi\\
(\slashed d_Ax^j)\slashed d^A\phi\\
\mathring L\varphi\\
(\slashed d_Ax^j)\slashed d^A\varphi\\
\end{array}
\right)
\left(
\begin{array}{ccc}
T\phi\\
T\varphi\\
\mu\mathring L\varphi\\
\mu (\slashed d_Bx^k)\slashed d^B\varphi\\
\end{array}
\right),
\end{split}
\end{equation}
where and below $f(\phi, \varphi, \mathring L^i)$ represents the generic smooth function with respect to
$\phi$, $\vp$, $\mathring L^1$, $\mathring L^2$, $\mathring L^3$ and $\mathring L^4$,
$f(\phi, \varphi, \mathring L^i)\left(
\begin{array}{ccc}
A_1\\
\vdots\\
A_n\\
\end{array}
\right)
\left(
\begin{array}{ccc}
B_1\\
\vdots\\
B_m\\
\end{array}
\right)$
stands for $\sum_{1\leq k\leq n,1\leq l\leq m}f_{kl}(\phi, \varphi, \mathring L^i)A_kB_l$.

An important observation here is that on the right hand side of \eqref{ge}, the worst factor $T\vp$ is always accompanied by
the ``good" multipliers such as  $\mathring L\phi$, $\mathring L\varphi$, $\slashed d\phi$ and $\slashed d\varphi$
which admit higher time-decay rates or orders of smallness.

\item
For the solution $\phi$ to \eqref{quasi}, it holds that
\begin{equation}\label{bphi}
\Box_g\phi=(\p_\al g^{\al\lambda})\varphi_\lambda+\f12 \big(g^{\al\beta}F_{\al\beta}\big)\big(g^{\kappa\lambda}\varphi_\kappa\varphi_\lambda\big)+\f12 \big(g^{\al\beta}G_{\al\beta}^\gamma\big)\big(g^{\kappa\lambda}\varphi_\lambda\p_\kappa\varphi_\gamma\big).
\end{equation}
Next one can treat each term on the right hand side of \eqref{bphi} as follows.

(a)	Note that
\begin{equation}\label{pgv}
	\begin{split}
	(\p_\al g^{\al\lambda})\varphi_\lambda=&(\mathring Lg^{0\gamma})\varphi_\gamma
-\mu^{-1}\mathring L_\al(Tg^{\al\lambda})\varphi_\lambda+g_{\al i}\slashed d^Ax^i(\slashed d_Ag^{\al\lambda})\varphi_\lambda\\
	=&(\mathring Lg^{0\gamma})\varphi_\gamma-\mu^{-1}\mathring L_\al(Tg^{\al\lambda})\varphi_\lambda
-g^{\al\lambda}\slashed d^Ax^i(\slashed d_Ag_{\al i})\varphi_\lambda
	\end{split}
	\end{equation}
Due to \eqref{pal}. In addition, it follows from \eqref{ggg} and \eqref{gab} that
	\begin{equation}\label{Lgv}
	\begin{split}
	&(\mathring Lg^{0\gamma})\varphi_\gamma=\big\{-(g^{0\beta}g^{\g\kappa}F_{\kappa\beta})\mathring L\phi-(g^{0\beta}g^{\g\kappa}G_{\kappa\beta}^\nu)\mathring L\varphi_\nu\big\}\varphi_{\gamma}\\
	=&-\{(\mathcal{FG})_{\mathring L\mathring L}({\mathring L})+2(\mathcal{FG})_{\mathring L\tilde T}({\mathring L})
+(\mathcal{FG})_{\tilde T\tilde T}({\mathring L})\}\mathring L\phi-\{(\mathcal{FG})_{\mathring L\mathring L}({\mathring L})\\
	&+(\mathcal{FG})_{\mathring L\tilde T}({\mathring L})\}\tilde T\phi+\{(\mathcal{FG})_{A\mathring L}({\mathring L})
+(\mathcal{FG})_{A\tilde T}({\mathring L})\}\slashed d^A\phi,
	\end{split}
	\end{equation}
	\begin{equation}
	\begin{split}
	&\mathring L_\al(Tg^{\al\lambda})\varphi_\lambda=-g^{\al\lambda}\mathring L^\beta(Tg_{\al\beta})\varphi_\lambda\\
	=&(\mathcal{FG})_{\mathring L\mathring L}(T)(\mathring L\phi+\tilde T\phi)+(\mathcal{FG})_{\tilde T\mathring L}(T)\mathring L\phi-(\mathcal{FG})_{A\mathring L}(T)\slashed d^A\phi\qquad\quad
	\end{split}
	\end{equation}
and
	\begin{equation}\label{gddv}
	\begin{split}
	g^{\al\lambda}\slashed d^Ax^i(\slashed d_Ag_{\al i})\varphi_\lambda
	=&-\slashed g^{AB}(\mathcal{FG})_{\mathring LA}({X_B})(\mathring L\phi+\tilde T\phi)\\
	&-\slashed g^{AB}(\mathcal{FG})_{\tilde TA}({X_B})\mathring L\phi+\slashed g^{AC }(\mathcal{FG})_{AB}({X_C})\slashed d^B\phi.
	\end{split}
	\end{equation}
Substituting \eqref{Lgv}-\eqref{gddv} into \eqref{pgv} yields
	\begin{equation}\label{pgva}
	\begin{split}
	&(\p_\al g^{\al\lambda})\varphi_\lambda\\
	=&-G_{\mathring L\mathring L}^\gamma(\tilde T\varphi_\gamma)\mathring L\phi
-G_{\tilde T\mathring L}^\gamma(\tilde T\varphi_\gamma)\mathring L\phi+G^\g_{A\mathring L}(\tilde T\varphi_\gamma)\slashed d^A\phi\\
	&-F_{\mathring L\mathring L}(\tilde T\phi)^2-G_{\mathring L\mathring L}^\gamma(\tilde T\varphi_\gamma)\tilde T\phi
+f(\phi, \varphi, \slashed dx^i,\mathring L^i)\left(
	\begin{array}{ccc}
	\mathring L\phi\\
	\slashed d\phi\\
	\tilde T\phi
	\end{array}
	\right)
	\left(
	\begin{array}{ccc}
	\mathring L\phi\\
	\slashed d\phi\\
	\mathring L\varphi\\
	\slashed d\varphi\\
	\end{array}
	\right).
	\end{split}
	\end{equation}

(b) Similarly, applying the equality \eqref{gab} to $g^{\al\beta}$ and $g^{\kappa\lambda}$ yields
	\begin{equation}\label{gFgv}
	\begin{split}
	&\f12 \big(g^{\al\beta}F_{\al\beta}\big)\big(g^{\kappa\lambda}\varphi_\kappa\varphi_\lambda\big)+\f12 \big(g^{\al\beta}G_{\al\beta}^\gamma\big)\big(g^{\kappa\lambda}\varphi_\lambda\p_\kappa\varphi_\gamma\big)\\
	=&\f12G_{\mathring L\mathring L}^\gamma(\tilde T\varphi_\gamma)\mathring L\phi
+G_{\mathring L\tilde T}^\gamma(\tilde T\varphi_\gamma)\mathring L\phi
-\f12\slashed g^{AB}G_{AB}^\gamma(\tilde T\varphi_\gamma)\mathring L\phi\\
	&+f(\phi, \varphi, \slashed dx^i,\mathring L^i)\left(
	\begin{array}{ccc}
	\mathring L\phi\\
	\slashed d\phi\\
	\tilde T\phi
	\end{array}
	\right)
	\left(
	\begin{array}{ccc}
	\mathring L\phi\\
	\slashed d\phi\\
	\mathring L\varphi\\
	\slashed d\varphi\\
	\end{array}
	\right).
	\end{split}
	\end{equation}
Thus, by putting \eqref{pgva}-\eqref{gFgv} into \eqref{bphi}, one can get
	\begin{equation}\label{muphi}
	\begin{split}
	&\mu\Box_g\phi\\
	=&-\f12G_{\mathring L\mathring L}^\gamma(T\varphi_\gamma)\mathring L\phi+G_{\mathring LA}^\gamma(T\varphi_\gamma)\slashed d^A\phi
-\f12\slashed g^{AB}G_{AB}^\gamma(T\varphi_\gamma)\mathring L\phi-F_{\mathring L\mathring L}(T\phi)\tilde T\phi\\
	&-G_{\mathring L\mathring L}^\gamma(T\varphi_\gamma)\tilde T\phi+f(\phi, \varphi, \slashed dx^i,\mathring L^i)\left(
	\begin{array}{ccc}
	\mu\mathring L\phi\\
	\mu\slashed d\phi\\
	T\phi
	\end{array}
	\right)
	\left(
	\begin{array}{ccc}
	\mathring L\phi\\
	\slashed d\phi\\
	\mathring L\varphi\\
	\slashed d\varphi\\
	\end{array}
	\right).
	\end{split}
	\end{equation}
\end{enumerate}

As explained in Remark \ref{R5.2}, $G_{\mathring L\mathring L}^\gamma(T\varphi_\gamma)$
will admit higher order smallness and faster time-decay rate than the ones for $T\varphi_\gamma$
because of the null condition \eqref{null}. Thus, the term $-G_{\mathring L\mathring L}^\gamma(T\varphi_\gamma)\tilde T\phi$
in the right hand side of \eqref{muphi} does not give rise to new  difficulties for the proof of global existence
in $A_{2\dl}$.

The following two Lemmas are needed to deal with commutators, which can be found from Lemmas 4.10, 8.9 and 8.11 in \cite{J}.

\begin{lemma}\label{com}
	In the frame $\{\mathring L,T,R_{ij}\}$, the following identities for the commutators hold:
	\begin{equation}\label{c}
	\begin{split}
	[\mathring L, R_{ij}]&={\leftidx{^{(R_{ij})}}{\slashed\pi}_{\mathring L}}^AX_A,\\
	[\mathring L, T]&={\leftidx{^{(T)}}{\slashed\pi}_{\mathring L}}^AX_A,\\
	[T,R_{ij}]&={\leftidx{^{(R_{ij})}}{\slashed\pi}_{T}}^AX_A.
	\end{split}
	\end{equation}
\end{lemma}

\begin{lemma}\label{commute}
For any vector field $Z\in\{\mathring L,T,R_{ij}\}$,

(a) if $f$ is a smooth function, then
 \begin{align*}
 \big([\slashed\nabla^2,\slashed{\mathcal L}_Z]f\big)_{AB}&
 =\big(\check{\slashed\nabla}_A\leftidx{^{(Z)}}{\slashed\pi_B}^C\big)\slashed d_Cf,\\
 [\slashed\triangle, Z]f&=\leftidx{^{(Z)}}{\slashed\pi}^{AB}\slashed\nabla_{AB}^2f
 +(\check{\slashed\nabla}_A\leftidx{^{(Z)}}{\slashed\pi}^{AB})\slashed d_Bf;
 \end{align*}

(b) if $\xi$ is a one-form on $S_{s, u}$, then
\[
([\slashed\nabla_A,\slashed{\mathcal L}_Z]\xi)_B=(\check{\slashed\nabla}_A\leftidx{^{(Z)}}{\slashed\pi_B}^C)\xi_C;
\]

(c) if  $\xi$ is a $(0,2)$-type symmetric tensor on $S_{s, u}$, then
\begin{align*}
([\slashed\nabla_A,\slashed{\mathcal L}_Z]\xi)_{BC}&=(\check{\slashed\nabla}_A\leftidx{^{(Z)}}{\slashed\pi_B}^D)\xi_{CD}
+(\check{\slashed\nabla}_A\leftidx{^{(Z)}}{\slashed\pi_C}^D)\xi_{BD};
\end{align*}

(d) if $\xi$ is a $(0,3)$-type tensor on $S_{s,u}$, then
\[
([\slashed\nabla_A, \slashed{\mathcal L}_Z]\xi)_{BCD}=(\check{\slashed\nabla}_A\leftidx{^{(Z)}}{\slashed\pi_B}^E)\xi_{ECD}
+(\check{\slashed\nabla}_A\leftidx{^{(Z)}}{\slashed\pi_C}^E)\xi_{BED}+(\check{\slashed\nabla}_A\leftidx{^{(Z)}}{\slashed\pi_D}^E)\xi_{BCE},
\]
where
\[
\check{\slashed\nabla}_A\leftidx{^{(Z)}}{\slashed\pi_{BC}}:=\f12\slashed\nabla_A\leftidx{^{(Z)}}{\slashed\pi_{BC}}
+\f12\slashed\nabla_B\leftidx{^{(Z)}}{\slashed\pi_{AC}}-\f12\slashed\nabla_C\leftidx{^{(Z)}}{\slashed\pi_{AB}}.
\]
\end{lemma}

\section{Bootstrap assumptions on $(\phi,\p\phi)$  near $C_0$ and some related estimates }\label{BA}

To show the global existence of the solution $\phi$ to \eqref{quasi} near $C_0$,
we will a bootstrap argument. To this end, the following bootstrap assumptions
are made in $D^{s,u}$:
\begin{equation*}
\begin{split}
&\delta^l\|\mathring LZ^\al\varphi_\gamma\|_{L^\infty(\Sigma_s^{u})}\leq  M\delta^{1-\varepsilon_0} s^{-5/2},\\
&\delta^l\|\slashed dZ^\al\varphi_\gamma\|_{L^\infty(\Sigma_s^{u})}\leq  M\delta^{1-\varepsilon_0} s^{-5/2},\\
&\delta^l\|\mathring{\underline L}Z^\al\varphi_\gamma\|_{L^\infty(\Sigma_s^{u})}\leq  M\delta^{-\varepsilon_0} s^{-3/2},\\
&\|\slashed\nabla^2\varphi_\gamma\|_{L^\infty(\Sigma_s^{u})}\leq  M\delta^{1-\varepsilon_0} s^{-7/2},\\
&\|\varphi_\gamma\|_{L^\infty(\Sigma_s^{u})}\leq  M\delta^{1-\varepsilon_0} s^{-3/2},\\
&\delta^l\|Z^\beta\phi\|_{L^\infty(\Sigma_s^{u})}\leq M\delta^{2-\varepsilon_0}s^{-3/2},\\
&\|\slashed d\phi\|_{L^\infty(\Sigma_s^{u})}\leq M\delta^{2-\varepsilon_0}s^{-5/2},\\
&\|\slashed\nabla^2\phi\|_{L^\infty(\Sigma_s^{u})}\leq M\delta^{2-\varepsilon_0}s^{-7/2},
\end{split}\tag{$\star$}
\end{equation*}
where $|\al|\leq N$, $|\beta|\leq N+1$,  $N$ is a fixed large positive integer, $M$ is some positive number to
be suitably chosen
(at least double bounds of the corresponding quantities on time $t_0$),
$Z\in\{\varrho\mathring L,T,R_{ij}\}$, and $l$ is the number of $T$ in $Z^\al$ or $Z^\beta$.

We now give a rough estimate of $\mu$ under the assumptions ($\star$). Note that
$1=g_{ij}\tilde T^i\tilde T^j=\big(1+O(M\delta^{1-\varepsilon_0} s^{-3/2})\big)\ds\sum_{i=1}^4|\tilde T^i|^2$
by \eqref{LTTT}.  This means
\begin{equation}\label{L}
|\tilde T^i|, |\mathring L^i|\leq 1+O( M\delta^{1-\varepsilon_0} s^{-3/2})
\end{equation}
due to $\mathring L^i=-g^{0i}-\tilde T^i$.
Moreover, by $|\slashed dx^i|^2=\slashed g^{ab}\slashed d_ax^i\slashed d_bx^i=g^{ii}+(g^{0i})^2-(\tilde T^i)^2$, $(\star)$ and \eqref{L},
it holds that
\begin{equation}\label{dx}
|\slashed dx^i|\lesssim 1.
\end{equation}
Then $|\mathring L(\varrho\check L^i)|\lesssim M\delta^{1-\varepsilon_0}s^{-3/2}$ holds by \eqref{LeL} and $(\star)$, which further gives
 \begin{equation}\label{chL}
 |\check L^i|+|\check T^i|\lesssim M\delta^{1-\varepsilon_0}s^{-1}
 \end{equation}
through integrating $\mathring L(\varrho\check L^i)$ along integral curves of $\mathring L$ and
using $\check T^i=-g^{0i}-\check L^i$.
In addition, by $\ds g_{ij}(\check T^i-\f{x^i}{\varrho})(\check T^j-\f{x^j}{\varrho})=1$, one has
 \[
 (g_{ij}\o^i\o^j)\f{r^2}{\varrho^2}-(2g_{ij}\check T^i\o^j)\f r\varrho+g_{ij}\check T^i\check T^j-1=0.
 \]
 Thus
 \begin{equation}\label{rrho}
 \check\varrho:=\f{r}{\varrho}-1=\f{1-g_{ij}\o^i\o^j-g_{ij}\check T^i\check T^j+2g_{ij}\check T^i\o^j}{\sqrt{g_{ij}\o^i\o^j-(g_{ij}\o^i\o^j)(g_{ab}\check T^a\check T^b)+(g_{ij}\check T^i\o^j)^2}
 +g_{ij}\o^i\o^j-g_{ij}\check T^i\o^j}.
 \end{equation}
It follows from \eqref{chL}  and \eqref{rrho} that
 \begin{equation}\label{cr}
 |\check\varrho|\lesssim M\delta^{1-\varepsilon_0}s^{-1},
 \end{equation}
 and hence  by \eqref{omega}, one has
 \begin{equation}\label{Ri}
 |\upsilon_{ij}|\lesssim\delta^{1-\varepsilon_0}M.
 \end{equation}

In order to estimate $\mu$, we pay special attention to the last term in \eqref{lmu} since it contains the factor $T\varphi_\gamma$
that seems to have the worse order of smallness and the slower time-decay rate due to $(\star)$.
By \eqref{pal} and \eqref{ggg}, one can arrive at
\begin{equation}\label{GT}
\begin{split}
G_{\mathring L\mathring L}^\gamma T\varphi_\gamma=&G_{\mathring L\mathring L}^0T^i(\mathring L\varphi_i)
-(\p_{\varphi_\gamma}g)_{\mathring L\mathring L}\mathring L_\gamma\tilde T^i T\varphi_i+g_{\gamma i}G_{\mathring L\mathring L}^\gamma (\slashed d^Ax^i)T^j\slashed d_A\varphi_j\\
=&G_{\mathring L\mathring L}^0T^i(\mathring L\varphi_i)
+\underline{(\p_{\varphi_\gamma}g^{\al\beta})\mathring L_\al\mathring L_\beta\mathring L_\gamma}\tilde T^i T\varphi_i
+g_{\gamma i}G_{\mathring L\mathring L}^\gamma (\slashed d^Ax^i)T^j\slashed d_A\varphi_j.
\end{split}
\end{equation}
Due to the first null condition \eqref{null},
the underline factor in \eqref{GT} can be written as
\begin{equation}\label{gLLL}
\begin{split}
&(\p_{\varphi_\gamma}g^{\al\beta})\mathring L_\al\mathring L_\beta\mathring L_\gamma\\
=&\{(\p_{\varphi_\gamma}g^{\al\beta})-\tilde g^{\al\beta,\gamma}\}\mathring L_\al\mathring L_\beta\mathring L_\gamma+\tilde g^{\al\beta,\gamma}(g_{\kappa\al}g_{\lambda\beta}g_{\nu\gamma}-m_{\kappa\al}m_{\lambda\beta}m_{\nu\gamma})\mathring L^\kappa\mathring L^\lambda\mathring L^\nu\\
&+\tilde g^{\al\beta,\gamma}m_{\kappa\al}m_{\lambda\beta}m_{\nu\gamma}\big\{(\check L^\kappa
+\f{\tilde x^\kappa}{\varrho})(\check L^\lambda+\f{\tilde x^\lambda}{\varrho})(\check L^\nu+\f{\tilde x^\nu}{\varrho})-\f{\tilde x^\kappa}{\varrho}\f{\tilde x^\lambda}{\varrho}\f{\tilde x^\nu}{\varrho}\big\}\\
&+\tilde g^{\al\beta,\gamma}\big\{(\o_\al+m_{\al i}\f{\check\varrho x^i}{(\check\varrho+1)\varrho})(\o_\beta+m_{\beta j}\f{\check\varrho x^j}{(\check\varrho+1)\varrho})(\o_\gamma+m_{\gamma k}\f{\check\varrho x^k}{(\check\varrho+1)\varrho})-\o_\al\o_\beta\o_\gamma\big\},
\end{split}
\end{equation}
where $(\tilde x^0, \tilde x^1, \tilde x^2, \tilde x^3, \tilde x^4)=(1, x^1, x^2, x^3, x^4)$
and $(\o_0, \o_1, \o_2, \o_3, \o_4)=(-1, \o^1, \o^2, \o^3, \o^4)$. Therefore,
$|(\p_{\varphi_\gamma}g^{\al\beta})\mathring L_\al\mathring L_\beta\mathring L_\gamma|\lesssim M\delta^{1-\varepsilon_0}s^{-1}$
holds by \eqref{L}-\eqref{chL}, \eqref{cr} and $(\star)$.
Subsequently, it follows from \eqref{GT} that
\begin{equation}\label{GTe}
|G_{\mathring L\mathring L}^\gamma T\varphi_\gamma|\lesssim \mu M\delta^{1-\varepsilon_0}s^{-5/2} +M^2\delta^{1-2\varepsilon_0}s^{-5/2}.
\end{equation}
This, together with $(\star)$ and \eqref{lmu},  implies that
 $|\mathring L\mu|\lesssim M\delta^{1-\varepsilon_0}s^{-5/2}\mu+M^2\delta^{1-2\varepsilon_0}s^{-3/2}$. When $\delta>0$ is small, by
integrating $\mathring L\mu$ along integral curves of $\mathring L$ and noting $\ds\mu=\f{1}{\sqrt{(g^{0i}\o_i)^2+g^{ij}\o_i\o_j}}
=1+O(\delta^{1-\varepsilon_0})$ on $\Sigma_{t_0}$,
one can get
\begin{equation}\label{phimu}
\mu=1+O( M^2\delta^{1-2\varepsilon_0}).
\end{equation}

To improve the assumptions ($\star$) and close the bootstrap arguments, we rewrite the equation \eqref{ge} in the new frame
$\{\mathring{\underline L}, \mathring L, X_1, X_2, X_3\}$ as
\begin{equation}\label{fequation}
\mathring L\mathring{\underline L}\varphi_\gamma+\f{3}{2\varrho}\mathring{\underline L}\varphi_\gamma
=\mu\slashed\triangle\varphi_\gamma+H_\gamma,
\end{equation}
here one has used the fact $\mu\Box_g\varphi_\g=-\mathscr D_{\mathring L\mathring {\underline L}}^2\varphi_\g
+\mu\slashed g^{AB}\mathscr D_{AB}^2\varphi_\g=-\mathring L\mathring{\underline L}\varphi_\g-2\mu\zeta^A(\slashed d_A\varphi_\g)+\mu\slashed\triangle\varphi_\g-\mu(\textrm{tr}\sigma+\textrm{tr}\chi)\mathring L\varphi_\g
-\textrm{tr}\chi T\varphi_\g$
due to \eqref{cdf} and \eqref{LuL}. In addition, by \eqref{zeta}-\eqref{theta},
\begin{equation}\label{H}
\begin{split}
H_\gamma=&-\f12\textrm{tr}\check\chi(\mathring {\underline L}\varphi_\gamma-\mu\mathring L\varphi_\gamma)+\f{3}{2\varrho}\mu\mathring L\varphi_\gamma\\
&+f_1(\phi, \varphi, \mathring L^i)\left(
\begin{array}{ccc}
\mu\mathring L\phi\\
\mu\mathring L\varphi\\
\mu(\slashed d_Ax^j)\slashed d^A\varphi\\
\mathring{\underline L}\phi\\
\mathring{\underline L}\varphi
\end{array}
\right)
\left(
\begin{array}{ccc}
\mathring L\phi\\
\mathring L\varphi\\
(\slashed d_Bx^k)\slashed d^B\phi\\
(\slashed d_Bx^k)\slashed d^B\varphi\\
\end{array}
\right)
+f_2(\phi, \varphi, \mathring L^i)\slashed d^A\varphi\left(\begin{array}{ccc}
\mu\slashed d_A\phi\\
\mu\slashed d_A\varphi
\end{array}\right).
\end{split}
\end{equation}
It should be observed here that in $H_\g$, wherever there is a term containing a factor of the form $\mathring{\underline L}\varphi$ (or $\mathring{\underline L}\phi$) which admits the ``bad" order of smallness or slow time decay rate, then the term must contain another factor with the ``good" order of smallness and fast time-decay rate, e.g.
$\textrm{tr}\check\chi, \mathring L\phi, \mathring L\varphi$ and so on. Consequently, $H_\g$ may
be expected to have some desired ``good" properties.

Similarly, the left side of \eqref{muphi} can be replaced by $\mu\Box_g\phi=
-\mathring L\mathring{\underline L}\phi-2\mu\zeta^A(\slashed d_A\phi)+\mu\slashed\triangle\phi-\mu(\textrm{tr}\sigma
+\textrm{tr}\chi)\mathring L\phi-\textrm{tr}\chi T\phi$, and hence,
\begin{equation}\label{fe}
\mathring L\mathring{\underline L}\phi+\f{3}{2\varrho}\mathring{\underline L}\phi=\mathring H+\mu\slashed\triangle\phi
\end{equation}
with
\begin{equation}\label{mH}
\begin{split}
\mathring H=&\f3{2\varrho}\mu\mathring L\phi-\f12\text{tr}\check\chi(\underline{\mathring L}\phi-\mu\mathring L\phi)
+\f12G_{\mathring L\mathring L}^\gamma(T\varphi_\gamma)\mathring L\phi+F_{\mathring L\mathring L}(\tilde T\phi)T\phi\\
&+G_{\mathring L\mathring L}^\gamma(T\varphi_\gamma)\tilde T\phi+f(\phi, \varphi, \slashed dx^i,\mathring L^i)\left(
\begin{array}{ccc}
\mu\mathring L\phi\\
\mu\slashed d\phi\\
T\phi
\end{array}
\right)
\left(
\begin{array}{ccc}
\mathring L\phi\\
\slashed d\phi\\
\mathring L\varphi\\
\slashed d\varphi\\
\end{array}
\right).
\end{split}
\end{equation}

Unless stated otherwise, in what follows until Section \ref{ert}, the pointwise estimates for the related quantities
are all made inside domain $D^{s,u}$.

It follows from \eqref{fequation} that $\mathring{\underline L}\varphi_\gamma$ can be estimated by integrating  \eqref{fequation}
along integral curves of $\mathring L$. To this end, one should estimate $\check{\chi}$ and other terms in
$H_{\g}$.

To treat $\check\chi$, we now derive a structure equation for $\chi$ by utilizing the Riemann curvature of
metric $g$.
As the Definition 11.1 in \cite{J}, the Riemann curvature tensor $\mathscr R$ of $g$ can be defined as follows:
\begin{equation}\label{cur}
\mathscr R_{WXYZ}:=-g(\mathscr D_{WX}^2Y-\mathscr D_{XW}^2Y,Z),
\end{equation}
where $W$, $X$, $Y$ and $Z$ are vector fields and $\mathscr D_{WX}^2Y:=W^\al X^\beta\mathscr D_\al\mathscr D_\beta Y$.
\begin{lemma}\label{curvature}
	In the frame $\{\mathring L, T, X_1, X_2, X_3\}$,
$\mathscr R_{TA\mathring LB}$ and $\mathscr R_{\mathring LA\mathring LB}$ are given as follows:
\begin{equation}\label{curT}
\begin{split}
\mathscr R_{TA\mathring LB}=&\f12\big\{\mu F_{B\tilde T}\slashed d_A\mathring L\phi
+\mu G_{B\tilde T}^\gamma\slashed d_A\mathring L\varphi_\gamma-\mu F_{\mathring L\tilde T}\slashed\nabla_{AB}^2\phi
-\mu G_{\mathring L\tilde T}^\gamma\slashed\nabla_{AB}^2\varphi_\gamma\\
&\quad-F_{AB}\mathring LT\phi-G_{AB}^\gamma\mathring LT\varphi_\gamma+F_{A\mathring L}\slashed d_BT\phi+G_{A\mathring L}^\gamma\slashed d_BT\varphi_\gamma\big\}\\
&+\f14\mu^{-1}\big\{(\mathcal{FG})_{A\mathring L}(T)\cdot(\mathcal{FG})_{B\mathring L}(T)
-2(\mathcal{FG})_{A\mathring L}(T)(\slashed d_B\mu)\big\}\\
&+\f12\big\{\mu (\mathcal{FG})_{A\mathring L}(X_C){\chi_{B}}^C-\mu (\mathcal {FG})_{B\tilde T}(X_C){\chi_{A}}^C+(\mathcal{FG})_{\mathring L\tilde T}(T)\chi_{AB}\big\}\\
&+f(\phi, \varphi, \slashed dx^i, \mathring L^i)\left(
\begin{array}{ccc}
\mu\mathring L\phi\\
\mu\slashed d\phi\\
T\phi\\
\mu \mathring L\varphi\\
\mu \slashed d\varphi\\
T{\varphi}\\
\end{array}
\right)
\left(
\begin{array}{ccc}
\mathring L\phi\\
\slashed d\phi\\
\mathring L\varphi\\
\slashed d\varphi\\
\end{array}
\right),
\end{split}
\end{equation}
and
\begin{equation}\label{curL}
\begin{split}
\mathscr R_{\mathring LA\mathring LB}=&\f12\big\{F_{B\mathring L}\slashed d_A\mathring L\phi
+G_{B\mathring L}^\gamma(\slashed d_A\mathring L\varphi_\gamma)-F_{\mathring L\mathring L}\slashed\nabla^2_{AB}\phi
-G_{\mathring L\mathring L}^\gamma\slashed\nabla^2_{AB}\varphi_\gamma-F_{AB}\mathring L^2\phi\\
&-G_{AB}^\gamma\mathring L^2\varphi_\gamma+F_{A\mathring L}\slashed d_B\mathring L\phi
+G_{A\mathring L}^\gamma\slashed d_B\mathring L\varphi_\gamma\}-\f12(\mathcal{FG})_{B\mathring L}(X_C){\chi_{A}}^C\\
&-\f12(\mathcal{FG})_{A\mathring L}(X_C){\chi_{B}}^C+\f12\mu^{-1}(\mathcal{FG})_{\mathring L\mathring L}(T)\chi_{AB}\\
&+f(\phi, \varphi, \slashed dx^i,\mathring L^i)\left(
\begin{array}{ccc}
\mathring L\phi\\
\slashed d\phi\\
\mathring L\varphi\\
\slashed d\varphi\\
\end{array}
\right)
\left(
\begin{array}{ccc}
\mathring L\phi\\
\slashed d\phi\\
\mathring L\varphi\\
\slashed d\varphi\\
\end{array}
\right),
\end{split}
\end{equation}
where $f(\phi, \varphi, \slashed dx^i,\mathring L^i)$ represents a generic smooth function of its arguments and similar convention as \eqref{ge} has been used.
\end{lemma}

\begin{proof}
Note that $\mathscr R$ is a $(0,4)$-type tensor field. Taking $W=\p_\kappa$, $X=\p_\lambda$, $Y=\p_\al$ and $Z=\p_\beta$ in \eqref{cur}, and
 applying $\mathscr D_\al\p_\beta=\Gamma^\nu_{\al\gamma\beta}\p_\nu$, one can get the components of $\mathscr R$ as
	 \[
	 \mathscr R_{\kappa\lambda\al\beta}=\p_\lambda\Gamma_{\al\beta\kappa}-\p_\kappa\Gamma_{\lambda\beta\al}
-g^{\nu\gamma}\Gamma_{\al\gamma\kappa}\Gamma_{\lambda\nu\beta}+g^{\nu\gamma}\Gamma_{\lambda\gamma\al}\Gamma_{\kappa\nu\beta}
	 \]
	 with the Christoffel symbol $\G_{\al\beta\kappa}:=g_{\beta\gamma}\G_{\al\kappa}^\gamma$.
	Therefore,
	 \begin{equation*}
	 \begin{split}
	 \mathscr R_{\kappa\lambda\al\beta}=&\f12\big\{F_{\beta\kappa}\mathscr D_{\lambda\al}^2\phi+G_{\beta\kappa}^\nu\mathscr D_{\lambda\al}^2\varphi_\nu-F_{\al\kappa}\mathscr D_{\lambda\beta}^2\phi-G_{\al\kappa}^\nu\mathscr D_{\lambda\beta}^2\varphi_\nu\\
	 &\quad-F_{\beta\lambda}\mathscr D_{\kappa\al}^2\phi-G_{\beta\lambda}^\nu\mathscr D_{\kappa\al}^2\varphi_\nu+F_{\al\lambda}\mathscr D_{\kappa\beta}^2\phi+G_{\al\lambda}^\nu\mathscr D_{\kappa\beta}^2\varphi_\nu\big\}\\
	 &-\f14g^{\tau\gamma}(F_{\lambda\al}\p_\gamma\phi+G_{\lambda\al}^\sigma\p_\gamma\varphi_\sigma)(F_{\kappa\beta}\p_\tau\phi+G_{\beta\kappa}^\nu
\p_\tau\varphi_\nu)\\
	 &+\f14g^{\tau\gamma}(F_{\lambda\beta}\p_\gamma\phi+G_{\lambda\beta}^\sigma\p_\gamma\varphi_\sigma)(F_{\al\kappa}\p_\tau\phi
+G_{\al\kappa}^\nu\p_\tau\varphi_\nu)\\
	 &-\f14g^{\tau\gamma}\big(F_{\gamma(\kappa}\p_{\al)}\phi+G_{\gamma(\kappa}^\sigma\p_{\al)}\varphi_\sigma\big)\big(F_{\tau(\lambda}\p_{\beta)}\phi
+G_{\tau(\lambda}^\nu\p_{\beta)}\varphi_\nu\big)\\
	 &+\f14g^{\tau\gamma}\big(F_{\gamma(\lambda}\p_{\al)}\phi+G_{\gamma(\lambda}^\sigma\p_{\al)}\varphi_\sigma\big)\big(F_{\tau(\kappa}\p_{\beta)}\phi
+G_{\tau(\kappa}^\nu\p_{\beta)}\varphi_\nu\big)\\
&+\f12\big\{(\p_\phi F_{\beta\kappa})(\p_\lambda\phi)(\p_\al\phi)+(\p_{\varphi_{\sigma}}F_{\beta\kappa})(\p_\lambda\varphi_\sigma)(\p_\al\phi)+(\p_\phi G_{\beta\kappa}^\sigma)(\p_\lambda\phi)(\p_\al\varphi_\sigma)\\
&
+(\p_{\varphi_{\sigma}}G_{\beta\kappa}^\gamma)(\p_\lambda\varphi_\sigma)(\p_\al\varphi_\gamma)-(\p_\phi F_{\al\kappa})(\p_\lambda\phi)(\p_\beta\phi)-(\p_{\varphi_\sigma}F_{\al\kappa})(\p_\lambda\varphi_\sigma)(\p_\beta\phi)\\
&-(\p_\phi G_{\al\kappa}^\gamma)(\p_\lambda\phi)(\p_\beta\varphi_\gamma)-(\p_{\varphi_\sigma}G_{\al\kappa}^\gamma)(\p_\lambda \varphi_\sigma)(\p_\beta\varphi_\gamma)-(\p_\phi F_{\lambda\beta})(\p_\kappa\phi)(\p_\al\phi)\\
&
-(\p_{\varphi_\sigma}F_{\lambda\beta})(\p_\kappa\varphi_\sigma)(\p_\al\phi)-(\p_\phi G_{\lambda\beta}^\gamma)(\p_\kappa\phi)(\p_\al\varphi_\gamma)
-(\p_{\varphi_\sigma}G_{\lambda\beta}^\gamma)(\p_\kappa\varphi_\sigma)(\p_\al\varphi_\gamma)\\
&+(\p_\phi F_{\lambda\al})(\p_\kappa\phi)(\p_\beta\phi)
+(\p_{\varphi_\sigma}F_{\lambda\al})(\p_\kappa\varphi_\sigma)(\p_\beta\phi)+(\p_\phi G_{\lambda\al}^\gamma)(\p_\kappa\phi)(\p_\beta\varphi_\gamma)\\
&+(\p_{\varphi_\sigma}G_{\lambda\al}^\gamma)(\p_\kappa\varphi_\sigma)(\p_\beta\varphi_\gamma)\big\},
\end{split}
\end{equation*}
where $F_{\kappa(\lambda}\p_{\al)}\phi:=F_{\kappa\lambda}\p_{\al}\phi+F_{\kappa\al}\p_{\lambda}\phi$ and $G_{\kappa(\lambda}^\nu\p_{\al)}\varphi_\nu:=G_{\kappa\lambda}^\nu\p_{\al}\varphi_\nu+G_{\kappa\al}^\nu\p_{\lambda}\varphi_\nu$.
In addition, by \eqref{gab}, one has that for any generic smooth functions $\psi_i$ $(i=1,2)$,
 \[
 g^{\tau\gamma}(\p_{\tau}\psi_1)(\p_{\gamma}\psi_2)=\slashed g^{AB}(\slashed d_A\psi_1)(\slashed d_B\psi_2)-(\mathring L\psi_1)(\mathring L\psi_2)-\mu^{-1}(T\psi_1)(\mathring L\psi_2)-\mu^{-1}(T\psi_2)(\mathring L\psi_1).
 \]
Then contracting $\mathscr R_{\kappa\lambda\al\beta}$ with $T^{\kappa}(\slashed d_Ax^\lambda)\mathring L^\al(\slashed d_Bx^\beta)$ yields
 \begin{equation*}\label{RT}
 \begin{split}
 \mathscr R_{TA\mathring LB}=&\f12\big(\mu F_{B\tilde T}\mathscr D_{A\mathring L}^2\phi+\mu G_{B\tilde T}^\nu\mathscr D_{A\mathring L}^2\varphi_\nu-\mu F_{\mathring L\tilde T}\mathscr D_{AB}^2\phi-\mu G_{\mathring L\tilde T}^\nu\mathscr D_{AB}^2\varphi_\nu\\
 &\quad-F_{AB}\mathscr D^2_{\mathring LT}\phi-G_{AB}^\nu\mathscr D^2_{\mathring LT}\varphi_\nu+F_{A\mathring L}\mathscr D_{BT}^2\phi+G_{A\mathring L}^\nu\mathscr D_{BT}^2\varphi_\nu\big)\\
 \end{split}
 \end{equation*}

 \begin{equation}\label{RT}
 \begin{split}
 &+f(\phi, \varphi, \slashed dx^i, \mathring L^i)\left(
 \begin{array}{ccc}
 \mu\mathring L\phi\\
 \mu\slashed d\phi\\
 T\phi\\
 \mu \mathring L\varphi\\
 \mu \slashed d\varphi\\
 T{\varphi}\\
 \end{array}
 \right)
 \left(
 \begin{array}{ccc}
 \mathring L\phi\\
 \slashed d\phi\\
 \mathring L\varphi\\
 \slashed d\varphi\\
 \end{array}
 \right).
 \end{split}
 \end{equation}
 In addition, it follows from \eqref{theta}, \eqref{zeta} and Lemma \ref{cd} that for any smooth function $\psi$,
 \begin{equation}\label{AL}
 \begin{split}
 \mathscr D_{A\mathring L}^2\psi=&\slashed d_A\mathring L\psi-(\mathscr D_A\mathring L)\psi\\
 =&\slashed d_A\mathring L\psi-\chi_{AB}\slashed d^B\psi+f(\phi, \varphi, \slashed dx^i, \mathring L^i)\left(
 \begin{array}{ccc}
 \mathring L\phi\\
 \slashed d\phi\\
 \tilde T\phi\\
 \mathring L\varphi\\
 \slashed d\varphi\\
 \tilde T{\varphi}\\
 \end{array}
 \right)
 \mathring L\psi.
 \end{split}
 \end{equation}
 Similarly,
 \begin{align}
 &\mathscr D_{BT}^2\psi=\slashed d_BT\psi-(\mu^{-1}T\psi)\slashed d_B\mu+(\mu\slashed d^D\psi)\chi_{BD}+\f12F_{B\mathring L}\tilde T\phi(T\psi)+\f12G_{B\mathring L}^\gamma\tilde T\varphi_\gamma(T\psi)\no\\
 &\qquad\qquad+f_1(\phi, \varphi, \slashed dx^i, \mathring L^i)\left(
 \begin{array}{ccc}
 \mu\mathring L\phi\\
 \mu\slashed d\phi\\
 T\phi\\
 \mu\mathring L\varphi\\
 \mu\slashed d\varphi\\
 T{\varphi}\\
 \end{array}
 \right)
 \left(
 \begin{array}{ccc}
 \mathring L\psi\\
 \slashed d\psi\\
 \end{array}
 \right)
 +f_2(\phi, \varphi, \slashed dx^i, \mathring L^i)\left(
 \begin{array}{ccc}
 \mathring L\phi\\
 \slashed d\phi\\
 \mathring L\varphi\\
 \slashed d\varphi\\
 \end{array}
 \right)T\psi,\label{BT}\\
 &\mathscr D_{\mathring L T}^2\psi=\mathring LT\psi
 +f(\phi, \varphi, \slashed dx^i, \mathring L^i)\left(
 \begin{array}{ccc}
 \mu\mathring L\phi\\
 \mu\slashed d\phi\\
 T\phi\\
 \mu\mathring L\varphi\\
 \mu\slashed d\varphi\\
 T{\varphi}\\
 \end{array}
 \right)
 \left(
 \begin{array}{ccc}
 \mathring L\psi\\
 \slashed d\psi\\
 \end{array}
 \right),\label{LT}\\
 &\mu\mathscr D_{AB}^2\psi=\mu\slashed\nabla_{AB}^2\psi-\chi_{AB}(T\psi)
 +f(\phi, \varphi, \slashed dx^i, \mathring L^i)\left(
 \begin{array}{ccc}
 \mu\mathring L\phi\\
 \mu\slashed d\phi\\
 T\phi\\
 \mu\mathring L\varphi\\
 \mu\slashed d\varphi\\
 T{\varphi}\\
 \end{array}
 \right)
 \mathring L\psi.\label{AB}
 \end{align}
 Taking $\psi=\phi$ or $\varphi_\gamma$ in \eqref{AL}-\eqref{AB}, and then substituting these resulting equations into \eqref{RT}, one gets \eqref{curT} immediately.

By an analogous argument for \eqref{curT}, one can show \eqref{curL}. Here details are omitted.
\end{proof}

In Lemma \ref{curvature}, the expressions of $\mathscr R_{TA\mathring LB}$ and $\mathscr R_{\mathring LA\mathring LB}$
are obtained through contracting $\mathscr R_{\kappa\lambda\al\beta}$ with respect to the corresponding vectorfields.
If  the definition \eqref{cur} is used directly, e.g. $\mathscr R_{\mathring LA\mathring LB}
=g\big(\mathscr D_A(\mathscr D_{\mathring L}\mathring L)-\mathscr D_{\mathring L}(\mathscr D_A\mathring L),X_B\big)$,
then one can find  the following structure equations for $\chi_{AB}$ and ${\check\chi}_{AB}$ with
the help of Lemmas \ref{curvature}.

\begin{lemma}\label{LTchi}
	For the second fundamental form $\chi$ and its ``error" form $\check\chi$ defined in \eqref{chith} and \eqref{errorv}
respectively, the following structure equations hold:
\begin{equation}\label{Lchi}
\begin{split}
\mathring L\chi_{AB}=&-\f12\big\{F_{B\mathring L}\slashed d_A\mathring L\phi+G_{B\mathring L}^\gamma\slashed d_A\mathring L\varphi_\gamma-F_{\mathring L\mathring L}\slashed\nabla_{AB}^2\phi-G_{\mathring L\mathring L}^\gamma\slashed\nabla_{AB}^2\varphi_\gamma-F_{AB}\mathring L^2\phi\\
&\qquad-G_{AB}^\gamma\mathring L^2\varphi_{\gamma}+F_{A\mathring L}\slashed d_B\mathring L\phi+G_{A\mathring L}^\gamma\slashed d_B\mathring L\varphi_\gamma+(\mathcal{FG})_{\mathring L\mathring L}(\mathring L)\chi_{AB}+2(\mathcal{FG})_{\tilde T\mathring L}(\mathring L)\chi_{AB}\\
&\qquad-(\mathcal{FG})_{B\mathring L}(X_C){\chi_A}^C-(\mathcal{FG})_{A\mathring L}(X_C){\chi_B}^C\big\}\\
&\quad+{\chi_A}^C\chi_{BC}+f(\phi, \varphi, \slashed dx^i,\mathring L^i)\left(
\begin{array}{ccc}
\mathring L\phi\\
\slashed d\phi\\
\mathring L\varphi\\
\slashed d\varphi\\
\end{array}
\right)
\left(
\begin{array}{ccc}
\mathring L\phi\\
\slashed d\phi\\
\mathring L\varphi\\
\slashed d\varphi\\
\end{array}
\right),
\end{split}
\end{equation}

\begin{equation*}\label{Tchi}
\begin{split}
\slashed{\mathcal L}_T\chi_{AB}=&\slashed\nabla_{AB}^2\mu-\f12\big\{\mu F_{\tilde T\tilde T}\slashed\nabla_{AB}^2\phi+\mu G_{\tilde T\tilde T}^\gamma\slashed\nabla_{AB}^2\varphi_\gamma+F_{B\mathring L}\slashed d_AT\phi+G_{B\mathring L}^\gamma\slashed d_AT\varphi_\gamma\\
&+F_{A\mathring L}\slashed d _BT\phi+G_{A\mathring L}^\gamma\slashed d_BT\varphi_\gamma-F_{AB}\mathring LT\phi-G_{AB}^\gamma\mathring LT\varphi_\gamma\big\}-\f12\Theta_B\slashed d_A\mu-\f12\Theta_A\slashed d_B\mu\\
&+\f12\big\{\mu(\mathcal{FG})_{\mathring L\mathring L}(\mathring L)+2\mu(\mathcal{FG})_{\tilde T\mathring L}(\mathring L)
-(\mathcal{FG})_{\mathring L\mathring L}(T)-(\mathcal{FG})_{\mathring L\tilde T}(T)\big\}\chi_{AB}\\
&-\mu{\chi_A}^C\chi_{BC}-\f14{\chi_B}^C\Upsilon_{CA}-\f14{\chi_A}^C\Upsilon_{CB}\\
\end{split}
\end{equation*}

\begin{equation}\label{Tchi}
\begin{split}
&+f_1(\phi, \varphi, \slashed dx^i, \mathring L^i)
\left(
\begin{array}{ccc}
\slashed d\mathring L^j\\
\mathring L\phi\\
\slashed d\phi\\
\mathring L\varphi\\
\slashed d\varphi\\
\end{array}
\right)\left(
\begin{array}{ccc}
\mu\mathring L\phi\\
\mu\slashed d\phi\\
T\phi\\
\mu\mathring L\varphi\\
\mu\slashed d\varphi\\
T{\varphi}\\
\end{array}
\right)
+f_2(\phi, \varphi, \mathring L^i)\slashed\nabla_{AB}^2x^j\left(
\begin{array}{ccc}
\mu\mathring L\phi\\
T\phi\\
\mu\mathring L\varphi\\
T{\varphi}\\
\end{array}
\right).
\end{split}
\end{equation}

And hence,
\begin{equation}\label{Lchi'}
\begin{split}
\mathring L\check\chi_{AB}=&-\f12\big\{F_{B\mathring L}\slashed d_A\mathring L\phi+G_{B\mathring L}^\gamma\slashed d_A\mathring L\varphi_\gamma-F_{\mathring L\mathring L}\slashed\nabla_{AB}^2\phi-G_{\mathring L\mathring L}^\gamma\slashed\nabla_{AB}^2\varphi_\gamma-F_{AB}\mathring L^2\phi\\
&-G_{AB}^\gamma\mathring L^2\varphi_{\gamma}+F_{A\mathring L}\slashed d_B\mathring L\phi+G_{A\mathring L}^\gamma\slashed d_B\mathring L\varphi_\gamma+(\mathcal{FG})_{\mathring L\mathring L}(\mathring L)\check\chi_{AB}+2(\mathcal{FG})_{\tilde T\mathring L}(\mathring L)\check\chi_{AB}\\
&-(\mathcal{FG})_{B\mathring L}(X_C){\check\chi_{A}}^C-(\mathcal{FG})_{A\mathring L}(X_C){\check\chi_{B}^C}+\f{1}{\varrho}(\mathcal{FG})_{\mathring L\mathring L}(\mathring L)\slashed g_{AB}+\f{2}\varrho (\mathcal{FG})_{\tilde T\mathring L}(\mathring L)\slashed g_{AB}\\
&-\f1\varrho (\mathcal{FG})_{B\mathring L}(X_A)-\f1\varrho (\mathcal{FG})_{A\mathring L}(X_B)\big\}+{\check\chi_A}^C\check\chi_{BC}\\
&+f(\phi, \varphi, \slashed dx^i,\mathring L^i)\left(
\begin{array}{ccc}
\mathring L\phi\\
\slashed d\phi\\
\mathring L\varphi\\
\slashed d\varphi\\
\end{array}
\right)
\left(
\begin{array}{ccc}
\mathring L\phi\\
\slashed d\phi\\
\mathring L\varphi\\
\slashed d\varphi\\
\end{array}
\right),
\end{split}
\end{equation}

\begin{equation}\label{Tchi'}
\begin{split}
\slashed{\mathcal L}_T\check{\chi}_{AB}=&\slashed\nabla_{AB}^2\mu-\f12\big\{\mu F_{\tilde T\tilde T}\slashed\nabla_{AB}^2\phi+\mu G_{\tilde T\tilde T}^\gamma\slashed\nabla_{AB}^2\varphi_\gamma+F_{B\mathring L}\slashed d_AT\phi+G_{B\mathring L}^\gamma\slashed d_AT\varphi_\gamma\\
&+F_{A\mathring L}\slashed d _BT\phi+G_{A\mathring L}^\gamma\slashed d_BT\varphi_\gamma-F_{AB}\mathring LT\phi-G_{AB}^\gamma\mathring LT\varphi_\gamma\big\}\\
&-\f12\Theta_B\slashed d_A\mu-\f12\Theta_A\slashed d_B\mu-\f14{\check\chi_B}^C\Upsilon_{CA}-\f14{\check\chi_A}^C\Upsilon_{CB}-\mu{\check\chi_A}^C\check\chi_{BC}+\f1{\varrho^2}(\mu-1)\slashed g_{AB}\\
&+\f12\big\{\mu(\mathcal{FG})_{\mathring L\mathring L}(\mathring L)+2\mu (\mathcal{FG})_{\tilde T\mathring L}(\mathring L)-(\mathcal{FG})_{\mathring L\mathring L}(T)-(\mathcal{FG})_{\mathring L\tilde T}(T)\big\}\big(\check\chi_{AB}+\f1\varrho\slashed g_{AB}\big)\\
&+\f{1}{4\varrho}\big\{\mu (\mathcal{FG})_{B\mathring L}(X_A)+\mu(\mathcal{FG})_{A\mathring L}(X_B)-2\mu (\mathcal{FG})_{AB}(\mathring L)+3\mu(\mathcal{FG})_{B\tilde T}(X_A)\\
&+3\mu (\mathcal{FG})_{A\tilde T}(X_B)-2(\mathcal{FG})_{AB}(T)\big\}\\
&+f_1(\phi, \varphi, \slashed dx^i, \mathring L^i)
\left(
\begin{array}{ccc}
\slashed d\mathring L^j\\
\mathring L\phi\\
\slashed d\phi\\
\mathring L\varphi\\
\slashed d\varphi\\
\end{array}
\right)\left(
\begin{array}{ccc}
\mu\mathring L\phi\\
\mu\slashed d\phi\\
T\phi\\
\mu\mathring L\varphi\\
\mu\slashed d\varphi\\
T{\varphi}\\
\end{array}
\right)
+f_2(\phi, \varphi, \mathring L^i)\slashed\nabla_{AB}^2x^j\left(
\begin{array}{ccc}
\mu\mathring L\phi\\
T\phi\\
\mu\mathring L\varphi\\
T{\varphi}\\
\end{array}
\right).
\end{split}
\end{equation}
Moreover,
\begin{equation}\label{dchi}
\begin{split}
\slashed\nabla_C\check\chi_{AB}=&\slashed\nabla_A\check\chi_{CB}+F_{ij}(\slashed d_C\phi)\slashed d_A\mathring L^i(\slashed d_Bx^j)+G_{ij}^\gamma(\slashed d_C\varphi_\gamma)\slashed d_A\mathring L^i(\slashed d_Bx^j)\\
&-F_{ij}(\slashed d_A\phi)\slashed d_C\mathring L^i(\slashed d_Bx^j)-G_{ij}^\gamma(\slashed d_A\varphi_\gamma)\slashed d_C\mathring L^i(\slashed d_Bx^j)-g_{ij}(\slashed d_C\check L^i)\slashed\nabla_{AB}^2x^j\\
&+g_{ij}(\slashed d_A\check L^i)\slashed\nabla_{CB}^2x^j+\varrho^{-1}(\mathcal{FG})_{BC}(X_A)-\varrho^{-1}(\mathcal{FG})_{AB}(X_C)\\
&+\slashed\nabla_A\Lambda_{CB}-\slashed\nabla_C\Lambda_{AB},
\end{split}
\end{equation}
where $\Lambda_{AB}$ is defined in \eqref{Lambda},
\begin{align}
\Theta_A=&(\mathcal{FG})_{\mathring L\tilde T}(X_A)+(\mathcal{FG})_{\tilde T\tilde T}(X_A)-(\mathcal{FG})_{A\tilde T}(\mathring L),\label{Theta}\\
\Upsilon_{AB}=&\mu \big\{(\mathcal{FG})_{A\mathring L}(X_B)+(\mathcal{FG})_{A\tilde T}(X_B)- (\mathcal{FG})_{AB}(\mathring L)+2(\mathcal{FG})_{B\mathring L}(X_A)\big\}-(\mathcal{FG})_{AB}(T),\label{Upsion}
\end{align}
and $f, f_1, f_2$ are generic smooth functions of their arguments.
\end{lemma}

\begin{proof}
 \begin{enumerate}
  \item
By \eqref{cur}, one can write $\mathscr R_{\mathring LA\mathring LB}$ as
 \begin{equation}\label{Z-0}
 \begin{split}
 \mathscr R_{\mathring LA\mathring LB}&=g\big(\mathscr D_A(\mathscr D_{\mathring L}\mathring L)
 -\mathscr D_{\mathring L}(\mathscr D_A\mathring L),X_B\big).
 \end{split}
 \end{equation}
Substituting the expression of $\mathscr D_{\mathring L}\mathring L$ and $\mathscr D_A\mathring L$ in Lemma \ref{cd} into
\eqref{Z-0} yields
 \begin{equation}\label{curL2}
 \mathscr R_{\mathring LA\mathring LB}=\mu^{-1}(\mathring L\mu){\chi}_{AB}-\mathring{L}{\chi}_{AB}+{{\chi}_A}^C{\chi}_{BC}.
 \end{equation}
Inserting \eqref{curL} and \eqref{lmu} into \eqref{curL2} yields \eqref{Lchi}.
On the other hand, \eqref{Lchi'} follows directly from ${\check{\chi}}_{AB}={\chi}_{AB}-\f{1}{\varrho}\slashed g_{AB}$
and $\mathring L\slashed g_{AB}=2{\chi}_{AB}$.

 \item
 It follows from the definition for Lie derivative that
 \begin{equation}\label{LTc}
 \slashed{\mathcal{L}}_T{\chi}_{AB}=\mathcal{L}_T{\chi}_{AB}=g(\mathscr D_A\mathscr D_T\mathring L,X_B)
 +g(\mathscr D_A\mathring L,\mathscr D_TX_B)-g(X_A,\mathscr D_{[T,X_B]}\mathring L)-\mathscr R_{TA\mathring L B},
 \end{equation}
where
 \begin{equation}\label{AT}
 \begin{split}
 &g(\mathscr D_A\mathscr D_T\mathring L,X_B)=g\big(\mathscr D_A(-\mathring L\mu\mathring L+{\eta}^CX_C),X_B\big)\\
 =&\slashed\nabla^2_{AB}\mu+\slashed\nabla_A(\mu{\zeta}_B)-(\mathring L\mu){\chi}_{AB}
 \end{split}
 \end{equation}
 and
 \begin{equation}\label{ALBT}
 \begin{split}
 g(\mathscr D_A\mathring L,\mathscr D_TX_B)=&-\zeta_Ag(\mathring L,\mathscr D_TX_B)+{\chi_A}^Dg(X_D,\mathscr D_TX_B)\\
 &=\zeta_A\eta_B+g(\mathscr D_{\slashed\Pi\mathscr D_TX_B}\mathring L,X_A)\\
 &=\zeta_A\eta_B+g(\mathscr D_{[T,X_B]}\mathring L,X_A)+g(\mathscr D_{\slashed\Pi\mathscr D_BT}\mathring L,X_A)\\
 &=\zeta_A\eta_B+g(\mathscr D_{[T,X_B]}\mathring L,X_A)+\mu{\sigma_B}^Cg(\mathscr D_C\mathring L,X_A),
 \end{split}
 \end{equation}
 where \eqref{cdf} has been used. Due to $\slashed{\mathcal{L}}_T{\chi}_{AB}=\slashed{\mathcal{L}}_T{\chi}_{BA}$,
 substituting \eqref{AT} and \eqref{ALBT} into \eqref{LTc} yields
 \begin{equation*}
 \begin{split}
 \slashed{\mathcal L}_T\chi_{AB}=&\slashed\nabla_{AB}^2\mu+\f12\slashed\nabla_A(\mu\zeta_B)+\f12\slashed\nabla_B(\mu\zeta_A)-(\mathring L\mu)\chi_{AB}+\f12\zeta_A\eta_B+\f12\zeta_B\eta_A\\
 &+\f12\mu{\sigma_B}^Cg(\mathscr D_C\mathring L,X_A)+\f12\mu{\sigma_A}^Cg(\mathscr D_C\mathring L,X_B)-\f12\mathscr R_{TA\mathring L B}-\f12\mathscr R_{TB\mathring L A}.
  \end{split}
 \end{equation*}
This gives \eqref{Tchi} with the help of \eqref{zeta}, \eqref{lmu}, \eqref{theta}, \eqref{curT} and \eqref{chith}.
 Analogously to the proof of \eqref{Lchi'}, one can apply ${\check{\chi}}_{AB}={\chi}_{AB}-\f{1}{\varrho}\slashed g_{AB}$
 and $\slashed{\mathcal L}_T\slashed g_{AB}=2\mu\sigma_{AB}$ to obtain \eqref{Tchi'}.

 \item
 \eqref{dL} implies
 \begin{equation}\label{chil}
 {\chi}_{AB}=g_{ij}\slashed d_Bx^j(\slashed g^{CD}{\chi}_{AD}\slashed d_Cx^i)=g_{ij}\slashed d_Bx^j(\slashed d_A\mathring L^i)-\Lambda_{AB}.
 \end{equation}
Taking $\slashed\nabla^B$ on both hand sides of \eqref{chil} yields \eqref{dchi}.
 \end{enumerate}
\end{proof}

It follows from \eqref{Lchi'} that $\check\chi$ can be estimated by integrating along integral curves of $\mathring L$ as soon as the right hand side of
\eqref{Lchi'} is estimated.

\begin{proposition}\label{chi'}
	Under the assumptions $(\star)$, when $\delta>0$ is small, it holds that
	\begin{equation}\label{echi'}
	|\check{\chi}|\lesssim M\delta^{1-\varepsilon_0} s^{-2}.
	\end{equation}
	Meanwhile,
	\begin{equation}\label{chi}
	|\chi|^2=\f{3}{\varrho^2}+O(M\delta^{1-\varepsilon_0} s^{-3}).
	\end{equation}
\end{proposition}

\begin{proof}
	Note that
	\[
	|\check{\chi}|^2=\slashed g^{AB}\slashed g^{CD}\check{\chi}_{AC}\check{\chi}_{BD}.
	\]
	Then
	\begin{equation}\label{Lrchi}
	\mathring L\big(\varrho^4|\check{\chi}|^2\big)=-4\varrho^4\check\chi^{AB}\slashed g^{CD}\check\chi_{AC}\check\chi_{BD}
+2\varrho^4(\mathring L\check\chi_{AB})\check\chi^{AB}.
	\end{equation}
Substituting \eqref{Lchi'} into \eqref{Lrchi}, and using $(\star)$, \eqref{L} and \eqref{dx} to estimate
the right hand side of \eqref{Lrchi} except $\check\chi$ itself, one arrives at
	\[
	|\mathring L\big(\varrho^4|\check{\chi}|^2\big)|\lesssim\varrho^4\big\{|\check{\chi}|^3+\delta^{1-\varepsilon_0} s^{-7/2} M|\check{\chi}|+\delta^{1-\varepsilon_0} s^{-5/2} M|\check{\chi}|^2\big\}.
	\]
	This shows that for small $\delta>0$,
	\[
	|\check{\chi}|\lesssim M\delta^{1-\varepsilon_0} s^{-2}.
	\]
\end{proof}

At the end of this section, we point out that the operator $R_{ij}$ is equivalent to the scaling operator $r\slashed\nabla$ under
the assumptions $(\star)$ and the estimate \eqref{echi'} in the following sense.

\begin{corollary}\label{12form}
 Under the assumptions $(\star)$, when $\delta>0$ is small,
\begin{enumerate}
	\item
  if $\xi$ is a 1-form on $S_{s, u}$, then
  \begin{align}
  &\sum_{i,j=1}^4(\xi_a R_{ij}^a)^2\sim r^2|\xi|^2\label{1-f},\\
  &\sum_{i,j=1}^4|\slashed{\mathcal L}_{R_{ij}}\xi|^2\sim r^2|\slashed\nabla\xi|^2+|\xi|^2;\label{1f}
  \end{align}
\item
  if $\xi$ is a 2-form on $S_{s, u}$, then
  \begin{equation}\label{2-f}
  \begin{split}
  \sum_{i,j=1}^4|\slashed{\mathcal L}_{R_{ij}}\xi|^2+(\textrm{tr}\xi)^2\sim r^2|\slashed\nabla\xi|^2+|\xi|^2
    \end{split}
  \end{equation}
  and
  \begin{equation}\label{2f}
  |\slashed\nabla^2\xi|\lesssim \varrho^{-2}|\slashed{\mathcal L}_R^{\leq 2}\xi|.
  \end{equation}
\end{enumerate}
\end{corollary}
This is similar to Lemma 12.22 in \cite{J} or
in Sect.3.3.3 of \cite{MY}.

\section{$L^\infty$ estimates on the higher order derivatives of $(\phi,\p\phi)$ and some other quantities near $C_0$}\label{ho}

With the structure equations for $\vp_\g$ and $\phi$ in the frame $\{\mathring L,\mathring {\underline L},R_{ij}\}$ already derived in Section \ref{BA} (see \eqref{fequation} and \eqref{fe} respectively), we are going to close the bootstrap assumption $(\star)$ by improving up to the $(N-1)^{th}$ order derivatives of $\vp_\g$ and $\phi$.
To this end, we start with some preliminary results which deal only with the rotational vector fields on $S_{s, u}$.

\begin{lemma}\label{Rh}
 Under the assumptions $(\star)$, when $\delta>0$ is small, it holds that for $k\leq N-1$,
 \begin{equation}\label{he}
 \begin{split}
 &|\slashed{\mathcal{L}}_{R}^{k+1}\slashed dx^j|\lesssim 1,\quad|\slashed{\mathcal{L}}_{R}^k\check{\chi}|\lesssim
  M\delta^{1-\varepsilon_0} s^{-2},\quad|R^{k+1} \mathring L^j|\lesssim 1,\\
 &|R^{k+1}\check{L}^j|\lesssim  M\delta^{1-\varepsilon_0} s^{-1},\quad|\slashed{\mathcal{L}}_{R}^k\leftidx{^{(R)}}{\slashed\pi}|\lesssim
  M\delta^{1-\varepsilon_0} s^{-1},\quad|\slashed{\mathcal{L}}_{R}^k\leftidx{^{(R)}}{\slashed\pi}_{\mathring L}|\lesssim M\delta^{1-\varepsilon_0} s^{-1},\\
 &|R^{k+1}\check\varrho|\lesssim M\delta^{1-\varepsilon_0}s^{-1},\quad|R^{k+1}\upsilon_{ij}|\lesssim
  M\delta^{1-\varepsilon_0},\quad|\slashed{\mathcal{L}}_{R}^{k+1}R|\lesssim s,
 \end{split}
 \end{equation}
 where $R\in\{R_{ij}: i,j=1,2,3,4\}$.
\end{lemma}

\begin{proof}
 This will be proved by induction with respect to $k$.
 \begin{enumerate}
  \item When $k=0$, the estimates for $\check{\chi}$, $R \mathring L^j$,
 $R\check{L}^j$, $\leftidx{^{(R)}}{\slashed\pi}$, $\leftidx{^{(R)}}{\slashed\pi}_{\mathring L}$, $R\check\varrho$
 and $R\upsilon_{ij}$ can be obtained easily from \eqref{echi'}, \eqref{dx}, \eqref{1-f}, \eqref{deL},
 \eqref{omega}, \eqref{rrho}
 and \eqref{Rpi}, respectively.

For $RR_{ij}x^a$, one has
$$RR_{ij}x^a=R({\Omega_{ij}}^a-\upsilon_{ij}\tilde{T}^a),$$
which means $|RR_{ij}x^a|\lesssim s$ by \eqref{Ri}, and hence
\begin{equation}\label{RRx}
|\slashed{\mathcal L}_{R}\slashed dx^a|=|\slashed dRx^a|\lesssim r^{-1}\sum_{i,j=1}^4|R_{ij}Rx^a|\lesssim 1
\end{equation}
if $\xi=\slashed d Rx^a$ in \eqref{1-f} is chosen.

Next we estimate $\slashed{\mathcal{L}}_{R_{ij}}R_{ab}=[R_{ij},R_{ab}]$.
Due to
  \begin{equation*}
  \begin{split}
  R_{ij}{R_{ab}}^c=&R_{ij}({\Omega_{ab}}^c-\upsilon_{ab}\tilde T^c)\\
  =&\delta_b^c({\Omega_{ij}}^a-\upsilon_{ij}\tilde T^a)-\delta_a^c({\Omega_{ij}}^b-\upsilon_{ij}\tilde T^b)-(R_{ij}\upsilon_{ab})\tilde T^c-\upsilon_{ab}R_{ij}\tilde T^c,
  \end{split}
  \end{equation*}
  then $\slashed\Pi[R_{ij},R_{ab}]=[R_{ij},R_{ab}]$ implies
  \begin{equation}\label{RiRj}
  \begin{split}
  \slashed{\mathcal{L}}_{R_{ij}}R_{ab}=&[R_{ij},R_{ab}]=\slashed\Pi_c^d[R_{ij},R_{ab}]^c\p_d=\slashed\Pi_c^d(R_{ij}{R_{ab}}^c
  -R_{ab}{R_{ij}}^c)\p_d\\
  =&\delta_j^aR_{ib}-\delta_i^aR_{jb}-\delta_j^bR_{ia}+\delta_i^bR_{ja}-\upsilon_{ij}\tilde T^ag_{bc}(\slashed d^Ax^c)X_A
  +\upsilon_{ij}\tilde T^bg_{ac}(\slashed d^Ax^c)X_A\\
  &+\upsilon_{ab}\tilde T^ig_{jc}(\slashed d^Ax^c)X_A-\upsilon_{ab}\tilde T^jg_{ic}(\slashed d^Ax^c)X_A-\upsilon_{ab}(R_{ij}\tilde T^c)g_{cd}(\slashed d^Ax^d)X_A\\
  &+\upsilon_{ij}(R_{ab}\tilde T^c)g_{cd}(\slashed d^Ax^d)X_A
  \end{split}
  \end{equation}
  with the help of $\slashed\Pi_c^d\p_d=g_{cd}(\slashed d^Ax^d)X_A$.
  Thus, $|\slashed{\mathcal{L}}_{R_{ij}}R_{ab}|\lesssim s$ holds.
  \item Assume that \eqref{he} holds for the orders less than or equal to $k-1$ $(1\leq k\leq N-1)$. We now
  prove that \eqref{he} holds true for $k$.
  \begin{enumerate}[(A)]
   \item {\bf Estimates of $\slashed{\mathcal{L}}_{R}^k\check{\chi}$ and
   $\slashed{\mathcal{L}}_{R}^k\leftidx{^{(R)}}{\slashed\pi}_{\mathring L}$}

   It follows from the expression of $\slashed{\mathcal{L}}_{R}^k\leftidx{^{(R)}}{\slashed\pi}_{\mathring LA}$ in \eqref{Rpi} and
   the induction assumptions that
   \begin{equation}\label{Y-35}
   |\slashed{\mathcal{L}}_{R}^k\leftidx{^{(R)}}{\slashed\pi}_{\mathring L}|\lesssim s|\slashed{\mathcal{L}}_{R}^k\check{\chi}|+ M\delta^{1-\varepsilon_0} s^{-1}.
   \end{equation}
   Then $\slashed{\mathcal{L}}_{R}^k\leftidx{^{(R)}}{\slashed\pi}_{\mathring L}$ can be bounded as soon as
   the $L^{\infty}$ norm of  $\slashed{\mathcal{L}}_{R}^k\check{\chi}$ is derived. Then it remains to estimate $\slashed{\mathcal{L}}_{R}^k\check{\chi}$.
  By $(\star)$, it follows from \eqref{Lchi'} and the induction assumptions up to the orders $k-1$  that
   \begin{equation}\label{RLchi}
   |\slashed{\mathcal{L}}_{R}^k\mathring L\check{\chi}|\lesssim  M\delta^{1-\varepsilon_0} s^{-2} |\slashed{\mathcal{L}}_{R}^k\check{\chi}|+ M\delta^{1-\varepsilon_0} s^{-7/2}.
   \end{equation}
   In addition, using the identity of  commutator $[\mathring L,R_{ij}]$ in \eqref{c} yields
   \begin{equation}\label{LRichi}
   \begin{split}
   &\slashed{\mathcal{L}}_{\mathring L}\slashed{\mathcal{L}}_{R}^k\check{\chi}_{AB}=\slashed{\mathcal{L}}_{R}^k(\mathring L\check{\chi}_{AB})+[\slashed{\mathcal{L}}_{\mathring L},\slashed{\mathcal{L}}_{R}^k]\check{\chi}_{AB}\\
   =&\slashed{\mathcal{L}}_{R}^k(\mathring L\check{\chi}_{AB})
   +\sum_{k_1+k_2=k-1}\slashed{\mathcal{L}}_{R}^{k_1}\slashed{\mathcal L}_{[\mathring L,R]}\slashed{\mathcal{L}}_{R}^{k_2}\check{ \chi}_{AB}\\
   =&\slashed{\mathcal{L}}_{R}^k(\mathring L\check{\chi}_{AB})+\sum_{k_1+k_2=k-1}\slashed{\mathcal{L}}_{R}^{k_1}\Big\{{\leftidx{^{(R)}}{\slashed\pi}_{\mathring L}}^C(\slashed\nabla_C\slashed{\mathcal{L}}_{R}^{k_2}\check{\chi}_{AB})\\
   &+\slashed{\mathcal{L}}_{R}^{k_2}\check{\chi}_{BC}(\slashed\nabla_A{\leftidx{^{(R)}}{\slashed \pi}_{\mathring L}}^C)+\slashed{\mathcal{L}}_{R}^{k_2}\check{\chi}_{AC}(\slashed\nabla_B{\leftidx{^{(R)}}{\slashed\pi}_{\mathring L}}^C)\Big\},
   \end{split}
   \end{equation}
   here the constant coefficients are neglected on the right hand of the second equality. Applying \eqref{RLchi} and
   the induction assumptions to estimate the right hand side of \eqref{LRichi},  one has
   \begin{equation}\label{eLRchi}
   |\slashed{\mathcal{L}}_{\mathring L}\slashed{\mathcal{L}}_{R}^k\check{\chi}|\lesssim  M\delta^{1-\varepsilon_0} s^{-2}|\slashed{\mathcal{L}}_{R}^k\check{\chi}|+ M\delta^{1-\varepsilon_0} s^{-7/2}.
   \end{equation}
Due to
   \begin{equation*}
   \mathring L\big(\varrho^4|\slashed{\mathcal{L}}_{R}^k\check{\chi}|^2\big)=\varrho^4
   \big\{-4\check{\chi}^{AB}(\slashed{\mathcal{L}}_{R}^k\check{\chi}_{AC})({\slashed{\mathcal{L}}_{R}^k\check{\chi}_B}^C)
   +2(\slashed{\mathcal{L}}_{\mathring L}\slashed{\mathcal{L}}_{R}^k\check{\chi}_{AB})(\slashed{\mathcal{L}}_{R}^k\check{\chi}^{AB})\big\},
   \end{equation*}
 then \eqref{chi'} and \eqref{eLRchi} give the estimate
 \begin{equation}\label{LrRchi}
 |\mathring L(\varrho^2|\slashed{\mathcal{L}}_{R}^k\check{\chi}|)|\lesssim  M\delta^{1-\varepsilon_0}s^{-2}(\varrho^2|\slashed{\mathcal{L}}_{R}^k\check{\chi}|)+ M\delta^{1-\varepsilon_0}s^{-3/2}.
 \end{equation}
 Integrating the inequality \eqref{LrRchi} along integral curves of $\mathring L$ and applying \eqref{Y-35} yield that for small $\delta>0$,
  $$|\slashed{\mathcal{L}}_{R}^k\check{\chi}|+s^{-1}|\slashed{\mathcal{L}}_{R}^k\leftidx{^{(R)}}{\slashed\pi}_{\mathring L}|\lesssim M\delta^{1-\varepsilon_0} s^{-2}.$$

   \item {\bf Estimates of $R^{k+1}\check{L}^j$, $R^{k+1}{\mathring L}^j$, $R^{k+1}\upsilon_{ij}$, $R^{k+1}\check\varrho$ and $\slashed{\mathcal{L}}_{R}^k\leftidx{^{(R)}}{\slashed\pi}$}

Since
   \begin{equation}\label{RiL}
   \begin{split}
   R^{k}R_{ab}\check{L}^j=R^k ({R_{ab}}^A\slashed d_A\check{L}^j)
   \overset{\eqref{deL}}=&R^k\big\{{R_{ab}}^A\check\chi_{AB}(\slashed d^Bx^j)-(\mathcal{FG})_{\mathring L\tilde T}(R_{ab})\tilde T^j\\
   &-\f12(\mathcal{FG})_{\tilde T\tilde T}(R_{ab})\tilde T^j+{R_{ab}}^A\Lambda_{AB}(\slashed d^Bx^i)\big\},
   	\end{split}
   	\end{equation}
   this implies $|R^{k+1}\check{L}^j|\lesssim  M\delta^{1-\varepsilon_0} s^{-1}$ by the induction assumptions, $(\star)$ and the estimate of $\slashed{\mathcal L}_{R}^k\check\chi$ in Part $(1)$. Hence $|R^{k+1}{\mathring L}^j|\lesssim 1$
   holds for small $\dl>0$.

   The estimates for $R^{k+1}\upsilon_{ij}$ and $R^{k+1}\check\varrho$ can be obtained by \eqref{omega} and \eqref{rrho} directly, and $|\slashed{\mathcal{L}}_{R}^k\leftidx{^{(R)}}{\slashed\pi}|\lesssim  M\delta^{1-\varepsilon_0}s^{-1}$ holds by \eqref{Rpi}.

   \item {\bf Estimate of $\slashed{\mathcal{L}}_{R}^{k+1}R$}

   Taking the Lie derivatives $R^k$ on both hand sides of \eqref{RiRj}, using the results in Part $(1)-(2)$, the induction assumptions for
$k-1$ and $(\star)$, one can estimate $\slashed{\mathcal{L}}_{R}^{k+1}R$ in \eqref{Rh}
directly.

   \item {\bf{Estimate of $R^{k+2}x^j$}}

Note that
   \begin{equation*}
   \begin{split}
   R^{k+1}R_{ab}x^j&=R^{k+1}(\Omega_{ab}-\upsilon_{ab}\tilde{T})x^j=R^{k+1}(x^a\delta_b^j-x^b\delta_a^j-\upsilon_{ab}\tilde T^j).
   \end{split}
   \end{equation*}
 Then by induction assumptions up to the orders of $k-1$, the estimate of $R^{k+1}\upsilon_{ab}$ and $R^{k+1}{\tilde T}^j=(-R^{k+1}g^{0j}-R^{k+1}\mathring L^j)$ in Part $(2)$, $|R^{k+2}x^j|\lesssim s$ is obtained, and hence $|\slashed{\mathcal{L}}_{R}^{k+1}\slashed dx^j|\lesssim 1$ holds as for \eqref{RRx}.
  \end{enumerate}
 \end{enumerate}
\end{proof}

With the help of Lemma \ref{Rh}, the estimates for the higher order derivatives of $\mu$ with the rotational
vector fields can be derived as follows.

\begin{proposition}\label{Rmuh}
 Under the same assumptions in Lemma \ref{Rh}, it holds that
 \begin{equation}\label{muh}
 |R^{k+1}\mu|\lesssim  M^2\delta^{1-2\varepsilon_0},\quad k\leq N-1.
 \end{equation}
\end{proposition}

\begin{proof}
  This can be proved by induction.

Note that
 \begin{equation}\label{Y-7}
 \mathring L\big(\varrho^2|\slashed d\mu|^2\big)=\varrho^2\big\{-2\check{\chi}^{AB}(\slashed d_A\mu)(\slashed d_B\mu)
 +2\slashed g^{AB}(\slashed d_A\mathring L\mu)(\slashed d_B\mu)\big\}.
 \end{equation}
 Substituting \eqref{lmu} into \eqref{Y-7}, and applying $(\star)$, \eqref{chi'} and \eqref{GT}, one gets
 \[
 |\mathring L\big(\varrho|\slashed d\mu|\big)|\lesssim M\delta^{1-\varepsilon_0} s^{-2}\big(\varrho|\slashed d\mu|\big)
 +M^2\delta^{1-2\varepsilon_0} s^{-3/2}.
 \]
 Hence,
 \begin{equation}\label{dmu}
 |\slashed d\mu|\lesssim M^2\delta^{1-2\varepsilon_0} s^{-1},
 \end{equation}
 which yields \eqref{muh} for $k=0$.

Assume that \eqref{muh} holds for the orders up to $k-1$ $(1\leq k\leq N-1)$. By \eqref {c}, one has
 \begin{equation}\label{LRimu}
 \begin{split}
 \mathring LR^{k+1}\mu&=[\mathring L,R^{k+1}]\mu+R^{k+1}\mathring L\mu\\
 &=\sum_{k_1+k_2=k}\slashed {\mathcal L}_{R}^{k_1}\big(\leftidx{^{(R)}}{{\slashed\pi_{\mathring L}}^A}\slashed d_AR^{k_2}\mu\big)
 +R^{k+1}\mathring L\mu.
 \end{split}
 \end{equation}
Using \eqref{lmu} and \eqref{GT}, the induction assumptions, Lemma \ref{Rh} and  $(\star)$ to
estimate the right hand side of \eqref{LRimu},
one can obtain
 \begin{equation}\label{LRmu}
 |\mathring LR^{k+1}\mu|\lesssim  M\delta^{1-\varepsilon_0} s^{-2}|R^{k+1}\mu|+ M^2\delta^{1-2\varepsilon_0}s^{-3/2}.
 \end{equation}
Integrating \eqref{LRmu} along integral curves of $\mathring L$ yields that  for small $\delta>0$,
$$|R^{k+1}\mu|\lesssim  M^2\delta^{1-2\varepsilon_0}.$$
\end{proof}

As a consequence of Lemma \ref{Rh} and Proposition \ref{Rmuh}, it follows from the expressions of $\leftidx{^{(R)}}{\slashed\pi}_{TA}$, $\leftidx{^{(T)}}{\slashed\pi}_{AB}$ and $\leftidx{^{(T)}}{\slashed\pi}_{\mathring LA}$
in \eqref{Rpi} and \eqref{Lpi} that for $k\leq N-1$,
\begin{equation}\label{piL}
|\slashed{\mathcal{L}}_{R}^k\leftidx{^{(R)}}{\slashed\pi}_{T}|\lesssim  M^2\delta^{1-2\varepsilon_0} s^{-1},\quad|\slashed{\mathcal{L}}_{R}^k\leftidx{^{(T)}}{\slashed\pi}|\lesssim M\delta^{-\varepsilon_0}s^{-1},\quad|\slashed{\mathcal{L}}_{R}^k\leftidx{^{(T)}}{\slashed\pi}_{\mathring L}|
\lesssim M\delta^{-\varepsilon_0}s^{-1}.
\end{equation}

For derivatives involving $T$, one can get the following estimates by similar analysis for Lemma \ref{Rh} and Proposition \ref{Rmuh}.

\begin{proposition}\label{TRh}
 Under the assumptions $(\star)$, for any operator $\bar Z\in \{T, R_{ij}\}$, and suitably small $\delta>0$,
it holds that for $k\leq N-1$,
 \begin{equation}\label{The}
 \begin{split}
 &|\slashed{\mathcal{L}}_{\bar Z}^{k+1;l}\slashed dx^j|\lesssim \delta^{1-l}s^{-1},
 \quad|\slashed{\mathcal{L}}_{\bar Z}^{k;l}\check{\chi}|\lesssim  M\delta^{1-l-\varepsilon_0}s^{-2},
 \quad |\slashed{\mathcal{L}}_{\bar Z}^{k+1;l}R|\lesssim  M^2\delta^{2-l-2\varepsilon_0}s^{-1},\\
 &|\slashed{\mathcal{L}}_{\bar Z}^{k;l}\leftidx{^{(R)}}{\slashed\pi}|\lesssim
  M\delta^{1-l-\varepsilon_0}s^{-1},
  \quad|\slashed{\mathcal{L}}_{\bar Z}^{k;l}\leftidx{^{(R)}}{\slashed\pi}_{\mathring L}|\lesssim  M\delta^{1-l-\varepsilon_0}s^{-1},\quad|\bar Z^{k+1;l}\check\varrho|\lesssim M\delta^{1-l-\varepsilon_0}s^{-1},\\
 &|\bar Z^{k+1;l}\check{L}^j|\lesssim M\delta^{1-l-\varepsilon_0}s^{-1},
 \quad|\bar Z^{k+1;l}\upsilon_{ij}|\lesssim  M\delta^{1-l-\varepsilon_0},\quad|\bar Z^{k+1;l}\mu|\lesssim M^2\delta^{1-l-2\varepsilon_0},\\
 &|\slashed{\mathcal{L}}_{\bar Z}^{k;l}\leftidx{^{(R)}}{\slashed\pi}_{T}|\lesssim  M^2\delta^{1-l-2\varepsilon_0}s^{-1},
 \quad |\slashed{\mathcal{L}}_{\bar Z}^{k;l}\leftidx{^{(T)}}{\slashed\pi}|\lesssim  M\delta^{-l-\varepsilon_0}s^{-1},  \quad |\slashed{\mathcal{L}}_{\bar Z}^{k;l}\leftidx{^{(T)}}{\slashed\pi}_{\mathring L}|\lesssim  M\delta^{-l-\varepsilon_0}s^{-1},
 \end{split}
 \end{equation}
 where the array $(k;l)$ means that the number of $\bar Z$ is $k$ and the number of $T$ is $l$ ($l\geq 1$).
\end{proposition}

\begin{proof}
We start with the special case of $l=1$.
 \begin{enumerate}
  \item
 Taking Lie derivative $R^{p}$ on both hand sides of \eqref{Tchi'},
 and applying Lemma \ref{Rh}, Proposition \ref{Rmuh} and $(\star)$, one has that for $p\leq N-2$,
 \begin{equation}\label{RTchi}
 |\slashed{\mathcal{L}}_{R}^{p}\slashed{\mathcal{L}}_T\check{\chi}|\lesssim  M\delta^{-\varepsilon_0} s^{-2}.
 \end{equation}
By $[T,R]=\leftidx{^{(R)}}{{\slashed\pi_T}^A}X_A$ in \eqref{c}
  and
 \[
 \slashed{\mathcal{L}}_{\bar Z}^{k;1}\check{\chi}
 =\slashed{\mathcal{L}}_{R}^{p_1}\slashed{\mathcal{L}}_T\slashed{\mathcal{L}}_{R}^{p_2}\check{\chi}
=\slashed{\mathcal{L}}_{R}^{p_1}\slashed{\mathcal{L}}_{R}^{p_2}\slashed{\mathcal{L}}_T\check{\chi}
+\sum_{q_1+q_2=p_2-1}\slashed{\mathcal{L}}_{R}^{p_1}
\slashed{\mathcal{L}}_{R}^{q_1}\slashed{\mathcal{L}}_{[T,R]}\slashed{\mathcal{L}}_{R}^{q_2}\check{\chi},
 \]
then $|\slashed{\mathcal{L}}_{\bar Z}^{k;1}\check{\chi}|\lesssim  M\delta^{-\varepsilon_0} s^{-2}$ is shown for $k\leq N-1$
and $l=1$ by making use of \eqref{RTchi}, \eqref{piL} and the estimate of $\check\chi$ in Lemma \ref{Rh}.
 \item Note that
 \begin{equation*}
 \begin{split}
 &\slashed{\mathcal L}_T\slashed dx^j=(\slashed d\mu)\tilde{T}^j+\mu\slashed d\tilde{T},\\
 &T\check L^j=T\mathring L^j-\f{\mu-1}{\varrho}\tilde{T}^j-\f{\check T^j}{\varrho},\\
 &\slashed{\mathcal L}_TR=[T,R]=\leftidx{^{(R)}}{\slashed\pi_T}^AX_A.
 \end{split}
 \end{equation*}
 Following similar arguments in the part (1), one can get the estimates in \eqref{The}
 for $\slashed{\mathcal{L}}_{\bar Z}^{k+1;1}\slashed dx^j$, $\slashed{\mathcal{L}}_{\bar Z}^{k+1;1}R$, $\bar Z^{k+1;1}\check{L}^j$, $\bar Z^{k+1;1}\check\varrho$ and $\bar Z^{k+1;1}\upsilon_{ij}$ by virtue of the equalities \eqref{TL}, \eqref{rrho} and \eqref{omega}.

\item
Now we turn to the estimate of $\mu$ by induction.

By \eqref{c}, one has
\begin{equation*}
\begin{split}
\mathring LT\mu&=[\mathring L,T]\mu+T\mathring L\mu=(-\slashed d^A\mu-2\mu\zeta^A)\slashed d_A\mu+T\mathring L\mu.
\end{split}
\end{equation*}
Then it follows from the equation \eqref{lmu}, $(\star)$, \eqref{GT} and \eqref{muh} that
\begin{equation}\label{LTmu}
|\mathring LT\mu|\lesssim  M^2\delta^{-2\varepsilon_0}s^{-3/2}+M\delta^{1-\varepsilon_0}s^{-5/2}|T\mu|.
\end{equation}
Integrating \eqref{LTmu} along integral curves of $\mathring L$ and applying Gronwall's inequality yield
\begin{equation}\label{Y-9}
|T\mu|\lesssim  M^2\delta^{-2\varepsilon_0}.
\end{equation}
Inductively, one can assume that for $q\leq p-1$ and $p\leq N-1$,
\begin{equation}\label{Y-10}
|R^q T\mu|\lesssim  M^2\delta^{-2\varepsilon_0}.
\end{equation}
Note that
\begin{equation*}
\begin{split}
\mathring LR^p T\mu=&\sum_{p_1+p_2=p-1} R^{p_1}\big(\leftidx{^{(R)}}{{\slashed\pi_{\mathring L}}^A}\slashed d_AR^{p_2}T\mu\big)
+R^p\big\{(-\slashed d^A\mu-2\mu\zeta^A)\slashed d_A\mu+T\mathring L\mu\big\}.
\end{split}
\end{equation*}
It follows from this, \eqref{Y-10}, Lemma \ref{Rh} and Proposition \ref{Rmuh} that
\[
 |\mathring LR^p T\mu|\lesssim  M\delta^{1-\varepsilon_0} s^{-2}|R^p T\mu|
 + M^2\delta^{-2\varepsilon_0}s^{-3/2},
\]
which implies $|R^{p}T\mu|\lesssim M^2\delta^{-2\varepsilon_0}$, and
hence $|\bar Z^{k+1;1}\mu|\lesssim M^2\delta^{-2\varepsilon_0}$. The estimates of $\slashed{\mathcal{L}}_{\bar Z}^{k;1}\leftidx{^{(R)}}{\slashed\pi}$, $\slashed{\mathcal{L}}_{\bar Z}^{k;1}\leftidx{^{(T)}}{\slashed\pi}$,
$\slashed{\mathcal{L}}_{\bar Z}^{k;1}\leftidx{^{(T)}}{\slashed\pi}_{\mathring L}$ and $\slashed{\mathcal{L}}_{\bar Z}^{k;1}\leftidx{^{(R)}}{\slashed\pi}_{\mathring L}$ can be obtained
directly by \eqref{Lpi} and \eqref{Rpi}.
Meanwhile, the estimate for $\slashed{\mathcal{L}}_{\bar Z}^{k;1}\leftidx{^{(R)}}{\slashed\pi}_{T}$
follows from \eqref{Rpi}.

\end{enumerate}

Analogously, if the number of the derivatives with respect to $T$ in \eqref{The} is $l$ $(l \geq 2)$, one can utilize an induction
argument on $l$ to get the estimates in \eqref{The}.
\end{proof}

We now improve $L^\infty$ estimates of some derivatives of $\varphi_\g$ and $\phi$
with respect to the vector field $\bar Z$ ($\bar Z\in \{T, R_{ij}\}$).

\begin{corollary}\label{barZ}
	Under the assumptions $(\star)$, for any operator $\bar Z\in\{T, R_{ij}\}$ and suitably small $\delta>0$, it holds that for $k\leq N-1$,
	\begin{align}
	&|\bar Z^{k;l}\mathring{\underline L}\varphi_\gamma(s, u, \vartheta)|+|\bar Z^{k;l}T\varphi_\gamma(s, u, \vartheta)|\lesssim\delta^{-l-\varepsilon_0}s^{-3/2},\label{barzt}
	\end{align}
	here $l$ is the number of $T$ in $\bar Z^k$.
\end{corollary}

\begin{proof}
	For any $p$ with $p+l\leq N-1$, by Lemma \ref{com}, one has
	\begin{equation}\label{RTphi}
	\begin{split}
	&R^{p}T^l(H_\g+\mu\slashed\triangle\varphi_\gamma)=R^{p}T^l\big(\mathring L\mathring{\underline L}\varphi_\gamma+\f{3}{2\varrho}\mathring{\underline L}\varphi_\gamma\big)\\
	=&\mathring L\big(R^p T^l\mathring{\underline L}\varphi_\gamma\big)+\f{3}{2\varrho}R^p T^l\mathring{\underline L}\varphi_\gamma-\sum_{p_1+p_2=p-1}R^{p_1}\big(\leftidx{^{(R)}}{\slashed\pi}_{\mathring L}^A\slashed d_AR^{p_2}T^{l}\mathring{\underline L}\varphi_\gamma\big)\\
	&-\sum_{l_1+l_2=l-1}R^{p}T^{l_1}\big(\leftidx{^{(T)}}{\slashed\pi}_{\mathring L}^A\slashed d_AT^{l_2}\mathring{\underline L}\varphi_\gamma\big)+\sum_{l_1+l_2=l,l_1\geq 1}T^{l_1}(\f{3}{2\varrho})
	\cdot R^{p}T^{l_2}\mathring{\underline L}\varphi_\gamma.
	\end{split}
	\end{equation}
Note that the terms in the last three summations on the right hand side of \eqref{RTphi} can be estimated by Lemma \ref{Rh} and Proposition \ref{TRh}.
In addition, one can estimate $R^{p}T^lH_\g$
	by the expression of $H_\gamma$ in \eqref{H}. Then it is derived from \eqref{RTphi} that
\begin{equation}\label{H7-1}
\begin{split}
|\mathring L\big(\varrho^{3/2} R^p T^l\mathring{\underline L}\varphi_\gamma\big)|\lesssim M^2\delta^{1-l-2\varepsilon_0}s^{-2}.
\end{split}
\end{equation}	
Subsequently, integrating \eqref{H7-1} along the integral
curves of $\mathring L$ yields
	\[
	|\varrho^{3/2} R^p T^l\mathring{\underline L}\varphi_\gamma(s, u, \vartheta)-{\varrho_0}^{3/2} R^p T^l\mathring{\underline L}\varphi_\gamma(t_0, u, \vartheta)|\lesssim M^2\delta^{1-l-2\varepsilon_0}.
	\]
	This implies that for small $\delta>0$,
	\begin{equation}\label{RTvphi}
	|R^p T^l\mathring{\underline L}\varphi_\gamma(s, u, \vartheta)|\lesssim\delta^{-l-\varepsilon_0}s^{-3/2}.
	\end{equation}
Thus, \eqref{barzt} follows from \eqref{RTvphi} and $\mathring{\underline L}=2T+\mu\mathring L$.
\end{proof}

To further improve the estimate of $\varphi_\gamma$, based on \eqref{barzt} and
 $T=\f{\p}{\p u}-\Xi^AX_A$, it is necessary to treat the components of $\Xi$.

\begin{lemma}\label{eXi}
	Under the assumptions $(\star)$, for suitably small $\delta>0$, it then holds
	\begin{equation}\label{Xi}
	|\Xi|\lesssim  M\delta^{-\varepsilon_0}s.
	\end{equation}
\end{lemma}

\begin{proof}
	Note that $\Xi$ is a vector field on $S_{s, u}$ and one has
	\begin{equation}\label{LXi}
	\mathring L\Xi^A=[T,\mathring L]^A=\big(\mathscr{D}_T\mathring L-\mathscr{D}_{\mathring L}T\big)^A=\slashed d^A\mu+2\mu\zeta^A.
	\end{equation}
	Then
	\[
	\mathring L(\varrho^{-2}|\Xi|^2)=\varrho^{-2}\big\{2\check\chi_{AB}\Xi^A\Xi^B+2\slashed g_{AB}(\slashed d^A\mu
	+2\mu\zeta^A)\Xi^B\big\}.
	\]
	This, together with the estimates \eqref{dmu}, \eqref{echi'}, $(\star)$ and \eqref{zeta}, yields
	\begin{equation}\label{Y1}
	|\mathring L(\varrho^{-1}|\Xi|)|\lesssim \varrho^{-1}\big\{ M\delta^{1-\varepsilon_0} s^{-2}|\Xi|
	+ M\delta^{-\varepsilon_0} s^{-1}\big\}.
	\end{equation}
	Thus \eqref{Xi} follows immediately from integrating \eqref{Y1} along integral curves of $\mathring L$.
\end{proof}

\begin{corollary}\label{Phi}
	Under the assumptions $(\star)$ with suitably small $\delta>0$, it holds that for $k\leq N-1$,
	\begin{align}
&|\bar Z^{k;l}\varphi_\gamma(s, u, \vartheta)|\lesssim \delta^{1-l-\varepsilon_0}s^{-3/2},\label{barz}\\
&|\bar Z^{k;l}\phi(s, u, \vartheta)|\lesssim \delta^{2-l-\varepsilon_0}s^{-3/2}.\label{barzp}
	\end{align}
\end{corollary}

\begin{proof}
		Notice that \eqref{barzt} implies
		\begin{equation}\label{Tvarphi}
		|T\bar Z^{k;l}\varphi_\gamma(s, u, \vartheta)|\lesssim\delta^{-l-\varepsilon_0}s^{-3/2}.
		\end{equation}
	Substituting $T=\f{\p}{\p u}-\Xi^AX_A$ into \eqref{Tvarphi}, and using $(\star)$ and \eqref{Xi}, one has
		\begin{align}\label{Z-1}
		|\f{\p}{\p u}\bar Z^{k;l}\varphi_\gamma(s, u,\vartheta)|\lesssim\delta^{-\varepsilon_0-l} s^{-3/2}.
		\end{align}
		Integrating \eqref{Z-1} from $0$ to $u$ leads to the estimate of $\bar Z^{k;l}\varphi_\gamma$ in \eqref{barz} since $\bar Z^{k;l}\varphi_\g$ vanishes on $C_0^s$.
		
		Next we derive the estimate \eqref{barzp}. Note that
		\begin{equation*}
		|\bar Z^{k;l}T\phi|=|\bar Z^{k;l}\big(\mu\tilde T^j\varphi_j\big)|\lesssim\delta^{1-l-\varepsilon_0}s^{-3/2}.
		\end{equation*}
		Then by the bootstrap assumption $(\star)$ for $\phi$, one has $
		|T\bar Z^{k;l}\phi|\lesssim\delta^{1-l-\varepsilon_0}s^{-3/2},$
		which implies
		\begin{equation}\label{Tphi}
		|\f{\p}{\p u}\bar Z^{k;l}\phi|\lesssim\delta^{1-l-\varepsilon_0}s^{-3/2}.
		\end{equation}
		Thus \eqref{barzp} holds after integrating \eqref{Tphi} from $0$ to $u$.
\end{proof}

It is noticed that the results in Proposition \ref{TRh} only involve only the vectorfields $\{T, R_{ij}\}$.
We next  deal with the corresponding quantities involving the remaining scaling vectorfield $\varrho\mathring L$ in the frame
$\{\varrho\mathring L, T, R_{ij}\}$. In fact, there are
explicit expressions of $\check\chi$ and $\varrho\check L^i$ under the action of $\mathring L$ (see \eqref{Lchi'} and \eqref{LeL}),
then by taking the Lie derivative of $\varrho\mathring L\check\chi$ with respect to $\bar Z$, one obtains the estimate
of $\slashed{\mathcal L}_{\bar Z}^k\slashed{\mathcal L}_{\varrho\mathring L}\check\chi$ as in \eqref{RTchi}.
On the other hand, the terms containing the derivatives of $\varrho\mathring L$ can be estimated similarly as in the proof of Proposition \ref{TRh}, we omit the tedious analysis and
give the conclusions directly in the following proposition:

\begin{proposition}\label{LTRh}
	Under the assumptions $(\star)$ with suitably small $\dl>0$, for any operate $Z\in\{\varrho\mathring L, T, R_{ij}\}$,
it holds that for $k\leq N-1$,
	\begin{equation}\label{z}
	\begin{split}
	&|\slashed{\mathcal{L}}_{Z}^{k+1;l,m}\slashed dx^j|\lesssim \delta^{-l},\quad|\slashed{\mathcal{L}}_{ Z}^{k;l,m}\check{\chi}|\lesssim M\delta^{1-l-\varepsilon_0}s^{-2},\quad |\slashed{\mathcal{L}}_{Z}^{k+1;l,m}R|\lesssim M\delta^{1-l-\varepsilon_0},\\
	&|\slashed{\mathcal{L}}_{Z}^{k;l,m}\leftidx{^{(R)}}{\slashed\pi}|\lesssim M\delta^{1-l-\varepsilon_0}s^{-1},\quad|\slashed{\mathcal{L}}_{Z}^{k;l,m}\leftidx{^{(R)}}{\slashed\pi}_{\mathring L}|\lesssim M\delta^{1-l-\varepsilon_0}s^{-1},\quad|Z^{k+1;l,m}\check\varrho|\lesssim M\delta^{1-l-\varepsilon_0}s^{-1},\\
	&|Z^{k+1;l,m}\check{L}^j|\lesssim M\delta^{1-l-\varepsilon_0}s^{-1},\quad| Z^{k+1;l,m}\upsilon_{ij}|\lesssim M\delta^{1-l-\varepsilon_0},\quad|Z^{k+1;l,m}\mu|\lesssim M^2\delta^{1-l-2\varepsilon_0}s^{-1},\\ &|\slashed{\mathcal{L}}_{Z}^{k;l,m}\leftidx{^{(R)}}{\slashed\pi}_{T}|\lesssim M^2\delta^{1-l-2\varepsilon_0}s^{-1},\quad |\slashed{\mathcal{L}}_{Z}^{k;l,m}\leftidx{^{(T)}}{\slashed\pi}|\lesssim M\delta^{-l-\varepsilon_0}s^{-1},\quad |\slashed{\mathcal{L}}_{Z}^{k;l,m}\leftidx{^{(T)}}{\slashed\pi}_{\mathring L}|\lesssim M\delta^{-l-\varepsilon_0}s^{-1},
	\end{split}
	\end{equation}
	 where the array $(k;l,m)$ means that the number of $Z$ is $k$, the number of $T$ is $l$ , and the number
of $\varrho \mathring L$ is $m$ ($m\geq 1$).
\end{proposition}

\eqref{barzt} implies that $|R^{p}T^l\mathring{\underline L}\varphi_\gamma(s, u, \vartheta)|\lesssim\delta^{-l-\varepsilon_0}s^{-3/2}$
for $p+l\leq N-1$. Assume that the following induction assumptions hold:
\begin{equation}\label{Y-13}
|R^p T^l(\varrho\mathring L)^{m-1}\mathring{\underline L}\varphi_\gamma(s, u, \vartheta)|\lesssim\delta^{-l-\varepsilon_0}s^{-3/2}
\quad\text{for $p+l+m-1\leq N-1$ and $m\geq 1$}.
\end{equation}
It follows from \eqref{fequation} that
\begin{equation}\label{Y-14}
R^p T^l(\varrho\mathring L)^m\mathring{\underline L}\varphi_\gamma=R^p T^l(\varrho\mathring L)^{m-1}\big(-\f32\mathring{\underline L}\varphi_\gamma+\varrho H_\gamma+\varrho\mu\slashed\triangle\varphi_\gamma\big).
\end{equation}
According to the explicit expression \eqref{H} of $H_\g$, the second term in the bracket of the right hand side of \eqref{Y-14}
can be estimated easily by \eqref{he}, \eqref{TRh} or \eqref{z}. This, together with
the assumption \eqref{Y-13}, implies that when $p+l+m\leq N$,
\begin{equation}\label{RTLT}
|R^p T^l(\varrho\mathring L)^m\mathring{\underline L}\varphi_\gamma|\lesssim\delta^{-l-\varepsilon_0}s^{-3/2}.
\end{equation}
This yields $|\f{\p}{\p u}R^p T^l(\varrho\mathring L)^m\varphi_\g|\lesssim\delta^{-l-\varepsilon_0}s^{-3/2}$ by the
same argument as for \eqref{Z-1}.
Thus, for $p+l+m\leq N$ and $m\geq 1$,
\begin{equation}\label{RTL}
|R^p T^l(\varrho\mathring L)^m\varphi_\gamma|\lesssim\delta^{1-l-\varepsilon_0}s^{-3/2}.
\end{equation}

Using Lemma \ref{com} and rearranging the orders of derivatives in \eqref{RTLT} and \eqref{RTL}, one can get
\begin{corollary}\label{Z}
	Under the assumptions $(\star)$ with suitably small $\dl>0$, for $Z\in\{\varrho\mathring L, T, R_{ij}\}$ and $k\leq N-1$, it then holds that
	\begin{align}
	&|Z^{k+2;l+1,m+1}\varphi_\gamma(s, u, \vartheta)|\lesssim\delta^{-l-\varepsilon_0}s^{-3/2},\label{ztl}\\
	&|Z^{k+1;l,m+1}\varphi_\gamma(s, u, \vartheta)|\lesssim\delta^{1-l-\varepsilon_0}s^{-3/2},\label{zl}\\
	&|Z^{k+1;l,m+1}\phi(s, u, \vartheta)|\lesssim\delta^{2-l-\varepsilon_0}s^{-3/2}.\label{zlphi}
	\end{align}
\end{corollary}

In the end of this section, it should be emphasized that under the assumptions $(\star)$,
one has obtained the estimates in Lemma \ref{Rh}, Proposition \ref{Rmuh}-\ref{LTRh} with
the corresponding bounds dependent on $M$.
Furthermore, improved results are obtained in Corollary \ref{barZ}-\ref{Z} with the bounds independent of $M$.
Therefore, if starting with Corollary \ref{barZ}-\ref{Z} and repeating
 the related procedure as in Section \ref{ho},  we can improve the conclusions in Lemma \ref{Rh} and
 Propositions \ref{Rmuh}-\ref{LTRh}
 such that all the related constants are  independent of $M$ when $k\leq N-3$.
From now on, when these propositions are applied, the special constant $M$ will be removed in the related $L^\infty$
 norms since $N$ could be chosen large enough.

\section{Energy estimates for the linearized equation}\label{EE}

It remains to close the bootstrap assumption $(\star)$ for the $N-th$ and $(N+1)-th$ order derivatives, for which the approach in Section \ref{ho} fails to apply.
To this end, we will construct some suitable energies for the higher order derivatives of $\varphi_\g$
to improve the related $L^\infty$ norms. Note that $\varphi_\gamma$ and $\phi$ satisfy the nonlinear equations \eqref{ge} and \eqref{muphi} respectively, and each
derivative of $\varphi_\g$ or $\phi$ also fulfills analogous equations with the same metric. Therefore, in this section, we focus on
the energy estimates for the smooth
solution $\Psi$ to the following linear equation
\begin{equation}\label{gel}
\mu\Box_g\Psi=\Phi,
\end{equation}
where $\Psi$ and its derivatives vanish on $C_0^s$.
We will choose a suitable multiplier $V\Psi$
($V$ is some vector field) and integrate $\mu (\Box_g\Psi)(V\Psi)$ in domain $D^{s, u}$
such that some appropriate energies can be found. To this end, set the {\it{energy-momentum tensor field}} of $\Psi$ in \eqref{gel} as
\begin{equation*}
Q_{\al\beta}=Q_{\al\beta}[\Psi]:=(\p_\al\Psi)(\p_\beta\Psi)-\f12 g_{\al\beta}g^{\nu\lambda}(\p_\nu\Psi)(\p_\lambda\Psi).
\end{equation*}
Then \eqref{gab} yields
$Q_{\al\beta}=(\p_\al\Psi)(\p_\beta\Psi)-\f12 g_{\al\beta}\big\{|\slashed d\Psi|^2-\mu^{-1}(\mathring{\underline L}\Psi)(\mathring L\Psi)\big\}$. Thus,  the components of $Q_{\al\beta}$ relative to $\{\mathring L, \mathring{\underline L}, X_1, X_2, X_3\}$ are
\begin{equation*}
\begin{split}
&Q_{\mathring L\mathring L}=(\mathring L\Psi)^2,\quad Q_{\mathring{\underline L}\mathring{\underline L}}=(\mathring{\underline L}\Psi)^2,\quad Q_{\mathring L\mathring{\underline L}}=\mu|\slashed d\Psi|^2,\\
&Q_{\mathring LA}=(\mathring L\Psi)(\slashed d_A\Psi),\quad Q_{\mathring{\underline L}A}=(\mathring{\underline L}\Psi)(\slashed d_A\Psi),\\
&Q_{AB}=(\slashed d_A\Psi)(\slashed d_B\Psi)-\f12\slashed g_{AB}\{|\slashed d\Psi|^2-\mu^{-1}(\mathring{\underline L}\Psi)(\mathring L\Psi)\big\}.
\end{split}
\end{equation*}
Let $V_1:=\varrho^{2\epsilon}\mathring L$ and $V_2:=\mathring{\underline L}$ with $0<\epsilon<\f12$
being any fixed constant. The energies $E_i[\Psi](s, u)$ and fluxes $F_i[\Psi](s, u)$ $(i=1,2)$
are defined as
\begin{align}
&E_1[\Psi](s, u):=\f12\ds\int_{\Sigma_s^{u}}\mu\varrho^{2\epsilon}\{(\mathring L\Psi)^2+|\slashed d\Psi|^2\},\label{E1}\\
&E_2[\Psi](s, u):=\f12\ds\int_{\Sigma_s^{u}}\{(\mathring{\underline L}\Psi)^2+\mu^2|\slashed d\Psi|^2\},\label{E2}\\
&F_1[\Psi](s, u):=\ds\int_{C_{u}^s}\varrho^{2\epsilon}(\mathring L\Psi)^2,\label{F1}\\
&F_2[\Psi](s,u):=\ds\int_{C_{u}^s}\mu |\slashed d\Psi|^2.\label{F2}
\end{align}
It follows from integration  by parts in $D^{s, u}$ for $(\mu\Box_g\Psi)(V_i\Psi)$ that
\begin{equation}\label{EI}
\begin{split}
&E_i[\Psi](s, u)-E_i[\Psi](t_0, u)+F_i[\Psi](s, u)\\
&=-\ds\int_{D^{s, u}}\Phi\cdot V_i\Psi-\int_{D^{s, u}}\f12\mu Q^{\al\beta}[\Psi]\leftidx{^{(V_i)}}\pi_{\al\beta},\qquad i=1,2,
\end{split}
\end{equation}
where $\leftidx{^{(V_i)}}\pi_{\al\beta}$ is the deformation tensor with respect to the vector field $V_i$.

To compute $Q^{\al\beta}\leftidx{^{(V_i)}}\pi_{\al\beta}$ in \eqref{EI}, one can use the
components of the metric in the frame $\{\mathring L,\underline{\mathring L}, X_1, X_2, X_3\}$, given in \eqref{gab}, to derive directly that
\begin{equation}\label{Y-36}
\begin{split}
&-\f12 \mu Q^{\al\beta}[\Psi]\leftidx{^{(V_1)}}\pi_{\al\beta}\\
=&-\f14\mu^{-1}\leftidx{^{(V_1)}}\pi_{\mathring L\mathring{\underline L}} Q_{\mathring L\mathring{\underline L}}-\f18\mu^{-1}\leftidx{^{(V_1)}}\pi_{\mathring{\underline L}\mathring{\underline L}}Q_{\mathring L\mathring L}+\f12\leftidx{^{(V_1)}}\pi_{\mathring{\underline L}}^AQ_{\mathring LA}-\f12\mu\leftidx{^{(V_1)}}{\slashed{\pi}}^{AB}Q_{AB}\\
=&(\f12\varrho^{2\epsilon}\mathring L\mu+\mu (\epsilon+\f12)\varrho^{2\epsilon-1})|\slashed d\Psi|^2+\big(\epsilon(\mu-2)\varrho^{2\epsilon-1}
-\f12\varrho^{2\epsilon}\mathring L\mu\big)(\mathring L\Psi)^2\\
&+\varrho^{2\epsilon}(\slashed d^A\mu+2\mu\zeta^A)(\mathring L\Psi)(\slashed d_A\Psi)
-\mu\varrho^{2\epsilon}\check\chi_{AB}(\slashed d^A\Psi)(\slashed d^B\Psi)+\f12\mu\varrho^{2\epsilon}\textrm{tr}\check\chi|\slashed d\Psi|^2\\
&-\f12\varrho^{2\epsilon}\textrm{tr}\chi(\mathring L\Psi)(\mathring{\underline L}\Psi)
\end{split}
\end{equation}
and
\begin{equation}\label{Y-37}
\begin{split}
&-\f12 \mu Q^{\al\beta}[\Psi]\leftidx{^{(V_2)}}\pi_{\al\beta}\\
=&\f12(\mathring{\underline L}\mu+\mu\mathring L\mu)|\slashed d\Psi|^2-(2\mu\zeta^A+\slashed d^A\mu)(\slashed d_A\Psi)(\mathring{\underline L}\Psi)-\mu\slashed d^A\mu(\mathring L\Psi)(\slashed d_A\Psi)\\
&-\f12\big(\textrm{tr}\leftidx{^{(T)}}{\slashed\pi}+\mu\textrm{tr}\chi\big)(\mathring L\Psi)(\mathring{\underline L}\Psi)-\mu\big(\mu{\chi}^{AB}+\leftidx{^{(T)}}{\slashed\pi}^{AB}\big)(\slashed d_A\Psi)(\slashed d_B\Psi)\\
&+\f12\mu\big(\textrm{tr}\leftidx{^{(T)}}{\slashed\pi}+\mu\textrm{tr}\chi\big)|\slashed d\Psi|^2.
\end{split}
\end{equation}

Next, applying the results in Sect. \ref{ho} to estimate all the coefficients in \eqref{Y-36} and \eqref{Y-37}, one has
 \begin{equation}\label{QV1}
\begin{split}
&\int_{D^{s, u}}-\f12 \mu Q^{\al\beta}[\Psi]\leftidx{^{(V_1)}}\pi_{\al\beta}\\
\lesssim&\int_0^u\int_{C_{u'}^s}|\slashed d\Psi|^2+\delta^{-1}\int_0^u\int_{C_{u'}^s}\varrho^{2\epsilon}|\mathring L\Psi|^2+\delta\int_{t_0}^s\tau^{2\epsilon-2}\int_{\Sigma_\tau^u}|\mathring{\underline L}\Psi|^2\\
\overset{0<\epsilon<\f12}\lesssim&\int_0^u F_2[\Psi](s,u')du'+\delta^{-1}\int_0^uF_1[\Psi](s,u')du'+\delta\int_{t_0}^s\tau^{2\epsilon-2}E_2[\Psi](\tau,u)d\tau
\end{split}
\end{equation}
and
 \begin{equation}\label{QV2}
 \begin{split}
 &\int_{D^{s, u}}-\f12 \mu Q^{\al\beta}[\Psi]\leftidx{^{(V_2)}}\pi_{\al\beta}\\
 \lesssim&\delta^{-2\varepsilon_0}\int_0^u\int_{C_{u'}^s}|\slashed d\Psi|^2
 +\delta^{-2\varepsilon_0}\int_0^u\int_{C_{u'}^s}\varrho^{2\epsilon}|\mathring L\Psi|^2
 +\int_{t_0}^s\tau^{-2}\int_{\Sigma_\tau^u}|\underline{\mathring L}\Psi|^2\\
 \lesssim&\delta^{-2\varepsilon_0}\int_0^uF_2[\Psi](s,u')du'+\delta^{-2\varepsilon_0}\int_0^uF_1[\Psi](s,u')du'
 +\int_{t_0}^s\tau^{-2}E_2[\Psi](\tau,u)d\tau.
  \end{split}
 \end{equation}
Substituting \eqref{QV1} and \eqref{QV2} into \eqref{EI} and using the Gronwall's inequality yield
 \begin{equation}\label{EF1}
\begin{split}
&E_1[\Psi](s, u)+F_1[\Psi](s, u)\\
\lesssim &E_1[\Psi](t_0, u)+\ds\int_{D^{s, u}}|\varrho^{2\epsilon}\Phi(\mathring L\Psi)|+\int_0^uF_2[\Psi](s,u')du'+\delta\int_{t_0}^s\tau^{2\epsilon-2}E_2[\Psi](\tau,u)d\tau
\end{split}
\end{equation}
and
 \begin{equation}\label{EF2}
 \begin{split}
 &E_2[\Psi](s, u)+F_2[\Psi](s, u)\\
 \lesssim &E_2[\Psi](t_0, u)+\ds\int_{D^{s, u}}|\Phi(\mathring{\underline L}\Psi)|
 +\delta^{-2\varepsilon_0}\int_0^uF_1[\Psi](s,u')du'.
 \end{split}
 \end{equation}

Taking the linear summation of \eqref{EF1} and \eqref{EF2}, and then utilizing Gronwall's inequality again, one
has that for $\epsilon\in(0,\f12)$,
\begin{equation}\label{e}
\begin{split}
&\delta E_2[\Psi](s, u)+\delta F_2[\Psi](s, u)+E_1[\Psi](s, u)+F_1[\Psi](s, u)\\
\lesssim& \delta E_2[\Psi](t_0, u)+E_1[\Psi](t_0, u)+\delta\int_{D^{s, u}}|\Phi\cdot \mathring{\underline L}\Psi|+\int_{D^{s, u}}\varrho^{2\epsilon}|\Phi\cdot \mathring L\Psi|.
\end{split}
\end{equation}

To obtain the energy estimate for $\vp_\g$ and its derivatives, we will apply \eqref{e} to $\Psi=\Psi_\gamma^{k+1}=Z^{k+1}\varphi_\gamma$ and
$\Phi=\Phi_\gamma^{k+1}=\mu\Box_g\Psi_\gamma^{k+1}$ ({$k\leq2N-6$}). Note that
\begin{equation}\label{Psi}
\begin{split}
\Phi_\gamma^{k+1}&=\mu[\Box_g,Z]\Psi_\gamma^{k}+Z\big(\mu\Box_g\Psi_\gamma^{k}\big)-(Z\mu)\Box_g\Psi_\gamma^{k}\\
&=\mu\mathscr D^\al({\leftidx{^{(Z)}}C_{\gamma}^{k}}_{,\al})+(Z+\leftidx{^{(Z)}}\lambda)\Phi_\gamma^{k},
\end{split}
\end{equation}
where
\begin{equation}\label{ClP}
\begin{split} &{\leftidx{^{(Z)}}C_{\gamma}^{k}}_{,\al}=\big(\leftidx{^{(Z)}}\pi_{\nu\al}
-\f12g_{\nu\al}(g_{\kappa\la}\leftidx{^{(Z)}}\pi^{\kappa\la})\big)g^{\nu\beta}\p_\beta\Psi_\gamma^{k},\\
&\leftidx{^{(Z)}}\lambda=\f12g_{\kappa\nu}\leftidx{^{(Z)}}\pi^{\kappa\nu}-\mu^{-1}Z\mu,\\
&\Psi_\gamma^0=\varphi_\gamma,\quad\Phi_\g^0=\mu\Box_g\varphi_\g,
\end{split}
\end{equation}
 and $\Phi_\g^0$ is given by
the right hand side of \eqref{ge}.
Consequently, for $\Psi_\g^{k+1}=Z_{k+1}Z_{k}\cdots Z_{1}\varphi_\gamma$ with $Z_m\in\{\varrho\mathring L, T, R_{ij}\}$,
then it follows from \eqref{Psi} and an induction argument that
\begin{equation}\label{Phik}
\begin{split}
\Phi_\gamma^{k+1}=&\sum_{m=1}^{k}\big(Z_{k+1}+\leftidx{^{(Z_{k+1})}}\lambda\big)\dots\big(Z_{k+2-m}
+\leftidx{^{(Z_{k+2-m})}}\lambda\big)\big(\mu\mathscr D{\leftidx{^{(Z_{k+1-m})}}C_{\gamma}^{k-m}}_{,\al}\big)\\
&+\mu\mathscr D^\al{\leftidx{^{(Z_{k+1})}}C_{\gamma}^{k}}_{,\al}+\big(Z_{k+1}+\leftidx{^{(Z_{k+1})}}\lambda\big)\dots\big(Z_{1}
+\leftidx{^{(Z_1)}}\lambda\big)\Phi_\gamma^0,\qquad k\geq 1,\\
\Phi_\gamma^{1}=&\big(Z_{1}+\leftidx{^{(Z_1)}}\lambda\big)\Phi_\gamma^0+\mu\mathscr D^\al{\leftidx{^{(Z_{1})}}C_{\gamma}^{0}}_{,\al}.
\end{split}
\end{equation}
In order to estimate $\leftidx{^{(Z)}}\lambda$ and $\mu\mathscr D^\al\leftidx{^{(Z)}}C_{\gamma,\al}^{k}$ more conveniently, one can rewrite
\begin{equation*}
g_{\kappa\nu}\leftidx{^{(Z)}}{\pi}^{\kappa\nu}=-2\mu^{-1}\leftidx{^{(Z)}}\pi_{\mathring LT}+\textrm{tr}\leftidx{^{(Z)}}{\slashed{\pi}}.
\end{equation*}
By \eqref{Lpi}-\eqref{Rpi}, one has
\begin{equation}\label{lamda}
\begin{split}
&\leftidx{^{(T)}}\lambda=\f12\textrm{tr}\leftidx{^{(T)}}{\slashed\pi},\qquad\leftidx{^{(\varrho\mathring L)}}\lambda=\varrho\textrm{tr}_{\slashed g}\check\chi+4,\qquad\leftidx{^{(R_{ij})}}\lambda=\f12\text{tr}\leftidx{^{(R_{ij})}}{\slashed\pi}.
\end{split}
\end{equation}
In addition, in the null frame $\{\mathring{\underline L},\mathring L,X_1,X_2,X_3\}$, the
term $\mu\mathscr D^\al{\leftidx{^{(Z)}}C_{\gamma}^{k}}_{,\al}$ can be written as
\begin{equation}\label{muC}
\begin{split}
\mu\mathscr D^\al{\leftidx{^{(Z)}}C_{\gamma}^{k}}_{,\al}=&-\f12\mathring L\big({\leftidx{^{(Z)}}C_\gamma^{k}}_{,\mathring{\underline L}}\big)-\f12\mathring{\underline L}\big({\leftidx{^{(Z)}}C_\gamma^{k}}_{,\mathring L}\big)
+\slashed\nabla^A\big(\mu {\leftidx{^{(Z)}}{\slashed C}_{\g}^{k}}_{,A}\big)\\
&-\f12\big(\mathring L\mu+\mu \textrm{tr}{\chi}+\textrm{tr}\leftidx{^{(T)}}{\slashed\pi}\big){\leftidx{^{(Z)}}C_\g^{k}}_{,\mathring L}-\f12\textrm{tr}\chi{\leftidx{^{(Z)}}C_\g^{k}}_{,\mathring{\underline L}},
\end{split}
\end{equation}
where
\begin{equation}\label{C}
\begin{split}
&{\leftidx{^{(Z)}}C_\g^{k}}_{,\mathring L}=\leftidx{^{(Z)}}{\slashed\pi}_{\mathring LA}(\slashed d^A\Psi_\g^{k})-\f12(\textrm{tr}\leftidx{^{(Z)}}{\slashed\pi})\mathring L\Psi_\g^{k},\\
&{\leftidx{^{(Z)}}C_\gamma^{k}}_{,\mathring{\underline L}}=-2(\leftidx{^{(Z)}}\pi_{LT}+\mu^{-1}\leftidx{^{(Z)}}\pi_{TT})(\mathring L\Psi_\gamma^{k})+\leftidx{^{(Z)}}{\slashed\pi}_{\mathring{\underline L}A}(\slashed d^A\Psi_\gamma^{k})-\f12(\textrm{tr}\leftidx{^{(Z)}}{\slashed\pi})\mathring{\underline L}\Psi_\gamma^{k},\\
&\mu{\leftidx{^{(Z)}}{\slashed C}_\gamma^{k}}_{,A}=-\f12\leftidx{^{(Z)}}{\slashed\pi}_{\mathring{\underline L}A}(\mathring L\Psi_\gamma^{k})-\f12\leftidx{^{(Z)}}{\slashed\pi}_{\mathring LA}(\mathring{\underline L}\Psi_\gamma^{k})
+\leftidx{^{(Z)}}{{\pi}}_{\mathring LT}(\slashed d_A\Psi_\gamma^{k})\\
&\qquad\qquad\quad+\mu(\leftidx{^{(Z)}}{\slashed\pi}_{AB}
-\f12\textrm{tr}\leftidx{^{(Z)}}{\slashed\pi}\slashed g_{AB})(\slashed d^B\Psi_\gamma^{k}).
\end{split}
\end{equation}
A direct substitution of \eqref{C} into \eqref{muC} would lead to a lengthy and complicated equation for $\mu\mathscr D^\al{\leftidx{^{(Z)}}C_{\gamma}^{k}}_{,\al}$
which seems to be difficult to analyze. To overcome this difficulty,
one can decompose $\mu\mathscr D^\al{\leftidx{^{(Z)}}C_{\gamma}^{k}}_{,\al}$ into the following three parts as in \cite{MY}:
\begin{equation}\label{muZC}
\mu\mathscr D^\al{\leftidx{^{(Z)}}C_{\gamma}^{k}}_{,\al}=\leftidx{^{(Z)}}D_{\gamma,1}^k+\leftidx{^{(Z)}}D_{\gamma,2}^k+\leftidx{^{(Z)}}D_{\gamma,3}^k,
\end{equation}
where
\begin{equation}\label{D1}
\begin{split}
\leftidx{^{(Z)}}D_{\gamma,1}^k=&\mu( \leftidx{^{(Z)}}{{\slashed\pi}}^{AB}-\f12\textrm{tr}\leftidx{^{(Z)}}{{\slashed\pi}}\slashed g^{AB})\slashed\nabla^2_{AB}\Psi_\gamma^k-\leftidx{^{(Z)}}{\slashed\pi}_{\mathring{\underline L}A}(\slashed d^A\mathring L\Psi_\gamma^{k})-\leftidx{^{(Z)}}{\slashed\pi}_{\mathring LA}(\slashed d^A \mathring{\underline L}\Psi_\gamma^{k})\\
&+(\leftidx{^{(Z)}}\pi_{\mathring LT}+\leftidx{^{(Z)}}\pi_{\tilde TT})(\mathring L^2\Psi_\gamma^k)+\leftidx{^{(Z)}}\pi_{\mathring LT}\slashed\triangle\Psi_\gamma^k+\f12\textrm{tr}\leftidx{^{(Z)}}{\slashed{\pi}}\big(\mathring L\mathring{\underline L}\Psi_\gamma^k+\f12\textrm{tr}{\chi}\mathring{\underline L}\Psi_\gamma^k\big),
\end{split}
\end{equation}
\begin{equation}\label{D2}
\begin{split}
\leftidx{^{(Z)}}D_{\gamma,2}^k=&\mathring L(\leftidx{^{(Z)}}\pi_{\mathring LT}+\leftidx{^{(Z)}}\pi_{\tilde TT})\mathring L\Psi_\gamma^k-\big(\f12\slashed\nabla^A\leftidx{^{(Z)}}{\slashed\pi}_{\mathring{\underline L}A}-\f14\mathring{\underline L}(\textrm{tr}\leftidx{^{(Z)}}{\slashed\pi})\big)\mathring L\Psi_\g^k\\
&-(\f12\slashed{\mathcal L}_{\mathring{\underline L}}\leftidx{^{(Z)}}{\slashed\pi}_{\mathring LA}-\slashed d_A\leftidx{^{(Z)}}{\pi}_{\mathring LT})\slashed d^A\Psi_\gamma^k+\slashed\nabla^B\big(\mu\leftidx{^{(Z)}}{{\slashed\pi}}_{AB}-\f{\mu}{2}\textrm{tr}\leftidx{^{(Z)}}{\slashed\pi}\slashed g_{AB}\big)\slashed d^A\Psi_\gamma^k\\
&-\f12(\slashed\nabla^A\leftidx{^{(Z)}}{\slashed\pi}_{\mathring LA})\mathring{\underline L}\Psi_\gamma^k-\f12(\slashed{\mathcal L}_{\mathring L}\leftidx{^{(Z)}}{\slashed\pi}_{\mathring{\underline L}A})\slashed d^A\Psi_\gamma^k+\f14\mathring L(\textrm{tr}\leftidx{^{(Z)}}{\slashed{\pi}})\mathring{\underline L}\Psi_\gamma^k,
\end{split}
\end{equation}
\begin{equation}\label{D3}
\begin{split}
\leftidx{^{(Z)}}D_{\gamma,3}^k=&\Big\{\textrm{tr}{\chi}(\leftidx{^{(Z)}}\pi_{\mathring LT}+\leftidx{^{(Z)}}\pi_{\tilde TT})+\f14(\mu\textrm{tr}{\chi}+\textrm{tr}\leftidx{^{(T)}}{\slashed{\pi}})\textrm{tr}\leftidx{^{(Z)}}{\slashed{\pi}}+\f12\slashed d^A\mu\leftidx{^{(Z)}}{\slashed\pi}_{\mathring LA}\Big\}\mathring L\Psi_\gamma^k\\
&+\f12\Big\{(\textrm{tr}\leftidx{^{(Z)}}{\slashed{\pi}})({\eta}_A+\mu\zeta_A)-(\mathring L\mu)\leftidx{^{(Z)}}{\slashed\pi}_{\mathring LA}-\textrm{tr}{\chi}\leftidx{^{(Z)}}{\slashed\pi}_{\mathring{\underline L}A}+2\leftidx{^{(Z)}}{\slashed\pi}_{\mathring{\underline L}}^B\chi_{AB}\\
&\qquad\quad+2\leftidx{^{(Z)}}{\slashed\pi}_{\mathring{ L}}^B(\mu\chi_{AB}+\leftidx{^{(T)}}{\slashed{\pi}}_{AB})-(\mu\textrm{tr}\chi
+\textrm{tr}\leftidx{^{(T)}}{\slashed{\pi}})\leftidx{^{(Z)}}{\slashed\pi}_{\mathring LA}\Big\}\slashed d^A\Psi_\gamma^k.
\end{split}
\end{equation}
Note that all the terms in $\leftidx{^{(Z)}}D_{\gamma,1}^k$ are the products of the deformation tensor and the
second order derivatives of $\Psi_\g^k$ except the last term. The reason  putting $\mathring L\mathring{\underline L}\Psi_\gamma^k+\f12\textrm{tr}{\chi}\mathring{\underline L}\Psi_\gamma^k$ together in the last term of \eqref{D1}
is that since $\Psi_\g^k$ is the
derivative of $\varphi_\g$, then by equation \eqref{fequation}, $\mathring L\mathring{\underline L}\varphi_\gamma+\f12\textrm{tr}{\chi}\mathring{\underline L}\varphi_\gamma=H_\g+\mu\slashed\triangle\varphi_\gamma+\f12\textrm{tr}\check\chi\mathring{\underline L}\varphi_\g$
admits the better smallness and faster time-decay rate
 than those of $\mathring L\mathring{\underline L}\varphi_\gamma$ and $\f12\textrm{tr}{\chi}\mathring{\underline L}\varphi_\gamma$ separately.
 In addition, $\leftidx{^{(Z)}}D_{\gamma,2}^k$ consists of all the products of the first order derivatives of the deformation
 tensor and the first order derivatives of $\Psi_\g^k$, and $\leftidx{^{(Z)}}D_{\gamma,3}^k$ consists of the other remaining terms.

 With the explicit expressions \eqref{Phik}-\eqref{lamda} and \eqref{muZC}-\eqref{D3} at hand, one can now estimate the last two integrals of \eqref{e}.
To this end, one can define the corresponding weighted energy and flux as in \cite{MY}:
\begin{align}
E_{i,m+1}(s , u)&=\sum_{\gamma=0}^4\sum_{Z\in\{\mathring L,T,R_{ab}\}}\delta^{2l}E_i[Z^{m}\varphi_\gamma](s, u),\quad i=1,2,\label{ei}\\
F_{i,m+1}(s, u)&=\sum_{\gamma=0}^4\sum_{Z\in\{\mathring L,T,R_{ab}\}}\delta^{2l}F_i[Z^{m}\varphi_\gamma](s, u),\quad i=1,2,\label{fi}\\
E_{i,\leq m+1}(s, u)&=\sum_{0\leq n\leq m}E_{i,n+1}(s, u),\quad i=1,2,\label{eil}\\
F_{i,\leq m+1}(s, u)&=\sum_{0\leq n\leq m}F_{i,n+1}(s, u),\quad i=1,2,\label{fil}
\end{align}
where $l$ is the number of $T$ in $Z^m$. These weighted energies will be estimated in the subsequent sections.

\section{Higher order $L^2$ estimates}\label{hoe}
In this section, we shall establish  the higher order $L^2$ estimates for some related quantities
so that the last two terms on the right hand side of \eqref{e} can be absorbed by the left hand side, and hence the higher order energy estimates on \eqref{ge} can
be derived.

Note that the energies and fluxes in Section \ref{EE} are all defined for $\varphi_\gamma$ (see \eqref{ei}-\eqref{fil}).
In order to use these  energies to control $L^2$ norms of $\phi$ and its derivatives, one requires the following results,
whose analogous statements
can be found in Lemma 7.3 of \cite{MY} for 3-D case.

\begin{lemma}\label{L2T}
 For any function $\psi\in C^1(D^{s, u})$ which vanishes on $C_0$, it holds that  for small $\delta>0$,
 \begin{align}
  \int_{S_{s, u}} \psi^2&\lesssim\delta\int_{\Sigma_s^u}|T\psi|^2\lesssim\delta\int_{\Sigma_{s}^{u}}\big(|\mathring{\underline L}\psi|^2+\mu^2|\mathring L\psi|^2\big)\label{SSi},\\
  \int_{\Sigma_{s}^{u}} \psi^2&\lesssim\delta^2\int_{\Sigma_s^u}|T\psi|^2\lesssim\delta^2\int_{\Sigma_{s}^{u}}\big(|\mathring{\underline L}\psi|^2+\mu^2|\mathring L\psi|^2\big).\label{SiSi}
 \end{align}
Therefore,
 \begin{align}
 \int_{S_{s, u}} \psi^2&\lesssim\delta\big(E_2[\psi](s, u)+\varrho^{-2\epsilon}E_1[\psi](s, u)\big),\label{SE}\\
 \int_{\Sigma_{s}^{u}} \psi^2&\lesssim\delta^2\big(E_2[\psi](s, u)+\varrho^{-2\epsilon}E_1[\psi](s, u)\big).\label{SiE}
 \end{align}
\end{lemma}

\begin{proof}
 For any function $f\in C^1(D^{s, u})$ vanishing on $C_0$, by \eqref{Lpi} and \eqref{theta}, one has
 \begin{equation*}
  \begin{split}
   \f{\p}{\p u}\int_{S_{s, u}}fd\nu_{\slashed g}&
   =\int_{S_{s, u}}\big(Tf+\f12\textrm{tr}\leftidx{^{(T)}}{\slashed\pi}\cdot f\big)d\nu_{\slashed g}.
  \end{split}
 \end{equation*}
 This leads to
 \begin{equation}\label{Y-38}
 |\f{\p}{\p u}\int_{S_{s, u}}fd\nu_{\slashed g}|\lesssim\int_{S_{s, u}}(|Tf|+\delta^{-\varepsilon_0}|f|)
 \end{equation}
 by \eqref{piL}.
 Setting $f=\psi^2$ in \eqref{Y-38} and then integrating  from $0$ to $u$ yield
 \begin{equation*}
  \begin{split}
   \int_{S_{s, u}} \psi^2d\nu_{\slashed g}\lesssim\int_{\Sigma_{s}^{ u}}(|T\psi|\cdot|\psi|+\delta^{-\varepsilon_0}\psi^2)\lesssim\int_0^{u}\int_{S_{s, u'}}\delta^{-1}\psi^2d\nu_{\slashed g}d u'+\int_0^{u}\int_{S_{s, u'}}\delta|T\psi|^2d\nu_{\slashed g}d u'.
  \end{split}
 \end{equation*}
Thus it follows from Gronwall's inequality that
\begin{align}\label{Y3}
 \int_{S_{s, u}} \psi^2d\nu_{\slashed g}\lesssim\int_0^{u}\int_{S_{s, u'}}\delta|T\psi|^2d\nu_{\slashed g}du'.
\end{align}
Taking $T=\f12\mathring{\underline L}-\f12\mu\mathring L$ in \eqref{Y3} leads to \eqref{SSi}, and hence \eqref{SiSi} holds.
\end{proof}

For any integer $k$, with the help of \eqref{SiSi}, one has
\begin{equation}\label{Zkp}
\begin{split}
&\delta^l\|Z^{k+1}\phi\|_{L^2(\Sigma_s^u)}\lesssim\delta^{1+l}\|[T,Z^{k+1}]\phi\|_{L^2(\Sigma_s^u)}+\delta^{1+l}\|Z^{k+1}(\mu\tilde T^i\varphi_i)\|_{L^2(\Sigma_s^u)}\\
\lesssim&\delta^{1-\varepsilon_0}s^{-1}\delta^{l_0}\|Z^{k_0}\phi\|_{L^2(\Sigma_s^u)}
+\delta^{3-\varepsilon_0}s^{-3/2}\delta^{l_1}\|\slashed{\mathcal L}_Z^{k_1}\leftidx{^{(T)}}{\slashed\pi}_{\mathring L}\|_{L^2(\Sigma_s^u)}+\delta^{3-\varepsilon_0}s^{-5/2}\delta^{l_1}\|\slashed{\mathcal L}_Z^{k_1}\leftidx{^{(R)}}{\slashed\pi}_T\|_{L^2(\Sigma_s^u)}\\
&+\delta^{3-2\varepsilon_0}s^{-5/2}\delta^{l_1}\|\slashed{\mathcal L}_Z^{k_1}\slashed g\|_{L^2(\Sigma_s^u)}+\delta^{2-\varepsilon_0}s^{-3/2}\delta^{l_0}\|Z^{k_0}\mu\|_{L^2(\Sigma_s^u)}
+\delta^{2-\varepsilon_0}s^{-3/2}\delta^{l_0}\|Z^{k_0}\check L^i\|_{L^2(\Sigma_s^u)}\\
&+\delta^{2-\varepsilon_0}s^{-5/2}\delta^{l_0}\|Z^{k_0}x^i\|_{L^2(\Sigma_s^u)}
+\delta\delta^{l_0}\|Z^{k_0}\varphi_\gamma\|_{L^2(\Sigma_s^u)},
\end{split}
\end{equation}
where $l$ and $l_m$ ($m=0,1$) are the numbers of $T$ in $Z^{k+1}$ and $Z^{k_m}$ respectively, and $1\leq k_m\leq k+1-m$.
Utilizing \eqref{SiE} to estimate the last term of \eqref{Zkp}, one then gets from \eqref{Zkp} that
\begin{equation}\label{zkp}
\begin{split}
&\delta^l\|Z^{k+1}\phi\|_{L^2(\Sigma_s^u)}\\
\lesssim&\delta^{3-\varepsilon_0}s^{-3/2}\delta^{l_1}\|\slashed{\mathcal L}_Z^{k_1}\leftidx{^{(T)}}{\slashed\pi}_{\mathring L}\|_{L^2(\Sigma_s^u)}+\delta^{3-\varepsilon_0}s^{-5/2}\delta^{l_1}\|\slashed{\mathcal L}_Z^{k_1}\leftidx{^{(R)}}{\slashed\pi}_T\|_{L^2(\Sigma_s^u)}\\
&+\delta^{3-2\varepsilon_0}s^{-5/2}\delta^{l_1}\|\slashed{\mathcal L}_Z^{k_1}\slashed g\|_{L^2(\Sigma_s^u)}+\delta^{2-\varepsilon_0}s^{-3/2}\delta^{l_0}\|Z^{k_0}\mu\|_{L^2(\Sigma_s^u)}\\
&+\delta^{2-\varepsilon_0}s^{-3/2}\delta^{l_0}\|Z^{k_0}\check L^i\|_{L^2(\Sigma_s^u)}+\delta^{2-\varepsilon_0}s^{-5/2}\delta^{l_0}\|Z^{k_0}x^i\|_{L^2(\Sigma_s^u)}\\
&+\delta^2s^{-\epsilon}\sqrt{E_{1,\leq k+2}(s,u)}+\delta^2\sqrt{E_{2,\leq k+2}(s,u)}.
\end{split}
\end{equation}
Based on  \eqref{zkp}, due to $T\phi=-\mu(g^{0i}+\check L^i+\f{x^i}{\varrho})\varphi_i$ and $\mathring L\phi=\varphi_0+\check L^i\varphi_i+\f{x^i}\varrho\varphi_i$, it holds that
\begin{equation}\label{ZTp}
\begin{split}
&\delta^l\|Z^{k+1}T\phi\|_{L^2(\Sigma_s^u)}\\
\lesssim&\delta^{1-\varepsilon_0}s^{-3/2}\delta^{l_0}\big(\|Z^{k_0}\mu\|_{L^2(\Sigma_s^u)}+\|Z^{k_0}\check L^i\|_{L^2(\Sigma_s^u)}+\|Z^{k_0}\phi\|_{L^2(\Sigma_s^u)}\big)+\delta^{l_0}\|Z^{k_0}\varphi\|_{L^2(\Sigma_s^u)}\\
&+\delta^{1-\varepsilon_0}s^{-5/2}\delta^{l_0}\|Z^{k_0}x^i\|_{L^2(\Sigma_s^u)}\\
\lesssim&\delta^{1-\varepsilon_0}s^{-3/2}\delta^{l_0}\big(\|Z^{k_0}\mu\|_{L^2(\Sigma_s^u)}+\|Z^{k_0}\check L^i\|_{L^2(\Sigma_s^u)}\big)+\delta^{4-3\varepsilon_0}s^{-4}\delta^{l_1}\|\slashed{\mathcal L}_Z^{k_1}\slashed g\|_{L^2(\Sigma_s^u)}\\
&+\delta^{4-2\varepsilon_0}s^{-4}\delta^{l_1}\big(\|\slashed{\mathcal L}_Z^{k_1}\leftidx{^{(R)}}{\slashed\pi}_T\|_{L^2(\Sigma_s^u)}+s\|\slashed{\mathcal L}_Z^{k_1}\leftidx{^{(T)}}{\slashed\pi}_{\mathring L}\|_{L^2(\Sigma_s^u)}\big)+\delta^{1-\varepsilon_0}s^{-5/2}\delta^{l_0}\|Z^{k_0}x^i\|_{L^2(\Sigma_s^u)}\\
&+\delta s^{-\epsilon}\sqrt{E_{1,\leq k+2}(s,u)}+\delta\sqrt{E_{2,\leq k+2}(s,u)}
\end{split}
\end{equation}
and
\begin{equation}\label{ZLp}
\begin{split}
\delta^l\|Z^{k+1}\mathring L\phi\|_{L^2(\Sigma_s^u)}\lesssim&\delta^{1-\varepsilon_0}s^{-3/2}\delta^{l_0}\big(\|Z^{k_0}\check L^i\|_{L^2(\Sigma_s^u)}+s^{-1}\|Z^{k_0}x^i\|_{L^2(\Sigma_s^u)}\big)\\
&+\delta s^{-\epsilon}\sqrt{E_{1,\leq k+2}(s,u)}+\delta\sqrt{E_{2,\leq k+2}(s,u)}.
\end{split}
\end{equation}
Similarly as for \eqref{zkp} and \eqref{ZTp}, one derives
	\begin{equation}\label{Rkp}
	\begin{split}
	&\|R^{k+1}\phi\|_{L^2(\Sigma_s^u)}\\
	\lesssim&\delta^{4-3\varepsilon_0}s^{-7/2}\|\slashed{\mathcal L}_R^{\leq k}\slashed g\|_{L^2(\Sigma_s^u)}+\delta^{3-\varepsilon_0}s^{-5/2}\|\slashed{\mathcal L}_R^{\leq k}\leftidx{^{(R)}}{\slashed\pi}_T\|_{L^2(\Sigma_s^u)}\\
	&+\delta^{2-\varepsilon_0}s^{-3/2}\|R^{ [1,k+1]}\mu\|_{L^2(\Sigma_s^u)}+\delta^{2-\varepsilon_0}s^{-3/2}\|R^{\leq k+1}\check L^i\|_{L^2(\Sigma_s^u)}\\
	&+\delta^{2-\varepsilon_0}s^{-5/2}\|R^{\leq k+1}x^i\|_{L^2(\Sigma_s^u)}+\delta^2s^{-\epsilon}\sqrt{E_{1,\leq k+2}(s,u)}
+\delta^2\sqrt{E_{2,\leq k+2}(s,u)}
	\end{split}
	\end{equation}
and
\begin{equation}\label{RTp}
\begin{split}
\|R^{k+1}T\phi\|_{L^2(\Sigma_s^u)}\lesssim&\delta^{1-\varepsilon_0}s^{-3/2}\big(\|R^{[1,k+1]}\mu\|_{L^2(\Sigma_s^u)}+\|R^{\leq k+1}\check L^i\|_{L^2(\Sigma_s^u)}+\|\slashed{\mathcal L}_R^{\leq k}\slashed dx^i\|_{L^2(\Sigma_s^u)}\big)\\
&+\delta^{4-2\varepsilon_0}s^{-4}\|\slashed{\mathcal L}_R^{\leq k}\leftidx{^{(R)}}{\slashed\pi}_T\|_{L^2(\Sigma_s^u)}+\delta^{5-4\varepsilon_0}s^{-5}\|\slashed{\mathcal L}_R^{\leq k}\slashed g\|_{L^2(\Sigma_s^u)}\\
&+\delta s^{-\epsilon}\sqrt{E_{1,\leq k+2}(s,u)}+\delta\sqrt{E_{2,\leq k+2}(s,u)},
\end{split}
\end{equation}
where $\|R^{[1,k+1]}\mu\|_{L^2(\Sigma_s^u)}=\ds\sum_{i=1}^{k+1}\|R^i\mu\|_{L^2(\Sigma_s^u)}$.
In addition,
\begin{equation*}
\begin{split}
\|R^k\mathring L\phi\|_{L^2(\Sigma_s^u)}\leq&\delta\|[T,R^k]\mathring L\phi\|_{L^2(\Sigma_s^u)}+\delta\|R^k[T,\mathring L]\phi\|_{L^2(\Sigma_s^u)}+\delta\|R^k\mathring L(\mu\tilde T^i\varphi_i)\|_{L^2(\Sigma_s^u)}\\
\lesssim&\delta^{3-\varepsilon_0}s^{-7/2}\|\slashed{\mathcal L}_R^{\leq k-1}\leftidx{^{(R)}}{\slashed\pi}_T\|_{L^2(\Sigma_s^u)}+\delta s^{-1}\|R^{\leq k}\mathring L\phi\|_{L^2(\Sigma_s^u)}\\
&+\delta^{3-\varepsilon_0}s^{-5/2}\|\slashed{\mathcal L}_R^{\leq k}\leftidx{^{(T)}}{\slashed\pi}_{\mathring L}\|_{L^2(\Sigma_s^u)}+\delta^{1-\varepsilon_0}s^{-2}\|R^{\leq k+1}\phi\|_{L^2(\Sigma_s^u)}\\
&+\delta^{3-2\varepsilon_0}s^{-7/2}\|\slashed{\mathcal L}_R^{\leq k}\slashed g\|_{L^2(\Sigma_s^u)}+\delta^{2-\varepsilon_0}s^{-3/2}\|R^{\leq k}(\mathring L\mu)\|_{L^2(\Sigma_s^u)}\\
&+\delta^{2-\varepsilon_0}s^{-5/2}\|R^{\leq k}\check L^i\|_{L^2(\Sigma_s^u)}+\delta s^{-1}\|R^{\leq k}\varphi\|_{L^2(\Sigma_s^u)}\\
&+\delta^{2-\varepsilon_0}s^{-5/2}\|R^{\leq k}\mu\|_{L^2(\Sigma_s^u)}+\delta\|R^{\leq k}\mathring L\varphi\|_{L^2(\Sigma_s^u)}\\
&+\delta^{2-\varepsilon_0}s^{-3/2}\|R^{\leq k}(\mathring L\mathring L^i)\|_{L^2(\Sigma_s^u)}+\delta^{2-\varepsilon_0}s^{-7/2}\|R^{\leq k}x^i\|_{L^2(\Sigma_s^u)}.
\end{split}
\end{equation*}
Note that $\mathring L\mu$ and $\mathring L\mathring L^i$ satisfy \eqref{lmu} and \eqref{LL} respectively. Then combining \eqref{Rkp} and \eqref{RTp} yields
\begin{equation}\label{RLp}
\begin{split}
\|R^k\mathring L\phi\|_{L^2(\Sigma_s^u)}\lesssim&\delta^{3-\varepsilon_0}s^{-7/2}\|\slashed{\mathcal L}_R^{\leq k}\leftidx{^{(R)}}{\slashed\pi}_T\|_{L^2(\Sigma_s^u)}+\delta^{3-\varepsilon_0}s^{-5/2}\|\slashed{\mathcal L}_R^{\leq k}\leftidx{^{(T)}}{\slashed\pi}_{\mathring L}\|_{L^2(\Sigma_s^u)}\\
&+\delta^{3-2\varepsilon_0}s^{-7/2}\|\slashed{\mathcal L}_R^{\leq k}\slashed g\|_{L^2(\Sigma_s^u)}+\delta^{2-2\varepsilon_0}s^{-5/2}\|R^{\leq k+1}\check L^i\|_{L^2(\Sigma_s^u)}\\
&+\delta^{2-\varepsilon_0}s^{-5/2}\|R^{\leq k+1}\mu\|_{L^2(\Sigma_s^u)}+\delta^{2-\varepsilon_0}s^{-7/2}\|R^{\leq k+1}x^i\|_{L^2(\Sigma_s^u)}\\
&+\delta^2 s^{-1-\epsilon}\sqrt{E_{1,\leq k+2}(s,u)}+\delta^2t^{-1}\sqrt{E_{2,\leq k+2}(s,u)}.
\end{split}
\end{equation}

Note that to estimate $\|\slashed{\mathcal L}_{Z}^k\check\chi\|_{L^2(\Sigma_s^u)}$
in Proposition \ref{L2chi} below, one needs to deal with $\|\slashed{\mathcal L}_{R}^k\check\chi\|_{L^2(\Sigma_s^u)}$ first, for which the following elementary estimate will be used.
\begin{lemma}\label{L2L}
	For any function $f\in C(D^{s, u})$, set
	\[
	F(s,u,\vartheta):=\int_{t_0}^sf(\tau,u,\vartheta)d\tau.
	\]
	Under the assumptions $(\star)$ with suitably small $\delta>0$, it holds that
	\begin{equation}\label{Ff}
	\|F\|_{L^2(\Sigma_s^u)}\lesssim\varrho(s,u)^{3/2}\int_{t_0}^s\f{1}{\varrho(\tau,u)^{3/2}}\|f\|_{L^2(\Sigma_{\tau}^u)}d\tau.
	\end{equation}
\end{lemma}

\begin{proof}
	The proof is exactly similar to that of Lemma 12.57 in \cite{J} and thus omitted here.
\end{proof}

Based on the  preparations above, we are ready to derive the $L^2$ estimates for the quantities including
${\check\chi}$, $\leftidx{^{(R)}}{\slashed\pi}_{\mathring L}$, $\leftidx{^{(R)}}{\slashed\pi}$, $\check{L^i}$, $\upsilon_{ij}$, $x^j$ and so on.

\begin{proposition}\label{L2chi}
 Under the assumptions $(\star)$ with suitably small $\delta>0$, it holds that for ${k\leq 2N-6}$,
 \begin{align*}
 &\delta^l\|\slashed{\mathcal L}_{Z}^k\check\chi\|_{L^2(\Sigma_s^u)}\lesssim\delta^{3/2-\varepsilon_0}s^{-1/2}
 +s^{-\epsilon}\sqrt{\tilde E_{1,\leq k+2}(s,u)}+\delta \sqrt{\tilde E_{2,\leq k+2}(s,u)},\\
 &\delta^l\|Z^{k+1}\mu\|_{L^2(\Sigma_s^u)}
 \lesssim\delta^{3/2-2\varepsilon_0}s^{3/2}+\delta^{-\varepsilon_0} s^{3/2}\sqrt{\tilde E_{1,\leq k+2}(s,u)}+\delta^{1-\varepsilon_0}s^{3/2}\sqrt{\tilde E_{2,\leq k+2}(s,u)},\\
 &\delta^l\|Z^{k+1}\check L^j\|_{L^2(\Sigma_s^u)}\lesssim\delta^{3/2-\varepsilon_0}s^{1/2}
 +s^{1-\epsilon}\sqrt{\tilde E_{1,\leq k+2}(s,u)}+\delta s\sqrt{\tilde E_{2,\leq k+2}(s,u)},\\
 &\delta^l\|Z^{k+1}\check\varrho\|_{L^2(\Sigma_s^u)}\lesssim\delta^{3/2-\varepsilon_0}s^{1/2}
 +s^{1-\epsilon}\sqrt{\tilde E_{1,\leq k+2}(s,u)}+\delta s\sqrt{\tilde E_{2,\leq k+2}(s,u)},\\
 &\delta^l\|\slashed{\mathcal L}_Z^{k+1}\slashed g\|_{L^2(\Sigma_s^u)}\lesssim\delta^{1/2}s^{3/2}
 +s^{1-\epsilon}\sqrt{\tilde E_{1,\leq k+2}(s,u)}+\delta s\sqrt{\tilde E_{2,\leq k+2}(s,u)},\\
  &\delta^l\|{Z}^{k+2}x^i\|_{L^2(\Sigma_s^u)}\lesssim\delta^{1/2}s^{5/2}
  +s^{2-\epsilon}\sqrt{\tilde E_{1,\leq k+2}(s,u)}+\delta s^2\sqrt{\tilde E_{2,\leq k+2}(s,u)},\\
  &\delta^l\|\slashed{\mathcal L}_Z^{k+1}R\|_{L^2(\Sigma_s^u)}\lesssim\delta^{1/2}s^{5/2}
  +s^{2-\epsilon}\sqrt{\tilde E_{1,\leq k+2}(s,u)}+\delta s^2\sqrt{\tilde E_{2,\leq k+2}(s,u)},\\
  &\delta^l\|Z^{k+1}\upsilon_{ij}\|_{L^2(\Sigma_s^u)}\lesssim\delta^{3/2-\varepsilon_0}s^{3/2}
  +s^{2-\epsilon}\sqrt{\tilde E_{1,\leq k+2}(s,u)}+\delta s^2\sqrt{\tilde E_{2,\leq k+2}(s,u)},\\
 &\delta^l\|\slashed{\mathcal L}_{Z}^k\leftidx{^{(R)}}{\slashed\pi}\|_{L^2(\Sigma_s^u)}\lesssim\delta^{3/2-\varepsilon_0}s^{1/2}
 +s^{1-\epsilon}\sqrt{\tilde E_{1,\leq k+2}(s,u)}+\delta s\sqrt{\tilde E_{2,\leq k+2}(s,u)},\\
 &\delta^l\|\slashed{\mathcal L}_{Z}^k\leftidx{^{(R)}}{\slashed\pi}_{\mathring L}\|_{L^2(\Sigma_s^u)}\lesssim\delta^{3/2-\varepsilon_0}s^{1/2}+s^{1-\epsilon}\sqrt{\tilde E_{1,\leq k+2}(s,u)}
 +\delta s\sqrt{\tilde E_{2,\leq k+2}(s,u)},\\
  &\delta^l\|\slashed{\mathcal L}_{Z}^k\leftidx{^{(T)}}{\slashed\pi}\|_{L^2(\Sigma_s^u)}\lesssim\delta^{1/2-\varepsilon_0}s^{1/2}
  +\delta^{-\varepsilon_0}s^{1/2}\sqrt{\tilde E_{1,\leq k+2}(s,u)}+s^{1/2}\sqrt{\tilde E_{2,\leq k+2}(s,u)},\\
  &\delta^l\|\slashed{\mathcal L}_{Z}^k\leftidx{^{(T)}}{\slashed\pi}_{\mathring L}\|_{L^2(\Sigma_s^u)}\lesssim\delta^{1/2-\varepsilon_0}s^{1/2}
  +\delta^{-\varepsilon_0}s^{1/2}\sqrt{\tilde E_{1,\leq k+2}(s,u)}+s^{1/2}\sqrt{\tilde E_{2,\leq k+2}(s,u)},\\
   &\delta^l\|\slashed{\mathcal L}_{Z}^k\leftidx{^{(R)}}{\slashed\pi}_T\|_{L^2(\Sigma_s^u)}\lesssim\delta^{3/2-2\varepsilon_0}s^{1/2}
   +\delta^{-\varepsilon_0}s^{1-\epsilon}\sqrt{\tilde E_{1,\leq k+2}(s,u)}+\delta^{1-\varepsilon_0} s\sqrt{\tilde E_{2,\leq k+2}(s,u)},
 \end{align*}
 where $\tilde E_{i,\leq k+2}(s,u)=\sup_{t_0\leq\tau\leq s}E_{i,\leq k+2}(\tau,u)$ $(i=1,2)$.
\end{proposition}

\begin{proof}
	It follows from Lemma \ref{4.2} and \eqref{TL} that explicit expressions of $\varrho\mathring L\check L^i$, $T\check L^i$ and $R_{jl}\check L^i$
can be given since $\varrho\mathring L\check L^i=\varrho\mathring L\mathring L^i-\check L^i$, $T\check L^i=T\mathring L^i+\f{\mu}{\varrho}g^{0i}+\f{\mu}{\varrho}\check L^i+\f{\mu-1}{\varrho^2}x^i$ and $R_{jl}\check L^i={R_{jl}^A}\slashed d_A\check L^i$. Therefore,
	\begin{align*}
	\delta^l\|Z^k(\varrho\mathring L\check L^i)\|_{L^2(\Sigma_s^u)}\lesssim&\delta^{l_1}\|Z^{k_1}\check L^i\|_{L^2(\Sigma_s^u)}+\delta^{1-\varepsilon_0}s^{-5/2}\delta^{l_{0}}\|Z^{k_{{0}}}x^i\|_{L^2(\Sigma_s^u)}
+\delta^{l_{0}}\|Z^{k_{{0}}}\phi\|_{L^2(\Sigma_s^u)}\\
	&+\delta^{l_{0}}\|Z^{k_{{0}}}\varphi_\gamma\|_{L^2(\Sigma_s^u)}+\delta^{1-\varepsilon_0}s^{-3/2}\delta^{l_1}\|\slashed{\mathcal L}_Z^{k_1}\slashed g\|_{L^2(\Sigma_s^u)},\\
	\delta^{l+1}\|Z^k(T\check L^i)\|_{L^2(\Sigma_s^u)}\lesssim&\delta s^{-1}\delta^{l_{0}}\|Z^{k_{0}}\mu\|_{L^2(\Sigma_s^u)}+\delta^{1-\varepsilon_0}s^{-1}\delta^{l_1}\|Z^{k_1}\check L^i\|_{L^2(\Sigma_s^u)}+\delta^{l_{0}}\|Z^{k_{{0}}}\phi\|_{L^2(\Sigma_s^u)}\\
	&+\delta^{l_{0}}\|Z^{k_{{0}}}\varphi_\gamma\|_{L^2(\Sigma_s^u)}
+\delta^{1-\varepsilon_0}s^{-2}\delta^{l_{0}}\|Z^{k_{{0}}}x^i\|_{L^2(\Sigma_s^u)}
+\delta^{1-\varepsilon_0}s^{-1}\delta^{l_1}\|\slashed{\mathcal L}_Z^{k_1}\slashed g\|_{L^2(\Sigma_s^u)},\\
	\delta^l\|Z^kR\check L^i\|_{L^2(\Sigma_s^u)}\lesssim&\delta^{1-\varepsilon_0}s^{-2}\delta^{l_1}\|\slashed{\mathcal L}_Z^{k_1}R\|_{L^2(\Sigma_s^u)}+s\delta^{l_1}\|\slashed{\mathcal L}_Z^{k_1}\check\chi\|_{L^2(\Sigma_s^u)}+\delta^{1-\varepsilon_0}s^{-2}\delta^{l_0}\|Z^{k_0}x^i\|_{L^2(\Sigma_s^u)}\\
	&+\delta^{1-\varepsilon_0}s^{-1}\delta^{l_1}\|\slashed{\mathcal L}_Z^{k_1}\slashed g\|_{L^2(\Sigma_s^u)}+\delta^{1-\varepsilon_0}s^{-3/2}\delta^{l_1}\|Z^{k_1}\check L^i\|_{L^2(\Sigma_s^u)}\\
	&+\delta^{l_0}\|Z^{k_0}\phi\|_{L^2(\Sigma_s^u)}+\delta^{l_0}\|Z^{k_0}\varphi_\gamma\|_{L^2(\Sigma_s^u)},		
	\end{align*}
where $l$ is the number of $T$ in $Z^k$, $l_m$ and $k_m$ are the same as in \eqref{Zkp}.
In addition, by utilizing \eqref{zkp} to estimate $\|Z^{k_0}\phi\|_{L^2(\Sigma_s^u)}$, we arrive at
\begin{equation}\label{ZkL}
\begin{split}
	\delta^l\|Z^{k+1}\check L^i\|_{L^2(\Sigma_s^u)}\lesssim&\delta^{l_1}\|Z^{k_1}\check L^i\|_{L^2(\Sigma_s^u)}+\delta^{1-\varepsilon_0}s^{-2}\delta^{l_1}\|\slashed{\mathcal L}_Z^{k_1}R\|_{L^2(\Sigma_s^u)}+s\delta^{l_1}\|\slashed{\mathcal L}_Z^{k_1}\check\chi\|_{L^2(\Sigma_s^u)}\\
	&+\delta^{1-\varepsilon_0}s^{-2}\delta^{l_0}\|Z^{k_0}x^i\|_{L^2(\Sigma_s^u)}+\delta^{1-\varepsilon_0}s^{-1}\delta^{l_1}\|\slashed{\mathcal L}_Z^{k_1}\slashed g\|_{L^2(\Sigma_s^u)}\\
&+\delta s^{-1}\delta^{l_{0}}\|Z^{k_{0}}\mu\|_{L^2(\Sigma_s^u)}+\delta^{3-\varepsilon_0}s^{-3/2}\delta^{l_1}\|\slashed{\mathcal L}_Z^{k_1}\leftidx{^{(T)}}{\slashed\pi}_{\mathring L}\|_{L^2(\Sigma_s^u)}\\
&+\delta^{3-\varepsilon_0}s^{-5/2}\delta^{l_1}\|\slashed{\mathcal L}_Z^{k_1}\leftidx{^{(R)}}{\slashed\pi}_T\|_{L^2(\Sigma_s^u)}+\delta s^{-\epsilon}\sqrt{E_{1,\leq k+2}(s,u)}\\
&+\delta\sqrt{E_{2,\leq k+2}(s,u)}.
\end{split}
\end{equation}
Note that $|\check L^i|\lesssim\delta^{1-\varepsilon_0}s^{-1}$ holds by Lemma \ref{Rh}.
This and $\|1\|_{L^2(\Sigma_s^u)}\lesssim\delta^{1/2}\varrho(s,u)^{3/2}$ yield
$\|\check L^i\|_{L^2(\Sigma_s^u)}\lesssim\delta^{3/2-\varepsilon_0}s^{1/2}$. Thus,
 	\begin{equation}\label{ZkcL}
 \begin{split}
 \delta^l\|Z^{k+1}\check L^i\|_{L^2(\Sigma_s^u)}\lesssim&\delta^{3/2-\varepsilon_0}s^{1/2}+\delta^{1-\varepsilon_0}s^{-2}\delta^{l_1}\|\slashed{\mathcal L}_Z^{k_1}R\|_{L^2(\Sigma_s^u)}+s\delta^{l_1}\|\slashed{\mathcal L}_Z^{k_1}\check\chi\|_{L^2(\Sigma_s^u)}\\
 &+\delta^{1-\varepsilon_0}s^{-2}\delta^{l_0}\|Z^{k_0}x^i\|_{L^2(\Sigma_s^u)}+\delta^{1-\varepsilon_0}s^{-1}\delta^{l_1}\|\slashed{\mathcal L}_Z^{k_1}\slashed g\|_{L^2(\Sigma_s^u)}\\
 &+\delta s^{-1}\delta^{l_{0}}\|Z^{k_{0}}\mu\|_{L^2(\Sigma_s^u)}+\delta^{3-\varepsilon_0}s^{-3/2}\delta^{l_1}\|\slashed{\mathcal L}_Z^{k_1}\leftidx{^{(T)}}{\slashed\pi}_{\mathring L}\|_{L^2(\Sigma_s^u)}\\
 &+\delta^{3-\varepsilon_0}s^{-5/2}\delta^{l_1}\|\slashed{\mathcal L}_Z^{k_1}\leftidx{^{(R)}}{\slashed\pi}_T\|_{L^2(\Sigma_s^u)}+\delta s^{-\epsilon}\sqrt{E_{1,\leq k+2}(s,u)}\\
 &+\delta\sqrt{E_{2,\leq k+2}(s,u)}.
 \end{split}
 \end{equation}
Analogously, due to $\varrho Lx^i=\varrho\check L^i+x^i$, $Tx^i=\mu(-g^{0i}-\check L^i-\f{x^i}\varrho)$ and $R_{ij}x^a=\O_{ij}^a+\upsilon_{ij}(g^{0a}+\check L^a+\f{x^a}\varrho)$, one can get
 \begin{equation}\label{Zkx}
 \begin{split}
 \delta^{l}\|Z^{k+2}x^i\|_{L^2(\Sigma_s^u)}\lesssim&\delta^{1/2}s^{5/2}+\delta^{l_0}\|Z^{k_0}\upsilon_{ab}\|_{L^2(\Sigma_s^u)}
 +s\delta^{l_0}\|Z^{k_0}\check L^j\|_{L^2(\Sigma_s^u)}+\delta\delta^{l_{0}}\|Z^{k_{0}}\mu\|_{L^2(\Sigma_s^u)}\\
 &+\delta^{4-3\varepsilon_0}s^{-5/2}\delta^{l_1}\|\slashed{\mathcal L}_Z^{k_1}\slashed g\|_{L^2(\Sigma_s^u)}+\delta^{4-2\varepsilon_0}s^{-3/2}\delta^{l_1}\|\slashed{\mathcal L}_Z^{k_1}\leftidx{^{(T)}}{\slashed\pi}_{\mathring L}\|_{L^2(\Sigma_s^u)}\\
 &+\delta^{4-2\varepsilon_0}s^{-5/2}\delta^{l_1}\|\slashed{\mathcal L}_Z^{k_1}\leftidx{^{(R)}}{\slashed\pi}_T\|_{L^2(\Sigma_s^u)}+\delta^{2-\varepsilon_0} s^{-\epsilon}\sqrt{E_{1,\leq k+2}(s,u)}\\
 &+\delta^{2-\varepsilon_0}\sqrt{E_{2,\leq k+2}(s,u)}.
 \end{split}
 \end{equation}
Similarly, it follows from $\slashed{\mathcal L}_{\varrho\mathring L}R=\varrho\leftidx{^{(R)}}{\slashed\pi}_{\mathring L}$, $\slashed{\mathcal L}_TR=\leftidx{^{(R)}}{\slashed\pi}_T$ and \eqref{RiRj} that
 \begin{equation}\label{ZR}
 \begin{split}
 \delta^l\|\slashed{\mathcal L}_Z^{k+1}R\|_{L^2(\Sigma_s^u)}\lesssim&\delta^{1/2}s^{5/2}+\delta^{1-\varepsilon_0}s^{-1}\delta^{l_0}\|Z^{k_0}x^i\|_{L^2(\Sigma_s^u)}
 +\delta^{1-\varepsilon_0}\delta^{l_0}\|Z^{k_0}\check L^j\|_{L^2(\Sigma_s^u)}\\
 &+\delta^{l_1}\|Z^{k_1}\upsilon_{ij}\|_{L^2(\Sigma_s^u)}+s\delta^{l_1}\|\slashed{\mathcal L}_Z^{k_1}\leftidx{^{(R)}}{\slashed\pi}_{\mathring L}\|_{L^2(\Sigma_s^u)}\\
 &+\delta^{4-2\varepsilon_0}s^{-3/2}\delta^{l_1}\|\slashed{\mathcal L}_Z^{k_1}\leftidx{^{(T)}}{\slashed\pi}_{\mathring L}\|_{L^2(\Sigma_s^u)}+\delta\delta^{l_1}\|\slashed{\mathcal L}_Z^{k_1}\leftidx{^{(R)}}{\slashed\pi}_T\|_{L^2(\Sigma_s^u)}\\
 &+\delta^{1-\varepsilon_0}\delta^{l_1}\|\slashed{\mathcal L}_Z^{k_1}\slashed g\|_{L^2(\Sigma_s^u)}+\delta^{3-2\varepsilon_0}s^{-3/2}\delta^{l_{0}}\|Z^{k_{0}}\mu\|_{L^2(\Sigma_s^u)}\\
 &+\delta^{2-\varepsilon_0} s^{-\epsilon}\sqrt{E_{1,\leq k+2}(s,u)}+\delta^{2-\varepsilon_0}\sqrt{E_{2,\leq k+2}(s,u)}.
 \end{split}
 \end{equation}
 In addition, by \eqref{omega} and \eqref{Lpi}-\eqref{Rpi}, we have
 \begin{equation}\label{Zku}
 \begin{split}
 \delta^{l}\|Z^{k+1}\upsilon_{ij}\|_{L^2(\Sigma_s^u)}\lesssim& \delta^{1-\varepsilon_0}s^{-1}\delta^{l_0}\|Z^{k_0}x^i\|_{L^2(\Sigma_s^u)}+\delta^{2-\varepsilon_0} s^{-1/2}\delta^{l_{0}}\|Z^{k_{0}}\mu\|_{L^2(\Sigma_s^u)}\\
 &+s\delta^{l_0}\|Z^{k_0}\check L^i\|_{L^2(\Sigma_s^u)}+\delta^{3-2\varepsilon_0}s^{-3/2}\delta^{l_1}\|\slashed{\mathcal L}_Z^{k_1}\slashed g\|_{L^2(\Sigma_s^u)}\\
 &+\delta^{3-\varepsilon_0}s^{-1/2}\delta^{l_1}\|\slashed{\mathcal L}_Z^{k_1}\leftidx{^{(T)}}{\slashed\pi}_{\mathring L}\|_{L^2(\Sigma_s^u)}+\delta^{3-\varepsilon_0}s^{-3/2}\delta^{l_1}\|\slashed{\mathcal L}_Z^{k_1}\leftidx{^{(R)}}{\slashed\pi}_T\|_{L^2(\Sigma_s^u)}\\
 &+\delta s^{1-\epsilon}\sqrt{E_{1,\leq k+2}(s,u)}+\delta s\sqrt{E_{2,\leq k+2}(s,u)},
 \end{split}
 \end{equation}
 \begin{equation}\label{ZkT}
 \begin{split}
 \delta^l\|\slashed{\mathcal L}_Z^k\leftidx{^{(T)}}{\slashed\pi}\|_{L^2(\Sigma_s^u)}\lesssim& s^{-1}\delta^{l_{0}}\|Z^{k_{0}}\mu\|_{L^2(\Sigma_s^u)}+\delta^{-\varepsilon_0}s^{-5/2}\delta^{l_0}\|Z^{k_0}x^i\|_{L^2(\Sigma_s^u)}
 +\delta^{l_1}\|\slashed{\mathcal L}_Z^{k_1}\check\chi\|_{L^2(\Sigma_s^u)}\\
 &+s^{-1}\delta^{l_1}\|\slashed{\mathcal L}_Z^{k_1}\slashed g\|_{L^2(\Sigma_s^u)}
 +\delta^{1-\varepsilon_0}s^{-3/2}\delta^{l_0}\|Z^{k_0}\check L^i\|_{L^2(\Sigma_s^u)}\\
  &+\delta^{2-\varepsilon_0}s^{-3/2}\delta^{l_1}\|\slashed{\mathcal L}_Z^{k_1}\leftidx{^{(T)}}{\slashed\pi}_{\mathring L}\|_{L^2(\Sigma_s^u)}+\delta^{2-\varepsilon_0}s^{-5/2}\delta^{l_1}\|\slashed{\mathcal L}_Z^{k_1}\leftidx{^{(R)}}{\slashed\pi}_T\|_{L^2(\Sigma_s^u)}\\
 &+s^{-\epsilon}\sqrt{E_{1,\leq k+2}(s,u)}+\sqrt{E_{2,\leq k+2}(s,u)},
 \end{split}
 \end{equation}

 \begin{equation}\label{ZkTL}
 \begin{split}
 \delta^l\|\slashed{\mathcal L}_Z^k\leftidx{^{(T)}}{\slashed\pi}_{\mathring L}\|_{L^2(\Sigma_s^u)}\lesssim&\delta^{5/2-2\varepsilon_0}s^{-1}+s^{-1}\delta^{l_{0}}\|Z^{k_{0}}\mu\|_{L^2(\Sigma_s^u)}
 +\delta^{-\varepsilon_0}s^{-3/2}\delta^{l_0}\|Z^{k_0}\check L^i\|_{L^2(\Sigma_s^u)}\\
 &+\delta^{-\varepsilon_0}s^{-5/2}\delta^{l_0}\|Z^{k_0}x^i\|_{L^2(\Sigma_s^u)}
 +\delta^{2-\varepsilon_0}s^{-5/2}\delta^{l_1}\|\slashed{\mathcal L}_Z^{k_1}\leftidx{^{(R)}}{\slashed\pi}_T\|_{L^2(\Sigma_s^u)}\\
 &+\delta^{2-2\varepsilon_0}s^{-5/2}\delta^{l_1}\|\slashed{\mathcal L}_Z^{k_1}\slashed g\|_{L^2(\Sigma_s^u)}
 +s^{-\epsilon}\sqrt{E_{1,\leq k+2}(s,u)}\\
 &+\sqrt{E_{2,\leq k+2}(s,u)},
 \end{split}
 \end{equation}
 \begin{equation}\label{ZkRT}
 \begin{split}
 \delta^{l}\|\slashed{\mathcal L}_Z^{k}\leftidx{^{(R)}}{\slashed\pi}_T\|_{L^2(\Sigma_s^u)}\lesssim&\delta^{9/2-4\varepsilon_0}s^{-2}
 +\delta^{1-\varepsilon_0}s^{-1}\Big\{\delta^{l_{0}}\|Z^{k_{0}}\mu\|_{L^2(\Sigma_s^u)}
 +s^{-1}\delta^{l_1}\|\slashed{\mathcal L}_Z^{k_1}R\|_{L^2(\Sigma_s^u)}\Big\}\\
 &+s\delta^{l_1}\|\slashed{\mathcal L}_Z^{k_1}\check\chi\|_{L^2(\Sigma_s^u)}+\delta^{-\varepsilon_0}s^{-3/2}\delta^{l_1}\|Z^{k_1}\upsilon_{ij}\|_{L^2(\Sigma_s^u)}
\\
 &+\delta^{1-2\varepsilon_0}s^{-2}\delta^{l_0}\|Z^{k_0}x^i\|_{L^2(\Sigma_s^u)}
 +\delta^{3-3\varepsilon_0}s^{-5/2}\delta^{l_1}\|\slashed{\mathcal L}_Z^{k_1}\slashed g\|_{L^2(\Sigma_s^u)}\\
 &+\delta^{3-2\varepsilon_0}s^{-3/2}\delta^{l_1}\|\slashed{\mathcal L}_Z^{k_1}\leftidx{^{(T)}}{\slashed\pi}_{\mathring L}\|_{L^2(\Sigma_s^u)} +\delta^{l_0}\|Z^{k_0}\check L^i\|_{L^2(\Sigma_s^u)}\\
 &+\delta^{1-\varepsilon_0}s^{-\epsilon}\sqrt{E_{1,\leq k+2}(s,u)}+\delta^{1-\varepsilon_0}\sqrt{E_{2,\leq k+2}(s,u)},
 \end{split}
 \end{equation}
 \begin{equation}\label{ZkRL}
 \begin{split}
 \delta^{l}\|\slashed{\mathcal L}_Z^{k}\leftidx{^{(R)}}{\slashed\pi}_{\mathring L}\|_{L^2(\Sigma_s^u)}\lesssim&\delta^{l_0}\|Z^{k_0}\check L^i\|_{L^2(\Sigma_s^u)}+\delta^{1-\varepsilon_0}s^{-2}\delta^{l_0}\|Z^{k_0}x^i\|_{L^2(\Sigma_s^u)}+s\delta^{l_1}\|\slashed{\mathcal L}_Z^{k_1}\check\chi\|_{L^2(\Sigma_s^u)}\\
 &+\delta^{1-\varepsilon_0}s^{-5/2}\delta^{l_1}\|Z^{k_1}\upsilon_{ij}\|_{L^2(\Sigma_s^u)}
 +\delta^{1-\varepsilon_0}s^{-2}\delta^{l_1}\|\slashed{\mathcal L}_Z^{k_1}R\|_{L^2(\Sigma_s^u)}\\
 &+\delta^{3-\varepsilon_0}s^{-3/2}\delta^{l_1}\|\slashed{\mathcal L}_Z^{k_1}\leftidx{^{(T)}}{\slashed\pi}_{\mathring L}\|_{L^2(\Sigma_s^u)}+\delta^{3-\varepsilon_0}s^{-5/2}\delta^{l_1}\|\slashed{\mathcal L}_Z^{k_1}\leftidx{^{(R)}}{\slashed\pi}_T\|_{L^2(\Sigma_s^u)}\\
 &+\delta^{3-2\varepsilon_0}s^{-5/2}\delta^{l_1}\|\slashed{\mathcal L}_Z^{k_1}\slashed g\|_{L^2(\Sigma_s^u)}+\delta^{2-\varepsilon_0}s^{-3/2}\delta^{l_{0}}\|Z^{k_{0}}\mu\|_{L^2(\Sigma_s^u)}\\
 &+\delta s^{-\epsilon}\sqrt{E_{1,\leq k+2}(s,u)}+\delta\sqrt{E_{2,\leq k+2}(s,u)},
 \end{split}
 \end{equation}
 \begin{equation}\label{ZkR}
 \begin{split}
 \delta^l\|\slashed{\mathcal L}_Z^{k}\leftidx{^{(R)}}{\slashed\pi}\|_{L^2(\Sigma_s^u)}\lesssim&\delta^{2-2\varepsilon_0}s^{-3/2}\delta^{l_0}\|Z^{k_0}\check L^i\|_{L^2(\Sigma_s^u)}+\delta^{1-\varepsilon_0}s^{-5/2}\delta^{l_0}\|Z^{k_0}x^i\|_{L^2(\Sigma_s^u)}\\
 &+\delta^{1-\varepsilon_0}\delta^{l_1}\|\slashed{\mathcal L}_Z^{k_1}\check\chi\|_{L^2(\Sigma_s^u)}+s^{-1}\delta^{l_1}\|Z^{k_1}\upsilon_{ij}\|_{L^2(\Sigma_s^u)}\\
 &+\delta^{3-\varepsilon_0}s^{-3/2}\delta^{l_1}\|\slashed{\mathcal L}_Z^{k_1}\leftidx{^{(T)}}{\slashed\pi}_{\mathring L}\|_{L^2(\Sigma_s^u)}+\delta^{3-\varepsilon_0}s^{-5/2}\delta^{l_1}\|\slashed{\mathcal L}_Z^{k_1}\leftidx{^{(R)}}{\slashed\pi}_T\|_{L^2(\Sigma_s^u)}\\
 &+\delta^{1-\varepsilon_0}s^{-1}\delta^{l_1}\|\slashed{\mathcal L}_Z^{k_1}\slashed g\|_{L^2(\Sigma_s^u)}+\delta^{2-\varepsilon_0}s^{-3/2}\delta^{l_{0}}\|Z^{k_{0}}\mu\|_{L^2(\Sigma_s^u)}\\
 &+\delta s^{-\epsilon}\sqrt{E_{1,\leq k+2}(s,u)}+\delta\sqrt{E_{2,\leq k+2}(s,u)}.
 \end{split}
 \end{equation}
Using $\slashed{\mathcal L}_{\varrho\mathring L}\slashed g=\varrho\check\chi+\slashed g$, $\slashed{\mathcal L}_T\slashed g=\leftidx{^{(T)}}{\slashed\pi}$ and $\slashed{\mathcal L}_R\slashed g=\leftidx{^{(R)}}{\slashed\pi}$, together with
\eqref{ZkT} and \eqref{ZkR}, one has
\begin{equation}\label{Zkg}
 \begin{split}
 \delta^l\|\slashed{\mathcal L}_Z^{k+1}\slashed g\|_{L^2(\Sigma_s^u)}\lesssim&\delta^{1/2}s^{3/2}+\delta^{2-2\varepsilon_0}s^{-3/2}\delta^{l_0}\|Z^{k_0}\check L^i\|_{L^2(\Sigma_s^u)}+\delta^{1-\varepsilon_0}s^{-5/2}\delta^{l_0}\|Z^{k_0}x^i\|_{L^2(\Sigma_s^u)}\\
 &+s\delta^{l_1}\|\slashed{\mathcal L}_Z^{k_1}\check\chi\|_{L^2(\Sigma_s^u)}+s^{-1}\delta^{l_1}\|Z^{k_1}\upsilon_{ij}\|_{L^2(\Sigma_s^u)}+\delta s^{-1}\delta^{l_{0}}\|Z^{k_{0}}\mu\|_{L^2(\Sigma_s^u)}\\
 &+\delta^{3-\varepsilon_0}s^{-3/2}\delta^{l_1}\|\slashed{\mathcal L}_Z^{k_1}\leftidx{^{(T)}}{\slashed\pi}_{\mathring L}\|_{L^2(\Sigma_s^u)}+\delta^{3-\varepsilon_0}s^{-5/2}\delta^{l_1}\|\slashed{\mathcal L}_Z^{k_1}\leftidx{^{(R)}}{\slashed\pi}_T\|_{L^2(\Sigma_s^u)}\\
 &+\delta s^{-\epsilon}\sqrt{E_{1,\leq k+2}(s,u)}+\delta\sqrt{E_{2,\leq k+2}(s,u)}.
 \end{split}
 \end{equation}

 Collecting \eqref{ZkcL}-\eqref{Zkg} yields
 \begin{align}
 \delta^l\|Z^{k+1}\check L^i\|_{L^2(\Sigma_s^u)}\lesssim&\delta^{3/2-\varepsilon_0}s^{1/2}+s\delta^{l_1}\|\slashed{\mathcal L}_Z^{k_1}\check\chi\|_{L^2(\Sigma_s^u)}+\delta s^{-1}\delta^{l_{0}}\|Z^{k_{0}}\mu\|_{L^2(\Sigma_s^u)}\no\\
 &+\delta s^{-\epsilon}\sqrt{E_{1,\leq k+2}(s,u)}+\delta\sqrt{E_{2,\leq k+2}(s,u)},\label{Zk1L}\\
  \delta^l\|Z^{k+2}x^i\|_{L^2(\Sigma_s^u)}\lesssim&\delta^{1/2}s^{5/2}+s^2\delta^{l_1}\|\slashed{\mathcal L}_Z^{k_1}\check\chi\|_{L^2(\Sigma_s^u)}+\delta\delta^{l_{0}}\|Z^{k_{0}}\mu\|_{L^2(\Sigma_s^u)}\no\\
  &+\delta s^{1-\epsilon}\sqrt{E_{1,\leq k+2}(s,u)}+\delta s\sqrt{E_{2,\leq k+2}(s,u)},\label{Zk1x}\\
\delta^l\|\slashed{\mathcal L}_Z^{k+1}\slashed g\|_{L^2(\Sigma_s^u)}\lesssim&\delta^{1/2}s^{3/2}+s\delta^{l_1}\|\slashed{\mathcal L}_Z^{k_1}\check\chi\|_{L^2(\Sigma_s^u)}+\delta s^{-1}\delta^{l_{0}}\|Z^{k_{0}}\mu\|_{L^2(\Sigma_s^u)}\no\\
&+\delta s^{-\epsilon}\sqrt{E_{1,\leq k+2}(s,u)}+\delta\sqrt{E_{2,\leq k+2}(s,u)},\label{Zk1g}\\
 \delta^{l}\|Z^{k+1}\upsilon_{ij}\|_{L^2(\Sigma_s^u)}\lesssim&\delta^{3/2-\varepsilon_0}s^{3/2}+s^2\delta^{l_1}\|\slashed{\mathcal L}_Z^{k_1}\check\chi\|_{L^2(\Sigma_s^u)}+\delta\delta^{l_{0}}\|Z^{k_{0}}\mu\|_{L^2(\Sigma_s^u)}\no\\
 &+\delta s^{1-\epsilon}\sqrt{E_{1,\leq k+2}(s,u)}+\delta s\sqrt{E_{2,\leq k+2}(s,u)},\label{Zk1u}\\
 \delta^l\|\slashed{\mathcal L}_Z^{k+1}R\|_{L^2(\Sigma_s^u)}\lesssim& \delta^{1/2}s^{5/2}+s^2\delta^{l_1}\|\slashed{\mathcal L}_Z^{k_1}\check\chi\|_{L^2(\Sigma_s^u)}+\delta\delta^{l_{0}}\|Z^{k_{0}}\mu\|_{L^2(\Sigma_s^u)}\no\\
 &+\delta s^{1-\epsilon}\sqrt{E_{1,\leq k+2}(s,u)}+\delta s\sqrt{E_{2,\leq k+2}(s,u)},\label{Zk1R}\\
 \delta^l\|\slashed{\mathcal L}_Z^k\leftidx{^{(T)}}{\slashed\pi}\|_{L^2(\Sigma_s^u)}\lesssim&\delta^{1/2-\varepsilon_0}s^{1/2}
 +\delta^{-\varepsilon_0}\delta^{l_1}\|\slashed{\mathcal L}_Z^{k_1}\check\chi\|_{L^2(\Sigma_s^u)}+s^{-1}\delta^{l_{0}}\|Z^{k_{0}}\mu\|_{L^2(\Sigma_s^u)}\no\\
 &+s^{-\epsilon}\sqrt{E_{1,\leq k+2}(s,u)}+\sqrt{E_{2,\leq k+2}(s,u)},\label{ZkTpi}\\
 \delta^l\|\slashed{\mathcal L}_Z^k\leftidx{^{(T)}}{\slashed\pi}_{\mathring L}\|_{L^2(\Sigma_s^u)}\lesssim&\delta^{1/2-\varepsilon_0}+\delta^{-\varepsilon_0}s^{-1/2}\delta^{l_1}\|\slashed{\mathcal L}_Z^{k_1}\check\chi\|_{L^2(\Sigma_s^u)}+ s^{-1}\delta^{l_{0}}\|Z^{k_{0}}\mu\|_{L^2(\Sigma_s^u)}\no\\
 &+s^{-\epsilon}\sqrt{E_{1,\leq k+2}(s,u)}+\sqrt{E_{2,\leq k+2}(s,u)},\label{ZkTLpi}\\
 \delta^{l}\|\slashed{\mathcal L}_Z^{k}\leftidx{^{(R)}}{\slashed\pi}_T\|_{L^2(\Sigma_s^u)}\lesssim&\delta^{3/2-2\varepsilon_0}s^{1/2}
 +\delta^{-\varepsilon_0}s\delta^{l_1}\|\slashed{\mathcal L}_Z^{k_1}\check\chi\|_{L^2(\Sigma_s^u)}+\delta^{1-\varepsilon_0} s^{-1}\delta^{l_{0}}\|Z^{k_{0}}\mu\|_{L^2(\Sigma_s^u)}\no\\
 &+\delta^{1-\varepsilon_0}s^{-\epsilon}\sqrt{E_{1,\leq k+2}(s,u)}+\delta^{1-\varepsilon_0}\sqrt{E_{2,\leq k+2}(s,u)},\label{ZkRLpi}\\
 \delta^{l}\|\slashed{\mathcal L}_Z^{k}\leftidx{^{(R)}}{\slashed\pi}_{\mathring L}\|_{L^2(\Sigma_s^u)}\lesssim&\delta^{3/2-\varepsilon_0}s^{1/2}+s\delta^{l_1}\|\slashed{\mathcal L}_Z^{k_1}\check\chi\|_{L^2(\Sigma_s^u)}
 +\delta s^{-1}\delta^{l_{0}}\|Z^{k_{0}}\mu\|_{L^2(\Sigma_s^u)}\no\\
 &+\delta s^{-\epsilon}\sqrt{E_{1,\leq k+2}(s,u)}+\delta\sqrt{E_{2,\leq k+2}(s,u)},\label{ZkRLp}\\
 \delta^l\|\slashed{\mathcal L}_Z^{k}\leftidx{^{(R)}}{\slashed\pi}\|_{L^2(\Sigma_s^u)}\lesssim&\delta^{3/2-\varepsilon_0}s^{1/2}+s\delta^{l_1}\|\slashed{\mathcal L}_Z^{k_1}\check\chi\|_{L^2(\Sigma_s^u)}+\delta s^{-1}\delta^{l_{0}}\|Z^{k_{0}}\mu\|_{L^2(\Sigma_s^u)}\no\\
 &+\delta s^{-\epsilon}\sqrt{E_{1,\leq k+2}(s,u)}+\delta\sqrt{E_{2,\leq k+2}(s,u)}.\label{ZkRp}
 \end{align}
 And hence, with the help of \eqref{zkp}-\eqref{ZLp} and \eqref{rrho}, we have obtained that
 \begin{align}
 &\delta^l\|Z^{k+1}\phi\|_{L^2(\Sigma_s^u)}+\delta\delta^l\|Z^{k+1}T\phi\|_{L^2(\Sigma_s^u)}\no\\
 \lesssim&\delta^{5/2-\varepsilon_0}+\delta^{2-\varepsilon_0}s^{-1/2}\delta^{l_1}\|\slashed{\mathcal L}_Z^{k_1}\check\chi\|_{L^2(\Sigma_s^u)}+\delta^{2-\varepsilon_0} s^{-3/2}\delta^{l_{0}}\|Z^{k_{0}}\mu\|_{L^2(\Sigma_s^u)}\label{zp}\\
 &+\delta^2 s^{-\epsilon}\sqrt{E_{1,\leq k+2}(s,u)}+\delta^2\sqrt{E_{2,\leq k+2}(s,u)},\no
 \end{align}
 \begin{equation}\label{zL}
 \begin{split}
 &\delta^l\|Z^{k+1}\mathring L\phi\|_{L^2(\Sigma_s^u)}\\
 \lesssim&\delta^{3/2-\varepsilon_0}+\delta^{1-\varepsilon_0}s^{-1/2}\delta^{l_1}\|\slashed{\mathcal L}_Z^{k_1}\check\chi\|_{L^2(\Sigma_s^u)}+\delta^{2-\varepsilon_0} s^{-5/2}\delta^{l_{0}}\|Z^{k_{0}}\mu\|_{L^2(\Sigma_s^u)}\\
 &+\delta s^{-\epsilon}\sqrt{E_{1,\leq k+2}(s,u)}+\delta\sqrt{E_{2,\leq k+2}(s,u)}
 \end{split}
 \end{equation}
 and
 \begin{equation}\label{Zkr}
 \begin{split}
 \delta^l\|Z^{k+1}\check\varrho\|_{L^2(\Sigma_s^u)}\lesssim&\delta^{3/2-\varepsilon_0}s^{1/2}+s\delta^{l_1}\|\slashed{\mathcal L}_Z^{k_1}\check\chi\|_{L^2(\Sigma_s^u)}+\delta s^{-1}\delta^{l_{0}}\|Z^{k_{0}}\mu\|_{L^2(\Sigma_s^u)}\\
 &+\delta s^{-\epsilon}\sqrt{E_{1,\leq k+2}(s,u)}+\delta\sqrt{E_{2,\leq k+2}(s,u)}.
 \end{split}
 \end{equation}
As in \eqref{Rkp}, all the vector fields
$Z$ in \eqref{Zk1L}-\eqref{ZkRp} can be replaced by the rotation vector fields $R_{ij}$ to obtain the following analogous estimates

 \begin{equation*}
 \begin{split}
\|R^{k+1}\check L^i\|_{L^2(\Sigma_s^u)}\lesssim&\delta^{3/2-\varepsilon_0}s^{1/2}+s\|\slashed{\mathcal L}_R^{\leq k}\check\chi\|_{L^2(\Sigma_s^u)}+\delta s^{-1}\|R^{[1, k+1]}\mu\|_{L^2(\Sigma_s^u)}\\
&+\delta s^{-\epsilon}\sqrt{E_{1,\leq k+2}(s,u)}+\delta\sqrt{E_{2,\leq k+2}(s,u)},\\
\end{split}
\end{equation*}

\begin{align}
\|R^{k+2}x^i\|_{L^2(\Sigma_s^u)}\lesssim&\delta^{1/2}s^{5/2}+s^2\|\slashed{\mathcal L}_R^{\leq k}\check\chi\|_{L^2(\Sigma_s^u)}+\delta\|R^{[1, k+1]}\mu\|_{L^2(\Sigma_s^u)}\no\\
&+\delta s^{1-\epsilon}\sqrt{E_{1,\leq k+2}(s,u)}+\delta s\sqrt{E_{2,\leq k+2}(s,u)},\no\\
\|\slashed{\mathcal L}_R^{k+1}\slashed g\|_{L^2(\Sigma_s^u)}\lesssim&\delta^{1/2}s^{3/2}+s\|\slashed{\mathcal L}_R^{\leq k}\check\chi\|_{L^2(\Sigma_s^u)}+\delta s^{-1}\|R^{[1,k+1]}\mu\|_{L^2(\Sigma_s^u)}\no\\
&+\delta s^{-\epsilon}\sqrt{E_{1,\leq k+2}(s,u)}+\delta\sqrt{E_{2,\leq k+2}(s,u)},\no\\
\|R^{k+1}\upsilon_{ij}\|_{L^2(\Sigma_s^u)}\lesssim&\delta^{3/2-\varepsilon_0}s^{3/2}+s^2\|\slashed{\mathcal L}_R^{\leq k}\check\chi\|_{L^2(\Sigma_s^u)}+\delta\|R^{[1,k+1]}\mu\|_{L^2(\Sigma_s^u)}\no\\
&+\delta s^{1-\epsilon}\sqrt{E_{1,\leq k+2}(s,u)}+\delta s\sqrt{E_{2,\leq k+2}(s,u)},\no\\
\|\slashed{\mathcal L}_R^{k+1}R\|_{L^2(\Sigma_s^u)}\lesssim& \delta^{1/2}s^{5/2}+s^2\|\slashed{\mathcal L}_R^{\leq k}\check\chi\|_{L^2(\Sigma_s^u)}+\delta\|R^{[1,k+1]}\mu\|_{L^2(\Sigma_s^u)}\no\\
&+\delta s^{1-\epsilon}\sqrt{E_{1,\leq k+2}(s,u)}+\delta s\sqrt{E_{2,\leq k+2}(s,u)},\no\\
\delta^l\|\slashed{\mathcal L}_R^k\leftidx{^{(T)}}{\slashed\pi}\|_{L^2(\Sigma_s^u)}\lesssim&\delta^{1/2-\varepsilon_0}s^{1/2}
+\delta^{-\varepsilon_0}\|\slashed{\mathcal L}_R^{\leq k}\check\chi\|_{L^2(\Sigma_s^u)}+s^{-1}\|R^{[1, k+1]}\mu\|_{L^2(\Sigma_s^u)}\label{RL2}\\
&+s^{-\epsilon}\sqrt{E_{1,\leq k+2}(s,u)}+\sqrt{E_{2,\leq k+2}(s,u)},\no\\
\|\slashed{\mathcal L}_R^k\leftidx{^{(T)}}{\slashed\pi}_{\mathring L}\|_{L^2(\Sigma_s^u)}\lesssim&\delta^{1/2-\varepsilon_0}+\delta^{-\varepsilon_0}s^{-1/2}\|\slashed{\mathcal L}_R^{\leq k}\check\chi\|_{L^2(\Sigma_s^u)}+ s^{-1}\|R^{[1, k+1]}\mu\|_{L^2(\Sigma_s^u)}\no\\
&+s^{-\epsilon}\sqrt{E_{1,\leq k+2}(s,u)}+\sqrt{E_{2,\leq k+2}(s,u)},\no\\
\|\slashed{\mathcal L}_R^{k}\leftidx{^{(R)}}{\slashed\pi}_T\|_{L^2(\Sigma_s^u)}\lesssim&\delta^{3/2-2\varepsilon_0}s^{1/2}
+\delta^{-\varepsilon_0}s\|\slashed{\mathcal L}_R^{\leq k}\check\chi\|_{L^2(\Sigma_s^u)}+\delta^{1-\varepsilon_0} s^{-1}\|R^{[1, k+1]}\mu\|_{L^2(\Sigma_s^u)}\no\\
&+\delta^{1-\varepsilon_0}s^{-\epsilon}\sqrt{E_{1,\leq k+2}(s,u)}+\delta^{1-\varepsilon_0}\sqrt{E_{2,\leq k+2}(s,u)},\no\\
\|\slashed{\mathcal L}_R^{k}\leftidx{^{(R)}}{\slashed\pi}_{\mathring L}\|_{L^2(\Sigma_s^u)}\lesssim&\delta^{3/2-\varepsilon_0}s^{1/2}+s\|\slashed{\mathcal L}_R^{\leq k}\check\chi\|_{L^2(\Sigma_s^u)}
+\delta s^{-1}\|R^{[1, k+1]}\mu\|_{L^2(\Sigma_s^u)}\no\\
&+\delta s^{-\epsilon}\sqrt{E_{1,\leq k+2}(s,u)}+\delta\sqrt{E_{2,\leq k+2}(s,u)},\no\\
\|\slashed{\mathcal L}_R^{k}\leftidx{^{(R)}}{\slashed\pi}\|_{L^2(\Sigma_s^u)}\lesssim&\delta^{3/2-\varepsilon_0}s^{1/2}+s\|\slashed{\mathcal L}_R^{\leq k}\check\chi\|_{L^2(\Sigma_s^u)}+\delta s^{-1}\|R^{[1, k+1]}\mu\|_{L^2(\Sigma_s^u)}\no\\
&+\delta s^{-\epsilon}\sqrt{E_{1,\leq k+2}(s,u)}+\delta\sqrt{E_{2,\leq k+2}(s,u)}.\no
\end{align}
This together with \eqref{Rkp}-\eqref{RLp} yields
\begin{equation}\label{RRT}
\begin{split}
&\|R^{k+1}\phi\|_{L^2(\Sigma_s^u)}+\delta\|R^{k+1}T\phi\|_{L^2(\Sigma_s^u)}\\
\lesssim&\delta^{5/2-\varepsilon_0}+\delta^{2-\varepsilon_0}s^{-1/2}\|\slashed{\mathcal L}_R^{\leq k}\check\chi\|_{L^2(\Sigma_s^u)}+\delta^{2-\varepsilon_0} s^{-3/2}\|R^{[1, k+1]}\mu\|_{L^2(\Sigma_s^u)}\\
&+\delta^2 s^{-\epsilon}\sqrt{E_{1,\leq k+2}(s,u)}+\delta^2\sqrt{E_{2,\leq k+2}(s,u)},
\end{split}
\end{equation}
\begin{equation}\label{RL}
\begin{split}
&\|R^{k}\mathring L\phi\|_{L^2(\Sigma_s^u)}\\
\lesssim&\delta^{5/2-\varepsilon_0}s^{-1}+\delta^{2-2\varepsilon_0}s^{-3/2}\|\slashed{\mathcal L}_R^{\leq k}\check\chi\|_{L^2(\Sigma_s^u)}+\delta^{2-\varepsilon_0} s^{-5/2}\|R^{[1, k+1]}\mu\|_{L^2(\Sigma_s^u)}\\
&+\delta^2 s^{-1-\epsilon}\sqrt{E_{1,\leq k+2}(s,u)}+\delta^2s^{-1}\sqrt{E_{2,\leq k+2}(s,u)},
\end{split}
\end{equation}
\begin{equation}\label{Rkr}
\begin{split}
\|R^{k+1}\check\varrho\|_{L^2(\Sigma_s^u)}\lesssim&\delta^{3/2-\varepsilon_0}s^{1/2}+s\|\slashed{\mathcal L}_R^{\leq k}\check\chi\|_{L^2(\Sigma_s^u)}+\delta s^{-1}\|R^{[1, k+1]}\mu\|_{L^2(\Sigma_s^u)}\\
&+\delta s^{-\epsilon}\sqrt{E_{1,\leq k+2}(s,u)}+\delta\sqrt{E_{2,\leq k+2}(s,u)}.
\end{split}
\end{equation}
 Note that all the terms in the left hand side of \eqref{Zk1L}-\eqref{ZkRp} are controlled by the $L^2-$norms of
 the derivatives of $\check\chi$ and $\mu$, which will be treated as follows.

\begin{itemize}
	\item
When $Z\in\{R_{ij}:1\leq i<j\leq 4\}$, as in the proof of Lemma \ref{Rh}, one needs to deal with each term in
$\slashed{\mathcal L}_{\mathring L}\slashed{\mathcal L}_{R}^k\check\chi$ (see \eqref{LRichi} and \eqref{Lchi'}).
Due to $\slashed{\mathcal L}_{R_{ij}}\slashed dx^a=\slashed d\O_{ij}^a+\slashed d\big(\upsilon_{ij}(g^{0a}+\check L^a
+\f{x^a}{\varrho})\big)$, then
\begin{equation}\label{Lx}
|\slashed{\mathcal L}_{R}^{k}\slashed dx^a|\lesssim 1+t^{-1}|R^{\leq k}\upsilon_{ij}|
+\delta^{1-\varepsilon_0} t^{-1}\big(|R^{\leq k}\phi|+|R^{\leq k}\varphi|+|R^{\leq k}\check L^a|\big).
\end{equation}
This implies
\begin{equation}\label{FdL}
\begin{split}
&|\slashed{\mathcal L}_{R}^k(F_{B\mathring L}\slashed d_A\mathring L\phi)|
+|\slashed{\mathcal L}_{R}^k(F_{A\mathring L}\slashed d_B\mathring L\phi)|\\
\lesssim&\sum_{k_1+k_2+k_3=k}|R^{k_1}(F_{i\beta}\mathring L^\beta)|\cdot|\slashed{\mathcal L}_R^{k_2}\slashed dx^i|\cdot|\slashed{\mathcal L}_R^{k_3}\slashed d\mathring L\phi|\\
\lesssim&\sum_{k_1+k_2\leq k}\big(|R^{k_1}\phi|+|R^{k_1}\varphi|+|R^{k_1}\check L^i|+|\slashed{\mathcal L}_R^{k_2}\slashed dx^i|\big)|\slashed{\mathcal L}_R^{k_3}\slashed d\mathring L\phi|\\
\lesssim&|\slashed dR^{\leq k}\mathring L\phi|+\delta^{2-\varepsilon_0}s^{-7/2}\big(s^{-1}|R^{\leq k}\upsilon_{ab}|+|R^{\leq k}\phi|+|R^{\leq k}\varphi|+|R^{\leq k}\check L^i|\big).
\end{split}
\end{equation}
Similarly, one has
\begin{equation}\label{GdL}
\begin{split}
&|\slashed{\mathcal L}_{R}^k(G_{B\mathring L}^\gamma\slashed d_A\mathring L\varphi_\gamma)|+|\slashed{\mathcal L}_{R}^k(G_{A\mathring L}^\gamma\slashed d_B\mathring L\varphi_\gamma)|+|\slashed{\mathcal L}_R^k\big(F_{AB}\mathring L^2\phi\big)|+|\slashed{\mathcal L}_R^k\big(G_{AB}^\gamma\mathring L^2\varphi_\gamma\big)|\\
\lesssim&|\slashed dR^{\leq k}\mathring L\varphi|+|R^{\leq k}\mathring L^2\phi|+|R^{\leq k}\mathring L^2\varphi|\\
& +\delta^{1-\varepsilon_0}s^{-7/2}\big(s^{-1}|R^{\leq k}\upsilon_{ab}|+|R^{\leq k}\phi|+|R^{\leq k}\varphi|
+|R^{\leq k}\check L^i|\big)
\end{split}
\end{equation}
and
\begin{equation}\label{RFL}
\begin{split}
&|R^k(F_{\mathring L\mathring L}\mathring L\phi)|+|R^k(G_{\mathring L\mathring L}^\gamma\mathring L\varphi_\gamma)|+|R^k(F_{\tilde T\mathring L}\mathring L\phi)|\\
&+|R^k(G_{\tilde T\mathring L}^\gamma\mathring L\varphi_\gamma)|+|\slashed{\mathcal L}_R^k(F_{B\mathring L}\slashed d_A\phi)|+|\slashed{\mathcal L}_R^k(G_{B\mathring L}^\gamma\slashed d_A\varphi_\gamma)|\\
\lesssim&|R^{\leq k}\mathring L\phi|+|R^{\leq k}\mathring L\varphi|+|\slashed dR^{\leq k}\phi|+|\slashed dR^{\leq k}\varphi|\\
&+\delta^{1-\varepsilon_0}s^{-5/2}\big(s^{-1}|R^{\leq k}\upsilon_{ab}|+|R^{\leq k}\phi|+|R^{\leq k}\varphi|+|R^{\leq k}\check L^i|\big).
\end{split}
\end{equation}

In addition, for the term $F_{\mathring L\mathring L}\slashed\nabla_{AB}^2\phi$, one can apply the part (a) of Lemma \ref{commute} and Corollary \ref{12form} to get
\begin{equation}\label{Fp}
\begin{split}
&|\slashed{\mathcal L}_R^k\big(F_{\mathring L\mathring L}\slashed\nabla_{AB}^2\phi\big)|\\
\lesssim&\sum_{k_1+k_2=k}|R^{k_1}F_{\mathring L\mathring L}|\cdot|[\slashed{\mathcal L}_R^{k_2},\slashed\nabla_{AB}^2]\phi+\slashed\nabla_{AB}^2R^{k_2}\phi|\\
\lesssim&s^{-2}|R^{\leq k+2}\phi|+\delta^{2-\varepsilon_0}s^{-7/2}|\slashed{\mathcal L}_R^{\leq k}\leftidx{^{(R)}}{\slashed\pi}|+\delta^{2-\varepsilon_0}s^{-9/2}|R^{\leq k}\upsilon_{ab}|\\
&+\delta^{2-\varepsilon_0}s^{-7/2}\big(|R^{\leq k}\phi|+|R^{\leq k}\varphi|+|R^{\leq k}\check L^i|\big)
\end{split}
\end{equation}
and
\begin{equation}\label{Fvp}
\begin{split}
&|\slashed{\mathcal L}_R^k\big(G_{\mathring L\mathring L}^\gamma\slashed\nabla_{AB}^2\varphi_\gamma\big)|\\
\lesssim&s^{-1}|\slashed dR^{\leq k+1}\varphi|+\delta^{1-\varepsilon_0}s^{-7/2}|\slashed{\mathcal L}_R^{\leq k}\leftidx{^{(R)}}{\slashed\pi}|+\delta^{1-\varepsilon_0}s^{-9/2}|R^{\leq k}\upsilon_{ab}|\\
&+\delta^{1-\varepsilon_0}s^{-7/2}\big(|R^{\leq k}\phi|+|R^{\leq k}\varphi|+|R^{\leq k}\check L^i|\big).
\end{split}
\end{equation}
Substituting \eqref{FdL}-\eqref{Fvp} into \eqref{Lchi'} yields
\begin{equation}\label{RLc}
\begin{split}
|\slashed{\mathcal L}_R^k\slashed{\mathcal L}_{\mathring L}\check\chi|\lesssim& s^{-1}|R^{\leq k}\mathring L\phi|+s^{-1}|R^{\leq k}\mathring L\varphi|+s^{-2}|R^{\leq k+2}\phi|+s^{-1}|\slashed dR^{\leq k+1}\varphi|\\
&+|\slashed dR^{\leq k}\mathring L\phi|+|\slashed dR^{\leq k}\mathring L\varphi|+|R^{\leq k}\mathring L^2\phi|+|R^{\leq k}\mathring L^2\varphi|\\
&+\delta^{1-\varepsilon_0}s^{-7/2}\big(s^{-1}|R^{\leq k}\upsilon_{ab}|+|R^{\leq k}\phi|+|R^{\leq k}\varphi|+|R^{\leq k}\check L^i|\big)\\
&+\delta^{1-\varepsilon_0}s^{-2}|\slashed{\mathcal L}_R^{\leq k}\check\chi|+\delta^{1-\varepsilon_0}s^{-5/2}|\slashed{\mathcal L}_R^{\leq k}\leftidx{^{(R)}}{\slashed\pi}|+\delta^{2-2\varepsilon_0}s^{-5}|\slashed{\mathcal L}_R^{\leq k}\slashed dx^i|.
\end{split}
\end{equation}

Next we come to estimate the $L^2$ norms of $R^{k+2}\phi$, $\slashed dR^k\mathring L\phi$ and $R^k\mathring L^2\phi$.
\begin{enumerate}
	\item As $R_{ij}=\O_{ij}-\upsilon_{ij}\tilde T$, then \eqref{Lx}, together
with \eqref{RL2} and \eqref{RRT}, implies that
	\begin{equation}\label{R2p}
	\begin{split}
	&\|R^{k+1}R_{ij}\phi\|_{L^2(\Sigma_s^u)}=\|R^{k+1}\big((\O_{ij}^l-\upsilon_{ij}\tilde T^l)\varphi_l\big)\|_{L^2(\Sigma_s^u)}\\
	\lesssim&\sum_{k_1+k_2=k}\|\big(s|\slashed{\mathcal L}_R^{k_1}\slashed dx|+|R^{k_1+1}(\upsilon_{ij}\tilde T^l)|\big)R^{k_2}\varphi\|_{L^2(\Sigma_s^u)}+s\|R^{k+1}\varphi\|_{L^2(\Sigma_s^u)}\\
	\lesssim&s\|R^{k+1}\varphi\|_{L^2(\Sigma_s^u)}+\delta^{1-\varepsilon_0}s^{-3/2}\|R^{\leq k+1}\upsilon_{ij}\|_{L^2(\Sigma_s^u)}+\delta^{2-2\varepsilon_0}s^{-5/2}\|R^{\leq k+1}x^a\|_{L^2(\Sigma_s^u)}\\
	&+\delta^{2-2\varepsilon_0}s^{-3/2}\big(\|R^{\leq k+1}\phi\|_{L^2(\Sigma_s^u)}+\|R^{\leq k+1}\varphi\|_{L^2(\Sigma_s^u)}
+\|R^{\leq k+1}\check L^l\|_{L^2(\Sigma_s^u)}\big)\\
	\lesssim&\delta^{5/2-2\varepsilon_0}+\delta^{1-\varepsilon_0}s^{1/2}\|\slashed{\mathcal L}_R^{\leq k}\check{\chi}\|_{L^2(\Sigma_s^u)}+\delta^{2-\varepsilon_0}s^{-3/2}\|R^{\leq k+1}\mu\|_{L^2(\Sigma_s^u)}\\
	&+\delta s^{1-\epsilon}\sqrt{E_{1,\leq k+2}(s,u)}+\delta s\sqrt{E_{2,\leq k+2}(s,u)}.
	\end{split}
	\end{equation}

	\item Note that
	\begin{equation}\label{dRLp}
    \slashed dR^k\mathring L\phi=\slashed{\mathcal L}_R^k\slashed d(\mathring L^\al\varphi_\al)
    =\slashed{\mathcal L}_R^k\big((\slashed d\mathring L^i)\varphi_i+(\slashed dx^i)\mathring L\varphi_i\big).
	\end{equation}
	One gets from \eqref{dL} that
	\begin{equation*}
	(\slashed d_A\mathring L^i)\varphi_i=\chi_{AB}{\slashed d}^B\phi
-\big\{(\mathcal{FG})_{\mathring L\tilde T}(X_A)+\f12(\mathcal{FG})_{\tilde T\tilde T}(X_A)\big\}\tilde T^i\varphi_i
+\Lambda_{AB}\slashed d^B\phi.
	\end{equation*}
	It follows from \eqref{RL2}-\eqref{RL} that
	\begin{equation}\label{Rd}
	\begin{split}
	&\|\slashed{\mathcal L}_R^k\big((\slashed d\mathring L^i)\varphi_i\big)\|_{L^2(\Sigma_s^u)}\\
	\lesssim&\delta^{2-2\varepsilon_0}s^{-4}\big(\|R^{\leq k}\check L^i\|_{L^2(\Sigma_s^u)}+\|\slashed{\mathcal L}_R^{\leq k}\slashed dx^i\|_{L^2(\Sigma_s^u)}\big)+s^{-2}\|R^{\leq k+1}\phi\|_{L^2(\Sigma_s^u)}\\
	&+\delta^{1-\varepsilon_0}s^{-5/2}\big(\|R^{\leq k+1}\varphi\|_{L^2(\Sigma_s^u)}+\delta\|R^{\leq k}\mathring L\phi\|_{L^2(\Sigma_s^u)}+\delta\|R^{\leq k}\mathring L\varphi\|_{L^2(\Sigma_s^u)}\big)\\
	&+\delta^{2-\varepsilon_0}s^{-5/2}\|\slashed{\mathcal L}_R^{\leq k}\check\chi\|_{L^2(\Sigma_s^u)}+\delta^{3-2\varepsilon_0} s^{-9/2}\|\slashed{\mathcal L}_R^{\leq k-1}\leftidx{^{(R)}}{\slashed\pi}\|_{L^2(\Sigma_s^u)}\\
	\lesssim&\delta^{5/2-2\varepsilon_0}s^{-5/2}+\delta^{2-2\varepsilon_0}s^{-5/2}\|\slashed{\mathcal L}_R^{\leq k}\check{\chi}\|_{L^2(\Sigma_s^u)}+\delta^{2-\varepsilon_0}s^{-7/2}\|R^{\leq k+1}\mu\|_{L^2(\Sigma_s^u)}\\
	&+\delta^{2-\varepsilon_0} s^{-2-\epsilon}\sqrt{E_{1,\leq k+2}(s,u)}+\delta^{2-\varepsilon_0} s^{-2}\sqrt{E_{2,\leq k+2}(s,u)}.
	\end{split}
	\end{equation}
	In addition, by \eqref{Lx},
	\begin{equation}\label{RdL}
	\begin{split}
	&\|\slashed{\mathcal L}_R^k\big((\slashed dx^i)\mathring L\varphi_i\big)\|_{L^2(\Sigma_s^u)}\\
	\lesssim&\|R^{\leq k}\mathring L\varphi\|_{L^2(\Sigma_s^u)}
+\delta^{1-\varepsilon_0}s^{-7/2}\|R^{\leq k}\upsilon_{jl}\|_{L^2(\Sigma_s^u)}\\
	&+\delta^{2-2\varepsilon_0}s^{-7/2}\big(\|R^{\leq k}\phi\|_{L^2(\Sigma_s^u)}
+\|R^{\leq k}\varphi\|_{L^2(\Sigma_s^u)}+\|R^{\leq k}\check L^i\|_{L^2(\Sigma_s^u)}\big)\\
    \lesssim&\delta^{5/2-2\varepsilon_0}s^{-2}+\delta^{1-\varepsilon_0}s^{-3/2}\|\slashed{\mathcal L}_R^{\leq k}\check{\chi}\|_{L^2(\Sigma_s^u)}+\delta^{2-\varepsilon_0}s^{-7/2}\|R^{\leq k+1}\mu\|_{L^2(\Sigma_s^u)}\\
    &+\delta s^{-1-\epsilon}\sqrt{E_{1,\leq k+2}(s,u)}+\delta s^{-1}\sqrt{E_{2,\leq k+2}(s,u)}.
	\end{split}
	\end{equation}
	Combining \eqref{Rd}-\eqref{RdL} and \eqref{dRLp} yields
	\begin{equation}\label{Rk1L}
	\begin{split}
	\|\slashed dR^k\mathring L\phi\|_{L^2(\Sigma_s^u)}    \lesssim&\delta^{5/2-2\varepsilon_0}s^{-2}+\delta^{1-\varepsilon_0}s^{-3/2}\|\slashed{\mathcal L}_R^{\leq k}\check{\chi}\|_{L^2(\Sigma_s^u)}+\delta^{2-\varepsilon_0}s^{-7/2}\|R^{\leq k+1}\mu\|_{L^2(\Sigma_s^u)}\\
	&+\delta s^{-1-\epsilon}\sqrt{E_{1,\leq k+2}(s,u)}+\delta s^{-1}\sqrt{E_{2,\leq k+2}(s,u)}.
	\end{split}
	\end{equation}
	
	\item Note that
	\begin{equation*}
	\begin{split}
	&|R^k\mathring L^2\phi|=|R^k\big((\mathring L\mathring L^i)\varphi_i+\mathring L^\al\mathring L\varphi_\al\big)|\\
	\lesssim&\sum_{k_1+k_2=k}\Big(|R^{k_1}(\mathring L\mathring L^i)R^{k_2}\varphi_i|
+(|R^{k_1}\check L^i|+s^{-1}|R^{k_1}x^i|)|R^{k_2}\mathring L\varphi_i|\Big).
	\end{split}
	\end{equation*}
	Since $\mathring L\mathring L^i$ satisfies \eqref{LL}, then
	\begin{equation*}
	\begin{split}
	|R^{k_1}(\mathring L\mathring L^i)|\lesssim &|R^{\leq k_1}\mathring L\phi|+|R^{\leq k_1}\mathring L\varphi|+s^{-1}|R^{\leq k_1+1}\phi|+s^{-1}|R^{\leq k_1+1}\varphi|\\
	&+\delta^{1-\varepsilon_0}s^{-5/2}(|R^{\leq k_1}\check L^i|+s^{-1}|R^{\leq k_1+1}x^i|)
	\end{split}
	\end{equation*}
	and it follows from \eqref{Lx} that
	\begin{equation*}
	\begin{split}
	s^{-1}|R^{k_1}x^i|\cdot|R^{k_2}\mathring L\varphi_i|\lesssim&|R^{\leq k}\mathring L\varphi|
+\delta^{1-\varepsilon_0}s^{-7/2}|R^{\leq k}\upsilon_{jl}|\\
	&+\delta^{2-2\varepsilon_0}s^{-7/2}(|R^{\leq k}\phi|+|R^{\leq k}\varphi|+|R^{\leq k}\check L^i|).
	\end{split}
	\end{equation*}
	Thus,
	\begin{equation}\label{RL2p}
	\begin{split}
	&\|R^k\mathring L^2\phi\|_{L^2(\Sigma_s^u)}\\
\lesssim&\|R^{\leq k}\mathring L\varphi\|_{L^2(\Sigma_s^u)}+\delta^{1-\varepsilon_0}s^{-7/2}\|R^{\leq k}\upsilon_{jl}\|_{L^2(\Sigma_s^u)}+\delta^{1-\varepsilon_0}s^{-5/2}\|R^{\leq k}\check L^i\|_{L^2(\Sigma_s^u)}\\
	&+\delta^{1-\varepsilon_0}s^{-5/2}\|R^{\leq k+1}\phi\|_{L^2(\Sigma_s^u)}+\delta^{1-\varepsilon_0}s^{-5/2}\|R^{\leq k+1}\varphi\|_{L^2(\Sigma_s^u)}\\
	&+\delta^{1-\varepsilon_0}s^{-3/2}|R^{\leq k}\mathring L\phi\|_{L^2(\Sigma_s^u)}+\delta^{2-2\varepsilon_0}s^{-5}\|R^{\leq k+1}x^i\|_{L^2(\Sigma_s^u)}\\
	\lesssim&\delta^{5/2-2\varepsilon_0}s^{-2}+\delta^{1-\varepsilon_0}s^{-3/2}\|\slashed{\mathcal L}_R^{\leq k}\check{\chi}\|_{L^2(\Sigma_s^u)}+\delta^{2-\varepsilon_0}s^{-7/2}\|R^{\leq k+1}\mu\|_{L^2(\Sigma_s^u)}\\
	&+\delta s^{-1-\epsilon}\sqrt{E_{1,\leq k+2}(s,u)}+\delta s^{-1}\sqrt{E_{2,\leq k+2}(s,u)}.
	\end{split}
	\end{equation}
	\end{enumerate}

Substituting \eqref{R2p}, \eqref{Rk1L} and \eqref{RL2p} into \eqref{RLc}, and using \eqref{RL2} and \eqref{LRichi}, we have
\begin{equation}\label{LRchi}
\begin{split}
&\|\slashed{\mathcal L}_{\mathring L}\slashed{\mathcal L}_R^k\check\chi\|_{L^2(\Sigma_s^u)}\\
\lesssim&
\|\slashed{\mathcal L}_R^k\slashed{\mathcal L}_{\mathring L}\check\chi\|_{L^2(\Sigma_s^u)}+\delta^{1-\varepsilon_0}s^{-2}(\|\slashed{\mathcal L}_R^{\leq k}\check{\chi}\|_{L^2(\Sigma_s^u)}+t^{-1}\|\slashed{\mathcal L}_R^{\leq k}\leftidx{^{(R)}}{\slashed\pi}_{\mathring L}\|_{L^2(\Sigma_s^u)})\\
&+\delta^{2-2\varepsilon_0}s^{-4}\|\slashed{\mathcal L}_R^{\leq k-1}\slashed g\|_{L^2(\Sigma_s^u)}\\	\lesssim&\delta^{5/2-2\varepsilon_0}s^{-2}+\delta^{1-\varepsilon_0}s^{-3/2}\|\slashed{\mathcal L}_R^{\leq k}\check{\chi}\|_{L^2(\Sigma_s^u)}+\delta^{2-\varepsilon_0}s^{-7/2}\|R^{\leq k+1}\mu\|_{L^2(\Sigma_s^u)}\\
&+s^{-1-\epsilon}\sqrt{E_{1,\leq k+2}(s,u)}+\delta s^{-1}\sqrt{E_{2,\leq k+2}(s,u)}.
\end{split}
\end{equation}

Next we estimate the $L^2-$norm of $\mathring LR^{k+1}\mu$. For any $\bar Z\in\{T,R_{ij}\}$, by virtue
of \eqref{LRimu}, \eqref{lmu} and \eqref{GT}, one can get
\begin{equation}\label{LbZmu}
\begin{split}
&\delta^l\|\mathring L\bar Z^{k+1}\mu\|_{L^2(\Sigma_s^u)}\\
\lesssim&\delta^{1-2\varepsilon_0}s^{-5/2}\big(\delta^{3/2}\delta^{l_1}\|\slashed{\mathcal L}_{\bar Z}^{k_1}\leftidx{^{(R)}}{\slashed\pi}_{\mathring L}\|_{L^2(\Sigma_s^u)}+\delta^{l_0}\|\slashed{\mathcal L}_{\bar Z}^{k_0}\slashed dx^i\|_{L^2(\Sigma_s^u)}+\delta^{\varepsilon_0}\delta^{l_1}\|\bar Z^{k_1}\upsilon_{ij}\|_{L^2(\Sigma_s^u)}\big)\\
&+\delta^{-\varepsilon_0}s^{-3/2}\delta^{l_0}\big(\|\bar Z^{k_0}\check L^i\|_{L^2(\Sigma_s^u)}+\|\bar Z^{ k_0}\varphi\|_{L^2(\Sigma_s^u)}+\|\bar Z^{ k_0}\phi\|_{L^2(\Sigma_s^u)}+\|\bar Z^{l_0}\check\varrho\|_{L^2(\Sigma_s^u)}\big)\\
&+\delta^{1-\varepsilon_0}s^{-2}\delta^{l_0}\big(s\|T\bar Z^{k_0}\varphi\|_{L^2(\Sigma_s^u)}+\|\bar Z^{k_0}\mu\|_{L^2(\Sigma_s^u)}\big)+\delta^{2-3\varepsilon_0}s^{-2}\delta^{l_1}\|\slashed{\mathcal L}_{\bar Z}^{k_1}\slashed g\|_{L^2(\Sigma_s^u)}\\
&+\delta^{l_0}\big(\|\slashed d\bar Z^{k_0}\varphi\|_{L^2(\Sigma_s^u)}+\|\mathring L\bar Z^{k_0}\varphi\|_{L^2(\Sigma_s^u)}+\|\bar Z^{k_0}\mathring L\phi\|_{L^2(\Sigma_s^u)}+\|\bar Z^{k_0}T\phi\|_{L^2(\Sigma_s^u)}\big)\\
&+\delta^{2-2\varepsilon_0}s^{-5/2}\delta^{l_1}\big(\|\slashed{\mathcal L}_{\bar Z}^{k_1}\leftidx{^{(R)}}{\slashed\pi}_T\|_{L^2(\Sigma_s^u)}+s^{3/2}\|\slashed{\mathcal L}_{\bar Z}^{k_1}\leftidx{^{(T)}}{\slashed\pi}_{\mathring L}\|_{L^2(\Sigma_s^u)}\big),
\end{split}
\end{equation}
where $l$ is the number of $T$ in $\bar Z^{k+1}$.
If all $\bar Z$ in \eqref{LbZmu} are chosen from $\{R_{ij}\}$, then it follows from \eqref{RL2}-\eqref{Rkr} and \eqref{Rk1L} that
\begin{equation}\label{LRkmu}
\begin{split}
\|\mathring LR^{k+1}\mu\|_{L^2(\Sigma_s^u)}
\lesssim&\delta^{3/2-2\varepsilon_0}+\delta^{-\varepsilon_0}\|\slashed{\mathcal L}_R^{\leq k}\check{\chi}\|_{L^2(\Sigma_s^u)}+\delta^{1-\varepsilon_0}s^{-3/2}\|R^{\leq k+1}\mu\|_{L^2(\Sigma_s^u)}\\
&+s^{-\epsilon}\sqrt{E_{1,\leq k+2}(s,u)}+\delta^{1-\varepsilon_0}\sqrt{E_{2,\leq k+2}(s,u)}.
\end{split}
\end{equation}

Let $F(s,u,\vartheta)=\varrho(s,u)^2|\slashed{\mathcal L}_{R}^{\le k}\check\chi(s,u,\vartheta)|-\varrho(t_0,u)^2|\slashed{\mathcal L}_{R}^{\leq k}\check\chi(t_0,u,\vartheta)|$ in \eqref{Ff}. Since for any 2-form $\xi$ on $S_{s,u}$, $|\mathring L(\varrho^2|\xi|)|\lesssim\varrho^2|\slashed{\mathcal L}_{\mathring L}\xi|+\delta^{1-\varepsilon_0}s^{-2}(\varrho^2|\xi|)$ holds, then
\begin{equation}\label{Y-15}
\begin{split}
\|\varrho^2\slashed{\mathcal L}_{R}^{\leq k}\check\chi\|_{L^2(\Sigma_s^u)}\lesssim\delta^{3/2-\varepsilon_0}\varrho^{3/2}+\varrho^{3/2}\int_{t_0}^s\tau^{-3/2}\Big(&\|\varrho^2(\tau, u)\slashed{\mathcal L}_{\mathring L}\slashed{\mathcal L}_{R}^{\leq k}\check\chi\|_{L^2(\Sigma_\tau^u)}\\
&+\delta^{1-\varepsilon_0}\tau^{-2}\|\varrho^2\slashed{\mathcal L}_{R}^{\leq k}\check\chi\|_{L^2(\Sigma_\tau^u)}\Big)d\tau.
\end{split}
\end{equation}

Applying \eqref{LRchi} and Grownwall's equality for \eqref{Y-15} yields
\begin{equation}\label{tchi}
\begin{split}
s^{1/2}\|\slashed{\mathcal L}_{R}^k\check\chi\|_{L^2(\Sigma_s^u)}\lesssim&\delta^{3/2-\varepsilon_0}
+\delta^{2-\varepsilon_0}\int_{t_0}^s\tau^{-3}\|R^{[1,k+1]}\mu\|_{L^2(\Sigma_\tau^u)}d\tau\\
&+\int_{t_0}^s\tau^{-1/2-\epsilon}\sqrt{E_{1,\leq k+2}(\tau,u)}d\tau+\delta\int_{t_0}^s \tau^{-1/2}\sqrt{E_{2,\leq k+2}(\tau,u)}d\tau.
\end{split}
\end{equation}

Analogously, \eqref{LRkmu} gives that
\begin{equation}\label{tRmu}
\begin{split}
s^{-3/2}\|R^{[1,k+1]}\mu&\|_{L^2(\Sigma_s^u)}\lesssim\delta^{3/2-2\varepsilon_0}
+\delta^{-\varepsilon_0}\int_{t_0}^s\tau^{-3/2}\|\slashed{\mathcal L}_R^{\leq k}\check\chi\|_{L^2(\Sigma_\tau^u)}d\tau\\
&+\int_{t_0}^s\tau^{-3/2-\epsilon}\sqrt{E_{1,\leq k+2}(\tau,u)}d\tau+\delta^{1-\varepsilon_0}\int_{t_0}^s \tau^{-3/2}\sqrt{E_{2,\leq k+2}(\tau,u)}d\tau.
\end{split}
\end{equation}
Let
\begin{equation*}
\begin{split}
G_1(s,u)=&\delta^{3/2-\varepsilon_0}+\delta^{2-\varepsilon_0}\int_{t_0}^s\tau^{-3}\|R^{[1,k+1]}\mu\|_{L^2(\Sigma_\tau^u)}d\tau\\
&+\int_{t_0}^s\tau^{-1/2-\epsilon}\sqrt{E_{1,\leq k+2}(\tau,u)}d\tau+\delta\int_{t_0}^s \tau^{-1/2}\sqrt{E_{2,\leq k+2}(\tau,u)}d\tau
\end{split}
\end{equation*}
and
\begin{equation*}
\begin{split}
G_2(s,u)=&\delta^{3/2-2\varepsilon_0}+\delta^{-\varepsilon_0}\int_{t_0}^s\tau^{-3/2}\|\slashed{\mathcal L}_R^{\leq k}\check\chi\|_{L^2(\Sigma_\tau^u)}d\tau\\
&+\int_{t_0}^s\tau^{-3/2-\epsilon}\sqrt{E_{1,\leq k+2}(\tau,u)}d\tau+\delta^{1-\varepsilon_0}\int_{t_0}^s \tau^{-3/2}\sqrt{E_{2,\leq k+2}(\tau,u)}d\tau.
\end{split}
\end{equation*}
Then
\begin{align*}
&\p_sG_1\lesssim\delta^{2-\varepsilon_0}s^{-3/2}G_2+s^{-1/2-\epsilon}\sqrt{E_{1,\leq k+2}}+\delta s^{-1/2}\sqrt{E_{2,\leq k+2}},\\
&\p_sG_2\lesssim\delta^{-\varepsilon_0}s^{-2}G_1+s^{-3/2-\epsilon}\sqrt{E_{1,\leq k+2}}
+\delta^{1-\varepsilon_0} s^{-3/2}\sqrt{E_{2,\leq k+2}}.
\end{align*}
This, together with $s^{1/2}\|\slashed{\mathcal L}_{R}^k\check\chi\|_{L^2(\Sigma_s^u)}\lesssim G_1$
and $s^{-3/2}\|R^{[1,k+1]}\mu\|_{L^2(\Sigma_s^u)}\lesssim G_2$ by \eqref{tchi} and \eqref{tRmu}, yields
\begin{equation*}
\p_s(s^{-1/4}G_1+\delta G_2)\lesssim\delta^{1-\varepsilon_0}s^{-7/4}(s^{-1/4}G_1+\delta G_2)
+s^{-3/4-\epsilon}\sqrt{E_{1,\leq k+2}}+\delta s^{-3/4}\sqrt{E_{2,\leq k+2}},
\end{equation*}
which gives
\begin{equation*}
s^{-1/4}G_1+\delta G_2\lesssim\delta^{3/2-\varepsilon_0}+s^{1/4}\sqrt{\tilde E_{1,\leq k+2}}
+\delta s^{1/4}\sqrt{\tilde E_{2,\leq k+2}}.
\end{equation*}
Hence,
\begin{align}
&\|\slashed{\mathcal L}_{R}^k\check\chi\|_{L^2(\Sigma_s^u)}\lesssim s^{-1/2}G_1\lesssim \delta^{3/2-\varepsilon_0}s^{-1/4}
+\sqrt{\tilde E_{1,\leq k+2}}+\delta\sqrt{\tilde E_{2,\leq k+2}},\label{LRc}\\
&\|R^{[1,k+1]}\mu\|_{L^2(\Sigma_s^u)}\lesssim s^{3/2}G_2\lesssim \delta^{1/2-\varepsilon_0}s^{3/2}
+\delta^{-1}s^{7/4}\sqrt{\tilde E_{1,\leq k+2}}+s^{7/4}\sqrt{\tilde E_{2,\leq k+2}}.\label{Rkmu}
\end{align}
Substituting \eqref{Rkmu} and \eqref{LRc} into \eqref{tchi} and \eqref{tRmu}, respectively, one has
\begin{align}
&\|\slashed{\mathcal L}_{R}^k\check\chi\|_{L^2(\Sigma_s^u)}\lesssim \delta^{3/2-\varepsilon_0}s^{-1/2}
+s^{-\epsilon}\sqrt{\tilde E_{1,\leq k+2}(s,u)}+\delta\sqrt{\tilde E_{2,\leq k+2}(s,u)},\label{LRch}\\
&\|R^{[1,k+1]}\mu\|_{L^2(\Sigma_s^u)}\lesssim \delta^{3/2-2\varepsilon_0}s^{3/2}
+\delta^{-\varepsilon_0}s^{3/2}\sqrt{\tilde E_{1,\leq k+2}(s,u)}
+\delta^{1-\varepsilon_0}s^{3/2}\sqrt{\tilde E_{2,\leq k+2}(s,u)}.\label{R1kmu}
\end{align}

\item If $\slashed{\mathcal L}_Z^k\check\chi=\slashed{\mathcal L}_Z^{p_1}\slashed{\mathcal L}_T\slashed{\mathcal L}_R^{p_2}\check\chi$ with $p_1+p_2=k-1$, then $\slashed{\mathcal L}_Z^k\check\chi=\slashed{\mathcal L}_Z^{p_1}[\slashed{\mathcal L}_T,\slashed{\mathcal L}_R^{p_2}]\check\chi+\slashed{\mathcal L}_Z^{p_1}\slashed{\mathcal L}_R^{p_2}\slashed{\mathcal L}_T\check\chi$, and hence
    it follows from Lemma \ref{commute}, \eqref{Tchi'} and \eqref{Zk1L}-\eqref{zp} that
\begin{equation}\label{TRchi}
\begin{split}
&\delta^l\|\slashed{\mathcal L}_Z^k\check\chi\|_{L^2(\Sigma_s^u)}\\
\lesssim&\delta s^{-2}\delta^{l_0}\|Z^{k_0}\mu\|_{L^2(\Sigma_s^u)}+\delta^{2-2\varepsilon_0}s^{-2}\big(\delta^{l_2}\|\slashed{\mathcal L}_Z^{k_2}\leftidx{^{(R)}}{\slashed\pi}\|_{L^2(\Sigma_s^u)}+\delta\delta^{l_2}\|\slashed{\mathcal L}_Z^{k_2}\leftidx{^{(T)}}{\slashed\pi}\|_{L^2(\Sigma_s^u)}\big)\\
&+\delta^{1-\varepsilon_0}s^{-5/2}\big(\delta^{l_1}\|Z^{k_1}\check L^i\|_{L^2(\Sigma_s^u)}+s^{-1}\delta^{l_0}\|Z^{k_0}x^i\|_{L^2(\Sigma_s^u)}\big)+s^{-1}\delta^{l_0}\|Z^{k_0}\phi\|_{L^2(\Sigma_s^u)}\\
&+s^{-1}\delta^{l_0}\|Z^{k_0}\varphi\|_{L^2(\Sigma_s^u)}+\delta^{1-\varepsilon_0}s^{-1}\delta^{l_2}\|\slashed{\mathcal L}_Z^{k_2}\check\chi\|_{L^2(\Sigma_s^u)}+\delta^{1-\varepsilon_0}s^{-2}\delta^{l_2}\|\slashed{\mathcal L}_Z^{k_2}\slashed g\|_{L^2(\Sigma_s^u)}\\
&+\delta^{2-\varepsilon_0}s^{-3}\delta^{l_2}\|\slashed{\mathcal L}_Z^{k_2}\leftidx{^{(R)}}{\slashed\pi}_T\|_{L^2(\Sigma_s^u)}\\
\lesssim&\delta^{3/2-\varepsilon_0}s^{-1/2}+\delta^{1-\varepsilon_0}s^{-1}\delta^{l_1}\|\slashed{\mathcal L}_Z^{k_1}\check\chi\|_{L^2(\Sigma_s^u)}+\delta s^{-2}\delta^{l_0}\|Z^{k_0}\mu\|_{L^2(\Sigma_s^u)}\\
&+\delta s^{-1-\epsilon}\sqrt{E_{1,\leq k+2}(s,u)}+\delta t^{-1}\sqrt{E_{2,\leq k+2}(s,u)}.
\end{split}
\end{equation}

If $\slashed{\mathcal L}_Z^k\check\chi=\slashed{\mathcal L}_Z^{p_1}\slashed{\mathcal L}_{\varrho\mathring L}\slashed{\mathcal L}_R^{p_2}\check\chi$ with $p_1+p_2=k-1$, then $\slashed{\mathcal L}_Z^k\check\chi=\slashed{\mathcal L}_Z^{p_1}[\slashed{\mathcal L}_{\varrho\mathring L},\slashed{\mathcal L}_R^{p_2}]\check\chi+\slashed{\mathcal L}_Z^{p_1}\slashed{\mathcal L}_R^{p_2}(\varrho\slashed{\mathcal L}_{\mathring L}\check\chi)$, and it follows from \eqref{Lchi'} that
\begin{equation}\label{rLRchi}
\begin{split}
&\delta^l\|\slashed{\mathcal L}_Z^k\check\chi\|_{L^2(\Sigma_s^u)}\\
&\lesssim\delta^{1-\varepsilon_0}s^{-2}\big(\delta^{l_2}\|\slashed{\mathcal L}_Z^{k_2}\leftidx{^{(R)}}{\slashed\pi}_{\mathring L}\|_{L^2(\Sigma_s^u)}+s\delta^{l_2}\|\slashed{\mathcal L}_Z^{k_2}\check\chi\|_{L^2(\Sigma_s^u)}+s^{-1/2}\delta^{l_1}\|\slashed{\mathcal L}_Z^{k_1}\slashed g\|_{L^2(\Sigma_s^u)}\big)\\
&+\delta^{1-\varepsilon_0}s^{-7/2}\big(s\delta^{l_2}\|Z^{k_2}\check L^i\|_{L^2(\Sigma_s^u)}+\delta^{l_1}\|Z^{k_1}x^i\|_{L^2(\Sigma_s^u)}\big)+s^{-1}\delta^{l_0}\|Z^{k_0}\phi\|_{L^2(\Sigma_s^u)}\\
&+s^{-1}\delta^{l_0}\|Z^{k_0}\varphi\|_{L^2(\Sigma_s^u)}\\
\lesssim&\delta^{3/2-\varepsilon_0}s^{-1}+\delta^{1-\varepsilon_0}s^{-1}\delta^{l_1}\|\slashed{\mathcal L}_Z^{k_1}\check\chi\|_{L^2(\Sigma_s^u)}+\delta^{2-\varepsilon_0} s^{-5/2}\delta^{l_0}\|Z^{k_0}\mu\|_{L^2(\Sigma_s^u)}\\
&+\delta s^{-1-\epsilon}\sqrt{E_{1,\leq k+2}(s,u)}+\delta s^{-1}\sqrt{E_{2,\leq k+2}(s,u)}.
\end{split}
\end{equation}

Therefore, for any $Z\in\{\varrho\mathring L,T,R_{ij}\}$, by \eqref{LRch}, \eqref{TRchi} and \eqref{rLRchi}, we obtain
\begin{equation}\label{zchi}
\begin{split}
\delta^l\|\slashed{\mathcal L}_Z^k\check\chi\|_{L^2(\Sigma_s^u)}\lesssim& \delta^{3/2-\varepsilon_0}s^{-1/2}+\delta s^{-2}\delta^{l_0}\|Z^{k_0}\mu\|_{L^2(\Sigma_s^u)}\\
&+s^{-\epsilon}\sqrt{\tilde E_{1,\leq k+2}(s,u)}+\delta\sqrt{\tilde E_{2,\leq k+2}(s,u)}.
\end{split}
\end{equation}

\item If $Z^{k+1}\mu=\bar Z^{k+1}\mu$ with $\bar Z\in\{T, R_{ij}\}$,
then substituting the corresponding estimates in \eqref{Zk1L}-\eqref{Zkr} into \eqref{LbZmu} and applying \eqref{zchi} show that
\begin{equation*}
\begin{split}
\delta^l\|\mathring L\bar Z^{k+1}\mu\|_{L^2(\Sigma_s^u)}\lesssim&\delta^{3/2-2\varepsilon_0}
+\delta^{1-\varepsilon_0} s^{-3/2}\delta^{l_0}\|Z^{k_0}\mu\|_{L^2(\Sigma_s^u)}\\
&+\delta^{-\varepsilon_0}s^{-\epsilon}\sqrt{E_{1,\leq k+2}(s,u)}+\delta^{1-\varepsilon_0}\sqrt{E_{2,\leq k+2}(s,u)}.
\end{split}
\end{equation*}
Taking $F(s,u,\vartheta)=\delta^l\bar Z^{k+1}\mu(s,u,\vartheta)-\delta^l\bar Z^{k+1}\mu(t_0,u,\vartheta)$ in \eqref{Ff} yields
\begin{equation}\label{bZmu}
\begin{split}
\delta^l\|\bar Z^{k+1}\mu\|_{L^2(\Sigma_s^u)}\lesssim&\delta^{3/2-2\varepsilon_0}s^{3/2}
+\delta^{1-\varepsilon_0}s^{3/2}\int_{t_0}^s\tau^{-3}\delta^{l_0}\|Z^{k_0}\mu\|_{L^2(\Sigma_\tau^u)}d\tau\\
&+\delta^{-\varepsilon_0}s^{3/2}\sqrt{\tilde E_{1,\leq k+2}(s,u)}
+\delta^{1-\varepsilon_0}s^{3/2}\sqrt{\tilde E_{2,\leq k+2}(s,u)}.
\end{split}
\end{equation}

If $Z^{k+1}\mu=Z^{p_1}(\varrho\mathring L)\bar Z^{p_2}\mu$ with $p_1+p_2=k$, it follows from \eqref{lmu}, \eqref{GT}, \eqref{Zk1L}-\eqref{zp} and \eqref{zchi} that
\begin{equation}\label{ZLmu}
\begin{split}
&\delta^l\|Z^{k+1}\mu\|_{L^2(\Sigma_s^u)}=\delta^l\|Z^{p_1}[\varrho\mathring L,\bar Z^{p_2}]\mu+Z^{p_1}\bar Z^{p_2}(\varrho\mathring L)\mu\|_{L^2(\Sigma_s^u)}\\
\lesssim&\delta^{1-\varepsilon_0}s^{-1}\delta^{l_1}\|Z^{k_1}\mu\|_{L^2(\Sigma_s^u)}+\delta^{-\varepsilon_0}s^{-1/2}\delta^{l_1}\|Z^{k_1}\check L^i\|_{L^2(\Sigma_s^u)}+\delta^{-\varepsilon_0}\delta^{l_0}\|Z^{k_0}\phi\|_{L^2(\Sigma_s^u)}\\
&+\delta^{1-2\varepsilon_0}s^{-3/2}\delta^{l_0}\|Z^{k_0}x^i\|_{L^2(\Sigma_s^u)}
+\delta^{-\varepsilon_0}\delta^{l_0}\|Z^{k_0}\varphi\|_{L^2(\Sigma_s^u)}+s\delta^{l_1}\|Z^{k_1}T\phi\|_{L^2(\Sigma_s^u)}\\
&+\delta^{1-2\varepsilon_0}\|\slashed{\mathcal L}_Z^{k_2}\leftidx{^{(R)}}{\slashed\pi}_{\mathring L}\|_{L^2(\Sigma_s^u)}+\delta^{2-2\varepsilon_0}\|\slashed{\mathcal L}_Z^{k_2}\leftidx{^{(T)}}{\slashed\pi}_{\mathring L}\|_{L^2(\Sigma_s^u)}+\delta^{2-3\varepsilon_0}s^{-1}\delta^{l_2}\|\slashed{\mathcal L}_Z^{k_2}\slashed g\|_{L^2(\Sigma_s^u)}\\
\lesssim&\delta^{3/2-2\varepsilon_0}s+\delta^{1-\varepsilon_0} s^{-1/2}\delta^{l_0}\|Z^{k_0}\mu\|_{L^2(\Sigma_s^u)}+\delta^{-\varepsilon_0} s^{1-\epsilon}\sqrt{\tilde E_{1,\leq k+2}(s,u)}\\
&+\delta^{1-\varepsilon_0}s\sqrt{\tilde E_{2,\leq k+2}(s,u)}.
\end{split}
\end{equation}
\end{itemize}
Thus, for any $Z\in\{\varrho\mathring L,T,R_{ij}\}$, \eqref{bZmu} and \eqref{ZLmu} yield
\begin{equation}\label{YH-1}
\begin{split}
&\delta^l\|Z^{k+1}\mu\|_{L^2(\Sigma_s^u)}\\
\lesssim&\delta^{3/2-2\varepsilon_0}s^{3/2}+\delta^{1-\varepsilon_0} s^{-1/2}\delta^{l_0}\|Z^{k_0}\mu\|_{L^2(\Sigma_s^u)}
+\delta^{1-\varepsilon_0}s^{3/2}\int_{t_0}^s\tau^{-3}\delta^{l_0}\|\bar Z^{k_0}\mu\|_{L^2(\Sigma_\tau^u)}d\tau\\
&+\delta^{-\varepsilon_0} s^{3/2}\sqrt{\tilde E_{1,\leq k+2}(s,u)}
+\delta^{1-\varepsilon_0}s^{3/2}\sqrt{\tilde E_{2,\leq k+2}(s,u)}.
\end{split}
\end{equation}
Applying Gronwall's inequality for $s^{-3/2}\delta^l\|Z^{k+1}\mu\|_{L^2(\Sigma_s^u)}$, one can get
\begin{equation}\label{Fmu}
\begin{split}
&\delta^l\|Z^{k+1}\mu\|_{L^2(\Sigma_s^u)}\\
\lesssim&\delta^{3/2-2\varepsilon_0}s^{3/2}+\delta^{-\varepsilon_0} s^{3/2}\sqrt{\tilde E_{1,\leq k+2}(s,u)}+\delta^{1-\varepsilon_0}s^{3/2}\sqrt{\tilde E_{2,\leq k+2}(s,u)}.
\end{split}
\end{equation}
This, together with \eqref{zchi}, gives
\begin{equation}\label{Fchi}
\begin{split}
\delta^l\|\slashed{\mathcal L}_Z^k\check\chi\|_{L^2(\Sigma_s^u)}\lesssim&
\delta^{3/2-\varepsilon_0}s^{-1/2}+s^{-\epsilon}\sqrt{\tilde E_{1,\leq k+2}(s,u)}+\delta\sqrt{\tilde E_{2,\leq k+2}(s,u)}.
\end{split}
\end{equation}

Substituting \eqref{Fmu}-\eqref{Fchi} into \eqref{Zk1L}-\eqref{ZkRp} and \eqref{Zkr} eventually
completes the proof of Proposition 9.1.
\end{proof}

On the other hand, the proof of Proposition 9.1 also gives the following energy estimates for $\phi$.
\begin{corollary}\label{Zphi}
	Under the assumption $(\star)$ with suitably small $\delta>0$, it holds that for $k\leq 2N-6$,
	\begin{align*}
	&\delta^l\|Z^{k+1}\phi\|_{L^2(\Sigma_s^u)}\lesssim\delta^{5/2-\varepsilon_0}+\delta^{2-2\varepsilon_0}\sqrt{\tilde E_{1,\leq k+2}(s,u)}+\delta^2\sqrt{\tilde E_{2,\leq k+2}(s,u)},\\
	&\delta^l\|Z^{k+1}T\phi\|_{L^2(\Sigma_s^u)}\lesssim\delta^{3/2-\varepsilon_0}+\delta^{1-2\varepsilon_0}\sqrt{\tilde E_{1,\leq k+2}(s,u)}+\delta\sqrt{\tilde E_{2,\leq k+2}(s,u)},\\
	&\delta^l\|Z^{k+1}\mathring L\phi\|_{L^2(\Sigma_s^u)}\lesssim\delta^{3/2-\varepsilon_0}
+\delta^{1-\varepsilon_0}s^{-\epsilon}\sqrt{\tilde E_{1,\leq k+2}(s,u)}+\delta\sqrt{\tilde E_{2,\leq k+2}(s,u)},\\
	&\delta^l\|Z^{k+1}R\phi\|_{L^2(\Sigma_s^u)}\lesssim\delta^{3/2-\varepsilon_0}s
+\delta^{1-\varepsilon_0}s^{1-\epsilon}\sqrt{\tilde E_{1,\leq k+2}(s,u)}+\delta s\sqrt{\tilde E_{2,\leq k+2}(s,u)}.
	\end{align*}
\end{corollary}
\begin{proof}
	The estimates for $Z^{k+1}\phi$, $Z^{k+1}T\phi$ and $Z^{k+1}\mathring L\phi$ can be directly
got by \eqref{zp}, \eqref{YH-1}, \eqref{Fmu} and \eqref{Fchi}. In addition, due to $R\phi=Rx^i\cdot\varphi_i$, then
	\begin{equation*}
	\delta^l\|Z^{k+1}R\phi\|_{L^2(\Sigma_s^u)}\lesssim\delta^{1-\varepsilon_0}s^{-3/2}\delta^{l_0}\|Z^{k_0}Rx^i\|_{L^2(\Sigma_s^u)}
+s\delta^{l_0}\|Z^{k_0}\varphi\|_{L^2(\Sigma_s^u)},
	\end{equation*}
	which gives the $L^2$ control of $Z^{k+1}R\phi$ with the help of Proposition \ref{L2chi}.
\end{proof}

\section{$L^2$ estimates for  the highest order  derivatives of $\slashed\nabla\chi$ and $\slashed\nabla^2\mu$}\label{L2chimu}

It is clear from the energy estimate \eqref{e} and the expression \eqref{Psi} that the highest order derivatives of $\varphi$
in the energy estimate for \eqref{Psi} are of $2N-4$. On the other hand, by the expression of $\leftidx{^{(Z)}}D_{\gamma,2}^k$
 in \eqref{D2}, the highest order derivatives of $\chi$ and $\mu$ in \eqref{e} are of $2N-5$ and $2N-4$ respectively. However,
 it follows from  Proposition \ref{L2chi} that the $L^2$ estimates on the $2N-5$ order derivatives of ${\chi}$
and the $2N-4$ order derivatives of
$\mu$ should be controlled by the energies of $\vp$ with the orders $2N-3$,
which goes beyond the range of energies containing up to $2N-4$ order derivatives of $\vp$.
To overcome this difficulty, as in \cite{C2} and \cite{J}, one needs to
treat $\textrm{tr}{\chi}$ and $\slashed\triangle\mu$ with the corresponding highest order derivatives.

Note that once the estimate of $\slashed\triangle\mu$ is established, one can use the elliptic estimate to treat $\slashed\nabla^2\mu$.
We start with the estimate on the Riemann curvature $\slashed{\mathcal R}$ of $\slashed g$.

\begin{lemma}\label{Gc}
	Under  the assumptions $(\star)$ with suitably small $\dl>0$, the Riemann curvature $\slashed{\mathcal R}$ of $\slashed g$ in $D^{s,u}$
	satisfies
	\begin{equation}\label{Ga}
	|\slashed{\mathcal R}|\lesssim s^{-2}.
	\end{equation}
\end{lemma}
\begin{proof}
	Let $\slashed\Gamma_{AB}^C$ be the Christoffel symbols of $\slashed g$, that is,
	\[
	\slashed\Gamma_{AB}^C=\f12\slashed g^{CD}\big\{X_A(\slashed g_{DB})+X_B(\slashed g_{AD})-X_D(\slashed g_{AB})\big\}.
	\]
	Take the derivative of $\slashed\Gamma_{AB}^C$ with respect to $s$ to get
	\begin{equation}\label{LGamma}
	\mathring L\slashed\Gamma_{AB}^C=\slashed g^{CD}(\slashed\nabla_A\check\chi_{BD}+\slashed\nabla_B\check\chi_{AD}-\slashed\nabla_D\check\chi_{AB}).
	\end{equation}
	As $|\slashed \Gamma|^2=\slashed g^{AA'}\slashed g^{BB'}\slashed g_{CC'}\slashed\Gamma_{AB}^C\slashed\Gamma_{A'B'}^{C'}$, thus
	\begin{equation*}
	\begin{split}
	\mathring L(\varrho^2|\slashed\Gamma|^2)=&\varrho^2\big(\mathring L(|\slashed\Gamma|^2)+\f2\varrho|\slashed\Gamma|^2\big)\\
	=&\varrho^2\big\{-4\check\chi^{AA'}\slashed g^{BB'}\slashed g_{CC'}\slashed\Gamma_{AB}^C\slashed\Gamma_{A'B'}^{C'}+2\slashed g^{AA'}\slashed g^{BB'}\check\chi_{CC'}\slashed\Gamma_{AB}^C\slashed\Gamma_{A'B'}^{C'}\\
	&\qquad+2\slashed g^{AA'}\slashed g^{BB'}\slashed g_{CC'}\slashed g^{CD}(\slashed\nabla_A\check\chi_{BD}+\slashed\nabla_B\check\chi_{AD}-\slashed\nabla_D\check\chi_{AB})\slashed\Gamma_{A'B'}^{C'}\big\},
	\end{split}
	\end{equation*}
	which implies
	\begin{equation*}
	|\mathring L(\varrho|\slashed\Gamma|)|\lesssim\delta^{1-\varepsilon_0}s^{-2}(\varrho|\slashed\Gamma|)+\delta^{1-\varepsilon_0}s^{-2}.
	\end{equation*}
	Hence, it holds that
	\begin{equation}\label{Gamma}
	|\slashed\Gamma|\lesssim s^{-1}.
	\end{equation}
	Similarly, one has
	\begin{equation*}
	\begin{split}
	\mathring L(\slashed\nabla_B\slashed\Gamma_{ADC})=&[\mathring L,\slashed\nabla_B]\slashed\Gamma_{ADC}+\slashed\nabla_B(\mathring L\slashed\Gamma_{ADC})\\
	=&-2(\check{\slashed\nabla}_B\check\chi_A^E)\slashed\Gamma_{EDC}-2(\check{\slashed\nabla}_B\check\chi_D^E)\slashed\Gamma_{AEC}
-2(\check{\slashed\nabla}_B\check\chi_C^E)\slashed\Gamma_{ADE}+2(\slashed\nabla_B\check\chi_{DE})\slashed\Gamma_{AC}^E\\
	&+2\check\chi_{DE}(\slashed\nabla_B\Gamma_{AC}^E)+2\varrho^{-1}\slashed\nabla_B\slashed\Gamma_{ADC}
+2\slashed\nabla_B\check{\slashed\nabla}_A\check\chi_{CD}
	\end{split}
	\end{equation*}
	due to Lemma \ref{commute} and \eqref{LGamma}, where $\slashed\Gamma_{ADC}:=\slashed g_{DE}\slashed\Gamma_{AC}^E$. Thus, \eqref{Gamma} implies that
	\begin{equation*}
	|\mathring L(\varrho^2|\slashed\nabla\Gamma|)|\lesssim\delta^{1-\varepsilon_0}s^{-2}(\varrho^2|\slashed\nabla\Gamma|)+\delta^{1-\varepsilon_0}s^{-2},
	\end{equation*}
	which gives
	\begin{equation}\label{dGamma}
	|\slashed\nabla\Gamma|\lesssim s^{-2}.
	\end{equation}
	In addition, according to the definition of Riemann curvature, one has
	\begin{equation}\label{Rie}
	\slashed{\mathcal R}_{ABCD}=X_B\slashed\Gamma_{CDA}-X_A\slashed\Gamma_{BDC}-\slashed\Gamma_{CA}^E\slashed\Gamma_{BED}+\slashed\Gamma_{BC}^E\slashed\Gamma_{AED}.
	\end{equation}
	Due to $X_A\slashed\Gamma_{BDC}=\slashed\nabla_A\slashed\Gamma_{BDC}+\slashed\Gamma_{AB}^E\slashed\Gamma_{EDC}
+\slashed\Gamma_{AD}^E\slashed\Gamma_{BEC}+\slashed\Gamma_{AC}^E\slashed\Gamma_{BDE}$, then by \eqref{Rie}, one has
	\begin{equation}\label{Riem}
	\slashed{\mathcal R}_{ABCD}=\slashed\nabla_B\slashed\Gamma_{ADC}-\slashed\nabla_A\Gamma_{BDC}+\slashed\Gamma_{BC}^E\slashed\Gamma_{ADE}
-\slashed\Gamma_{AC}^E\slashed\Gamma_{BDE}.
	\end{equation}
Substituting the estimates \eqref{Gamma} and \eqref{dGamma} into \eqref{Riem} yields the desired result \eqref{Ga}.
\end{proof}

Now, the following elliptic estimate will be derived.

\begin{proposition}\label{elliptic}
	For any function $f\in C^3(D^{s, u})$, it holds that
	\begin{equation}\label{ee}
	\int_{S_{s, u}}|\slashed\nabla^2f|^2d\nu_{\slashed g}
	\lesssim\int_{S_{s, u}}(\slashed\triangle f)^2d\nu_{\slashed g}+s^{-2}\int_{S_{s, u}}|\slashed d f|^2d\nu_{\slashed g}.
	\end{equation}
\end{proposition}
\begin{proof}
		
	For any function $f\in C^3(D^{s, u})$, the Ricci identity gives $\slashed\nabla_A\slashed\triangle f=\slashed\nabla^B\slashed\nabla_{AB}^2f-\mathcal R_{AB}\slashed d^Bf$, where $\mathcal R_{AB}=\slashed g^{CD}\slashed{\mathcal R}_{ACBD}$ is the Ricci curvature tensor. Therefore,
	\begin{equation}\label{n2f}
	|\slashed\nabla^2f|^2=\slashed\nabla^A(\slashed\nabla^Bf\slashed\nabla_{AB}^2f)-\slashed\nabla_B(\slashed d^Bf\slashed\triangle f)+(\slashed\triangle f)^2-\mathcal R_{AB}\slashed d^Af\slashed d^Bf.
	\end{equation}
	Integrating \eqref{n2f} over $S_{s,u}$ yields
	\begin{equation*}
	\int_{S_{s, u}}|\slashed\nabla^2f|^2d\nu_{\slashed g}=\int_{S_{s, u}}|(\slashed\triangle f)^2d\nu_{\slashed g}-\int_{S_{s, u}}\mathcal R_{AB}\slashed d^Af\slashed d^Bfd\nu_{\slashed g},
	\end{equation*}
	which implies \eqref{ee} due to \eqref{Ga}.
\end{proof}

\subsection{Estimates on derivatives of $\slashed d\textrm{tr}\chi$ and $\slashed\nabla\chi$}\label{trchi}
If there is at least one vector field $\varrho\mathring L$ in $Z^k$, that is, $\slashed{\mathcal L}_Z^k=\slashed{\mathcal L}_Z^{p_1}\slashed{\mathcal L}_{\varrho\mathring L}\slashed{\mathcal L}_{\bar Z}^{p_2}$ with $p_1+p_2=k-1$ and $\bar Z\in\{T, R_{ij}\}$, then
it follows from Proposition \ref{L2chi} and Corollary \ref{Zphi} that
\begin{equation}\label{nz}
\begin{split}
&\delta^l\|\slashed\nabla\slashed{\mathcal L}_{Z}^k\check\chi\|_{L^2(\Sigma_s^{u})}\\
=&\delta^l\|\slashed\nabla\slashed{\mathcal L}_Z^{p_1}[\slashed{\mathcal L}_{\varrho\mathring L},\slashed{\mathcal L}_{\bar Z}^{p_2}]\check\chi+\slashed\nabla\slashed{\mathcal L}_Z^{p_1}\slashed{\mathcal L}_{\bar Z}^{p_2}\slashed{\mathcal L}_{\varrho\mathring L}\check\chi\|_{L^2(\Sigma_s^{u})}\\
\lesssim&s^{-2}\delta^{l_1}\|\slashed{\mathcal L}_Z^{k_1}\check\chi\|_{L^2(\Sigma_s^{u})}+\delta^{1-\varepsilon_0}s^{-3}\delta^{l_1}(\|\slashed{\mathcal L}_Z^{k_1}\leftidx^{{(R)}}\slashed\pi_{\mathring L}\|_{L^2(\Sigma_s^{u})}+\delta\|\slashed{\mathcal L}_Z^{k_1}\leftidx^{{(T)}}\slashed\pi_{\mathring L}\|_{L^2(\Sigma_s^{u})})\\
&+\delta^{1-\varepsilon_0}s^{-7/2}\delta^{l_1}(\|\slashed{\mathcal L}_Z^{k_1}\leftidx^{{(R)}}\slashed\pi\|_{L^2(\Sigma_s^{u})}+\delta\|\slashed{\mathcal L}_Z^{k_1}\leftidx^{{(T)}}\slashed\pi\|_{L^2(\Sigma_s^{u})}+\|Z^{k_1}\check L^i\|_{L^2(\Sigma_s^{u})})\\
&+\delta^{1-\varepsilon_0}s^{-9/2}\delta^{l_0}\|Z^{k_0}x^i\|_{L^2(\Sigma_s^{u})}
+s^{-2}\delta^{l_0}(\|Z^{k_0}\phi\|_{L^2(\Sigma_s^{u})}+\|Z^{k_0}\varphi\|_{L^2(\Sigma_s^{u})})\\
&+\delta^{1-\varepsilon_0}s^{-7/2}\delta^{l_1}\|\slashed{\mathcal L}_Z^{k_1}\slashed g\|_{L^2(\Sigma_s^{u})}
+s^{-1}\delta^{l_0}(\|Z^{k_0}\mathring L\phi\|_{L^2(\Sigma_s^{u})}+s^{-1}\|Z^{k_0}R\phi\|_{L^2(\Sigma_s^{u})})\\
&+\delta s^{-2}\|Z^{k_0}T\phi\|_{L^2(\Sigma_s^{u})}+s^{-1}\delta^{l_0}\|\slashed dZ^{k_0}\varphi\|_{L^2(\Sigma_s^{u})}+s^{-1}\delta^{l_0}\|\mathring LZ^{k_0}\varphi\|_{L^2(\Sigma_s^{u})}\\
\lesssim&\delta^{3/2-\varepsilon_0}s^{-1}+s^{-1-\epsilon}\sqrt{\tilde E_{1,\leq k+2}(s,u)}+\delta s^{-1}\sqrt{\tilde E_{2,\leq k+2}(s,u)}.
\end{split}
\end{equation}
Thus,
\begin{equation}\label{dch}
\begin{split}
&\delta^l\|\slashed dZ^k\text{tr}\chi\|_{L^2(\Sigma_s^{u})}\\
\lesssim&\delta^l\|\slashed\nabla\slashed{\mathcal L}_{Z}^k\check\chi\|_{L^2(\Sigma_s^{u})}+s^{-1}\delta^{l_1}\|\slashed{\mathcal L}_Z^{k_1}\check\chi\|_{L^2(\Sigma_s^{u})}+\delta^{1-\varepsilon_0}s^{-3}\delta^{l_0}\|\slashed {\mathcal L}_Z^{k_0}\slashed g\|_{L^2(\Sigma_s^{u})}\\
\lesssim&\delta^{3/2-\varepsilon_0}s^{-1}+s^{-1-\epsilon}\sqrt{\tilde E_{1,\leq k+2}(s,u)}+\delta s^{-1}\sqrt{\tilde E_{2,\leq k+2}(s,u)}.
\end{split}
\end{equation}

When all vectorfields $Z$ are in $\{T, R_{ij}\}$, by the expression \eqref{Lchi}, one can get the transport equation for $tr\chi$ as follows
\begin{equation}\label{Ltrchi}
\begin{split}
\mathring L(\textrm{tr}{\chi})=&-|{\check\chi}|^2-\f2\varrho\text{tr}\chi+\f{3}{\varrho^2}-F_{A\mathring L}\slashed d^A\mathring L\phi-G_{A\mathring L}^\gamma\slashed d^A\mathring L\varphi_\gamma
+\f12F_{\mathring L\mathring L}\slashed\triangle\phi\\
&+\f12G_{\mathring L\mathring L}^\gamma\slashed\triangle\varphi_\gamma+\f12\slashed g^{AB}F_{AB}\mathring L^2\phi+\f12\slashed g^{AB}G_{AB}^\gamma\mathring L^2\varphi_{\gamma}\\
&-\f12(\mathcal{FG})_{\mathring L\mathring L}(\mathring L)\text{tr}\chi-(\mathcal{FG})_{\tilde T\mathring L}(\mathring L)\text{tr}\chi+(\mathcal{FG})_{B\mathring L}(X_A)\chi^{AB}\\
&+f(\phi, \varphi, \slashed dx^i,\mathring L^i)\left(
\begin{array}{ccc}
\mathring L\phi\\
\slashed d\phi\\
\mathring L\varphi\\
\slashed d\varphi\\
\end{array}
\right)
\left(
\begin{array}{ccc}
\mathring L\phi\\
\slashed d\phi\\
\mathring L\varphi\\
\slashed d\varphi\\
\end{array}
\right).
\end{split}
\end{equation}
In addition, by \eqref{fe} and \eqref{fequation}, one can rewrite $\slashed\triangle\phi$ and $\slashed\triangle\varphi_\g$ as
\begin{equation}\label{triphi}
\begin{split}
	&\slashed\triangle\phi=\mu^{-1}\big(\mathring L\mathring{\underline L}\phi+\f{3}{2\varrho}\mathring {\underline L}\phi-\mathring {H}\big),\\
&\slashed\triangle\varphi_\g=\mu^{-1}\big(\mathring L\mathring{\underline L}\varphi_\g+\f{3}{2\varrho}\mathring {\underline L}\varphi_\g-{H}_\g\big).
\end{split}
\end{equation}
Thus, \eqref{Ltrchi} can be written as a new form which contains only the first order
or zero-th order derivatives of $\varphi_\g$ on the right hand side of the following  equality
\begin{equation}\label{Ltr}
\begin{split}
\mathring L\big(\textrm{tr}\chi-E)=\big(-\f{2}{\varrho}+\mathcal E\big)\textrm{tr}\chi+\f{3}{\varrho^2}-|\check\chi|^2+e,
\end{split}
\end{equation}
where
\begin{equation*}
\begin{split}
E=&-\slashed g^{AB}(\mathcal{FG})_{A\mathring L}(X_B)+\f12\mu^{-1}(\mathcal{FG})_{\mathring L\mathring L}(\mathring{\underline L})+\f12\slashed g^{AB}(\mathcal{FG})_{AB}(\mathring L),\\
\mathcal E=&-\f12(\mathcal{FG})_{\mathring L\mathring L}(\mathring L)-(\mathcal{FG})_{\tilde T\mathring L}(\mathring L)+\f12\mu^{-1}(\mathcal{FG})_{\mathring L\mathring L}(T),\\
e=&\f12\mu^{-2}(G_{\mathring L\mathring L}^\gamma T\varphi_\gamma)(\mathcal{FG})_{\mathring L\mathring L}(T)
+f(\phi, \varphi, \slashed dx^i,\mathring L^i)\left(
\begin{array}{ccc}
\mathring L\phi\\
\slashed d\phi\\
\mathring L\varphi\\
\slashed d\varphi\\
\tilde T\phi\\
\tilde T\varphi
\end{array}
\right)
\left(
\begin{array}{ccc}
\mathring L\phi\\
\slashed d\phi\\
\mathring L\varphi\\
\slashed d\varphi\\
\end{array}
\right).\\
\end{split}
\end{equation*}
Let $F^k=\slashed d\bar Z^k\textrm{tr}\chi-\slashed d\bar Z^k E$ with $\bar Z\in\{T,R_{ij}\}$. Then by an induction
argument on \eqref{Ltr}, one has
\begin{equation}\label{LFal}
\begin{split}
\slashed{\mathcal L}_{\mathring L}F^k=&(-\f{2}{\varrho}+\mathcal E)F^k+(-\f{2}{\varrho}+\mathcal E)\slashed d\bar Z^k E-\slashed d\bar Z^{k}(|\check\chi|^2)+e^k,
\end{split}
\end{equation}
where
\begin{equation}\label{eal}
\begin{split}
e^k=&\slashed{\mathcal L}_{\bar Z}^k e^0+\sum_{p_1+p_2=k-1}{\tiny {\tiny {\tiny {\tiny }}}}\slashed{\mathcal L}_{Z}^{p_1}\slashed{\mathcal L}_{[\mathring L,\bar Z]}F^{p_2}+\sum_{\tiny\begin{array}{c}p_1+p_2=k\\p_1\geq 1\end{array}}\bar Z^{p_1}\big(-\f{2}{\varrho}+\mathcal E)\slashed d\bar Z^{p_2}\textrm{tr}\chi
\end{split}
\end{equation}
and
\begin{equation}\label{e0}
\begin{split}
e^0=&\slashed de+(\slashed d\mathcal E)\textrm{tr}\chi.
\end{split}
\end{equation}
Note that for any one-form $\xi$ on $S_{s,u}$, it holds that
\begin{equation}\label{Lxi}
\mathring L(\varrho^2|\xi|^2)=-2\varrho^2\check{\chi}^{AB}\xi_A\xi_B+2\varrho^2\slashed g^{AB}(\slashed{\mathcal L}_{\mathring L}\xi_A)\xi_B.
\end{equation}
Taking $\xi=\varrho^2F^k$ in \eqref{Lxi} and using \eqref{LFal} show
\begin{equation*}
\begin{split}
&\mathring L\big(\varrho^6|F^k|^2\big)\\
=&2\varrho^6\big\{-\check\chi^{AB}F^k_AF^k_B+\mathcal E|F^k|^2+\slashed g^{AB}e^k_AF^k_B+(-\f{2}{\varrho}+\mathcal E)\slashed d^A\bar Z^k E\cdot F^k_A-\slashed d^A\bar Z^k(|\check\chi|^2)\cdot F^k_A\big\}.
\end{split}
\end{equation*}
Then
\begin{equation}\label{LrhoF}
\begin{split}
|\mathring L(\varrho^3|F^k|)|\lesssim&\delta^{1-2\varepsilon_0}s^{-3/2}\varrho^3|F^k|+\varrho^2|\slashed d\bar Z^k E|+\varrho^3|\slashed d\bar Z^k(|\check\chi|^2)|+\varrho^3|e^k|.
\end{split}
\end{equation}
It follows from \eqref{Ff} and \eqref{LrhoF} that
\begin{equation*}
\begin{split}
\delta^l\varrho^3\|F^k\|_{L^2(\Sigma_s^{u})}\lesssim&\delta^l\|F^k(t_0,\cdot,\cdot)\|_{L^2(\Sigma_{s}^{ u})}+\delta^l\varrho^{3/2}\int_{t_0}^s\Big\{\delta^{1-2\varepsilon_0}\tau^{-3/2}\varrho^{3/2}\|F^k\|_{L^2(\Sigma_{\tau}^{ u})}\\
&+\tau^{1/2}\|\slashed d\bar Z^k E\|_{L^2(\Sigma_{\tau}^{u})}+\tau^{3/2}\|\slashed d\bar Z^k(|\check\chi|^2)\|_{L^2(\Sigma_{\tau}^u)}+\tau^{3/2}\|e^k\|_{L^2(\Sigma_{\tau}^u)}\Big\}d\tau.
\end{split}
\end{equation*}
This, together with Gronwall's inequality, yields
\begin{equation}\label{Eal}
\begin{split}
&\varrho^{3/2}\delta^l\|F
^k\|_{L^2(\Sigma_s^{u})}\\
\lesssim&\delta^{3/2-2\varepsilon_0}+\delta^l\int_{t_0}^s\Big\{\tau^{1/2}\|\slashed d\bar Z^k E\|_{L^2(\Sigma_{\tau}^{u})}
+\tau^{3/2}\|\slashed d\bar Z^k(|\check\chi|^2)\|_{L^2(\Sigma_{\tau}^u)}+\tau^{3/2}\|e^k\|_{L^2(\Sigma_{\tau}^u)}\Big\}{d}\tau.
\end{split}
\end{equation}

Next, we come to estimate each term in the integrand in \eqref{Eal}.
\begin{enumerate}
	
	\item {\bfseries The estimate on $\|\slashed d\bar Z^k E\|_{L^2(\Sigma_{s}^{u})}$}
	
Due to $E=-\slashed g^{AB}(\mathcal{FG})_{A\mathring L}(X_B)+\f12\mu^{-1}(\mathcal{FG})_{\mathring L\mathring L}(\mathring{\underline L})+\f12\slashed g^{AB}(\mathcal{FG})_{AB}(\mathring L)$, by virtue of \eqref{GT}, it is direct to get
	\begin{equation}\label{sdE}
	\begin{split}
	&\delta^l\|\slashed d\bar Z^k E\|_{L^2(\Sigma_s^{u})}\\
	\lesssim&\delta^{1-2\varepsilon_0}s^{-7/2}\big(\delta^{l_{-1}}\|Z^{k_{-1}}x^i\|_{L^2(\Sigma_s^{u})}
+s\delta^{l_0}\|Z^{k_0}\mu\|_{L^2(\Sigma_s^{u})}\big)
+\delta^{-\varepsilon_0}s^{-5/2}\delta^{l_0}\|Z^{k_0}\check L^i\|_{L^2(\Sigma_s^{u})}\\
	&+s^{-1}\delta^{l_0}\big(\|Z^{k_0}\mathring L\phi\|_{L^2(\Sigma_s^{u})}+\|Z^{k_0} T\phi\|_{L^2(\Sigma_s^{u})}+\|\slashed dZ^{k_0}\varphi\|_{L^2(\Sigma_s^{u})}+\|\mathring LZ^{k_0}\varphi\|_{L^2(\Sigma_s^{u})}\big)\\
	&+\delta^{1-\varepsilon_0}s^{-2}\delta^{l_0}\|\mathring {\underline L}Z^{k_0}\varphi\|_{L^2(\Sigma_s^{u})}+\delta^{-\varepsilon_0}s^{-5/2}\delta^{l_0}(\|Z^{k_0}\phi\|_{L^2(\Sigma_s^{u})}
+\|Z^{k_0}\varphi\|_{L^2(\Sigma_s^{u})})\\
	&+s^{-2}\delta^{l_0}\|Z^{k_0}R\phi\|_{L^2(\Sigma_s^{u})}
+\delta^{1-\varepsilon_0}s^{-7/2}(\delta^{l_0}\|\slashed{\mathcal L}_Z^{k_0}\slashed g\|_{L^2(\Sigma_s^{u})}
+\delta^{l_1}\|\slashed{\mathcal L}_Z^{k_1}\leftidx{^{(R)}}{\slashed\pi}_{\mathring L}\|_{L^2(\Sigma_s^{u})})\\
	&+\delta^{2-2\varepsilon_0}s^{-9/2}\delta^{l_1}(\|\slashed{\mathcal L}_Z^{k_1}\leftidx{^{(R)}}{\slashed\pi}_{T}\|_{L^2(\Sigma_s^{u})}
+s\|\slashed{\mathcal L}_Z^{k_1}\leftidx{^{(T)}}{\slashed\pi}_{\mathring L}\|_{L^2(\Sigma_s^{u})})\\
	\lesssim&\delta^{3/2-2\varepsilon_0}s^{-1}+\delta^{-\varepsilon_0}s^{-1}\sqrt{\tilde E_{1,\leq k+2}}
+\delta^{1-\varepsilon_0}s^{-1}\sqrt{\tilde E_{2,\leq k+2}},
	\end{split}
	\end{equation}
where one has used the related $L^\infty$ estimates in Section \ref{ho} and the $L^2$ estimates in
Proposition \ref{L2chi}, Corollary \ref{Zphi} and Lemma \ref{L2T},  $l_m$ is the number of $T$
in $Z^{k_m}$ ($m=-1,0,1$) and $1\leq k_m\leq k+1-m$.

 \item {\bfseries The estimate on $\|\slashed d\bar Z^k(|\check\chi|^2)\|_{L^2(\Sigma_{s}^u)}$}

It follows from the structure equation \eqref{dchi} that
\begin{equation}\label{divchi'}
\begin{split}
&\slashed\nabla_C\check\chi_{AB}=\slashed\nabla_A\check\chi_{CB}+\mathcal I_{CAB},\\
&\slashed\nabla^B\check{\chi}_{AB}=\slashed d_A(\textrm{tr}\chi)+\slashed g^{BC}\mathcal I_{CAB},
\end{split}
\end{equation}
where $\mathcal I$ is a 3-form on $S_{s,u}$, whose components are
\begin{equation}\label{IA}
\begin{split}
\mathcal I_{CAB}=&F_{ij}(\slashed d_C\phi)\slashed d_A\mathring L^i(\slashed d_Bx^j)
+G_{ij}^\gamma(\slashed d_C\varphi_\gamma)\slashed d_A\mathring L^i(\slashed d_Bx^j)
-F_{ij}(\slashed d_A\phi)\slashed d_C\mathring L^i(\slashed d_Bx^j)\\
&-G_{ij}^\gamma(\slashed d_A\varphi_\gamma)\slashed d_C\mathring L^i(\slashed d_Bx^j)-g_{ij}(\slashed d_C\check L^i)\slashed\nabla_{AB}^2x^j+g_{ij}(\slashed d_A\check L^i)\slashed\nabla_{CB}^2x^j\\
&+\varrho^{-1}(\mathcal{FG})_{BC}(X_A)-\varrho^{-1}(\mathcal{FG})_{AB}(X_C)
+\slashed\nabla_A\Lambda_{CB}-\slashed\nabla_C\Lambda_{AB}.
\end{split}
\end{equation}
By commuting ${\slashed{\mathcal L}}_{\bar Z}^k$ with $\slashed\nabla_C$ and $\slashed\nabla^B$
and utilizing Lemma \ref{commute}, one has
\begin{equation}\label{divchi}
\begin{split}
&\slashed\nabla_C\slashed{\mathcal L}_{\bar Z}^k\check\chi_{AB}=\slashed\nabla_A\slashed{\mathcal L}_{\bar Z}^k\check\chi_{CB}+\mathcal{I}^k_{CAB},\\
&\slashed\nabla^B{\slashed{\mathcal L}}_{\bar Z}^k\check{\chi}_{AB}=\slashed d_A(\bar Z^k\textrm{tr}\chi)+\tilde{\mathcal{I}}_A^k
\end{split}
\end{equation}
with
\begin{equation}\label{I}
\begin{split}
\mathcal I_{CAB}^k=\slashed{\mathcal L}_{\bar Z}^k \mathcal I_{CAB}+&\sum_{p_1+p_2=k-1}\slashed{\mathcal L}_{\bar Z}^{p_1}\Big\{(\check{\slashed\nabla}_C\leftidx^{{(\bar Z)}}\slashed\pi_A^D)\slashed{\mathcal L}_{\bar Z}^{p_2}\check\chi_{BD}+(\check{\slashed\nabla}_C\leftidx^{{(\bar Z)}}\slashed\pi_B^D)\slashed{\mathcal L}_{\bar Z}^{p_2}\check\chi_{AD}\\
&\qquad\qquad\qquad-(\check{\slashed\nabla}_A\leftidx^{{(\bar  Z)}}\slashed\pi_C^D)\slashed{\mathcal L}_{\bar Z}^{p_2}\check\chi_{BD}-(\check{\slashed\nabla}_A\leftidx^{{(\bar Z)}}\slashed\pi_B^D)\slashed{\mathcal L}_{\bar Z}^{p_2}\check\chi_{CD}\Big\},\\
\tilde{\mathcal I}_A^k=\slashed{\mathcal L}_{\bar Z}^k (\slashed g^{BC}\mathcal I_{CAB})+&\sum_{p_1+p_2=k-1}\slashed{\mathcal L}_{\bar Z}^{p_1}\Big\{\slashed g^{BC}(\check{\slashed\nabla}_B\leftidx^{{(\bar Z)}}\slashed\pi_A^D)\slashed{\mathcal L}_{ Z}^{p_2}\check\chi_{CD}\\
&\qquad\qquad+\slashed g^{BC}(\check{\slashed\nabla}_B\leftidx^{{(\bar  Z)}}\slashed\pi_C^D)\slashed{\mathcal L}_{\bar Z}^{p_2}\check\chi_{AD}
+\leftidx^{{(\bar Z)}}\slashed\pi^{BC}\slashed\nabla_B\slashed{\mathcal L}_{\bar Z}^{p_2}\check\chi_{AC}\Big\},
\end{split}
\end{equation}
here the constant coefficients are neglected in the summation above.
Therefore,
\begin{equation}\label{nzchi}
\begin{split}
|\slashed\nabla\slashed{\mathcal L}_{\bar Z}^k\check\chi|^2=&(\slashed\nabla_A\slashed{\mathcal L}_{\bar Z}^k\check\chi_{CB}+\mathcal{I}^k_{CAB})(\slashed\nabla^C\slashed{\mathcal L}_{\bar Z}^k\check\chi^{AB})\\
=&\slashed\nabla_A\big((\slashed{\mathcal L}_{\bar Z}^k\check\chi_{CB})(\slashed\nabla^C\slashed{\mathcal L}_{\bar Z}^k\check\chi^{AB})\big)-\slashed\nabla_C\big((\slashed{\mathcal L}_{\bar Z}^k\check\chi^{CB})(\slashed\nabla_A\slashed{\mathcal L}_{\bar Z}^k\check\chi_{B}^A)\big)\\
&-(\slashed{\mathcal L}_{\bar Z}^k\check\chi^{CB})\slashed g^{AD}(\slashed{\mathcal L}_{\bar Z}^k\check\chi_{ED}{\slashed{\mathcal R}_{ACB}^E}+\slashed{\mathcal L}_{ \bar Z}^k\check\chi_{BE}{\slashed{\mathcal R}_{ACD}^E})\\
&+(\slashed\nabla_C{\slashed{\mathcal L}}_{\bar Z}^k\check{\chi}^{CB})(\slashed\nabla^A{\slashed{\mathcal L}}_{\bar Z}^k\check{\chi}_{AB})+\mathcal{I}^k_{CAB}(\slashed\nabla^C\slashed{\mathcal L}_{\bar Z}^k\check\chi^{AB}),
\end{split}
\end{equation}
where the Ricci identity $\slashed\nabla_A\slashed\nabla_C\slashed{\mathcal L}_{\bar Z}^k\check\chi_{BD}-\slashed\nabla_C\slashed\nabla_A\slashed{\mathcal L}_{\bar Z}^k\check\chi_{BD}=\slashed{\mathcal L}_{\bar Z}^k\check\chi_{ED}{\slashed{\mathcal R}_{ACB}^E}+\slashed{\mathcal L}_{\bar Z}^k\check\chi_{BE}{\slashed{\mathcal R}_{ACD}^E}$ with $\slashed{\mathcal R}_{ACB}^E=\slashed g^{EF}\slashed{\mathcal R}_{ACBF}$ is used. According to the
second equation of \eqref{divchi} and the expressions \eqref{I}, one can integrate \eqref{nzchi} on $\Sigma_s^u$ to get
\begin{equation}\label{nbzchi}
\begin{split}
&\delta^{2l}\|\slashed\nabla\slashed{\mathcal L}_{\bar Z}^k\check\chi\|_{L^2(\Sigma_s^{u})}^2\\
\lesssim&\delta^{2l}\|\slashed d \bar Z^k\text{tr}\chi\|_{L^2(\Sigma_s^{u})}^2+s^{-2}\delta^{2l}\|\slashed{\mathcal L}_{ \bar Z}^k\check\chi\|_{L^2(\Sigma_s^{u})}^2+\delta^{2l}(\|\tilde{\mathcal I}^k\|_{L^2(\Sigma_s^{u})}^2+\|{\mathcal I}^k\|_{L^2(\Sigma_s^{u})}^2)\\
\lesssim&\delta^{2l}\|\slashed d \bar Z^k\text{tr}\chi\|_{L^2(\Sigma_s^{u})}^2+s^{-2}\delta^{2l_1}\|\slashed{\mathcal L}_{ \bar Z}^{k_1}\check\chi\|_{L^2(\Sigma_s^{u})}^2+\delta^{2l_1}\|\slashed{\mathcal L}_{ \bar Z}^{k_1}\mathcal I\|_{L^2(\Sigma_s^{u})}^2\\
&+\delta^{2-2\varepsilon_0}s^{-6}\delta^{2l_1}\|\slashed{\mathcal L}_{\bar Z}^{k_1}\slashed g\|_{L^2(\Sigma_s^{u})}^2+\delta^{2-2\varepsilon_0} s^{-6}\delta^{2l_1}\|\slashed{\mathcal L}_{\bar Z}^{k_1}\leftidx^{{(R)}}\slashed\pi\|_{L^2(\Sigma_s^{u})}^2\\
&+\delta^{4-2\varepsilon_0}s^{-6}\delta^{2l_1}\|\slashed{\mathcal L}_{ \bar Z}^{k_1}\leftidx^{{(T)}}\slashed\pi\|_{L^2(\Sigma_s^{u})}^2.
\end{split}
\end{equation}
In addition, substituting \eqref{IA} into \eqref{nbzchi}, due to $\slashed d\bar Z^k\textrm{tr}\chi=F^k+\slashed d\bar Z^k E$,
one can then obtain by Proposition \ref{L2chi}, Corollary \ref{Zphi} and \eqref{sdE} that
\begin{equation}\label{nbz}
\begin{split}
&\delta^l\|\slashed\nabla\slashed{\mathcal L}_{\bar Z}^k\check\chi\|_{L^2(\Sigma_s^{u})}\\
\lesssim&\delta^{l}\|\slashed d \bar Z^k\text{tr}\chi\|_{L^2(\Sigma_s^{u})}+s^{-1}\delta^{l_1}\|\slashed{\mathcal L}_{ Z}^{k_1}\check\chi\|_{L^2(\Sigma_s^{u})}+s^{-2}\delta^{l_0}(\| Z^{k_0}\phi\|_{L^2(\Sigma_s^{u})}+\|Z^{k_0}\check L^i\|_{L^2(\Sigma_s^{u})})\\
&+\delta^{1-\varepsilon_0}s^{-4}\delta^{l_{-1}}\| Z^{k_{-1}}x^i\|_{L^2(\Sigma_s^{u})}+s^{-2}\delta^{l_0}\| Z^{k_0}\varphi\|_{L^2(\Sigma_s^{u})}+s^{-2}\delta^{l_0}\| Z^{k_0}R\phi\|_{L^2(\Sigma_s^{u})}\\
&+s^{-1}\delta^{l_0}\|\slashed d Z^{k_0}\varphi\|_{L^2(\Sigma_s^{u})}+\delta^{1-\varepsilon_0}s^{-3}\delta^{l_1}\|\slashed{\mathcal L}_{Z}^{k_1}\slashed g\|_{L^2(\Sigma_s^{u})}+\delta^{1-\varepsilon_0}s^{-3}\delta^{l_1}\|\slashed{\mathcal L}_{ Z}^{k_1}\leftidx^{{(R)}}\slashed\pi\|_{L^2(\Sigma_s^{u})}\\
&+\delta^{3-\varepsilon_0}s^{-9/2}\delta^{l_1}\|\slashed{\mathcal L}_{ Z}^{k_1}\leftidx^{{(R)}}\slashed\pi_T\|_{L^2(\Sigma_s^{u})}+\delta^{2-\varepsilon_0}s^{-3}\delta^{l_1}\|\slashed{\mathcal L}_{Z}^{k_1}\leftidx^{{(T)}}\slashed\pi\|_{L^2(\Sigma_s^{u})}\\
\lesssim&\delta^{l}\|\slashed d \bar Z^k\text{tr}\chi\|_{L^2(\Sigma_s^{u})}+\delta^{3/2-\varepsilon_0}s^{-1}+s^{-1-\epsilon}\sqrt{\tilde E_{1,\leq k+2}}+\delta s^{-1}\sqrt{\tilde E_{2,\leq k+2}}\\
\lesssim&\delta^{l}\|F^k\|_{L^2(\Sigma_s^{u})}+\delta^{3/2-2\varepsilon_0}s^{-1}+\delta^{-\varepsilon_0}s^{-1}\sqrt{\tilde E_{1,\leq k+2}}
+\delta^{1-\varepsilon_0}s^{-1}\sqrt{\tilde E_{2,\leq k+2}}.
\end{split}
\end{equation}

It follows from \eqref{nbz} and Proposition \ref{L2chi} that
\begin{equation}\label{Y-19}
\begin{split}
&\delta^l\|\slashed d\bar Z^k(|\check\chi|^2)\|_{L^2(\Sigma_s^u)}\\
\lesssim&\delta^l\|\slashed d(\slashed g^{AB}\slashed g^{CD}\slashed{\mathcal L}_{\bar Z}^k\check\chi_{AC}\cdot\check\chi_{BD})\|_{L^2(\Sigma_s^u)}\\
&+\delta^l\sum_{\mbox{\tiny$\begin{array}{c}p_1+p_2+p_3+p_4=k\\
		p_3\leq k-1,p_4\leq k-1\end{array}$}}\|\slashed d(\slashed{\mathcal L}_{\bar Z}^{p_1}\slashed g^{AB}\cdot\slashed{\mathcal L}_{\bar Z}^{p_2}\slashed g^{CD}\cdot\slashed{\mathcal L}_{\bar Z}^{p_3}\check\chi_{AC}\cdot\slashed{\mathcal L}_{\bar Z}^{p_4}\check\chi_{BD})\|_{L^2(\Sigma_s^u)}\\
	\lesssim&\delta^{1-\varepsilon_0} s^{-2}(\delta^l\|\slashed\nabla\slashed{\mathcal L}_{\bar Z}^k\check\chi\|_{L^2(\Sigma_s^u)}+s^{-1}\delta^{l_1}\|\slashed{\mathcal L}_{\bar Z}^{k_1}\check\chi\|_{L^2(\Sigma_s^u)})+\delta^{2-2\varepsilon_0}s^{-5}\delta^{l_0}\|\slashed{\mathcal L}_{\bar Z}^{k_0}\slashed g\|_{L^2(\Sigma_s^u)}\\
	\lesssim&\delta^{1-\varepsilon_0}s^{-2}\delta^{l}\|F^k\|_{L^2(\Sigma_s^{u})}+\delta^{5/2-3\varepsilon_0}s^{-3}
+\delta^{1-2\varepsilon_0}s^{-3}\sqrt{\tilde E_{1,\leq k+2}}+\delta^{2-2\varepsilon_0}s^{-3}\sqrt{\tilde E_{2,\leq k+2}}.
\end{split}
\end{equation}

\item {\bfseries The estimate on $\|e^k\|_{L^2(\Sigma_{s}^u)}$}

We first deal with the term $\slashed{\mathcal L}_{\bar Z}^{p_1}\slashed{\mathcal L}_{[\mathring L,\bar Z]}F^{p_2}$
in \eqref{eal} when $p_1+p_2=k-1$. Due to
$[\mathring L, \bar Z]=\leftidx^{{(\bar Z)}}\slashed\pi_{\mathring L}^AX_A$, then
$\slashed{\mathcal L}_{[\mathring L,\bar Z]}F^{p_2}=\leftidx^{{(\bar Z)}}{\slashed\pi_{\mathring L}}^C\slashed\nabla_CF^{p_2}+F^{p_2}_C\slashed\nabla\leftidx^{{(\bar Z)}}{\slashed\pi_{\mathring L}}^C$.
This implies
\begin{equation}\label{LF}
\begin{split}
&\delta^l\|\slashed{\mathcal L}_{\bar Z}^{p_1}\slashed{\mathcal L}_{[\mathring L, \bar Z]}F^{p_2}\|_{L^2(\Sigma_s^u)}\\
\lesssim&\delta^{1-\varepsilon_0} s^{-2}\delta^l\|F^k\|_{L^2(\Sigma_s^u)}+\delta^{1-\varepsilon_0}s^{-3}\delta^{l_1}\|\slashed{\mathcal L}_{ \bar Z}^{k_1}\check\chi\|_{L^2(\Sigma_s^u)}+\delta^{1-\varepsilon_0}s^{-2}\delta^{l_2}\|\slashed d\bar Z^{k_2}E\|_{L^2(\Sigma_s^u)}\\
&+\delta^{1-2\varepsilon_0}s^{-7/2}\delta^{l_1}(\|\slashed{\mathcal L}_{ \bar Z}^{k_1}\leftidx^{{(R)}}\slashed\pi_{\mathring L}\|_{L^2(\Sigma_s^{u})}+\delta\|\slashed{\mathcal L}_{ \bar Z}^{k_1}\leftidx^{{(T)}}\slashed\pi_{\mathring L}\|_{L^2(\Sigma_s^{u})})+\delta^{2-3\varepsilon_0}s^{-9/2}\delta^{l_1}\|\slashed{\mathcal L}_{\bar Z}^{k_1}\slashed g\|_{L^2(\Sigma_s^{u})}\\
\lesssim&\delta^{1-\varepsilon_0} s^{-2}\delta^l\|F^k\|_{L^2(\Sigma_s^u)}+\delta^{5/2-3\varepsilon_0}s^{-3}+\delta^{1-2\varepsilon_0}s^{-5/2-\epsilon}\sqrt{\tilde{E}_{1,\leq k+2}}+\delta^{2-2\varepsilon_0}s^{-5/2}\sqrt{\tilde{E}_{2,\leq k+2}}.
\end{split}
\end{equation}

For the term $\bar Z^{p_1}(-\f2\varrho+\mathcal E)\slashed d\bar Z^{p_2}\textrm{tr}\chi$
in \eqref{eal} with $p_1+p_2\leq k$ and $p_1\geq 1$, using the expression of $\mathcal E$ and \eqref{GT}, one can get by Proposition \ref{L2chi} and Corollary \ref{Zphi} that
\begin{equation}\label{Zrho}
\begin{split}
&\delta^l\|\bar Z^{p_1}(-\f2\varrho+\mathcal E)\slashed d\bar Z^{p_2}\textrm{tr}\chi\|_{L^2(\Sigma_s^u)}\\
\lesssim&s^{-5/2}\delta^{l_1}\|\slashed{\mathcal L}_{\bar Z}^{k_1}\check\chi\|_{L^2(\Sigma_s^{u})}+\delta^{-\varepsilon_0}s^{-3}\delta^{l_0}\|Z^{k_0}\phi\|_{L^2(\Sigma_s^{u})}
+\delta^{1-2\varepsilon_0}s^{-4}\delta^{l_0}\|Z^{k_0}\varphi\|_{L^2(\Sigma_s^{u})}\\
&+\delta^{1-\varepsilon_0}s^{-9/2}\delta^{l_1}(\|\slashed{\mathcal L}_{\bar Z}^{k_1}\slashed g\|_{L^2(\Sigma_s^{u})}+\delta^{-\varepsilon_0}\|\bar Z^{k_1}\check L^i\|_{L^2(\Sigma_s^{u})}+\delta^{-\varepsilon_0}\|\bar Z^{k_1}\check\varrho\|_{L^2(\Sigma_s^{u})})\\
&+\delta^{2-3\varepsilon_0}s^{-9/2}\delta^{l_1}\|Z^{k_1}\mu\|_{L^2(\Sigma_s^{u})}
+\delta^{2-3\varepsilon_0}s^{-11/2}\delta^{l_0}\|Z^{k_0}x^i\|_{L^2(\Sigma_s^{u})}\\
\lesssim&\delta^{3/2-\varepsilon_0}s^{-3}+s^{-5/2-\epsilon}\sqrt{\tilde{E}_{1,\leq k+2}(s,u)}
+\delta s^{-5/2}\sqrt{\tilde E_{2,\leq k+2}(s,u)}.
\end{split}
\end{equation}

The remaining term to be estimated in \eqref{eal} is $\slashed{\mathcal L}_Z^k e^0$.
By \eqref{e0}, a direct computation yields
\begin{equation}\label{Le0}
\begin{split}
\delta^l\|\slashed{\mathcal L}_{\bar Z}^k e^0\|_{L^2(\Sigma_s^u)}\lesssim&\delta^{3/2-2\varepsilon_0}s^{-2}
+\delta^{-\varepsilon_0}s^{-2}\sqrt{\tilde E_{1,\leq k+2}}+\delta^{1-\varepsilon_0}s^{-2}\sqrt{\tilde E_{2,\leq k+2}}.
\end{split}
\end{equation}

Combining the estimates \eqref{LF}-\eqref{Le0} with \eqref{eal} leads to
\begin{equation}\label{Y-21}
\begin{split}
\delta^l\|e^k\|_{L^2(\Sigma_s^u)}\lesssim&\delta^{1-\varepsilon_0}s^{-2}\delta^l\|F^k\|_{L^2(\Sigma_s^u)}
+\delta^{3/2-2\varepsilon_0}s^{-2}\\
&+\delta^{-\varepsilon_0}s^{-2}\sqrt{\tilde E_{1,\leq k+2}(s,u)}+\delta^{1-\varepsilon_0} s^{-2}\sqrt{\tilde E_{2,\leq k+2}(s,u)}.
\end{split}
\end{equation}
\end{enumerate}

By inserting \eqref{sdE}, \eqref{Y-19} and \eqref{Y-21} into \eqref{Eal},  one arrives at
\begin{equation}\label{Fal}
\begin{split}
\delta^l\|F^k\|_{L^2(\Sigma_s^u)}\lesssim&\delta^{3/2-2\varepsilon_0}s^{-1}
+\delta^{-\varepsilon_0}s^{-1}\sqrt{\tilde E_{1,\leq k+2}(s,u)}+\delta^{1-\varepsilon_0} s^{-1}\sqrt{\tilde E_{2,\leq k+2}(s,u)}.
\end{split}\end{equation}
In addition, one has from the definition of $F^k$ that $\slashed d\bar Z^k\textrm{tr}\chi=F^k+\slashed d \bar Z^k E$. Hence,
by \eqref{Fal}, \eqref{sdE} and \eqref{nbz}, it holds that
\begin{equation}\label{d}
\begin{split}
&\delta^l\|\slashed d{\bar Z}^k\textrm{tr}\chi\|_{L^2(\Sigma_s^u)}
+\delta^l\|\slashed\nabla{\slashed{\mathcal L}}_{\bar Z}^k\check{\chi}\|_{L^2(\Sigma_s^u)}\\
\lesssim&\delta^{3/2-2\varepsilon_0}s^{-1}
+\delta^{-\varepsilon_0}s^{-1}\sqrt{\tilde E_{1,\leq k+2}(s,u)}
+\delta^{1-\varepsilon_0} s^{-1}\sqrt{\tilde E_{2,\leq k+2}(s,u)}.
\end{split}
\end{equation}

In summary, for any $Z\in\{\mathring L, T, R_{ij}\}$, it follows from  \eqref{nz}, \eqref{dch} and \eqref{d} that
\begin{equation}\label{dnchi}
\begin{split}
&\delta^l\|\slashed dZ^k\textrm{tr}\chi\|_{L^2(\Sigma_s^u)}
+\delta^l\|\slashed\nabla{\slashed{\mathcal L}}_{Z}^k\check{\chi}\|_{L^2(\Sigma_s^u)}\\
\lesssim&\delta^{3/2-2\varepsilon_0}s^{-1}
+\delta^{-\varepsilon_0}s^{-1}\sqrt{\tilde E_{1,\leq k+2}(s,u)}+\delta^{1-\varepsilon_0} s^{-1}\sqrt{\tilde E_{2,\leq k+2}(s,u)}.
\end{split}
\end{equation}

\begin{remark}
	It should be noted here that due to $\slashed dZ^k\textrm{tr}\check\chi=\slashed dZ^k\textrm{tr}\chi$
by \eqref{errorv}, then \eqref{dnchi} also gives the $L^2$ estimate of $\slashed dZ^k\textrm{tr}\check\chi$.
\end{remark}

\subsection{Estimates on the derivatives of $\slashed\triangle\mu$ and $\slashed\nabla^2\mu$}\label{trimu}

Analogously to Section \ref{trchi}, one can utilize the transport equation \eqref{lmu} for $\mu$ to estimate $\slashed\triangle\mu$.
It follows from $[Z,\slashed\triangle]=
-\leftidx^{{(Z)}}{\slashed\pi}^{AB}\slashed\nabla_{AB}^2-(\check{\slashed\nabla}_A\leftidx^{{(Z)}}{\slashed\pi}^{AB})\slashed d_B$
 of Lemma \ref{commute} with $Z\in\{T,\mathring L\}$ that
\begin{equation}\label{Ltrimu}
\begin{split}
&\mathring L\slashed\triangle\mu=[\mathring L,\slashed\triangle]\mu+\slashed\triangle\mathring L\mu=\mathcal J_1
+\mathcal J_2+\mathcal J_3+\mathcal J_4,
\end{split}
\end{equation}
where $\mathcal J_1$-$\mathcal J_3$ consist of all the terms whose $k$-order derivatives can not be controlled by
the estimates in Proposition \ref{L2chi}, and the remaining terms are put in $\mathcal J_4$.
\begin{itemize}
	\item $\mathcal J_1$ is comprised of all the terms containing the second order derivatives of $\mu$:
	$$
	\mathcal J_1=-2\check\chi^{AB}\slashed\nabla_{AB}^2\mu-2\varrho^{-1}\slashed\triangle\mu-\f12(\slashed\triangle\mu)\big\{(\mathcal {FG})_{\mathring L\mathring L}(\mathring L)+2(\mathcal{FG})_{\tilde T\mathring L}(\mathring L)\big\}.
	$$
	\item $\mathcal J_2$ is given as follows
	$$
	\mathcal J_2=-\f12\mu (F_{\mathring L\mathring L}+2F_{\tilde T\mathring L})\underline{\mathring L\slashed\triangle\phi}
-\f12\mu (G_{\mathring L\mathring L}^\gamma+2G_{\tilde T\mathring L}^\gamma)\underline{\mathring L\slashed\triangle\varphi_\gamma}
+\f12F_{\mathring L\mathring L}\underline{T\slashed\triangle\phi}+\f12G_{\mathring L\mathring L}^\gamma\underline{T\slashed\triangle\varphi_\gamma},
	$$
	where the terms with underlines in $\mathcal J_2$ are the ones including the third order derivatives of $\phi$ or $\varphi$.
	The following careful decomposition of $\mathcal J_2$ will be used later.

It follows from \eqref{lmu} and \eqref{LL} that
\begin{equation}\label{FLLLp}
\begin{split}
-\f12\mu F_{\mathring L\mathring L}\mathring L\slashed\triangle\phi=&\mathring L(-\f12\mu F_{\mathring L\mathring L}\slashed\triangle\phi)+\f12\mathring L(\mu F_{\mathring L\mathring L})\slashed\triangle\phi\\
&=\mathring L(-\f12\mu F_{\mathring L\mathring L}\slashed\triangle\phi)+f(\phi,\varphi,L^i)\left(
\begin{array}{ccc}
\mu\mathring L\phi\\
\mu\mathring L\varphi\\
\mu\slashed d_A\phi\cdot\slashed d^Ax^j\\
\mu\slashed d_A\varphi\cdot\slashed d^Ax^j\\
T\phi\\
T\varphi
\end{array}
\right)
\slashed\triangle\phi.
\end{split}
\end{equation}
Similar decomposition for $-\mu F_{\tilde T\mathring L}\mathring L\slashed\triangle\phi$ and $-\f12\mu (G_{\mathring L\mathring L}^\gamma+2G_{\tilde T\mathring L}^\gamma)\mathring L\slashed\triangle\varphi_\gamma$ can be derived.

In addition, due to the equation \eqref{fe}, one has
\begin{equation}\label{FLLTp}
\begin{split}
&\f12F_{\mathring L\mathring L}T\slashed\triangle\phi\\
=&\f12\mu^{-1}F_{\mathring L\mathring L}\Big(-(T\mu)\slashed\triangle\phi+\mathring LT\underline{\mathring L}\phi-\leftidx^{{(T)}}{\slashed\pi}_{\mathring L}^A\slashed d_A\underline{\mathring L}\phi+\f{3}{2\varrho^2}\underline{\mathring L}\phi+\f{3}{2\varrho}T\underline{\mathring L}\phi-T\mathring H\Big)\\
=&\mathring L(-\f12\mu^{-1}F_{\mathring L\mathring L}T\underline{\mathring L}\phi)+\f12\mathring L(\mu^{-1}F_{\mathring L\mathring L})T\underline{\mathring L}\phi-\f12\mu^{-1}F_{\mathring L\mathring L}\Big((T\mu)\slashed\triangle\phi
\\
&\qquad\qquad\qquad\qquad\qquad\qquad+\leftidx^{{(T)}}{\slashed\pi}_{\mathring L}^A\slashed d_A\underline{\mathring L}\phi-\f{3}{2\varrho^2}\underline{\mathring L}\phi
-\f{3}{2\varrho}T\underline{\mathring L}\phi+T\mathring H\Big).
\end{split}
\end{equation}
Analogously, $\f12G_{\mathring L\mathring L}^\gamma T\slashed\triangle\varphi_\gamma$ can be decomposed.

It should be emphasized that the terms like $\mathring L(-\f12\mu F_{\mathring L\mathring L}\slashed\triangle\phi)$ in \eqref{FLLLp}
and $\mathring L(-\f12\mu^{-1}F_{\mathring L\mathring L}T\underline{\mathring L}\phi)$
in \eqref{FLLTp} will be moved to the left hand side of \eqref{Ltrimu}.
\item $\mathcal J_3$ can be rewritten as
\begin{equation*}
\begin{split}
\mathcal J_3=&-\f12\mu\big\{\uwave{\slashed\triangle(F_{\mathring L\mathring L})}\mathring L\phi+\uwave{\slashed\triangle(G_{\mathring L\mathring L}^\gamma)}\mathring L\varphi_\gamma+\uwave{2\slashed\triangle(F_{\tilde T\mathring L})}\mathring L\phi+\uwave{2\slashed\triangle(G_{\tilde T\mathring L}^\gamma)}\mathring L\varphi_\gamma
\big\}\\
&+\f12\uwave{\slashed\triangle(F_{\mathring L\mathring L})}T\phi+\f12\uwave{\slashed\triangle(G_{\mathring L\mathring L}^\gamma)}T\varphi_\gamma.
\end{split}
\end{equation*}
In $\mathcal J_3$, the factors with wavy line can be decomposed as follows. First, \eqref{divchi'} implies
\begin{equation}\label{Y-22}in
\begin{split}
\slashed\triangle(F_{\mathring L\mathring L})=&(\slashed\triangle F_{\al\beta})\mathring L^\al\mathring L^\beta+2(\slashed d_AF_{\al\beta})\slashed d^A(\mathring L^\al\mathring L^\beta)+2F_{ij}(\slashed d_A\mathring L^i)\slashed d^A\mathring L^j\\
&+2F_{i\beta}(\slashed\triangle\check L^i+\f{\slashed\triangle x^i}{\varrho})\mathring L^\beta,
\end{split}
\end{equation}
where
\begin{equation*}
\begin{split}
\slashed\triangle\check L^i=&(\slashed{\nabla}^B\check\chi_{AB})\slashed d^Ax^i+\check\chi^{AB}\slashed\nabla_{AB}^2x^i\\
&+\slashed\nabla^A\big\{-(\mathcal{FG})_{\mathring L\tilde T}(X_A)\tilde T^i
-\f12(\mathcal{FG})_{\tilde T\tilde T}(X_A)\tilde T^i+\Lambda_{AB}\slashed d^Bx^i\big\}
\end{split}
\end{equation*}
holds with the help of \eqref{deL}.
Next, $\slashed\triangle(G_{\mathring L\mathring L}^\gamma)$, $\slashed\triangle(F_{\tilde T\mathring L})$ and $\slashed\triangle(G_{\tilde T\mathring L}^\gamma)$ have the similar expressions as that of $\slashed\triangle(F_{\mathring L\mathring L})$.
Here it is pointed out that $\|Z^k(\slashed\triangle(F_{\mathring L\mathring L}))\|_{L^2(\Sigma_s^u)}$
can be estimated with the help of \eqref{Y-22}, \eqref{dnchi} and Proposition \ref{L2chi}.
\item $\mathcal J_4$ is decomposed as
\begin{equation*}
\begin{split}
\mathcal J_4=&-2({\check{\slashed\nabla}_A\check\chi^{AB}})\slashed d_B\mu-\slashed d\mu\cdot\slashed d\big((\mathcal {FG})_{\mathring L\mathring L}(\mathring L)+2(\mathcal{FG})_{\tilde T\mathring L}(\mathring L)\big)\\
&-\mu (F_{\mathring L\mathring L}+2F_{\tilde T\mathring L})\big(\chi^{AB}\slashed\nabla_{AB}^2\phi+(\check{\slashed\nabla}_A\check\chi^{AB})\slashed d_B\phi\big)\\
&-\mu (G_{\mathring L\mathring L}^\gamma+2G_{\tilde T\mathring L}^\gamma)\big(\chi^{AB}\slashed\nabla_{AB}^2\varphi_\gamma+(\check{\slashed\nabla}_A\check\chi^{AB})\slashed d_B\varphi_\gamma\big)\\
&+\slashed d^A(F_{\mathring L\mathring L})\slashed d_AT\phi+\f12F_{\mathring L\mathring L}\big(\leftidx^{{(T)}}{\slashed\pi}^{AB}\slashed\nabla_{AB}^2\phi+(\check{\slashed\nabla}_A\leftidx^{{(T)}}{\slashed\pi}^{AB})\slashed d_B\phi\big)\\
& +\slashed d^A(G_{\mathring L\mathring L}^\gamma)\slashed d_AT\varphi_\gamma+\f12G_{\mathring L\mathring L}^\gamma\big(\leftidx^{{(T)}}{\slashed\pi}^{AB}\slashed\nabla_{AB}^2\varphi_\gamma
+(\check{\slashed\nabla}_A\leftidx^{{(T)}}{\slashed\pi}^{AB})\slashed d_B\varphi_\gamma\big)\\
&-\mu\big\{\slashed d^A(F_{\mathring L\mathring L})\slashed d_A\mathring L\phi
+\slashed d^A(G_{\mathring L\mathring L}^\gamma)\slashed d_A\mathring L\varphi_\gamma
+2\slashed d^A(F_{\tilde T\mathring L})\slashed d_A\mathring L\phi
+2\slashed d^A(G_{\tilde T\mathring L}^\gamma)\slashed d_A\mathring L\varphi_\gamma\big\}.
\end{split}
\end{equation*}
\end{itemize}

Finally, set
\begin{equation*}
\begin{split}
\tilde E=&-\f12\mu (F_{\mathring L\mathring L}+2F_{\tilde T\mathring L})\slashed\triangle\phi
-\f12\mu (G_{\mathring L\mathring L}^\gamma+2G_{\tilde T\mathring L}^\gamma)\slashed\triangle\varphi_\gamma
+\f12\mu^{-1}F_{\mathring L\mathring L}T\underline{\mathring L}\phi+\f12\mu^{-1}G_{\mathring L\mathring L}^\gamma T\underline{\mathring L}\varphi_\gamma,\\
\tilde{\mathcal E}=&-\f12(\mathcal{FG})_{\mathring L\mathring L}(\mathring L)-(\mathcal{FG})_{\tilde T\mathring L}(\mathring L),\\
\tilde F=&\slashed\triangle\mu-\tilde E.
\end{split}
\end{equation*}
Then by the expressions of $\mathcal J_1$-$\mathcal J_4$, \eqref{Ltrimu} can be rewritten as
\begin{equation*}
\begin{split}
\mathring L\tilde F=&-2\check\chi^{AB}\slashed\nabla_{AB}^2\mu+(-\f2\varrho+\tilde{\mathcal E})\tilde F-(\slashed\nabla_B\check\chi^{AB})\mathcal Q_A+\slashed d^A(\text{tr}\check\chi)\mathcal Y_A+\tilde e
\end{split}
\end{equation*}
with
\begin{align*}
\mathcal Q_A=&2\slashed d_A\mu+\mu(\mathcal {FG})_{A\tilde T}(\mathring L)-(\mathcal{FG})_{A\mathring L}(T)+2\mu(\mathcal{FG})_{\tilde T\mathring L}(X_A)+2\mu(\mathcal{FG})_{\mathring L\mathring L}(X_A),\\
\mathcal Y_A=&\slashed d_A\mu+\mu(\mathcal {FG})_{\tilde T\mathring L}(X_A)+\mu(\mathcal{FG})_{\mathring L\mathring L}(X_A)
\end{align*}
and
\begin{equation}\label{te}
\begin{split}
\tilde e=&2(\slashed d^A\mu)\slashed d_A\tilde{\mathcal E}+(-\f2\varrho+\tilde{\mathcal E})\tilde E-\f12\mathring L(\mu^{-1}F_{\mathring L\mathring L})T\underline{\mathring L}\phi-\f12\mathring L(\mu^{-1}G_{\mathring L\mathring L}^\gamma)T\underline{\mathring L}\varphi_\gamma\\
&-\f12\mu^{-1}F_{\mathring L\mathring L}\big\{T\mu(\slashed\triangle\phi)+\leftidx^{{(T)}}{\slashed\pi}_{\mathring L}^A\slashed d_A\underline{\mathring L}\phi-\f{3}{2\varrho^2}\underline{\mathring L}\phi-\f{3}{2\varrho}T\underline{\mathring L}\phi+T\mathring H\big\}\\
&-\f12\mu^{-1}G_{\mathring L\mathring L}^\gamma\big\{T\mu(\slashed\triangle\varphi_\gamma)+\leftidx^{{(T)}}{\slashed\pi}_{\mathring L}^A\slashed d_A\underline{\mathring L}\varphi_\gamma-\f{3}{2\varrho^2}\underline{\mathring L}\varphi_\gamma-\f{3}{2\varrho}T\underline{\mathring L}\varphi_\gamma+T H_\gamma\big\}\\
&+f_1(\phi,\varphi,L^i,\slashed dx^j)\left(
\begin{array}{ccc}
\mu\mathring L\phi\\
\mu\mathring L\varphi\\
\mu\slashed d\phi\\
\mu\slashed d\varphi\\
T\phi\\
T\varphi\\
\mu\chi
\end{array}
\right)
\left(\begin{array}{ccc}
\slashed\nabla^2\phi\\
\slashed\nabla^2\varphi\\
\varrho^{-1}\slashed\triangle x^k\\
\slashed d\mathring L\phi\\
\slashed d\mathring L\varphi
\end{array}
\right)
+f_2(\phi,\varphi,L^i,\slashed dx^j)\left(
\begin{array}{ccc}
\slashed d\phi\\
\slashed d\varphi\\
\slashed d\mathring L^k\\
\slashed\nabla^2x^k
\end{array}
\right)
\left(
\begin{array}{ccc}
\slashed d\phi\\
\slashed d\varphi\\
\mathring L\phi\\
\mathring L\varphi\\
\slashed d\mathring L^l\\
\check\chi
\end{array}
\right)
\left(
\begin{array}{ccc}
\mu\mathring L\phi\\
\mu\mathring L\varphi\\
\mu\slashed d\phi\\
\mu\slashed d\varphi\\
T\phi\\
T\varphi
\end{array}
\right)\\
&+f_3(\phi,\varphi,L^i,\slashed dx^j)\slashed d\mu\left(
\begin{array}{ccc}
\slashed d\phi\\
\slashed d\varphi\\
\chi
\end{array}
\right)
\left(
\begin{array}{ccc}
\slashed d\phi\\
\slashed d\varphi\\
\mathring L\phi\\
\mathring L\varphi
\end{array}
\right)
+f_4(\phi,\varphi,L^i,\slashed dx^j)\left(
\begin{array}{ccc}
\slashed dT\phi\\
\slashed dT\varphi
\end{array}
\right)
\left(
\begin{array}{ccc}
\slashed d\phi\\
\slashed d\varphi\\
\slashed d\mathring L^k
\end{array}
\right).
\end{split}
\end{equation}

As $F^k$ in Sect. \ref{trchi}, one can set $\bar F^k=\bar Z^k\slashed\triangle\mu-\bar Z^k\tilde E$
with $\bar Z$ being any vector field in $\{T,R_{ij}\}$.  Then by an induction argument, one has that for $k\geq 1$,
\begin{equation}\label{LbarF}
\begin{split}
\mathring L\bar F^k=&-2\check\chi^{AB}\slashed{\mathcal L}_{\bar Z}^k\slashed\nabla_{AB}^2\mu+(-\f2\varrho+\tilde{\mathcal E})\bar F^k-(\slashed\nabla_B\slashed{\mathcal L}_{\bar Z}^k\check\chi^{AB})\mathcal Q_A\\
&+(\slashed d^A\bar Z^k\textrm{tr}\chi)\mathcal Y_A+\leftidx^{{(\bar Z)}}{\slashed\pi}_{\mathring L}^A(\slashed d_A\bar F^{k-1})+\bar e^k,
\end{split}
\end{equation}
where
\begin{equation}\label{bare}
\begin{split}
\bar e^k=&\underbrace{\sum_{{\mbox{\tiny$\begin{array}{c}k_1+k_2=k-1\\k_1\geq 1\end{array}$}}}(\slashed{\mathcal L}_{\bar Z}^{k_1}\leftidx^{{(\bar Z)}}{\slashed\pi}_{\mathring L}^A)(\slashed d_A\bar F^{k_2})}_{\textrm{vanish when}\ k=1}+\bar Z^{k}\t e+\sum_{{\mbox{\tiny$\begin{array}{c}k_1+k_2=k\\k_1\geq 1\end{array}$}}}\Big\{-2(\slashed{\mathcal L}_{\bar Z}^{k_1}\check\chi^{AB})(\slashed{\mathcal L}_{\bar Z}^{k_2}\slashed\nabla_{AB}^2\mu)\\
&+\bar Z^{k_1}(-\f2\varrho+\tilde{\mathcal E})\bar F^{k_2}-(\slashed\nabla_B\slashed{\mathcal L}_{\bar Z}^{k_2}\check\chi^{AB})\slashed{\mathcal L}_{\bar Z}^{k_1}\mathcal Q_A+(\slashed d^A\bar Z^{k_2}\textrm{tr}\chi)\slashed{\mathcal L}_{\bar Z}^{k_1}\mathcal Y_A\Big\}\\
&+\sum_{k_1+k_2+k_3=k-1}\bar Z^{k_1}\Big\{\leftidx^{{(\bar Z)}}{\slashed\pi}^{BC}(\slashed\nabla_C\slashed{\mathcal L}_{\bar Z}^{k_2}\check\chi_{AB})\slashed{\mathcal L}_{\bar Z}^{k_3}\mathcal Q^A+\slashed g^{BC}\big(\check{\slashed\nabla}_C\leftidx^{{(\bar Z)}}{\slashed\pi}_A^D(\slashed{\mathcal L}_{\bar Z}^{k_2}\check\chi_{BD})\\
&\qquad\qquad\qquad\qquad\qquad+(\check{\slashed\nabla}_C\leftidx^{{(\bar Z)}}{\slashed\pi}_B^D)\slashed{\mathcal L}_{\bar Z}^{k_2}\check\chi_{AD}\big)\slashed{\mathcal L}_{\bar Z}^{k_3}\mathcal Q^A\Big\}
\end{split}
\end{equation}
when $k\geq 1$.

Due to
\begin{equation}\label{GTL}
\begin{split}
G_{\mathring L\mathring L}^\gamma T\underline{\mathring L}\varphi_\gamma=G_{\mathring L\mathring L}^\gamma T(\mu\mathring L\varphi_\gamma)+2T\big(G_{\mathring L\mathring L}^\gamma T\varphi_\gamma\big)-2(TG_{\mathring L\mathring L}^\gamma)T\varphi_\gamma,
\end{split}
\end{equation}
then $\delta^l|\bar F^k(t_0,u,\vartheta)|\lesssim\delta^{-2\varepsilon_0}$ holds with the help of \eqref{GT}.
Setting $F(s,u,\vartheta)=\varrho^2\bar F^k(s,u,\vartheta)-\varrho_0^2\bar F^k(t_0,u,\vartheta)$ in \eqref{Ff},
and then one can use \eqref{LbarF} to get
\begin{equation}\label{tildeF}
\begin{split}
&\delta^l\varrho^{1/2}\|\bar F^k\|_{L^2(\Sigma_s^u)}\\
\lesssim&\delta^{1/2-2\varepsilon_0}+\int_{t_0}^s\delta^{1-\varepsilon_0}\tau^{-3/2}\delta^l\|\slashed{\mathcal L}_{\bar Z}^k\slashed\nabla^2\mu\|_{L^2(\Sigma_\tau^u)}d\tau+\int_{t_0}^s\tau^{1/2}\delta^l\|\bar e^k\|_{L^2(\Sigma_\tau^u)}d\tau\\
&+\int_{t_0}^s\delta^{1-2\varepsilon_0+l}\tau^{-1/2}\|\slashed d\bar Z^k\textrm{tr}\chi\|_{L^2(\Sigma_\tau^u)}d\tau+\int_{t_0}^s\delta^{-\varepsilon_0+l}\tau^{-1/2}\|\slashed\nabla\slashed{\mathcal L}_{\bar Z}^k\check\chi\|_{L^2(\Sigma_\tau^u)}d\tau.
\end{split}
\end{equation}
In the right hand side of \eqref{tildeF}, the last two integrals can be estimated by \eqref{d}
as
\begin{equation}\label{Y-23}
\begin{split}
&\int_{t_0}^s\delta^{1-2\varepsilon_0+l}\tau^{-1/2}\|\slashed d\bar Z^k\textrm{tr}\chi\|_{L^2(\Sigma_\tau^u)}d\tau+\int_{t_0}^s\delta^{-\varepsilon_0+l}\tau^{-1/2}\|\slashed\nabla\slashed{\mathcal L}_{\bar Z}^k\check\chi\|_{L^2(\Sigma_\tau^u)}d\tau\\
\lesssim&\delta^{3/2-3\varepsilon_0}+\delta^{-2\varepsilon_0}\sqrt{\tilde E_{1,\leq k+2}(s,u)}+\delta^{1-2\varepsilon_0}\sqrt{\tilde E_{2,\leq k+2}(s,u)}.
\end{split}
\end{equation}
In addition, it follows from \eqref{ee}, Lemma \ref{commute} and Proposition \ref{L2chi} that
\begin{equation*}\label{Y-24}
\begin{split}
&\int_{t_0}^s\delta^{1-\varepsilon_0}\tau^{-3/2}\delta^l\|\slashed{\mathcal L}_{\bar Z}^k\slashed\nabla^2\mu\|_{L^2(\Sigma_\tau^u)}d\tau\\
\lesssim&\delta^{1+l-\varepsilon_0}\int_{t_0}^s\tau^{-3/2}\|[\slashed{\mathcal L}_{\bar Z}^k,\slashed\nabla^2]\mu\|_{L^2(\Sigma_\tau^u)}d\tau+\delta^{1+l-\varepsilon_0}\int_{t_0}^s\tau^{-3/2}\|[\bar Z^k,\slashed\triangle]\mu\|_{L^2(\Sigma_\tau^u)}d\tau\\
&+\delta^{1+l-\varepsilon_0}\int_{t_0}^s\tau^{-3/2}\|\bar Z^k\slashed\triangle\mu\|_{L^2(\Sigma_\tau^u)}d\tau+\delta^{1+l-\varepsilon_0}\int_{t_0}^s\tau^{-5/2}\|\slashed d\bar Z^k\mu\|_{L^2(\Sigma_\tau^u)}d\tau\\
\end{split}
\end{equation*}
\begin{equation}\label{Y-24}
\begin{split}
\lesssim&\delta^{2-3\varepsilon_0}\int_{t_0}^s\tau^{-7/2}\delta^{l_1}\|\slashed{\mathcal L}_{\bar Z}^{k_1}\leftidx^{{(R)}}{\slashed\pi}\|_{L^2(\Sigma_\tau^u)}d\tau+\delta^{3-3\varepsilon_0}\int_{t_0}^s\tau^{-7/2}\delta^{l_1}\|\slashed{\mathcal L}_{\bar Z}^{k_1}\leftidx^{{(T)}}{\slashed\pi}\|_{L^2(\Sigma_\tau^u)}d\tau\\
&+\delta^{1-\varepsilon_0+l}\int_{t_0}^s\tau^{-3/2}\|\bar Z^k\tilde E\|_{L^2(\Sigma_\tau^u)}d\tau+\delta^{1+l-\varepsilon_0}\int_{t_0}^s\tau^{-5/2}\|\slashed d\bar Z^k\mu\|_{L^2(\Sigma_\tau^u)}d\tau\\
&+\delta^{1-\varepsilon_0+l}\int_{t_0}^s\tau^{-3/2}\|\bar F^k\|_{L^2(\Sigma_\tau^u)}d\tau\\
\lesssim&\delta^{1-\varepsilon_0+l}\int_{t_0}^s\tau^{-3/2}\|\bar F^k\|_{L^2(\Sigma_\tau^u)}d\tau+\delta^{3/2-3\varepsilon_0}+\delta^{-2\varepsilon_0}\sqrt{\tilde E_{1,\leq k+2}}+\delta^{1-2\varepsilon_0}\sqrt{\tilde E_{2,\leq k+2}},
\end{split}
\end{equation}
where one has used the following fact that
\begin{equation}\label{ZkE}
\begin{split}
&\delta^l\|\bar Z^k\tilde E\|_{L^2(\Sigma_s^u)}\\
\lesssim&\delta^{-2\varepsilon_0}s^{-3/2}\delta^{l_0}\|Z^{k_0}\mu\|_{L^2(\Sigma_s^u)}+\delta^{-\varepsilon_0-1}s^{-3/2}\delta^{l_0}\|Z^{k_0}\check L^i\|_{L^2(\Sigma_s^u)}+\delta^{-2\varepsilon_0}s^{-5/2}\delta^{l_{-1}}\|Z^{k_{-1}}x^i\|_{L^2(\Sigma_s^u)}\\
&+\delta^{1-\varepsilon_0}s^{-7/2}\delta^{l_1}\|\slashed{\mathcal L}_{ Z}^{k_1}\leftidx^{{(R)}}{\slashed\pi}\|_{L^2(\Sigma_s^u)}+\delta^{2-\varepsilon_0}s^{-7/2}\delta^{l_1}\|\slashed{\mathcal L}_{Z}^{k_1}\leftidx^{{(T)}}{\slashed\pi}\|_{L^2(\Sigma_s^u)}+s^{-1}\delta^{l_0}\|\slashed dZ^{k_0}\varphi\|_{L^2(\Sigma_s^u)}\\
&+\delta^{-\varepsilon_0}s^{-1}\delta^{l_0}\|TZ^{k_0}\varphi\|_{L^2(\Sigma_s^u)}+\delta^{-1}\delta^{l_0}\|Z^{k_0}T\phi\|_{L^2(\Sigma_s^u)}+\delta^{-1}\delta^{l_0}\|Z^{k_0}\mathring L\phi\|_{L^2(\Sigma_s^u)}\\
&+s^{-2}\delta^{l_0}\|Z^{k_0}R\phi\|_{L^2(\Sigma_s^u)}+\delta^{-1-\varepsilon_0}s^{-1}\delta^{l_0}\|Z^{k_0}\varphi\|_{L^2(\Sigma_s^u)}+\delta^{-1-\varepsilon_0}s^{-3/2}\delta^{l_0}\|Z^{k_0}\phi\|_{L^2(\Sigma_s^u)}\\
&+\delta^{1-2\varepsilon_0}s^{-7/2}\delta^{l_1}\|\slashed{\mathcal L}_{ Z}^{k_1}\leftidx^{{(R)}}{\slashed\pi}_T\|_{L^2(\Sigma_s^u)}\\
\lesssim&\delta^{1/2-2\varepsilon_0}+\delta^{-\varepsilon_0-1}\sqrt{\tilde E_{1,\leq k+2}(s,u)}+\delta^{-\varepsilon_0}\sqrt{\tilde E_{2,\leq k+2}(s,u)},
\end{split}
\end{equation}
which is derived by \eqref{GTL}, \eqref{GT}, Proposition \ref{L2chi} and Corollary \ref{Zphi}.
By \eqref{bare} and direct computations, one can estimate the $L^2$ norm of $\bar e^k$ on $\Sigma_s^u$ as
\begin{equation}\label{barek}
\begin{split}
\delta^l\|\bar e^k\|_{L^2(\Sigma_s^u)}\lesssim&\delta^{1/2-2\varepsilon_0}s^{-1}+\delta^{-\varepsilon_0-1}s^{-1}\sqrt{\tilde E_{1,\leq k+2}}+\delta^{-2\varepsilon_0}s^{-1}\sqrt{\tilde E_{2,\leq k+2}}\\
&+\delta^{1-2\varepsilon_0}s^{-3/2}\delta^l\|\bar F^k\|_{L^2(\Sigma_s^u)}.
\end{split}
\end{equation}
Note that the term $\delta^{1-2\varepsilon_0}s^{-3/2}\delta^l\|\bar F^k\|_{L^2(\Sigma_t^u)}$ appeared
in \eqref{barek} is derived from the term $\f12\mu^{-1}(G_{\mathring L\mathring L}^\gamma T\varphi_\gamma)\bar Z^k T\text{tr}\check\chi$ in $\bar e^k$, which originates from $\bar Z^k(-\f12\mu^{-1}G_{\mathring L\mathring L}^\gamma TH_\gamma)$ in $\bar Z^k\tilde e$.
By the identity \eqref{Tchi'}, $\f12\mu^{-1}(G_{\mathring L\mathring L}^\gamma T\varphi_\gamma)\bar Z^k T\text{tr}\check\chi$ contains
the term $\f12\mu^{-1}(G_{\mathring L\mathring L}^\gamma T\varphi_\gamma)\bar Z^k\slashed\triangle\mu$ which cannot be estimated
directly by Proposition \ref{L2chi} since the resulting orders of derivatives of $\mu$ exceed $k+1$. In order to overcome this
difficulty, one can treat  the term $\bar F^k+\bar Z^k\tilde E$ instead of $Z^k\slashed\triangle\mu$.

Substituting \eqref{Y-23}-\eqref{barek}  into \eqref{tildeF} and applying Gronwall's inequality yield
\begin{equation*}
\begin{split}
\delta^l\varrho^{1/2}\|\bar F^k\|_{L^2(\Sigma_s^u)}\lesssim&\delta^{1/2-2\varepsilon_0}s^{1/2}+\delta^{-\varepsilon_0-1}s^{1/2}\sqrt{\tilde E_{1,\leq k+2}}+\delta^{-2\varepsilon_0}s^{1/2}\sqrt{\tilde E_{2,\leq k+2}},
\end{split}
\end{equation*}
and hence,
\begin{equation*}
\begin{split}
\delta^l\|\bar Z^k\slashed\triangle\mu\|_{L^2(\Sigma_s^u)}\lesssim&\delta^{1/2-2\varepsilon_0}+\delta^{-\varepsilon_0-1}\sqrt{\tilde E_{1,\leq k+2}(s,u)}+\delta^{-2\varepsilon_0}\sqrt{\tilde E_{2,\leq k+2}(s,u)}
\end{split}
\end{equation*}
with the help of \eqref{ZkE}.

For other cases  containing at least one $\varrho\mathring L$ in $Z^k$, for examples, $Z^k\slashed\triangle\mu=Z^{p_1}(\varrho\mathring L)Z^{p_2}\slashed\triangle\mu$ with $p_1+p_2=k-1$, one can make use of the
commutators $[\mathring L,\bar Z]$ and $[\mathring L,\slashed\triangle]$ and subsequently
utilize the transport equation \eqref{lmu}
to obtain  the related $L^2$ norm estimates.
Therefore, by \eqref{ee} and Proposition \ref{L2chi}, we eventually obtain that
\begin{equation}\label{Zmu}
\begin{split}
&\delta^l\|Z^k\slashed\triangle\mu\|_{L^2(\Sigma_s^u)}+\delta^l\|\slashed{\mathcal L}_Z^k\slashed\nabla^2\mu\|_{L^2(\Sigma_s^u)}\\
\lesssim&\delta^{1/2-2\varepsilon_0}+\delta^{-\varepsilon_0-1}\sqrt{\tilde E_{1,\leq k+2}(s,u)}
+\delta^{-2\varepsilon_0}\sqrt{\tilde E_{2,\leq k+2}(s,u)},
\end{split}
\end{equation}
where $Z\in\{\varrho\mathring L,R_{ij},T\}$.

\section{Estimates for the error terms}\label{ert}
After the preparations for the $L^2$ estimates of the related quantities in Section \ref{hoe} and \ref{L2chimu}, we are ready to
handle the error terms $\delta\int_{D^{s, u}}|\Phi\cdot\mathring{\underline L}\Psi|$ and
$\int_{D^{s, u}}\varrho^{2\epsilon}|\Phi\cdot \mathring L\Psi|$ in \eqref{e}, and then get the final energy estimates for $\vp$.

Recall $\Psi=Z^{k+1}\vp_\g$ and $\Phi=\Phi_\g^{k+1}$, which has been explicitly given in \eqref{Phik}. We need to
treat each term in \eqref{Phik}.

\subsection{Treatments on the first line of \eqref{Phik}}\label{tot}

First, rewrite $Z^{k+1}$ as $Z_{k+1}Z_{k}\cdots Z_2Z_1$ with $Z_i\in\{\varrho\mathring L,T,R_{ij}\}$ and set $\varphi_\gamma^n=\begin{cases}Z_n\cdots Z_1\varphi_\gamma,\ &n\geq 1,\\\varphi_\gamma,\ &n=0.\end{cases}$
 \begin{enumerate}[(1)]
  \item By \eqref{Phik}, one needs to treat the derivatives of $\mu\mathscr D^{\al}{\leftidx{^{(Z)}}C_\gamma^{n}}_{,\al}$ $(0\leq n\leq k)$. To this end, one focuses first on $\leftidx{^{(Z)}}D_{\gamma,1}^{n}$ and $\leftidx{^{(Z)}}D_{\gamma,3}^{n}$ in \eqref{muZC}, which do not contain the top order derivatives of $\varphi_\gamma$.  Substituting \eqref{Lpi}-\eqref{Rpi} into \eqref{D1} and \eqref{D3} directly yields
\begin{equation}\label{T1}
\begin{split}
\leftidx{^{(T)}}D_{\gamma,1}^{n}=&(T\mu)\mathring L^2\varphi_\gamma^{n}+\mu(\slashed d_A\mu+2\mu\zeta_A)\slashed d^A\mathring L\varphi_\gamma^{n}+\f12\textrm{tr}\leftidx{^{(T)}}{\slashed\pi}(\mathring L\mathring{\underline L}\varphi_\gamma^{n}+\f12\textrm{tr}\chi\mathring{\underline L}\varphi_\gamma^{n})\\
&+(\slashed d_A\mu+2\mu\zeta_A)\slashed d^A\mathring{\underline L}\varphi_\gamma^{n}-(T\mu)\slashed\triangle\varphi_\gamma^{n}+\mu\big(\leftidx{^{(T)}}{\slashed\pi}^{AB}
-\f12\text{tr}\leftidx{^{(T)}}{\slashed\pi}\slashed g^{AB}\big)\slashed\nabla_{AB}^2\varphi_\gamma^{n},
\end{split}
\end{equation}
\begin{equation}\label{T3}
\begin{split}
\leftidx{^{(T)}}D_{\gamma,3}^{n}=&\big\{\textrm{tr}\chi T\mu+\f14(\mu\textrm{tr}\chi+\textrm{tr}\leftidx{^{(T)}}{\slashed \pi})\textrm{tr}\leftidx{^{(T)}}{\slashed \pi}-\f12|\slashed d\mu|^2-\mu\zeta_A(\slashed d^A\mu)\big\}\mathring L\varphi_\gamma^{n}
+\big\{(\f12\mathring L\mu\\
&+\textrm{tr}\leftidx{^{(T)}}{\slashed \pi}+\mu\textrm{tr}\chi)(\slashed d_A\mu+2\mu\zeta_A)-(\slashed d^B\mu+2\mu\zeta^B)(2\mu\chi_{AB}+\leftidx{^{(T)}}{\slashed \pi}_{AB})\big\}\slashed d^A\varphi_\gamma^{n},
\end{split}
\end{equation}
\begin{equation}\label{rL1}
\begin{split}
\leftidx{^{(\varrho\mathring L)}}D_{\gamma,1}^{n}=&(2-\mu+\varrho\mathring L\mu)\mathring L^2\varphi_\gamma^{n}-2\varrho(\slashed d_A\mu+2\mu\zeta_A)\slashed d^A\mathring L\varphi_\gamma^{n}+\varrho\textrm{tr}\chi(\mathring L\mathring{\underline L}\varphi_\gamma^{n}+\f12\textrm{tr}\chi\mathring{\underline L}\varphi_\gamma^{n})\\
&-(2\mu+\varrho\mathring L\mu)\slashed\triangle\varphi_\gamma^{n}+\mu\varrho(2\check\chi^{AB}-\textrm{tr}\check\chi\slashed g^{AB})\slashed\nabla_{AB}^2\varphi_\gamma^{n},
\end{split}
\end{equation}
\begin{equation}\label{rL3}
\begin{split}
\leftidx{^{(\varrho\mathring L)}}D_{\gamma,3}^{n}=&\textrm{tr}\chi\big\{2-\mu+\varrho\mathring L\mu+\f12\varrho\text{tr}\leftidx{^{(T)}}{\slashed\pi}+\f12\varrho\mu\textrm{tr}\chi\big\}\mathring L\varphi_\gamma^{n}+2\varrho\chi_{AB}(2\mu\zeta^B+\slashed d^B\mu)\slashed d^A\varphi_\gamma^n,
\end{split}
\end{equation}
\begin{equation}\label{R1}
\begin{split}
\leftidx{^{(R_{ij})}}D_{\gamma,1}^{n}=&(R_{ij}\mu)\mathring L^2\varphi_\gamma^{n}-\leftidx{^{(R_{ij})}}{\slashed\pi}_{\mathring{\underline L}A}\slashed d^A\mathring L\varphi_\gamma^{n}+\f12\textrm{tr}\leftidx{^{(R_{ij})}}{\slashed\pi}(\mathring L\mathring{\underline L}\varphi_\gamma^{n}+\f12\textrm{tr}\chi\mathring{\underline L}\varphi_\gamma^{n})\qquad\quad\qquad\\
&-\leftidx{^{(R_{ij})}}{\slashed\pi}_{\mathring LA}\slashed d^A\mathring{\underline L}\varphi_\gamma^{n}-(R_{ij}\mu)\slashed\triangle\varphi_\gamma^{n}
+\mu(\leftidx{^{(R_{ij})}}{\slashed\pi}^{AB}-\f12\big(\text{tr}\leftidx{^{(R_{ij})}}{\slashed\pi}\big)\slashed g^{AB})\slashed\nabla^2_{AB}\varphi_\gamma^{n},
\end{split}
\end{equation}
\begin{equation}\label{R3}
\begin{split}
\leftidx{^{(R_{ij})}}D_{\gamma,3}^{n}=&\big\{\textrm{tr}\chi R_{ij}\mu+\f14(\textrm{tr}\leftidx{^{(T)}}{\slashed\pi}+\mu\textrm{tr}\chi)\textrm{tr}\leftidx{^{(R_{ij})}}{\slashed\pi}+\f12\slashed d^A\mu\leftidx{^{(R_{ij})}}{\slashed\pi}_{\mathring LA}\big\}\mathring L\varphi_\gamma^{n}\\
&+\big\{2\chi_{AB}\leftidx{^{(R_{ij})}}{\slashed\pi}_{T}^B+\textrm{tr}\leftidx{^{(R_{ij})}}{\slashed\pi}(\mu\zeta_A+\f12\slashed d_A\mu)+(2\mu\chi_{AB}+\leftidx{^{(T)}}{\slashed\pi}_{AB })\leftidx{^{(R_{ij})}}{\slashed\pi}_{\mathring L}^B\quad\\
&-\textrm{tr}\chi\leftidx{^{(R_{ij})}}{\slashed\pi}_{TA}-(\f12\mathring L\mu+\mu\text{tr}\chi+\f12\text{tr}\leftidx{^{(T)}}{\slashed\pi})\leftidx{^{(R_{ij})}}{\slashed\pi}_{\mathring LA}\big\}\slashed d^A\varphi_\gamma^{n}.
\end{split}
\end{equation}

It can be checked directly that each term in $\leftidx{^{(Z)}}D_{\gamma,1}^{n}$
contains the term $\mathring L\mathring{\underline L}\varphi_\gamma^{n}+\f12\textrm{tr}\chi\mathring{\underline L}\varphi_\gamma^{n}$
which can be treated with the help of \eqref{fequation} as explained after \eqref{D3}.
For the first line of \eqref{Phik}, it is clear that at most
 the $(k-n)^{th}$ derivatives of $\leftidx{^{(Z)}}D_{\gamma,i}^{n}$ $(i=1,3)$
 are taken. Thus, the $L^2$ norms of all terms in the first line of $\eqref{Phik}$ can be estimated by
the corresponding $L^\infty$ estimates in Section \ref{ho}
and the related $L^2$ estimates in Proposition \ref{L2chi}, which do not need the  results on the highest order derivatives
in Section \ref{L2chimu}.
Therefore,

\begin{equation}\label{D11}
\begin{split}
&\delta^{2l+1}|\int_{D^{s, u}}\sum_{m=1}^{k}\big(Z_{k+1}+\leftidx{^{(Z_{k+1})}}\lambda\big)\dots\big(Z_{k+2-m}
+\leftidx{^{(Z_{k+2-m})}}\lambda\big)\leftidx{^{(Z_{k+1-m})}}D_{\gamma,1}^{k-m}\cdot\mathring{\underline L}\varphi_\gamma^{k+1}|\\
\lesssim&\delta^{2l+1}\int_{t_0}^s\sum_{m=1}^{k}\|\big(Z_{k+1}+\leftidx{^{(Z_{k+1})}}\lambda\big)\dots\big(Z_{k+2-m}
+\leftidx{^{(Z_{k+2-m})}}\lambda\big)\leftidx{^{(Z_{k+1-m})}}D_{\gamma,1}^{k-m}\|_{L^2(\Sigma_{\tau}^{u})}\\
&\qquad\qquad\cdot\|\mathring{\underline L}Z^{k+1}\varphi_\gamma\|_{L^2(\Sigma_{\tau}^{u})}d\tau\\
\lesssim&\delta^{3-2\varepsilon_0}+\int_{t_0}^s\tau^{-1-\epsilon}\tilde E_{1,\leq k+2}(\tau,u)d\tau+\delta^{2-2\varepsilon_0}\int_{t_0}^s\tau^{-1-\epsilon}\tilde E_{2,\leq k+2}(\tau,u)d\tau.
\end{split}
\end{equation}
Similarly,
\begin{equation}\label{D33}
\begin{split}
&\delta^{2l+1}|\int_{D^{s, u}}\sum_{m=1}^{k}\big(Z_{k+1}+\leftidx{^{(Z_{k+1})}}\lambda\big)\dots\big(Z_{k+2-m}
+\leftidx{^{(Z_{k+2-m})}}\lambda\big)\leftidx{^{(Z_{k+1-m})}}D_{\gamma,3}^{k-m}\cdot\mathring{\underline L}\varphi_\gamma^{k+1}|\\
\lesssim&\delta^{3-2\varepsilon_0}+\delta^{2-2\varepsilon_0}\int_{t_0}^s\tau^{-2}\tilde E_{1,\leq k+2}(\tau,u)d\tau+\delta^{2-2\varepsilon_0}\int_{t_0}^s\tau^{-2}\tilde E_{2,\leq k+2}(\tau,u)d\tau.
\end{split}
\end{equation}

The corresponding terms in $\delta^{2l}\int_{D^{s, u}}\varrho^{2\epsilon}|\Phi\cdot \mathring L\Psi|$ can also be estimated as follows
\begin{equation*}
\begin{split}
&\delta^{2l}\int_{D^{s, u}}|\sum_{m=1}^{k}\varrho^{2\epsilon}\big(Z_{k+1}+\leftidx{^{(Z_{k+1})}}\lambda\big)\dots\big(Z_{k+2-m}
+\leftidx{^{(Z_{k+2-m})}}\lambda\big)\big(\leftidx{^{(Z_{k+1-m})}}D_{\gamma,1}^{k-m}\\
&\qquad\qquad\qquad+\leftidx{^{(Z_{k+1-m})}}D_{\gamma,3}^{k-m}\big)\mathring L\varphi_\gamma^{k+1}|\\
\lesssim&\delta^{2l+1}\int_{D^{s, u}}\varrho^{2\epsilon}\Big\{\sum_{m=1}^{k}\big(Z_{k+1}+\leftidx{^{(Z_{k+1})}}\lambda\big)\dots\big(Z_{k+2-m}
+\leftidx{^{(Z_{k+2-m})}}\lambda\big)\big(\leftidx{^{(Z_{k+1-m})}}D_{\gamma,1}^{k-m}\\
&\qquad+\leftidx{^{(Z_{k+1-m})}}D_{\gamma,3}^{k-m}\big)\Big\}^2+\delta^{2l-1}\int_{D^{s, u}}\varrho^{2\epsilon}|\mathring LZ^{k+1}\varphi_\gamma|^2\\
\end{split}
\end{equation*}

\begin{equation}\label{D13}
\begin{split}
\lesssim&\delta^{4-4\varepsilon_0}+\delta^{1-2\varepsilon_0}\int_{t_0}^s\tau^{-2}\tilde E_{1,\leq k+2}(\tau,u)d\tau+\delta^{3-2\varepsilon_0}\int_{t_0}^s\tau^{-3+2\epsilon}\tilde E_{2,\leq k+2}(\tau,u)d\tau\\
&+\delta^{-1}\int_0^uF_{1,k+2}(s,u')du'.
\end{split}
\end{equation}

\item Next, we treat the terms $\leftidx{^{(Z)}}D_{\gamma,2}^n$ $(0\leq n\leq k)$. The most special case is $n=0$,
which corresponds to $m=k$ in the summation of \eqref{Phik}. In this case, the highest order derivatives
in $\leftidx{^{(Z)}}D_{\gamma,2}^0$ are of order $k$. Thus there appear some terms containing the $(k+1)^{th}$
order derivatives of the deformation tensors. But it can be checked by \eqref{Lpi}-\eqref{Rpi} that $\leftidx{^{(Z)}}D_{\gamma,2}^0$ contains terms such as $\slashed\triangle\mu$ and $\slashed\nabla\check\chi$,
which makes it difficult to apply Proposition \ref{L2chi} to estimate the $L^2$ norm of $Z^k(\leftidx{^{(Z)}}D_{\gamma,2}^0)$
directly. Otherwise in the energy inequalities \eqref{e} the factors $\tilde E_{1,\leq k+3}(s,u)$ and $\tilde E_{2,\leq k+3}(s,u)$ will appear in the right
hand side, which can not be absorbed by the lower order energies on the left hand side. Thus, it is necessary to  examine the expression of $\leftidx{^{(Z)}}D_{\gamma,2}^k$ carefully and apply  the estimates in Section \ref{L2chimu} to deal with the higher order
derivatives of $\chi$ and $\mu$.
 Indeed, it follows from direct computations that
\begin{equation}\label{T2}
\begin{split}
\leftidx{^{(T)}}D_{\gamma,2}^{n}=&\big\{\mathring L T\mu+\f14\uwave{\mathring{\underline L}(\text{tr}\leftidx{^{(T)}}{\slashed\pi})}+\boxed{\slashed\nabla_A\big(\f12\mu\slashed d^A\mu}+\mu^2\zeta^A\big)\big\}\mathring L\varphi_\gamma^n+\big\{\f14\mathring L(\textrm{tr}\leftidx{^{(T)}}{\slashed\pi})\\
&+\f12\boxed{\slashed\nabla_A(\slashed d^A\mu}+2\mu\zeta^A)\big\}\mathring{\underline L}\varphi_\gamma^n+\big\{\f12\slashed{\mathcal L}_{\mathring L}(2\mu^2\zeta_A+\mu\slashed d_A\mu)-\uline{\slashed d_AT\mu}\\
&+\underline{\f12\slashed{\mathcal L}_{\mathring{\underline L}}(\slashed d_A\mu}+2\mu\zeta_A)+\underbrace{\slashed\nabla^B(\mu\leftidx{^{(T)}}{\slashed\pi}_{AB}
-\f12\mu(\text{tr}\leftidx{^{(T)}}{\slashed\pi})\slashed g_{AB})}\big\}\slashed d^A\varphi_\gamma^n,
\end{split}
\end{equation}
\begin{equation}\label{rL2}
\begin{split}
\leftidx{^{(\varrho\mathring L)}}D_{\gamma,2}^{n}=&\big\{\mathring L\big(\varrho\mathring L\mu-\mu\big)+\uwave{\f12\mathring{\underline L}\big(\varrho\textrm{tr}\check\chi\big)}-\varrho\boxed{\slashed\nabla_A\big(\slashed d^A\mu}+2\mu\zeta^A\big)\big\}\mathring L\varphi_\gamma^n+\f12\mathring L\big(\varrho\textrm{tr}\check\chi\big)\mathring{\underline L}\varphi_\gamma^n\\
-\big\{&\slashed{\mathcal L}_{\mathring L}\big(\varrho\slashed d_A\mu+2\varrho\mu\zeta_A\big)+\slashed d_A\big(\mu+\varrho\mathring L\mu\big)-\varrho\underbrace{\slashed\nabla^B\big(2\mu\chi_{AB}-\mu\textrm{tr}\chi\slashed g_{AB}\big)}\big\}\slashed d^A\varphi_\gamma^n,
\end{split}
\end{equation}
\begin{equation}\label{R2}
\begin{split}
\leftidx{^{(R_{ij})}}D_{\gamma,2}^n=&\big\{\mathring LR\mu
-\f12\underbrace{\slashed\nabla^A\leftidx{^{(R_{ij})}}{\slashed\pi}_{\mathring{\underline L}A}}+\f14\uwave{\mathring{\underline L}(\textrm{tr}\leftidx{^{(R_{ij})}}{\slashed\pi})}\big\}\mathring L\Psi_\g^n+\big\{\f14\mathring L(\textrm{tr}\leftidx{^{(R_{ij})}}{\slashed{\pi}})\\
&-\f12\underbrace{\slashed\nabla^A\leftidx{^{(R_{ij})}}{\slashed\pi}_{\mathring LA}}\big\}\mathring{\underline L}\Psi_\gamma^n+\big\{-\f12\slashed{\mathcal L}_{\mathring L}\leftidx{^{(R_{ij})}}{\slashed\pi}_{\mathring{\underline L}A}-\f12\uwave{\slashed{\mathcal L}_{\mathring{\underline L}}\leftidx{^{(R_{ij})}}{\slashed\pi}_{\mathring LA}}
-\boxed{\slashed d_AR_{ij}\mu}\\
&+\underbrace{\slashed\nabla^B\big(\mu\leftidx{^{(R_{ij})}}{{\slashed\pi}}_{AB}
-\f{\mu}{2}\textrm{tr}\leftidx{^{(R_{ij})}}{\slashed\pi}\slashed g_{AB}\big)}\big\}\slashed d^A\Psi_\gamma^n.
\end{split}
\end{equation}

Some attentions are needed to deal with the terms with underlines, wavy lines, boxes, or braces
in the above three equalities \eqref{T2}-\eqref{R2}. In $\leftidx{^{(T)}}D_{\gamma,2}^{n}$,
due to $\f12\slashed{\mathcal L}_{\mathring{\underline L}}\slashed d_A\mu=\slashed d_AT\mu+\f12\mu\slashed d_A\mathring L\mu$ by $\mathring{\underline L}=\mu\mathring L+2T$, then the
 corresponding underline part is
\begin{equation}\label{Y-25}
-\slashed d_AT\mu+\f12\slashed{\mathcal L}_{\mathring{\underline L}}\slashed d_A\mu=\f12\mu\slashed d_A\mathring L\mu,
\end{equation}
which can be estimated by using \eqref{lmu}.
Note that $\textrm{tr}\leftidx{^{(T)}}{\slashed\pi}=-2\mu\slashed g^{AB}\check\chi_{AB}+\dots$, $\textrm{tr}\leftidx{^{(R_{ij})}}{\slashed\pi}=2\upsilon_{ij}\slashed g^{AB}\check\chi_{AB}+\cdots$ and $\leftidx{^{(R_{ij})}}{\slashed\pi}_{\mathring LA}=-R_{ij}^B\check\chi_{AB}+\cdots$ hold true by \eqref{Lpi}, \eqref{zeta} and \eqref{Rpi},
respectively.
Then all terms with wavy lines in \eqref{T2}-\eqref{R2} contain the same factor $\slashed{\mathcal L}_{\mathring{\underline L}}\check\chi_{AB}$. Thanks to \eqref{Tchi'} and \eqref{Lchi'}, $\slashed{\mathcal L}_{\mathring{\underline L}}\check\chi$ can be replaced by $\slashed\nabla^2\mu+\cdots$. This leads to that in the process of estimating $\|Z^k(\leftidx{^{(T)}}D_{\gamma,2}^{0})\|_{L^2(\Sigma_s^u)}$, $\|Z^k(\leftidx{^{(\varrho\mathring L)}}D_{\gamma,2}^{0})\|_{L^2(\Sigma_s^u)}$ and $\|Z^k(\leftidx{^{(R_{ij})}}D_{\gamma,2}^{0})\|_{L^2(\Sigma_s^u)}$, the most troublesome factor related to the wavy lines in
\eqref{T2}-\eqref{R2}  is $\|\slashed{\mathcal L}_Z^k\slashed{\nabla}^2\mu\|_{L^2(\Sigma_s^u)}$ but which can be treated well
by \eqref{Zmu}.
In addition, the terms with boxes and braces
can be estimated by utilizing \eqref{Zmu} and \eqref{d}.

On the other hand, it is noticed that in \eqref{T2}-\eqref{R2}, there are some terms whose factors are the derivatives of $\leftidx{^{(Z)}}{\pi}$
with respect to $\mathring L$, for example, the term $\f14{\mathring L}(\text{tr}\leftidx{^{(T)}}{\slashed\pi})\underline{\mathring L}\varphi_\gamma^n$ appears in \eqref{T2}. In fact, these terms are not the ``bad" ones since the derivatives of $\mathring L$
for the involved deformation tensors are
 actually equipped with the ``good" quantities due to \eqref{lmu} and \eqref{Lchi'} after examining
 each term in \eqref{Lpi} and \eqref{Rpi}.

In summary, we eventually arrive at
\begin{equation}\label{D22}
\begin{split}
&\delta^{2l+1}|\int_{D^{s, u}}\sum_{m=1}^{k}\big(Z_{k+1}+\leftidx{^{(Z_{k+1})}}\lambda\big)\dots\big(Z_{k+2-m}
+\leftidx{^{(Z_{k+2-m})}}\lambda\big)\leftidx{^{(Z_{k+1-m})}}D_{\gamma,2}^{k-m}\cdot\mathring{\underline L}\varphi_\gamma^{k+1}|\\
\lesssim&\delta^{4-6\varepsilon_0}+\int_{t_0}^s\tau^{-3/2}\tilde E_{1,\leq k+2}(\tau,u)d\tau
+\delta\int_{t_0}^s\tau^{-3/2}\tilde E_{2,\leq k+2}(\tau,u)d\tau
\end{split}
\end{equation}
and
\begin{equation}\label{D22L}
\begin{split}
&\delta^{2l}\int_{D^{s, u}}\sum_{m=1}^{k}|\varrho^{2\epsilon}\big(Z_{k+1}+\leftidx{^{(Z_{k+1})}}\lambda\big)\dots\big(Z_{k+2-m}
+\leftidx{^{(Z_{k+2-m})}}\lambda\big)\leftidx{^{(Z_{k+1-m})}}D_{\gamma,2}^{k-m}\cdot\mathring L\varphi_\gamma^{k+1}|\\
\lesssim&\delta^{4-6\varepsilon_0}+\int_{t_0}^s\tau^{-3+2\epsilon}\tilde E_{1,\leq k+2}(\tau,u)d\tau
+\delta\int_{t_0}^s\tau^{-3+2\epsilon}\tilde E_{2,\leq k+2}(\tau,u)d\tau\\
&\qquad\qquad\qquad+\delta^{-1}\int_0^uF_{1,k+2}(s,u')du'.
\end{split}
\end{equation}

\end{enumerate}

\subsection{Treatments on the other terms in \eqref{Phik}}\label{l}
It remains to treat the terms in the second and third lines in \eqref{Phik}, which do not contain the highest order derivatives of $\textrm{tr}{\chi}$ and $\slashed\triangle\mu$. Therefore, according to Proposition \ref{L2chi} and the expressions
of $\leftidx{^{(Z)}}D_{\g, j}^k$ in \eqref{T1}-\eqref{R3} and \eqref{T2}-\eqref{R2}, we have
\begin{equation}\label{ZL}
\begin{split}
&\delta^{2l+1}\int_{D^{s,u}}|\sum_{j=1}^3\leftidx{^{(Z_{k+1})}}D_{\gamma,j}^k\cdot \mathring{\underline L}{\varphi}_\gamma^{k+1}|\\
\lesssim&\delta^{4-4\varepsilon_0}+\int_{t_0}^s\tau^{-1-\epsilon}\tilde E_{1,\leq k+2}(\tau,u)d\tau
+\delta\int_{t_0}^s\tau^{-1-\epsilon}\tilde E_{2,\leq k+2}(\tau,u)d\tau
\end{split}
\end{equation}
and
\begin{equation}\label{ZrL}
\begin{split}
&\delta^{2l}\int_{D^{s,u}}|\sum_{j=1}^3\leftidx{^{(Z_{k+1})}}D_{\gamma,j}^k\cdot \varrho^{2\epsilon}\mathring L{\varphi}_\gamma^{k+1}|\\
\lesssim&\delta^{4-4\varepsilon_0}+\int_{t_0}^s\tau^{-2}\tilde E_{1,\leq k+2}(\tau,u)d\tau
+\delta\int_{t_0}^s\tau^{-4+2\epsilon}\tilde E_{2,\leq k+2}(\tau,u)d\tau\\
&+\delta^{-1}\int_0^uF_{1,k+2}(s,u')du'.
\end{split}
\end{equation}

In addition, $\Phi_\g^0$ is equal to $\mu\Box_g\varphi_\g$ which is given explicitly in \eqref{ge}.
Then, Proposition \ref{L2chi} and \eqref{lamda} lead to
\begin{equation}\label{Phi0}
\begin{split}
&\delta^{2l+1}\int_{D^{s,u}}|(Z_{k+1}+\leftidx{^{(Z_{k+1})}}\lambda)\cdots(Z_{1}
+\leftidx{^{(Z_{1})}}\lambda)\Phi_\gamma^0\cdot\mathring{\underline L}\varphi_\gamma^{k+1}|\\
\lesssim&\delta^{4-4\varepsilon_0}+\int_{t_0}^s\tau^{-3/2}\tilde E_{1,\leq k+2}(\tau,u)d\tau
+\delta\int_{t_0}^s\tau^{-3/2}\tilde E_{2,\leq k+2}(\tau,u)d\tau
\end{split}
\end{equation}
and
\begin{equation}\label{rPhi}
\begin{split}
&\delta^{2l}\int_{D^{s,u}}\varrho^{2\epsilon}|(Z_{k+1}+\leftidx{^{(Z_{k+1})}}\lambda)\cdots(Z_{1}
+\leftidx{^{(Z_{1})}}\lambda)\Phi_\gamma^0\cdot\mathring L\varphi_\gamma^{k+1}|\\
\lesssim&\delta^{4-4\varepsilon_0}+\int_{t_0}^s\tau^{-3+2\epsilon}\tilde E_{1,\leq k+2}(\tau,u)d\tau
+\delta\int_{t_0}^s\tau^{-3+2\epsilon}\tilde E_{2,\leq k+2}(\tau,u)d\tau\\
&+\delta^{-1}\int_0^uF_{1,k+2}(s,u')du'.
\end{split}
\end{equation}

\section{Global existence of the solution $\phi$ in $A_{2\dl}$}\label{YY}

By substituting \eqref{D11}-\eqref{D13} and \eqref{D22}-\eqref{rPhi} into \eqref{e}, then one gets from
Gronwall's inequality that
\begin{equation}\label{E}
\delta\tilde E_{2,\leq 2N-4}(s,u)+\delta F_{2,\leq 2N-4}(s,u)+\tilde E_{1,\leq 2N-4}(s,u)+F_{1,\leq 2N-4}(s,u)
\lesssim\delta^{2-2\varepsilon_0}.
\end{equation}
Hence the $L^2$ norm of $\phi$ on $\Sigma_s^u$ can be estimated by \eqref{E} and the first inequality of Corollary \ref{Zphi}.
Namely, for $k\leq 2N-5$, it holds that
\begin{equation}\label{Ep}
\delta^l\|Z^{k}\phi\|_{L^2(\Sigma_s^u)}\lesssim\delta^{5/2-\varepsilon_0}.
\end{equation}
Based on \eqref{E} and \eqref{Ep}, one can adopt the following Sobolev-type  embedding formula to close the bootstrap
assumptions $(\star)$ in Section \ref{BA}.
\begin{lemma}\label{HC-F1}
 For any function $f\in H^2(S_{s,u})$, under the assumptions $(\star)$ with suitably small $\delta>0$, then
 \begin{equation}\label{et}
 \|f\|_{L^\infty(S_{s,u})}\lesssim\f{1}{s^{3/2}}\sum_{|\beta|\leq 2}\|R^\beta f\|_{L^2(S_{s, u})},
 \end{equation}
 where $R\in\{R_{ij}:1\leq i<j\leq 4\}$.
\end{lemma}

\begin{proof}
	Its proof follows from Proposition 18.10 in \cite{J}, details are omitted here.
\end{proof}

It follows \eqref{et}, \eqref{E} and \eqref{Ep} that for $k\leq 2N-7$,
\begin{equation}\label{im}
\delta^l|Z^k\varphi_\gamma|\lesssim\f{\delta^l}{s^{3/2}} \sum_{|\beta|\leq 2}\|R^\beta Z^k\varphi_\gamma\|_{L^2(S_{s, u})}\overset{\eqref{SSi}}{\lesssim}\f{\delta^{1/2}}{s^{3/2}}\big(\sqrt{E_{1,\leq 2N-4}}+\sqrt{E_{2,\leq 2N-4}}\big)
\lesssim\delta^{1-\varepsilon_0} s^{-3/2},
\end{equation}
and for $k\leq 2N-8$,
\begin{equation}\label{imp}
\delta^l|Z^k\phi|\lesssim\delta^{1/2}s^{-3/2}\delta^l\sum_{|\beta|\leq 2}(\|\underline{\mathring L}R^\beta Z^k\phi\|_{L^2(\Sigma_s^u)}
+\|{\mathring L}R^\beta Z^k\phi\|_{L^2(\Sigma_s^u)})\lesssim\delta^{2-\varepsilon_0}s^{-3/2}.
\end{equation}

Note that \eqref{im} and \eqref{imp} are independent of $M$ and hence improve the bootstrap assumptions $(\star)$. Thus we have proved the existence of the solution
$\phi$ to equation \eqref{quasi} with \eqref{Y1-1} in the domain $D^{s,4\delta}$ (see Figure 6 in Section \ref{p}).

Moreover, one can refine the estimates on $\check{\varrho}$ (see \eqref{rrho}), $\upsilon_{ij}$ (see \eqref{omega}) and $g_{ij}\check L^i\o^j$,
which will be used to obtain some more precise $L^\infty$-estimates on $\phi$ under the actions of $\Gamma\in\{L,\underline L,\O_{ij}\}$.
This is one of key points to solve the global Goursat problem for
\eqref{quasi} in $B_{2\dl}$.
\begin{lemma}\label{orL}
	In $D^{s,4\dl}$, the quantities $\check L^i$ and $\check{\varrho}$ admit the following estimates:
	\begin{equation}\label{rhoL}
	\delta^l|Z(g_{\al\beta}\o^\al\o^\beta)|\lesssim \delta^{2-2\varepsilon_0}s^{-3/2},\qquad\delta^l|Z^k\check{\varrho}|
+\delta^l|Z^k(g_{ij}\check L^i\o^j)|\lesssim\delta^{2-2\varepsilon_0}s^{-1},
	\end{equation}
	where $l$ is the number of $T$ in $Z^k$, and $k\leq 2N-9$.
\end{lemma}
\begin{proof}
	Similar to \eqref{g}, one can write $g_{\al\beta}=m_{\al\beta}+\tilde g_{\al\beta}^\gamma\p_\gamma\phi+O(|\phi|+|\vp|^2)$ with $\tilde g_{\al\beta}^\gamma=-\tilde g^{\lambda\nu,\gamma}m_{\al\lambda}m_{\nu\beta}$. By virtue of \eqref{pal}, it holds that
	\begin{equation*}
	\begin{split}
	g_{\al\beta}\o^\al\o^\beta&=-\tilde g^{\al\beta,\gamma}\p_\gamma\phi\o_\al\o_\beta+f(\phi,\varphi,\o)\left(
	\begin{array}{ccc}
	\phi\\
	\vp^2
	\end{array}
	\right)\\
	&=\mu^{-1}\tilde g^{\al\beta,\gamma}\mathring L_\gamma\o_\al\o_\beta(T\phi)+f(\phi,\varphi,\slashed dx^i,\o)\left(
	\begin{array}{ccc}
	\phi\\
	\vp^2\\
	\mathring L\phi\\
	\slashed d\phi
	\end{array}
	\right).
	\end{split}
	\end{equation*}
As $\mathring L_\gamma=g_{\al\gamma}\mathring L^\al=\o_\gamma+f(\phi,\varphi,\o,\mathring L^j)\left(
	\begin{array}{ccc}
	\phi\\
	\vp\\
	\check L^a\\
	\check\varrho
	\end{array}
	\right)$, the null condition \eqref{null} gives
	\begin{equation*}
	g_{\al\beta}\o^\al\o^\beta=f(\phi,\varphi,\slashed dx^i,\mathring L^j,\o)\left(
	\begin{array}{ccc}
	\phi\\
	\vp^2\\
	\mathring L\phi\\
	\slashed d\phi\\
	\check L^a\vp\\
	\check\varrho\vp
	\end{array}
	\right).
	\end{equation*}
	This, together with \eqref{im}, \eqref{imp} and $L^\infty$ estimates in Section \ref{ho}, yields
	\[
	\delta^l|Z^k(g_{\al\beta}\o^\al\o^\beta)|\lesssim\delta^{2-2\varepsilon_0}s^{-3/2}.
	\]

In order to improve the estimate \eqref{cr} for $\check{\varrho}$, we start with a part of the numerator in \eqref{rrho}.
Due to \eqref{pal}, it holds that
	\begin{equation}\label{1-go}
	\begin{split}
	&1-g_{ij}\o^i\o^j+2g_{ij}\check T^i\o^j\\
	=&\tilde g^{ij,\gamma}\o_i\o_j(\p_\gamma\phi)-2\tilde g^{0i,\gamma}\o_i(\p_\gamma\phi)-2g_{ij}\check L^i\o^j+f(\phi,\varphi,\o)\left(
	\begin{array}{ccc}
	\phi\\
	\vp^2
	\end{array}
	\right)\\
	=&-\mu^{-1}\tilde g^{\al\beta,\gamma}\o_\al\o_\beta\mathring L_\gamma(T\phi)-2g_{ij}\check L^i\o^j+f(\phi,\varphi,\slashed dx^i,\o)\left(
	\begin{array}{ccc}
	\phi\\
	\vp^2\\
	\mathring L\phi\\
	\slashed d\phi
	\end{array}
	\right)\\
	=&-2g_{ij}\check L^i\o^j+f(\phi,\varphi,\slashed dx^i,\mathring L^j,\o)\left(
	\begin{array}{ccc}
	\phi\\
	\vp^2\\
	\mathring L\phi\\
	\slashed d\phi\\
	\check L^a\vp\\
	\check\varrho\vp
	\end{array}
	\right).
	\end{split}
	\end{equation}
	
	Under the action of $\mathring L$, $\varrho g_{ij}\check L^i\o^j$ satisfies
	\begin{equation}\label{LrLo}
	\begin{split}
	\mathring L(\varrho g_{ij}\check L^i\o^j)=\big(F_{ij}\mathring L\phi+G_{ij}^\gamma \mathring L\varphi_\gamma\big)\varrho\check L^i\o^j
+\f \varrho rg_{ij}\mathring L(\varrho\check L^i)(\mathring L^j-\check L^j)+\f \varrho r g_{ij}\check L^i(\check L^j-\o^j\o_a\check L^a),
	\end{split}
	\end{equation}
	where $
	g_{ij}\mathring L(\varrho\check L^i)\mathring L^j=-\f12\varrho F_{\mathring L\mathring L}\mathring L\phi
-\f12\varrho G_{\mathring L\mathring L}^\gamma\mathring L\varphi_\gamma+f(\phi,\varphi,\slashed dx^i,\mathring L^j)\left(
\begin{array}{ccc}
\varrho\mathring L\phi\\
\varrho\slashed d\phi
\end{array}
\right)\left(
\begin{array}{ccc}
\phi\\
\varphi
\end{array}
\right)$ holds by \eqref{LeL}.
Similar to \eqref{GT}, $G_{\mathring L\mathring L}^\gamma\mathring L\varphi_\gamma$ admits the better smallness
than $\mathring L\varphi_\gamma$ due to the null condition \eqref{null}. Indeed,
	\begin{equation*}
	\begin{split}
	G_{\mathring L\mathring L}^\gamma\mathring L\varphi_\gamma=&
-(\p_{\varphi_\gamma}g^{\kappa\lambda})\mathring L_\kappa\mathring L_\lambda\mathring L^{\nu}\p_\gamma\varphi_\nu\\
	=&-(\p_{\varphi_0}g^{\kappa\lambda})\mathring L_\kappa\mathring L_\lambda\big(\mathring L^2\phi-(\mathring L\mathring L^i)\varphi_i\big)+\big((\p_{\varphi_\gamma}g^{\kappa\lambda})\mathring L_\kappa\mathring L_\lambda\mathring L_{\gamma}\big)\tilde T^i\mathring L\varphi_i\\
	&-(\p_{\varphi_\gamma}g^{\kappa\lambda})\mathring L_\kappa\mathring L_\lambda g_{\gamma j}(\slashed d^Ax^j)\big(\slashed d_A\mathring L\phi-(\slashed d_A\mathring L^i)\varphi_i\big),
	\end{split}
	\end{equation*}
	where $\mathring L\mathring L^i=O(|\mathring L\phi|+|\mathring L\varphi|+|\slashed d\phi|+|\slashed d\varphi|)$ and $(\slashed d_A\mathring L^i)\varphi_i=\chi_{AB}\slashed d^B\phi+O(|\mathring L\phi|+|\mathring L\varphi|+|\slashed d\phi|+|\slashed d\varphi|)$
due to \eqref{LL} and \eqref{dL} respectively, and  $(\p_{\varphi_\gamma}g^{\kappa\lambda})\mathring L_\kappa\mathring L_\lambda\mathring L_{\gamma}$ satisfies \eqref{gLLL}.

	At the initial time $t_0$,
	\begin{align}
	&\mu=\f{1}{\sqrt{(g^{0i}\o_i)^2+g^{ij}\o_i\o_j}},\label{m}\\
	&\tilde L^i=-g^{0i}\big(-g^{0j}\o_j+\sqrt{(g^{0j}\o_j)^2+g^{ab}\o_a\o_b}\big)+g^{ij}\o_j.\label{tL}
	\end{align}
Note that $-1-2g^{0j}\o_j+g^{ab}\o_a\o_b=g^{\al\beta}\o_\al\o_\beta
=\tilde g^{\al\beta,\gamma}\o_\al\o_\beta\p_\gamma\phi+O(|\phi|+|\vp|^2)$.
	Then it follows from \eqref{errorv} that
	\begin{equation}\label{iL}
	\begin{split}
	\check L^i|_{t_0}
	=&\Big\{\tilde g^{\al i,\gamma}\o_\al\p_\gamma\phi+\f{1-g^{ab}\o_a\o_b}{2}\o^i+f(\phi,\varphi,\o)\left(
	\begin{array}{ccc}
	\phi\\
	\vp^2
	\end{array}
	\right)\Big\}\Big|_{t_0}\\
	=&\Big\{\tilde g^{\al i,\gamma}\o_\al\p_\gamma\phi-(\tilde g^{0j,\gamma}\o_j\p_\gamma\phi)\o^i-\f12(\tilde g^{\al\beta,\gamma}\o_\al\o_\beta\p_\gamma\phi)\o^i+f(\phi,\varphi,\o)\left(
	\begin{array}{ccc}
	\phi\\
	\vp^2
	\end{array}
	\right)\Big\}\Big|_{t_0}.
	\end{split}
	\end{equation}
	
In addition, due to $\p_t=-\f12\o_0\underline L+\f12L$, $\p_i=-\f12\o_i\underline L+\f12\o_i L-\f1r\o^j\O_{ij}$
and the null condition \eqref{null}, then it holds that
	\begin{equation}\label{gLo}
	\begin{split}
	g_{ij}\check L^i\o^j|_{t_0}=\Big\{\f12\tilde g^{\al\beta,\gamma}\o_\al\o_\beta\p_\gamma\phi+f(\phi,\varphi,\o)\left(
	\begin{array}{ccc}
	\phi\\
	\vp^2\\
	\phi\vp
	\end{array}
	\right)\Big\}\Big|_{t_0}=f(\phi,\varphi,\o)\left(
	\begin{array}{ccc}
	\phi\\
	\vp^2\\
	\phi\vp\\
	L\phi\\
	\O\phi
	\end{array}
	\right)\Bigg|_{t_0}.
	\end{split}
	\end{equation}
	Furthermore, the definition of $\upsilon_{ij}$ and \eqref{iL} imply that
	\begin{equation}\label{iu}
	\begin{split}
	\upsilon_{ij}|_{t_0}=&\{-g_{0j}x^i-g_{aj}\check L^ax^i+g_{0i}x^j+g_{ai}\check L^ax^j-(g_{aj}-m_{aj})\o^ax^i+(g_{ai}-m_{ai})\o^ax^j\}|_{t_0}\\
	=&r\Bigg\{\tilde g^{\al j,\gamma}\o_\al\o_i\p_\gamma\phi-\check L^j\o_i-\tilde g^{\al i,\gamma}\o_\al\o_j\p_\gamma\phi+\check L^i\o_j+rf(\phi,\varphi,\o)\left(
	\begin{array}{ccc}
	\phi\\
	\vp^2\\
	\phi\vp
	\end{array}
	\right)\Bigg\}\Bigg|_{t_0}\\
	=&f(\phi,\varphi,\o)\left(
	\begin{array}{ccc}
	\phi\\
	\vp^2\\
	\phi\vp
	\end{array}
	\right)\Bigg|_{t_0}.
	\end{split}
	\end{equation}

For $Z^k=\bar Z^k$ with $\bar Z\in\{T=-\mu(g^{0i}+\mathring L^i)\p_i, R_{ij}=\O_{ij}-\upsilon_{ij}\tilde T\}$,
it follows from \eqref{m} and \eqref{iL}-\eqref{iu} that
	\begin{equation}\label{bZ}
	\delta^l|\bar Z^k(g_{ij}\check L^i\o^j)|_{t_0}\lesssim\delta^{2-2\varepsilon_0}.
	\end{equation}
	
On the other hand, for $k\leq 2N-9$, one can get from \eqref{LrLo} that
	\begin{equation*}
	\begin{split}
	\delta^l|\mathring L\bar Z^k(\varrho g_{ij}\check L^i\o^j)|\lesssim&\underbrace{\delta^l\sum_{p_1+p_2=k-1}|\big(\slashed{\mathcal L}_{\bar Z}^{p_1}\leftidx{^{(\bar Z)}}{\slashed\pi}_{\mathring L}^A\big)\slashed d_A\bar Z^{p_2}(\varrho g_{ij}\check L^i\o^j)|}_{\text{vanish when}\ k=0}+\delta^l|\bar Z^k\mathring L(\varrho g_{ij}\check L^i\o^j)|\\
	\lesssim&\delta^{2-2\varepsilon_0}s^{-3/2}.
	\end{split}
	\end{equation*}
		This, together with \eqref{bZ}, yields
	\begin{equation}\label{LbZ}
	\delta^l|\bar Z^k(g_{ij}\check L^i\o^j)|\lesssim\delta^{2-2\varepsilon_0}s^{-1}.
	\end{equation}
	
	Note that if there is at least one $\varrho\mathring L$ in $Z^k$ $(k\leq 2N-8)$, i.e., $Z^k=Z^{p_1}(\varrho\mathring L)\bar Z^{p_2}$
with $p_1+p_2=k-1$, then one has
\begin{equation*}
(\varrho\mathring L)\bar Z^{p_2}(g_{ij}\check L^i\o^j)=\underbrace{\sum_{q_1+q_2=p_2-1}\varrho\bar Z^{q_1}\leftidx{^{(\bar Z)}}{\slashed\pi}_{\mathring L}^A\cdot\slashed d_A\bar Z^{q_2}(g_{ij}\check L^i\o^j)}_{\text{vanish when}\ p_2=0}+\varrho\bar Z^{p_2}\big(\varrho^{-1}\mathring L(\varrho g_{ij}\check L^i\o^j)-\varrho^{-1}g_{ij}\check L^i\o^j\big).
\end{equation*}
Hence by an induction argument on the number of $\varrho\mathring L$ in $Z^{p_1}$, one can get from \eqref{LrLo} and \eqref{LbZ} that
	\begin{equation}\label{rLg}
	|Z^{p_1}(\varrho\mathring L)\bar Z^{p_2}(g_{ij}\check L^i\o^j)|\\
	\lesssim\delta^{2-2\varepsilon_0-l}s^{-1}.
	\end{equation}
	
Therefore, for any vector $Z\in\{\varrho\mathring L,T,R_{ij}\}$, when $k\leq 2N-9$, it follows from \eqref{LbZ} and \eqref{rLg} that
	\begin{equation}\label{ZgLo}
	\begin{split}
	\delta^l|Z^k(g_{ij}\check L^i\o^j)|\lesssim\delta^{2-2\varepsilon_0}s^{-1}.
	\end{split}
	\end{equation}

	Substituting \eqref{1-go} into \eqref{rrho} and applying \eqref{ZgLo} show that for $k\leq 2N-9$,
	\begin{equation*}
	\delta^l|Z^k\check\varrho|\lesssim\delta^{2-2\varepsilon_0}s^{-1}.
	\end{equation*}
	
\end{proof}

\begin{lemma}\label{n}
	For $k\leq 2N-10$, it holds that
	\begin{equation}\label{Zu}
	\delta^l|Z^k\upsilon_{ij}|\lesssim\delta^{2-2\varepsilon_0},
	\end{equation}
	where $l$ is the number of $T$ in $Z^k$.
\end{lemma}
\begin{proof}
	\eqref{omega} yields
	\begin{equation}\label{upsilon}
	\begin{split}
	\upsilon_{ij}=&-g_{0j}x^i-g_{aj}\check L^ax^i+g_{0i}x^j+g_{ai}\check L^ax^j-g_{aj}\f{x^ax^i}{\varrho}+g_{ai}\f{x^ax^j}{\varrho}\\
	=&r\big(-g_{\al j}\o^\al\o^i+g_{\al i}\o^\al\o^j-g_{aj}\check L^a\o^i+g_{ai}\check L^a\o^j-\check\varrho g_{aj}\o^a\o^i+\check\varrho g_{ai}\o^a\o^j\big)\\
	=&r\big(-\tilde g^{0j,\gamma}(\p_\gamma\phi)\f{x^i}{\varrho}+\tilde g^{aj,\gamma}(\p_\gamma\phi)\f{x^a}{\varrho}\f{x^i}{\varrho}+\tilde g^{0i,\gamma}(\p_\gamma\phi)\f{x^j}{\varrho}-\tilde g^{ai,\gamma}(\p_\gamma\phi)\f{x^a}{\varrho}\f{x^j}{\varrho}\\
	&-\check L^j\f{x^i}{\varrho}+\check L^i\f{x^j}{\varrho}\big)+rf(\phi,\varphi,\o,\mathring L^i,\check\varrho)\left(
	\begin{array}{ccc}
	\phi\\
	\vp^2\\
	\check L^a\vp\\
	\check\varrho
	\end{array}
	\right).
	\end{split}
	\end{equation}
	Let
	\begin{equation}\label{A}
	A_{ij}=-\tilde g^{0j,\gamma}(\p_\gamma\phi)\f{x^i}{\varrho}+\tilde g^{aj,\gamma}(\p_\gamma\phi)\f{x^a}{\varrho}\f{x^i}{\varrho}+\tilde g^{0i,\gamma}(\p_\gamma\phi)\f{x^j}{\varrho}-\tilde g^{ai,\gamma}(\p_\gamma\phi)\f{x^a}{\varrho}\f{x^j}{\varrho}-\check L^j\f{x^i}{\varrho}
+\check L^i\f{x^j}{\varrho}.
	\end{equation}
Then
\begin{equation}\label{TA}
\begin{split}
TA_{ij}=&\tilde g^{\lambda j,\gamma}\o_\lambda(T\varphi_\gamma)\f{x^i}{\varrho}-\tilde g^{\lambda i,\gamma}\o_\lambda(T\varphi_\gamma)\f{x^j}{\varrho}-(T\mathring L^j)\f{x^i}{\varrho}+(T\mathring L^i)\f{x^j}{\varrho}\\
&+\check\varrho\tilde g^{aj,\gamma}\o_a(T\varphi_\gamma)\f{x^i}\varrho-\check\varrho\tilde g^{ai,\gamma}\o_a(T\varphi_\gamma)\f{x^j}\varrho
+f(\phi,\varphi,\o,\mathring L^a)T(\f{x^b}{\varrho}).
\end{split}
\end{equation}
Note that \eqref{TA} contains a factor $T\mathring L^i$ and $T\mathring L^i$ satisfies \eqref{TL}.
To analyze $T\mathring L^i$, one first gets
from the definition of $G_{A\mathring L}^\gamma$ in \eqref{FG} and \eqref{gab} that
	\begin{equation}\label{GTd}
	\begin{split}
	&-G_{A\mathring L}^\gamma(T\varphi_\gamma)\slashed d^Ax^i=-G_{a\beta}^\gamma(\slashed d_Ax^a\cdot\slashed d^Ax^i)\mathring L^\beta(T\varphi_\gamma)\\
	=&-G_{i\beta}^\gamma\mathring L^\beta(T\varphi_\gamma)-G^\g_{\mathring L\mathring L}(T\varphi_\gamma)\tilde T^i
+G_{0\beta}^\gamma\mathring L^\beta\tilde T^i(T\varphi_\gamma)+f(\phi,\varphi,\mathring L^i)\left(
	\begin{array}{ccc}
	\phi\\
	\vp
	\end{array}
	\right)T\varphi_\gamma\\
	=&-G_{i\beta}^\gamma\o^\beta(T\varphi_\gamma)-G_{0\beta}^\gamma\o^\beta\o^i(T\varphi_\gamma)+f(\phi,\varphi,\o,\mathring L^i,\f{x^j}{\varrho})\left(\begin{array}{ccc}\left(
	\begin{array}{ccc}
	\phi\\
	\vp\\
	\check L^a\\
	\check\varrho
	\end{array}
	\right)T\varphi\\
	\mathring L\varphi\\
	\slashed dx^a\cdot\slashed d\varphi
	\end{array}
	\right),
	\end{split}
	\end{equation}
	where  \eqref{GT} is used in the last equality. Then substituting \eqref{GTd}
into \eqref{TL} gives
	\begin{equation}\label{TLi}
	T\mathring L^i=\tilde g^{i\lambda,\gamma}\o_\lambda(T\varphi_\gamma)-\tilde g^{0\lambda,\gamma}\o_\lambda\o^i(T\varphi_\gamma)
+f(\phi,\varphi,\o,\mathring L^i,\f{x^j}{\varrho},\mu)\left(\begin{array}{ccc}\left(
	\begin{array}{ccc}
	\phi\\
	\vp\\
	\check L^a\\
	\check\varrho
	\end{array}
	\right)T\varphi\\
	T\phi\\
	\mathring L\phi\\
	\mathring L\varphi\\
	\slashed dx^a\cdot\slashed d\phi\\
	\slashed dx^a\cdot\slashed d\varphi\\
	\slashed dx^a\cdot\slashed d\mu
	\end{array}
	\right).
	\end{equation}

For $k\leq 2N-10$, one can substitute \eqref{TLi} into \eqref{TA} and use $T(\f{x^a}{\varrho})=\f{\mu\check T^a}{\varrho}+\f{(1-\mu)x^a}{\varrho^2}$ to derive that
	\begin{equation*}
	\delta^l|Z^kTA_{ij}|\lesssim\delta^{1-2\varepsilon_0}s^{-1},
	\end{equation*}
	where the corresponding estimates in Section \ref{ho} have been used.
Thus,
	\begin{equation*}
	\delta^l|TZ^kA_{ij}|\leq\delta^l|[T,Z^k]A_{ij}|+\delta^l|Z^kTA_{ij}|\lesssim\delta^{1-2\varepsilon_0}s^{-1},
	\end{equation*}
	which gives
	\begin{equation}\label{uZA}
	\delta^l|\p_uZ^kA_{ij}|\leq\delta^l|TZ^kA_{ij}|+\delta^l|\Xi^A\slashed d_AZ^kA_{ij}|\lesssim\delta^{1-2\varepsilon_0}s^{-1}.
	\end{equation}
	Integrating \eqref{uZA} with respect to $u$ yields
	\begin{equation}\label{ZkA}
	\delta^l|Z^kA_{ij}|\lesssim\delta^{2-2\varepsilon_0}s^{-1},\qquad k\leq 2N-10.
	\end{equation}
	Thanks to \eqref{upsilon} and \eqref{A}, it follows from \eqref{ZkA} and \eqref{orL} that for $k\leq 2N-10$,
	\begin{equation*}
	\delta^l|Z^k\upsilon_{ij}|\lesssim\delta^{2-2\varepsilon_0}.
	\end{equation*}
	\end{proof}

Note that \eqref{rhoL} implies $\check\varrho=\f{r}{\varrho}-1=O(\delta^{2-2\varepsilon_0}s^{-1})$.
Then the distance between $C_0$ and $C_{4\delta}$ on the hypersurface $\Sigma_t$ is $4\delta
+O(\delta^{2-2\varepsilon_0})$ and the characteristic
surface $C_u$ $(0\leq u\leq 4\delta)$ is almost straight with the error $O(\delta^{2-2\varepsilon_0})$ away from
the corresponding  light
conic surface.

On the other hand, for the standard vectors $\{L,\underline L,\O_{ij}\}$ defined in the end of Section \ref{in},
by \eqref{pal} and \eqref{R},
it holds that
\begin{equation}\label{sL}
\begin{split}
&L=\mathring L-\mu^{-1}\o^\al\mathring L_\al T+\o^\al g_{\al j}(\slashed d^Ax^j)X_A,\\
&\underline L=\mathring L+\mu^{-1}(\o^i\mathring L_i-\mathring L_0)T+(g_{0j}\slashed d^Ax^j-\o^i g_{ij}\slashed d^Ax^j)X_A,\\
&\Omega_{ij}=R_{ij}+\mu^{-1}\upsilon_{ij}T.
\end{split}
\end{equation}

By the identity
\[
\o^\al\mathring L_\al=g_{\al\beta}\o^\al\o^\beta+g_{0i}\check L^i+g_{ij}\o^i\check L^j+\check\varrho(g_{0i}\o^i+g_{ij}\o^i\o^j)
\]
and Lemma \ref{orL}, when $k\leq 2N-9$, one can show that
\begin{equation}\label{oL}
\delta^l|Z^k(\o^\al\mathring L_\al)|\lesssim\delta^{2-2\varepsilon_0}s^{-1}.
\end{equation}

Therefore, combining \eqref{oL}, \eqref{Zu}, \eqref{sL} with \eqref{im} and \eqref{imp} yields that in $D^{s,u}$,
\begin{equation}\label{gv}
|\Gamma^k\varphi_\gamma|\lesssim\delta^{1-\varepsilon_0-l}s^{-3/2},
\qquad|\Gamma^k\phi|\lesssim\delta^{2-\varepsilon_0-l}s^{-3/2},\quad k\leq 2N-9,
\end{equation}
where $\Gamma\in\{(s+r)L,\underline L,\Omega\}$ and $l$ is the number of $\underline L$ in $\Gamma^k$.

For any point $P\in\tilde C_{2\delta}=\{(t,x):t\geq t_0,t-|x|=2\delta\}$, there is an integral line of
$L$ across this point, and the corresponding initial point is denoted by $P_0(t_0,x_0)$ on $\Sigma_{t_0}$ with $|x_0|=1$.
It then follows from \eqref{local3-3} that $|\underline L^kL^l\O^\kappa\phi(P_0)|\lesssim\delta^{2-\varepsilon_0}$ holds.
Next we derive the desired smallness and decay estimate on $\underline L^kL^l\O^\kappa\phi$ at the point $P$.

\begin{proposition}\label{YHC-2}
	For any $(t,x)\in\tilde C_{2\delta}$ and $3k+l+|\kappa|\leq 2N-9$, it holds that
	\begin{equation}\label{LLOp}
	|\underline L^kL^l\O^\kappa\phi(t,x)|\lesssim\delta^{2-\varepsilon_0}t^{-3/2-l}.
	\end{equation}
Moreover, when $k+l\leq 2$,
		\begin{equation}\label{modifiedboundary}
		|\underline L^kL^l\O^\kappa\phi|_{\tilde C_{2\delta}}\lesssim\delta^{3-3\varepsilon_0}t^{-3/2-l}.
		\end{equation}
\end{proposition}
\begin{proof}
According to \eqref{me}, it holds that
\begin{equation}\label{Y-28}
| L(r^{3/2}{\underline L}\phi)|\lesssim\delta^{2-2\varepsilon_0}t^{-3/2}+\delta^{1-2\varepsilon_0} t^{-3}| r^{3/2}{\underline L}\phi|.
\end{equation}
Integrating \eqref{Y-28} along integral curves of $L$ and applying the initial data estimate \eqref{local3-3} on $t_0$
yield that on $\tilde C_{2\delta}$,
\begin{equation}\label{PuL}
|{\underline L}\phi|\lesssim\delta^{2-2\varepsilon_0}t^{-3/2}.
\end{equation}
This, together with \eqref{me} again, shows $|L(r^{3/2}{\p_t\underline L}\phi)|\lesssim\delta^{1-2\varepsilon_0}t^{-3/2}+\delta^{1-2\varepsilon_0} t^{-3}|r^{3/2}{\p_t\underline L}\phi|$.
This means that on $\tilde C_{2\delta}$,
\begin{equation}\label{ptL}
|\p_t\underline L\phi|\lesssim\delta^{1-2\varepsilon_0}t^{-3/2}.
\end{equation}
Similarly, on $\tilde C_{2\delta}$,
\begin{equation}\label{prL}
|\p_r\underline L\phi|\lesssim\delta^{1-2\varepsilon_0}t^{-3/2}.
\end{equation}
Thus it holds that for $(t,x)\in\tilde C_{2\delta}$, $|\bar\p\underline L\phi(t,x)|\lesssim\delta^{1-2\varepsilon_0}t^{-3/2}$
with $\bar\p\in\{\p_t,\p_r\}$. Acting $\O^\kappa$ on \eqref{me} and arguing as for \eqref{Y-28} -\eqref{prL},
one can get that for $|\kappa|+3\leq 2N-9$,
\begin{equation*}
|\bar\p\O^\kappa\underline L\phi(t,x)|\lesssim\delta^{1-2\varepsilon_0}t^{-3/2},\quad(t,x)\in\tilde C_{2\delta}.
\end{equation*}

Now we claim that when $c+2+|\kappa|\leq 2N-9$, there holds on $\tilde C_{2\delta}$,
\begin{equation}\label{poL}
|\bar\p^c\O^\kappa\underline L\phi|\lesssim\delta^{2-2\varepsilon_0-c}t^{-3/2}
\end{equation}
(both here and below, one concerns only with the numbers of $\bar\p$). In fact, if one assumes that
for some positive integer $c_0$ and $c\leq c_0$ and $c+|\kappa|+2\leq 2N-9$, it holds that on $\tilde C_{2\delta}$,
\[
|\bar\p^c\O^\kappa\underline L\phi|\lesssim\delta^{2-2\varepsilon_0-c}t^{-3/2}.
\]
It then follows from \eqref{me} that for $c_0+3+|\kappa|\leq 2N-9$,
\begin{equation}\label{Lrbp}
\begin{split}
|L(r^{3/2}\bar\p^{c_0+1}\O^\kappa\underline L\phi)|=&|r^{3/2}\big\{\f32\O^\kappa[r^{-1},\bar\p^{c_0+1}]\underline L\phi+\O^\kappa\bar\p^{c_0+1}(L\underline L\phi+\f32r^{-1}\underline L\phi)\big\}|\\
\lesssim&\delta^{2-2\varepsilon_0-(c_0+1)}t^{-2}.
\end{split}
\end{equation}
Integrating \eqref{Lrbp} along integral curves of $L$ yields
\[
|\bar\p^{c_0+1}\O^\kappa\underline L\phi|_{\tilde C_{2\delta}}\lesssim\delta^{2-2\varepsilon_0-(c_0+1)}t^{-3/2},\quad c_0+3+|\kappa|\leq 2N-9,
\]
which proves \eqref{poL} by induction.

Due to \eqref{poL}, by checking all the terms in $L(r^{3/2}\bar\p^c\O^\kappa\underline L\phi)$ one can
obtain $| L(r^{3/2}\bar\p^c\O^\kappa\underline L\phi)|\lesssim\delta^{2-\varepsilon_0-c}t^{-2}$
for $c+|\kappa|+3\leq 2N-9$. Therefore, on $\tilde C_{2\delta}$,
\begin{equation*}
|\bar\p^c\O^\kappa\underline L\phi|_{\tilde C_{2\delta}}\lesssim\delta^{2-\varepsilon_0-c}t^{-3/2},\quad c+|\kappa|+3\leq 2N-9.
\end{equation*}
In addition, it follows from \eqref{me} that on $\tilde C_{2\delta}$,
\begin{equation*}
|\bar\p^c\O^\kappa L\underline L\phi|\lesssim\delta^{2-\varepsilon_0-c}t^{-5/2},\quad c+|\kappa|+3\leq 2N-9.
\end{equation*}
This, together with an induction argument, derives that for $c+3k+|\kappa|\leq 2N-9$ and on $\tilde C_{2\delta}$,
\begin{align}
&|\bar\p^c\O^\kappa\underline L^k\phi|\lesssim\delta^{2-\varepsilon_0-c}t^{-3/2},\label{pOL}\\
&|L\bar\p^c\O^\kappa\underline L^k\phi|\lesssim\delta^{2-\varepsilon_0-c}t^{-5/2}.\label{LpOL}
\end{align}

Based on \eqref{LpOL}, one can assume that for some positive integer $l_0$ and any $l$ such that $1\leq l\leq l_0$ and $l+c+3k-1+|\kappa|\leq 2N-9$,
there holds on $\tilde C_{2\delta}$,
\begin{equation}\label{ass}
|L^l\bar\p^c\O^\kappa\underline L^k\phi|\lesssim\delta^{2-\varepsilon_0-c}t^{-3/2-l}.
\end{equation}
For $l_0+c+3k+|\kappa|\leq 2N-9$ and $k\geq 1$, it follows from \eqref{gv}, \eqref{pOL}, \eqref{ass} and \eqref{me} that on $\tilde C_{2\delta}$,
\begin{equation*}
\begin{split}
&| L^{l_0+1}\bar\p^c\O^\kappa\underline L^{k}\phi|=| L^{l_0}\bar\p^c\O^\kappa\underline L^{k-1}(L\underline L\phi)|\\
\lesssim&\delta^{2-\varepsilon_0-c}t^{-5/2-l_0}+\sum_{c_1+c_2=c}\delta^{-c_2}t^{-1}|L^{l_0+1}\bar\p^{c_1}\O^{\leq|\kappa|}\underline L^{\leq k-1}\phi|\\
&+\sum_{c_1+c_2=c}\delta^{2-\varepsilon_0-c_2}t^{-3/2}| L^{l_0+1}\bar\p^{c_1}
\O^{\leq |\kappa|}\underline L^{\leq k-1}\p\phi|.
\end{split}
\end{equation*}
Due to $\p_i=\o_i\p_r-\f1r\o^j\O_{ij}$, then
\begin{align*}
&|L^{l_0+1}\bar\p^{c_1}\O^{\leq |\kappa|}\underline L^{\leq k-1}\p\phi|\\
&\lesssim|L^{l_0+1}\bar\p^{c_1+1}\O^{\leq|\kappa|}\underline L^{\leq k-1}\phi|+t^{-1}| L^{l_0+1}\bar\p^{\leq c_1}\O^{\leq|\kappa|+1}\underline L^{\leq k-1}\phi|+\delta^{2-\varepsilon_0-c_1}t^{-7/2-l_0}.
\end{align*}
By reducing the number of $\underline L$ gradually, one eventually obtains that on $\tilde C_{2\delta}$,
\begin{equation*}
\begin{split}
|L^{l_0+1}\bar\p^c\O^\kappa\underline L^{k}\phi|\lesssim&\delta^{2-\varepsilon_0-c}t^{-5/2-l_0}+\sum_{|\nu|+n\leq c+|\kappa|+k}\delta^{n-c}| L^{l_0+1}\bar\p^n\O^{\nu}\phi|\\
\lesssim&\delta^{2-\varepsilon_0-c}t^{-5/2-l_0}.
\end{split}
\end{equation*}
This means that \eqref{ass} holds for any positive $l$ by an induction argument.
Then \eqref{LLOp} follows from the second estimates in \eqref{gv} together with \eqref{pOL} and \eqref{ass} (here $c=0$ is chosen).

In addition, it follow from
	\begin{equation*}
	\begin{split}
	&\underline L(\f 3r\phi+2L\phi)\\
	=&\f 3{r^2}\phi+\f3rL\phi+\f2{r^2}\sum_{i}\o^j\o^k\O_{ij}\O_{ik}\phi+2(g^{\al\beta,0}\phi+\tilde g^{\al\beta,\gamma}\p_\gamma\phi+h^{\al\beta}(\phi,\p\phi))\p_{\al\beta}^2\phi
	\end{split}
	\end{equation*}
	and \eqref{gv} that
	\begin{equation}\label{LLa}
	|\underline LL^a(\f 3r
	\O^\kappa\phi+2L\O^\kappa\phi)|_{\tilde D_{2\dl}}\lesssim\delta^{2-3\ve_0}t^{-3-a},\ 0\leq a\leq 1\ \text{and}\ a+|\kappa|\leq 2N-11,
	\end{equation}
where $\tilde D_{2\dl}=\{(t,x)|0\leq t-|x|\leq 2\dl, t+|x|\geq 2+2\dl\}$.

Integrating \eqref{LLa} along integral curves of $\underline L$
yields that in $\tilde D_{2\dl}$
	\begin{equation}\label{La}
	|L^a(\f 3r
	\O^\kappa\phi+2L\O^\kappa\phi)|_{\tilde C_{2\dl}}\lesssim\delta^{3-3\ve_0}t^{-3-a},\ 0\leq a\leq 1\ \text{and}\ a+|\kappa|\leq 2N-11,
	\end{equation}
	which means $|L(r^{3/2}\O^\kappa\phi)|_{\tilde C_{2\dl}}\lesssim\delta^{3-3\ve_0}t^{-3/2}$.
Then we have
	\begin{equation}\label{Ophi}
	|\O^\kappa\phi|_{\tilde C_{2\dl}}\lesssim\delta^{3-3\ve_0}t^{-3/2}\ \text{for}\ |\kappa|\leq 2N-11
	\end{equation}
	with the help of \eqref{local3-2}.
On the other hand, collecting \eqref{Ophi} and \eqref{La} deduces
	\begin{equation}\label{LaL}
	|L^aL\O^\kappa\phi|_{\tilde C_{2\dl}}\lesssim\delta^{3-3\ve_0}t^{-5/2-a},\ 0\leq a\leq 1\ \text{and}\ a+|\kappa|\leq 2N-11.
	\end{equation}
	Utilizing \eqref{me} again, it follows from \eqref{Ophi}-\eqref{LaL} and \eqref{LLOp} that
	\begin{equation*}
	|L(r^{3/2}\underline L\O^{\kappa}\phi)|_{\tilde C_{2\dl}}\lesssim\delta^{3-3\ve_0}t^{-3/2},\ |\kappa|\leq 2N-15.
	\end{equation*}	
This implies that by \eqref{local2-2} and when $|\kappa|\leq 2N-15$,
	\begin{equation}\label{uL}
	|\underline L\O^{\kappa}\phi|_{\tilde C_{2\dl}}\lesssim\delta^{3-3\ve_0}t^{-3/2}.
	\end{equation}
Hence, by virtue of \eqref{me},
	\begin{equation}\label{LuLp}
	|L\underline L\O^{\kappa}\phi|_{\tilde C_{2\dl}}\lesssim\delta^{3-3\ve_0}t^{-5/2},\ |\kappa|\leq 2N-15.
	\end{equation}
With the same argument as above, it holds that by the fact of $|L\underline L\O^\kappa(r^{3/2}\underline L\phi)|_{\tilde C_{2\dl}}\lesssim\delta^{3-3\ve_0}t^{-3/2}$ for $|\kappa|\leq 2N-18$ and \eqref{local2-2},
	\begin{equation}\label{uL2}
	|\underline L^2\O^\kappa\phi|_{\tilde C_{2\dl}}\lesssim\delta^{3-3\ve_0}t^{-3/2},\ |\kappa|\leq 2N-18.
	\end{equation}
	Therefore, \eqref{modifiedboundary}  is shown.
\end{proof}

\section{Global existence of $\phi$ in $B_{2\dl}$ and the proof of Theorem \ref{main}}\label{inside}

In this section, we will show the global existence of the solution $\phi$ to \eqref{quasi} in $B_{2\dl}$. Set
$$
D_{t}:=\{(\bar t,x): \bar t-|x|\geq 2\delta, t_0\leq \bar t\leq t\},
$$
see Figure 7 below. Different from the problem of 3D small value solutions inside the cone in \cite{Godin07} and \cite{MPY}, here the solution
$\phi$ to \eqref{quasi} in $D_{t}$ remains large due to the structure of the initial data on time $t_0$ (see Theorem \ref{Th2.1}).

\vskip 0.1 true cm
\centerline{\includegraphics[scale=0.29]{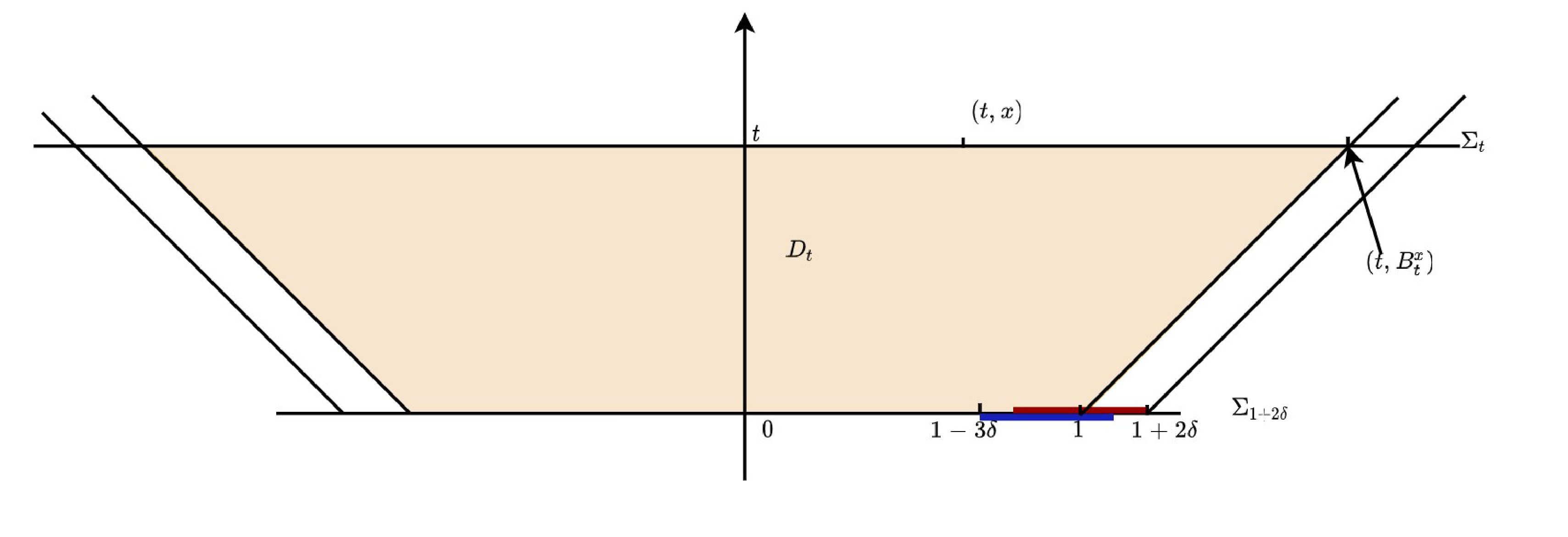}}

\centerline{\bf Figure 7. The domain $D_{t}$}
\vskip 0.3 true cm

We will adopt the energy method to get the global existence of $\phi$ in $B_{2\dl}$. To this end, getting suitable rate of
time-decay of $\phi$ by the Klainerman-Sobolev type $L^\infty-L^2$ inequality is crucial.
However, since $D_t$ has finite lateral boundary, then the classical Klainerman-Sobolev type  $L^\infty-L^2$  inequality
cannot be applied directly in $D_t$. Inspired by the works in \cite{Godin97}, \cite{LWY} and \cite{MPY},
we need the following modified Klainerman-Sobolev inequalities.

\begin{lemma}\label{KS}
	For any function $f(t,x)\in C^\infty(\mathbb R^{1+4})$, $t_0\leq t\leq T$, the following inequalities hold:
	\begin{equation}\label{leq}
	| f(t,x)|\lesssim\sum_{i=0}^3 t^{-2}\delta^{(i-2)\nu}\|\bar{\Gamma}^{i}f(t,\cdot)\|_{L^2(r\leq t/2)},\ |x|\leq\f14t,
	\end{equation}
	\begin{equation}\label{geq}
	| f(t,x)|\lesssim | f(t,B_t^x)|+\sum_{|\al|\leq 2,
		|\beta|\leq 1}t^{-3/2}\|\Omega^\al\p^\beta f(t,\cdot)\|_{L^2(t/4\leq r\leq t-2\delta)},\ |x|\geq\f14t,
	\end{equation}
	for $(t,x)\in D_T$, where $\bar{\Gamma}\in\{S,H_i,\Omega_{ij}\}$, $\O\in\{\O_{ij}:i,j=1,2,3,4\}$, $(t,B_t^x)$ is the intersection point of
	the boundary $\tilde C_{2\delta}$ and the ray crossing $(t,x)$ which emanates from $(t,0)$,
	and $\nu$ is the any nonnegative constant in \eqref{leq}.
\end{lemma}

It should be remarked that though this Lemma and its proof are similar to those of Proposition  3.1 in \cite{MPY},
 yet the refined inner estimate, \eqref{leq} is new (the appearance of factor $\delta^{(i-2)\nu}$) and is crucial for
 the treatment of the short pulse initial data here that is not needed for the small data case in \cite{MPY,Godin97,LWY}
 (note that in \cite{MPY}, the smallness of the corresponding solution holds true in the interior domain $D(\dl)=\{(t,x): t-|x|\geq\dl\}$
 of the light cone due to the largeness of $N_0$ in \eqref{HC-03}).

\begin{proof}
	The basic approach here is similar to those for Proposition 3.1 in \cite{MPY} by separating the inner estimate from the outer ones.
	However, to get the refined inner estimate \eqref{leq}, we will use a $\delta$-dependent scaling.
	Let $\chi\in C_c^\infty(\mathbb R_+)$ be a nonnegative cut-off function such that $\chi(r)\equiv1$ fro $r\in[0,\f14]$ and $\chi(r)\equiv0$ for $r\geq\f12$.
	
	Define $f_1(t,x)=\chi(\f{|x|}t)f(t,x)$ and $f_2=f-f_1$. Then $\text{supp}\ f_1\subset\{(t,x):|x|\leq\f{t}2\}$ and $\text{supp}\ f_2\subset\{(t,x):|x|\geq\f{t}4\}$. One can obtain the inner estimate \eqref{leq} and the outer estimate \eqref{geq} separately as follows.
	
	First, for any point $(t,x)$ satisfying $|x|\leq\f14t$, $f_1(t,x)=f(t,x)$. One can rescale the variable as $x=t\delta^\nu y$,
	and then use the Sobolev embedding theorem for $y$ to get
		\begin{equation}\label{ins}
	\begin{split}
	|f(t,x)|&=|f_1(t,t\delta^\nu y)|\lesssim\sum_{|\al|\leq 3}\Big(\int_{\mathbb R^4}|\p_y^\al\big(f(t,t\delta^\nu y)\chi(\delta^\nu y)\big)|^2dy\Big)^{1/2}\\
	&\lesssim\sum_{|\al|\leq 3}\Big(\int_{|\delta^\nu y|\leq 1/2}(t\delta^\nu)^{2|\al|}|(\p_x^\al f)(t,t\delta^\nu y)|^2dy\Big)^{\f12}\\
	&\lesssim\sum_{|\al|\leq 3}\Big(\int_{|z|\leq\f12t}(t\delta^\nu)^{2|\al|-4}|\p_z^\al f(t,z)|^2dz\Big)^{1/2}.
	\end{split}
	\end{equation}
	Note that
	\begin{equation}\label{11.4}
	\p_{i}=-\f{1}{t-r}\big(\f{x^i}{t+r}S-\f{t}{t+r}H_i-\f{x^j}{t+r}\O_{ij}\big),\quad\p_t=\f{1}{t-r}\big(\f{t}{t+r}S-\f{x^i}{t+r}H_i\big)
	\end{equation}
	and $t\sim t-|z|$ in the domain $\{(t,z): |z|\leq\f12t\}$. Then one can get
	\begin{equation}\label{tf}
	\begin{split}
	&|t\p_zf(t,z)|\lesssim|(t-|z|)\p_zf(t,z)|\lesssim|\bar{\Gamma} f(t,z)|,\\
	&|t^2\p_z^2f(t,z)|\lesssim\sum_{1\leq|\al|\leq 2}|\bar\Gamma^\al f(t,z)|,\\
	&|t^3\p_z^3f(t,z)|\lesssim\sum_{1\leq|\al|\leq 3}|\bar\Gamma^\al f(t,z)|.
	\end{split}
	\end{equation}
	Therefore, substituting \eqref{tf} into \eqref{ins} yields \eqref{leq}.

	Next, for $(t,x)$ satisfying $|x|\geq\f14t$, by the Newton-Leibnitz formula and the Sobolev embedding
	theorem on the circle $S_t^\rho$
	with  radius $\rho$ and center at the origin on $\Sigma_t$, one has
	\begin{equation*}
	\begin{split}
	f^2(t,x)&=f^2(t,B_t^x)-\int_{|x|}^{t-2\delta}\p_\rho\big(f^2(t,\rho\omega)\big)d\rho\\
	&\lesssim f^2(t,B_t^x)+\int_{|x|}^{t-2\delta}\f{1}{\rho^3}\sum_{|\al|,|\beta|\leq 2}\|\Omega^\al f\|_{L^2(S_t^\rho)}\|\Omega^{{\beta}} \p f\|_{L^2(S_t^\rho)}d\rho\\
	&\lesssim f^2(t,B_t^x)+\sum_{|\al|\leq 2,|\beta|\leq 1}t^{-3}\|\Omega^\al\p^\beta f(t,\cdot)\|^2_{L^2(t/4\leq r\leq t-2\delta)},
	\end{split}
	\end{equation*}
	which implies \eqref{geq}. Thus Lemma \ref{KS} is verified.
\end{proof}

On the hypersurface in $D_{T}$, one can have the following estimate which is similar to Lemma 2.3 for 3D case in \cite{DY}.

\begin{lemma}\label{L2}
	Let $f$ is the function as in Lemma \ref{KS}. Then for any $1\leq\bar t\leq t-2\delta$, it holds that
	\begin{equation}\label{L2L2}
	\|\f{f(t,\cdot)}{1+t-|\cdot|}\|_{L^2(\bar t\leq |x|\leq t-2\delta)}
\lesssim t^{3/2}\|f(t,B_t^{\cdot})\|_{L^\infty(\bar t\leq |x|\leq t-2\delta)}+\|\p f(t,\cdot)\|_{L^2(\bar t\leq |x|\leq t-2\delta)}.
	\end{equation}
\end{lemma}
\begin{proof}
	For any point $(t,x)$ satisfying $|x|\geq\bar t$ in $D_T$, one has
	\[
	f(t,x)=f(t,B_t^x)-\int_r^{t-2\delta}\nabla_xf(t,s\omega)\cdot\omega ds,
	\]
where $(t,B_t^x)$ is given in Lemma \ref{KS}.
	Then it follows from  Schwartz inequality that
\begin{equation}\label{Y-29}
	|f(t,x)|^2\lesssim|f(t,B_t^x)|^2+(1+t-r)^{1/2}\int_r^{t-2\delta}|\p f(t,s\omega)|^2(1+t-s)^{1/2}ds.
\end{equation}
	Multiplying the factor $(1+t-r)^{-2}r^3$ on both sides of \eqref{Y-29} and integrating them from $\bar t$ to $t-2\delta$ yield
	\begin{equation}\label{SoFn-1}
	\begin{split}
	\int_{\bar t}^{t-2\delta}|\f{f(t,r\omega)}{1+t-r}|^2r^3dr\lesssim t^3\|f(t,B_t^{\cdot})\|^2_{L^\infty(\bar t\leq |x|\leq t-2\delta)}
+\int_{\bar t}^{t-2\delta}|\p f(t,s\omega)|^2s^3ds,
	\end{split}
	\end{equation}
	which gives \eqref{L2L2} directly by integrating \eqref{SoFn-1} on $\mathbb S^3$.
\end{proof}

With these preparations and the crucial estimates in Proposition \ref{YHC-2}, one can use the standard energy method to prove the global estimate of the solution $\phi$ to \eqref{quasi} in $B_{2\dl}$. Set
\begin{equation}\label{Linfty}
E_{k,l}(t)=\|\p\tilde\Gamma^k\Omega^l\phi(t,\cdot)\|_{L^2(\Sigma_t\cap D_{t})}^2,
\end{equation}
where $\tilde\Gamma\in\{\p,S,H_i\}$ and $k+l\leq 6$. By
the estimates \eqref{local3-2} on $\Sigma_{t_0}$, one can make the following bootstrap assumptions:

For $t\ge t_0$ and $0<\ve_0\leq\f{7}{144}$, there exists a uniform constant $M_0>0$ such that
\begin{equation}\label{EA}
\begin{split}
&E_{0,l}(t)\leq {M_0}^2\delta^{6-6\varepsilon_0},\qquad E_{1,l}(t)\leq {M_0}^2\delta^{44/9-10\varepsilon_0},\qquad E_{2,l}(t)\leq {M_0}^2\delta^{8/3-2\varepsilon_0},\\
&E_{3,l}(t)\leq {M_0}^2\delta^{1-2\varepsilon_0},\qquad E_{4,l}(t)\leq {M_0}^2\delta^{-1-2\varepsilon_0},\qquad\quad E_{5,l}(t)\leq {M_0}^2\delta^{-3-2\varepsilon_0},\\
&E_{6,l}(t)\leq {M_0}^2\delta^{-5-2\varepsilon_0}.
\end{split}
\end{equation}

Lemma \ref{KS}, \ref{L2},  and the assumptions \eqref{EA} imply
the following $L^\infty$ estimates.

\begin{proposition}\label{P4.1}
	Under the assumptions \eqref{EA} with suitably small $\delta>0$, it holds that
\begin{equation}\label{pphi}
	\begin{split}
	&|\p\Omega^{\leq 3}\phi|\lesssim M_0\delta^{(4-5\varepsilon_0)/3}t^{-3/2},\ \ \quad |\p\tilde\Gamma\Omega^{\leq 2}\phi|\lesssim M_0\delta^{5/12-\varepsilon_0}t^{-3/2},\\
	&|\p\tilde\Gamma^2\Omega^{\leq 1}\phi|\lesssim M_0\delta^{-5/9-\varepsilon_0}t^{-3/2},\quad|\p\tilde{\Gamma}^3\phi|
\lesssim M_0\delta^{-3/2-\varepsilon_0}t^{-3/2}.
	\end{split}
	\end{equation}
\end{proposition}
\begin{proof}
		First, for $|x|\leq\f t4$, one gets from \eqref{leq} that
		\begin{equation}\label{Y-30}
		\begin{split}
		|\p\Omega^{\leq 3}\phi|\lesssim t^{-2}\Big\{&\delta^{-2\nu}\|\p\Omega^{\leq 3}\phi\|_{L^2(\Sigma_t\cap D_t)}+\delta^{-\nu}\|\bar{\Gamma}\p\Omega^{\leq 3}\phi\|_{L^2(\Sigma_t\cap D_t)}\\
		&+\|\bar{\Gamma}^2\p\Omega^{\leq 3}\phi\|_{L^2(\Sigma_t\cap D_t)}
+\delta^{\nu}\|\bar{\Gamma}^3\p\Omega^{\leq 3}\phi\|_{L^2(\Sigma_t\cap D_t)}\big\}.
		\end{split}
		\end{equation}
Choosing $\nu=\f56-\f23\ve_0$ in \eqref{Y-30} and utilizing the assumptions \eqref{EA} yield
		\begin{equation}\label{1s}
	|\p\Omega^{\leq 3}\phi|\lesssim M_0\delta^{(4-5\varepsilon_0)/3} t^{-2}.
		\end{equation}
		Next, for $\f t4\leq |x|\leq t-2\delta$, choosing $f(t,x)$ as $\p\O^{\leq 3}\phi(t,x)$ in \eqref{geq}, one can get,
		\begin{equation}\label{1lp}
		\begin{split}
		|\p\Omega^{\leq 3}\phi(t,x)|\lesssim &|\p\Omega^{\leq 3}\phi(t,B_t^x)|+t^{-3/2}\|\O^{\leq2}\p^{\leq1}\p\O^{\leq 3}\phi(t,\cdot)\|_{L^2(\f{t}4\leq r\leq t)}\\
		\lesssim&
		 M_0\delta^{44/9-10\varepsilon_0}t^{-3/2}
		 \end{split}
		\end{equation}
		with the help of \eqref{modifiedboundary}.
		Combining \eqref{1s} and \eqref{1lp} yields
		\begin{align*}
		|\p\Omega^{\leq 3}\phi|\lesssim M_0\delta^{(4-5\varepsilon_0)/3}t^{-3/2}.
		\end{align*}
		
		Finally, the other cases can be treated similarly as in \eqref{Y-30} by choosing different $\nu$ such as $\nu=\f{11}{12}$
		for $\p\tilde{\Gamma}\Omega^{\leq 2}\phi$, $\nu=\f{17}{18}$ for $\p\tilde{\Gamma}^2\Omega^{\leq 1}\phi$, and $\nu=1$ for $\p\tilde{\Gamma}^3\phi$. Due to the direct but tedious computations, the resulting details are omitted.
\end{proof}

\begin{remark}
	It is pointed out that the assumption of $\ve_0\leq\f{7}{144}$ is used to derive $|\p\tilde\Gamma\Omega^{\leq 2}\phi|\lesssim M_0\delta^{5/12-\varepsilon_0}t^{-3/2}$ in \eqref{pphi} when $\nu=\f{11}{12}$ is chosen according to the analogous inequality \eqref{Y-30}.
\end{remark}

\begin{corollary}\label{C4.1}
	Under the same conditions in Proposition \ref{P4.1}, it holds that
\begin{equation}\label{p2}
	\begin{split}
	&|\p^2\Omega^{\leq 2}\phi|\lesssim M_0\delta^{5/12-\varepsilon_0}t^{-3/2}(1+t-r)^{-1},\\
	&|\p^2\tilde\Gamma\Omega^{\leq 1}\phi|\lesssim M_0\delta^{-5/9-\varepsilon_0}t^{-3/2}(1+t-r)^{-1},\\
	&|\p^2\t\Gamma^2\phi|\lesssim M_0\delta^{-3/2-\varepsilon_0}t^{-3/2}(1+t-r)^{-1}
	\end{split}
	\end{equation}
	and
	\begin{equation}\label{O}
	\begin{split}
	&|\O^{\leq3}\phi|\lesssim M_0\delta^{(4-5\varepsilon_0)/3}t^{-3/2}(1+t-r),\\
	&|\tilde\Gamma\O^{\leq 2}\phi|\lesssim M_0\delta^{5/12-\varepsilon_0}t^{-3/2}(1+t-r),\\
	&|\tilde\Gamma^2\O^{\leq 1}\phi|\lesssim M_0\delta^{-5/9-\varepsilon_0}t^{-3/2}(1+t-r),\\
	&|\tilde\Gamma^3\phi|\lesssim M_0\delta^{-3/2-\varepsilon_0}t^{-3/2}(1+t-r).
	\end{split}
	\end{equation}
\end{corollary}
\begin{proof}
	\eqref{p2} follow from \eqref{11.4} and \eqref{pphi} directly. \eqref{O} can be shown by using
		\[
		f(t,x)=f(t,B_t^x)-\int_r^{t-2\delta}\nabla_xf(t,s\omega)\cdot\omega ds
		\]
		and \eqref{pphi}.
\end{proof}

We now carry out the energy estimates on $\phi$ in $B_{2\dl}$.
 Direct computations show
 \begin{equation}\label{Wvv}
 \begin{split}
 (\p_tv)g^{\al\beta}\p_{\al\beta}^2v=&\f12\p_t\big((g^{00}(\p_t v)^2-g^{ij}\p_iv\p_jv)\big)+\p_i\big((g^{0i}(\p_tv)^2
 +g^{ij}\p_tv\p_jv)\big)\\
 &+\big(-(\p_ig^{0i})(\p_tv)^2-(\p_ig^{ij})\p_tv\p_jv+\f12(\p_tg^{ij})\p_iv\p_jv\big),
 \end{split}
 \end{equation}
 where
 \begin{equation}\label{Dt2}
 \begin{split}
 -(\p_ig^{0i})(\p_tv)^2-(\p_ig^{ij})\p_tv\p_jv+\f12(\p_tg^{ij})\p_iv\p_jv
 =O\big((|\p\phi|+|\p^2\phi|)|\p v|^2\big).
 \end{split}
 \end{equation}
In addition, integrating \eqref{Wvv} over $D_t$ and using integration by parts give rise to a term on the lateral boundary $\tilde C_{2\delta}$ as
 \begin{equation}\label{lb}
 \begin{split}
 &-\f{\sqrt 2}4(g^{00}(\p_t v)^2-g^{ij}\p_iv\p_jv)+\f{\sqrt 2}2(g^{0i}(\p_tv)^2+g^{ij}\p_tv\p_jv)\o_i\\
 =&\f{\sqrt 2}4(Lv)^2+\f{\sqrt 2}{4r^2}\sum_i(\o^j\O_{ij}v)^2+O\big((|\phi|+|\p\phi|)|\p v|^2\big).
 \end{split}
 \end{equation}
With the help of \eqref{modifiedboundary}, it can be checked that the term in \eqref{lb} can be controlled on $\tilde C_{2\delta}$ by
\begin{equation}\label{C2}
(Lv)^2+\sum_{i=1}^4(\f{\omega^j}{r}\O_{ij}v)^2+\delta^{3-3\varepsilon_0}t^{-3/2}(\p v)^2,
\end{equation}
where  the unimportant constant coefficients in \eqref{C2} are neglected.

Due to \eqref{Dt2} and \eqref{C2}, integrating \eqref{Wvv} on $D_t$ and
utilizing Proposition \ref{P4.1} and Corollary \ref{C4.1},
 one then gets from  Gronwall's inequality that for small $\delta>0$,
\begin{equation}\label{W}
\begin{split}
&\int_{\Sigma_t\cap D_t}\big((\p_tv)^2+|\nabla v|^2\big)\\
\lesssim&\int_{\Sigma_{t_0}\cap D_t}\big((\p_tv)^2+|\nabla v|^2\big)+\iint_{D_t}| \p_tvg^{\al\beta}\p_{\al\beta}^2v|\\
&+\int_{\tilde C_{2\delta}\cap D_t}\big\{(Lv)^2+\sum_{i=1}^4(\f{\omega^j}{r}\O_{ij}v)^2
+\delta^{3-3\varepsilon_0}tau^{-3/2}|\p v|^2\big\}.
\end{split}
\end{equation}

To close the bootstrap assumptions \eqref{EA}, one can set $v=\tilde\Gamma^k\Omega^l\phi$ $(k+l\leq 6)$ in \eqref{W}.
Note that \eqref{LLOp} and \eqref{modifiedboundary} implies that $(L\tilde\Gamma^k\Omega^l\phi)^2+\sum_{i=1}^4(\f{\omega^j}{r}\O_{ij}\tilde\Gamma^k\Omega^l\phi)^2
 +\delta^{3-3\varepsilon_0}t^{-3/2}|\p \tilde\Gamma^k\Omega^l\phi|^2\lesssim\delta^{6-6\varepsilon_0}t^{-9/2}$ for $k\leq 1$ and $(L\tilde\Gamma^k\Omega^l\phi)^2+\sum_{i=1}^4(\f{\omega^j}{r}\O_{ij}\tilde\Gamma^k\Omega^l\phi)^2
 +\delta^{3-3\varepsilon_0}t^{-3/2}|\p \tilde\Gamma^k\Omega^l\phi|^2\lesssim\delta^{4-2\varepsilon_0}t^{-9/2}$ for $2\leq k\leq 6$ hold
 on $\tilde C_{2\delta}$. Therefore,
\begin{equation*}
\int_{\tilde C_{2\delta}\cap D_t}\big\{(L\tilde\Gamma^k\Omega^l\phi)^2+\sum_{i=1}^4(\f{\omega^j}{r}\O_{ij}\tilde\Gamma^k\Omega^l\phi)^2+\delta^{3-3\varepsilon_0}\tau^{-3/2}|\p \tilde\Gamma^k\Omega^l\phi|^2\big\}\lesssim
\left\{
\begin{aligned}
&\delta^{6-6\ve_0},\ k\leq 1,\\
&\delta^{4-2\varepsilon_0},\ 2\leq k\leq 6.
\end{aligned}
\right.
\end{equation*}
In addition, on the initial hypersurface $\Sigma_{t_0}\cap D_t$,
it holds that $|\p\tilde\Gamma^k\Omega^l\phi|\lesssim\delta^{3-k-\varepsilon_0}$
for $0\leq k\leq 6-l$ by \eqref{local3-2}. Hence, \eqref{W} gives that
\begin{align}
&E_{0,l}(t)\lesssim\delta^{6-6\varepsilon_0}+\iint_{D_t}| (\p_t\Omega^l\phi)(g^{\al\beta}\p_{\al\beta}^2\Omega^l\phi)|,\quad l\leq 6,\label{E01}\\
&E_{k,l}(t)\lesssim\delta^{7-2k-2\varepsilon_0}+\iint_{D_t}| (\p_t\tilde\Gamma^k\Omega^l\phi)(g^{\al\beta}\p_{\al\beta}^2\tilde\Gamma^k\Omega^l\phi)|,\quad 1\leq k\leq 6-l.\label{E2-6}
\end{align}

It remains to estimate $\iint_{D_t}|(\p_t\tilde\Gamma^k\Omega^l\phi)(g^{\al\beta}\p_{\al\beta}^2\tilde\Gamma^k\Omega^l\phi)|$
in the right hand sides of \eqref{E01}-\eqref{E2-6} and further obtain the global existence of $\phi$ in $B_{2\dl}$.

\begin{theorem}\label{T4.1}
	When $\delta>0$ is small, there exists a smooth solution $\phi$ to \eqref{quasi} in $B_{2\dl}$.
\end{theorem}
\begin{proof}
	Acting $\tilde\Gamma^k\Omega^l$ on \eqref{quasi} and commuting it with $g^{\al\beta}\p_{\al\beta}^2$ yield
	\begin{equation*}
	\begin{split}
	|g^{\al\beta}\p_{\al\beta}^2\tilde\Gamma^k\Omega^l\phi|\lesssim\sum_{\mbox{\tiny$\begin{array}{cc}k_1+k_2\leq k,l_1+l_2\leq l\\k_2+l_2<k+l\end{array}$}}|\tilde\Gamma^{k_1}\O^{l_1}(g^{\al\beta}-m^{\al\beta})|\cdot|\p^2\tilde\Gamma^{k_2}\O^{l_2}\phi|.
	\end{split}
	\end{equation*}
	The right hand side can be estimated separately as follows.
\begin{enumerate}
	\item {\bf $k=0$ and $l\leq 6$.}
	\begin{enumerate}
		\item When $l_1\leq l_2$, then $l_1\leq 3$ holds. It follows from Proposition \ref{P4.1} and Corollary \ref{C4.1} that
		\[
		|\Omega^{l_1}(g^{\al\beta}-m^{\al\beta})|\lesssim|\O^{\leq l_1}\phi|+|\O^{\leq l_1}\p\phi|\lesssim M_0\delta^{(4-5\varepsilon_0)/3}t^{-3/2}(1+t-r).
		\]
		Thus, since $l_2\leq l-1$, one has by \eqref{EA} that
\begin{equation}\label{l1}
\begin{split}
&\iint_{D_t}|\p\Omega^l\phi|\cdot|\tilde\O^{l_1}(g^{\al\beta}-m^{\al\beta})|\cdot|\p^2\O^{l_2}\phi|\\
\lesssim&\int_{t_0}^t M_0\delta^{(4-5\varepsilon_0)/3}\tau^{-3/2}\|\p\Omega^l\phi\|_{L^2(\Sigma_{\tau}\cap D_t)}\|(1+\tau-r)\p^2\Omega^{l_2}\phi\|_{L^2(\Sigma_{\tau}\cap D_t)}d\tau\\
\lesssim&M_0\delta^{(4-5\varepsilon_0)/3}\int_{t_0}^t\tau^{-3/2}\sqrt{E_{0,\leq 6}(\tau)}\big(\sqrt{E_{0,\leq 6}(\tau)}
+\sqrt{E_{1,\leq 5}(\tau)}\big)d\tau\\
\lesssim&\delta^{6-6\varepsilon_0}.
\end{split}
\end{equation}

		\item When $l_1>l_2$, then $l_2\leq 2$ holds. It follows from Corollary \ref{C4.1} that
		\[
		|\p^2\Omega^{l_2}\phi|\lesssim M_0\delta^{5/12-\varepsilon_0}t^{-3/2}(1+t-r)^{-1}.
		\]
		This, together with \eqref{EA} and Lemma \ref{L2}, yields
		\begin{equation}\label{l2}
		\begin{split}
		&\iint_{D_t}|\p\Omega^l\phi|\cdot|\tilde\O^{l_1}(g^{\al\beta}-m^{\al\beta})|\cdot|\p^2\O^{l_2}\phi|\\
		\lesssim&M_0\delta^{5/12-\varepsilon_0}\iint_{D_t}\tau^{-3/2}(1+\tau-r)^{-1}|\p\Omega^l\phi|\big(|\O^{\leq l_1}\phi|+|\O^{\leq l_1}\p\phi|\big)\\
		\lesssim&M_0\delta^{5/12-\varepsilon_0}\int_{t_0}^t\tau^{-3/2}\|\p\O^l\phi(\tau,\cdot)\|_{L^2(\Sigma_{\tau}\cap D_t)}\Big(\|\f{\O^{\leq l_1}\phi}{1+\tau-|\cdot|}\|_{L^2(\Sigma_{\tau}\cap D_t)}\\
		&\qquad\qquad\qquad+\|\O^{\leq l_1}\p\phi\|_{L^2(\Sigma_{\tau}\cap D_t)}\Big)d\tau\\
		\lesssim& M_0\delta^{5/12-\varepsilon_0}\int_{t_0}^t\tau^{-3/2}\|\p\O^l\phi(\tau,\cdot)\|_{L^2(\Sigma_{\tau}\cap D_t)}(\delta^{3-3\varepsilon_0}+\|\p\O^{\leq l_1}\phi\|_{L^2(\Sigma_{\tau}\cap D_t)})d\tau\\
		\lesssim&\delta^{6-6\varepsilon_0}.
		\end{split}
		\end{equation}
		
	\end{enumerate}
Inserting \eqref{l1} and \eqref{l2} into \eqref{E01} leads to
\begin{equation}\label{E0}
E_{0,l}(t)\lesssim\delta^{6-6\varepsilon_0},\quad l\leq 6.
\end{equation}

\item{\bf $k=1$ and $l\leq 5$.}
\begin{enumerate}
	\item When $k_1+l_1\leq k_2+l_2$, then $k_1+l_1\leq 3$ holds. Due to $k_1\leq1$, it follows from
Proposition \ref{P4.1} and Corollary \ref{C4.1} that
	\[
	|\tilde\Gamma\Omega^{l_1}(g^{\al\beta}-m^{\al\beta})|\lesssim|\tilde\Gamma\O^{\leq l_1}\phi|
+|\tilde\Gamma\O^{\leq l_1}\p\phi|\lesssim M_0\delta^{5/12-\varepsilon_0}t^{-3/2}(1+t-r).
	\]
	Then, since $k_2+l_2\leq 5$ and $\delta>0$ is small, one has
	\begin{equation}\label{k11}
	\begin{split}
	&\iint_{D_t}|\p\tilde\Gamma\Omega^l\phi|\cdot|\tilde\Gamma^{k_1}\O^{l_1}(g^{\al\beta}-m^{\al\beta})|\cdot|\p^2\tilde\Gamma^{k_2}\O^{l_2}\phi|\\
	\lesssim&\iint_{D_t}|\p\tilde\Gamma\Omega^l\phi|\big\{|\tilde\Gamma\O^{l_1}(g^{\al\beta}-m^{\al\beta})|\cdot|\p^2\O^{l_2}\phi|\\
	&\qquad\qquad\qquad\qquad+|\O^{l_1}(g^{\al\beta}-m^{\al\beta})|\cdot|\p^2\tilde\Gamma^{\leq 1}\O^{l_2}\phi|\big\}\\
	\lesssim& M_0\delta^{5/12-\varepsilon_0}\int_{t_0}^t\tau^{-3/2}\|\p\tilde\G\O^l\phi\|_{L^2(\Sigma_{\tau}\cap D_t)}\|\p\tilde\G^{\leq 1}\O^{l_2}\phi\|_{L^2(\Sigma_{\tau}\cap D_t)}d\tau\\
	&+M_0\delta^{(4-5\varepsilon_0)/3}\int_{t_0}^t\tau^{-3/2}\|\p\tilde\G\O^l\phi\|_{L^2(\Sigma_{\tau}\cap D_t)}\|\p\tilde\G^{\leq 2}\O^{l_2}\phi\|_{L^2(\Sigma_{\tau}\cap D_t)}d\tau\\
	\lesssim&\delta^{9/44-10\varepsilon_0}.
	\end{split}
	\end{equation}
	\item When $k_1+l_1>k_2+l_2$, then $k_2+l_2\leq 2$ holds. As $k_2\leq 1$, then Corollary \ref{C4.1} implies
	\[
	|\p^2\tilde\G\O^{l_2}\phi|\lesssim M_0\delta^{-5/9-\varepsilon_0}t^{-3/2}(1+t-r)^{-1}.
	\]
	This, together with \eqref{EA}, yields
	\begin{equation}\label{k12}
	\begin{split}
	&\iint_{D_t}|\p\tilde\Gamma\Omega^l\phi|\cdot|\tilde\Gamma^{k_1}\O^{l_1}(g^{\al\beta}-m^{\al\beta})|\cdot|\p^2\tilde\Gamma^{k_2}\O^{l_2}\phi|\\
	\lesssim& M_0\delta^{5/12-\varepsilon_0}\iint_{D_t}\tau^{-3/2}(1+\tau-r)^{-1}|\p\tilde\Gamma\O^l\phi|\big(|\tilde\G\O^{\leq l_1}\phi|+|\tilde\G\O^{\leq l_1}\p\phi|\big)\\
	&+M_0\delta^{-5/9-\varepsilon_0}\iint_{D_t}\tau^{-3/2}(1+\tau-r)^{-1}|\p\tilde\Gamma\O^l\phi|\big(|\O^{\leq l_1}\phi|+|\O^{\leq l_1}\p\phi|\big)\\
	\lesssim&\delta^{9/44-10\varepsilon_0}.
	\end{split}
	\end{equation}
\end{enumerate}
 Similarly to Case (1), substituting \eqref{k11} and \eqref{k12} into \eqref{E2-6} gives
 \begin{equation}\label{E1l}
 E_{1,l}(t)\lesssim\delta^{9/44-10\varepsilon_0},\quad l\leq 5.
 \end{equation}

\item {\bf $k\geq 2$ and $k+l\leq 6$.}

When $k+l\leq 6$ $(k\geq 2)$, as in the case (1) and (2), one can make use of Proposition \ref{P4.1} and Corollary \ref{C4.1} to estimate the
related integrals in  \eqref{E01}-\eqref{E2-6} for $k_1+l_1\leq k_2+l_2$ ($k_1+l_1\leq3$ and $k_2+l_2\leq5$), meanwhile utilize Corollary \ref{C4.1}
to deal with the remaining terms ($k_2+l_2\leq2$).

\end{enumerate}

Consequently, one has obtained the following desired estimates:
\begin{equation}\label{Y-34}
\begin{split}
&E_{0,l}(t)\leq \delta^{6-6\varepsilon_0},\qquad E_{1,l}(t)\leq\delta^{44/9-10\varepsilon_0},\qquad E_{2,l}(t)\leq \delta^{8/3-2\varepsilon_0},\\
&E_{3,l}(t)\leq \delta^{1-2\varepsilon_0},\qquad E_{4,l}(t)\leq \delta^{-1-2\varepsilon_0},\qquad\quad E_{5,l}(t)\leq\delta^{-3-2\varepsilon_0},\\
&E_{6,l}(t)\leq \delta^{-5-2\varepsilon_0}.
\end{split}
\end{equation}
Since the constants in \eqref{Y-34} are all independent of $M_0$, the bootstrap
assumptions  \eqref{EA} can be proved as long as \eqref{EA} holds for the time $t=t_0$. By combining  the
local existence of the solution $\phi$
to the problem \eqref{quasi} by a continuous induction argument,  the global existence
of $\phi$ in $B_{2\dl}$ is established.
\end{proof}

We now prove Theorem \ref{main}.
\begin{proof}
By Theorem \ref{Th2.1}, one has obtained the local existence of smooth solution $\phi$ to equation \eqref{quasi} with \eqref{Y1-1}.
On the other hand,  the global existence of $\phi$ in $A_{2\dl}$
and $B_{2\dl}$  are established in Section \ref{YY} and Section \ref{inside} respectively.
Then it follows from the uniqueness of smooth solutions to \eqref{quasi} with \eqref{Y1-1}
that the proof of $\phi\in C^\infty([1,+\infty)\times\mathbb R^4)$ is finished.
In addition, $|\p\phi|\lesssim\delta^{1-\varepsilon_0}t^{-3/2}$ comes from \eqref{Lle}, \eqref{imp} and the first inequality in \eqref{pphi};
$|\phi|\lesssim\delta^{\f76-\varepsilon_0}t^{-1/2}$ is derived from \eqref{Lle}, \eqref{imp} and the first inequality in \eqref{O}. Thus Theorem \ref{main} is proved.
\end{proof}

\end{document}